\documentclass[12pt,reqno]{amsart}

\usepackage[toc, page]{appendix}

\usepackage{amsmath}
\usepackage{amscd}
\usepackage{amsthm}
\usepackage{mathdots}
\usepackage{mathtools}
\usepackage{bm}
\usepackage{amssymb}
\usepackage{comment}
\usepackage{todonotes} 
\usepackage{yhmath}

\makeatletter
\newcommand{\myitem}[1]{%
\item[(#1)]\protected@edef\@currentlabel{#1}%
}
\makeatother

\usepackage[linktocpage=true, backref=page]{hyperref}

\usepackage{mathrsfs}

\usepackage{derivative} 

\usepackage{upgreek} 

\usepackage{microtype}

\usepackage{tikz,tikz-cd, color}
\usepackage[all]{xy}

\usepackage{xcolor}

\newcommand{\llangle}{\langle\hspace{-2.7pt}\langle}
\newcommand{\rrangle}{\rangle\hspace{-2.7pt}\rangle}

\usepackage{wasysym, stackengine, makebox, graphicx}

\usepackage{adjustbox}
\usetikzlibrary{matrix}
\usetikzlibrary{decorations.pathmorphing}

\tikzset{
  symbol/.style={
    draw=none,
    every to/.append style={
      edge node={node [sloped, allow upside down, auto=false]{$#1$}}}
  }
}

\usepackage{stmaryrd}
\usepackage{enumerate}

\usepackage{xcolor}

\usepackage{booktabs}
\usepackage{float}

\mathchardef\mhyphen="2D 

\usepackage{fullpage}

\usepackage{colonequals}

\newtheorem{thm}{Theorem}
\newtheorem*{thm*}{Theorem}

\newtheorem{prop}[thm]{Proposition}
\newtheorem{cor}[thm]{Corollary}
\newtheorem{corollary}[thm]{Corollary}
\newtheorem{thm&defn}[thm]{Theorem \& Definition}

\newtheorem{thr}[thm]{Theorem}
\newtheorem{theorem}[thm]{Theorem}
\newtheorem{lemma}[thm]{Lemma}
\newtheorem{proposition}[thm]{Proposition}

\theoremstyle{definition}

\newtheorem{definition}[thm]{Definition}
\newtheorem*{notation*}{Notation}
\newtheorem{convention}[thm]{Convention}
\newtheorem*{convention*}{Convention}
\newtheorem{example}[thm]{Example}

\newtheorem*{claim*}{Claim}

\theoremstyle{remark}

\newtheorem*{rem*}{Remark}
\newtheorem{remark}[thm]{Remark}
\newtheorem{rmk}[thm]{Remark}

\numberwithin{equation}{section}
\numberwithin{thm}{section}

\newcommand{\bZ}{\mathbb{Z}}
\newcommand{\bN}{\mathbb{N}}
\newcommand{\bQ}{\mathbb{Q}}
\newcommand{\bR}{\mathbb{R}}
\newcommand{\bC}{\mathbb{C}}
\newcommand{\bH}{\mathbb{H}}
\newcommand{\bP}{\mathbb{P}}

\newcommand{\bF}{\mathbb{F}}

\newcommand{\cC}{\mathcal{C}}
\newcommand{\cI}{\mathcal{I}}

\newcommand{\cL}{\mathcal{L}}
\newcommand{\cO}{\mathscr{O}}
\newcommand{\cV}{\mathcal{V}} 
\newcommand{\cW}{\mathcal{W}} 
\newcommand{\cX}{\mathcal{X}}
\newcommand{\cY}{\mathcal{Y}}

\newcommand{\cU}{\mathcal{U}} 
\newcommand{\cZ}{\mathcal{Z}}

\newcommand{\fV}{\mathfrak{V}} 
\newcommand{\fS}{\mathfrak{S}}
\newcommand{\fX}{\mathfrak{X}}

\newcommand{\fg}{\mathfrak{g}}

\newcommand{\rf}{\mathrm{f}} 
\newcommand{\rg}{\mathrm{g}} 

\newcommand{\rj}{\bm{j}} 
\newcommand{\rk}{\bm{k}} 
\newcommand{\rE}{\mathbb{E}} 
\newcommand{\rI}{\mathrm{I}} 
\newcommand{\rX}{\mathrm{X}} 
\DeclareMathOperator{\rL}{\mathrm{L}} 
\newcommand{\rF}{\mathrm{F}} 

\DeclareMathOperator{\Eis}{E}

\newcommand{\T}{T} 

\newcommand{\sI}{\mathsf{I}} 
\newcommand{\sR}{\mathsf{R}} 
\newcommand{\sS}{\mathsf{S}} 

\newcommand{\oM}{\overline{\mathcal{M}}} 
\newcommand{\oY}{\bar{Y}}
\newcommand{\oU}{\bar{U}} 

\newcommand{\obH}{\bH^*} 

\newcommand{\ocU}{\bar{\mathcal{U}}} 
\newcommand{\tcU}{\widetilde{\mathcal{U}}} 

\newcommand{\hcO}{\hat{\cO}} 

\newcommand{\moe}{\mu} 

\newcommand{\PSR}[1]{\llbracket #1\rrbracket}
\newcommand{\Laurent}[1]{(\hspace{-0.25em}(#1)\hspace{-0.25em})}

\newcommand{\Xres}{Y} 
\newcommand{\Xsing}{\bar{Y}} 
\newcommand{\Xsm}{X} 
\newcommand{\crpcon}{\phi} 

\newcommand{\tauq}{\tau_q} 
\newcommand{\trpt}[1]{\tau_{q,#1}^\infty} 

\makeatletter 
\renewcommand*\env@matrix[1][\arraystretch]{%
  \edef\arraystretch{#1}%
  \hskip -\arraycolsep
  \let\@ifnextchar\new@ifnextchar
  \array{*\c@MaxMatrixCols c}}
\makeatother

\newcommand{\mx}[4]{\begin{pmatrix}[0.8]
    #1 & #2 \\
    #3 & #4
\end{pmatrix}}
\newcommand{\smx}[4]{\left(\begin{smallmatrix}
    #1 & #2 \\
    #3 & #4
\end{smallmatrix}\right)}

\newcommand{\Def}{\mathrm{Def}}
\newcommand{\Sing}{\mathrm{Sing}}

\newcommand{\ev}{\mathrm{ev}}
\newcommand{\loc}{\mathrm{loc}}
\newcommand{\vir}{\mathrm{vir}}

\newcommand{\pt}{\mathrm{pt}}
\newcommand{\reg}{\mathrm{reg}}

\newcommand{\chitop}{\chi_\mathrm{top}} 

\newcommand{\typeII}{\Xres \xlongrightarrow{\crpcon} \Xsing \rightsquigarrow \Xsm} 

\newcommand{\Rloc}{\sR_\loc}
\newcommand{\Iloc}{\sI_\loc}
\newcommand{\Xloc}{X_\loc}
\newcommand{\Yloc}{Y_\loc}

\newcommand{\ii}{\mathtt{i}} 

\newcommand{\I}{\mathrm{I}}
\newcommand{\II}{\mathrm{I\!I}}

\DeclareMathOperator{\Proj}{Proj}

\DeclareMathOperator{\Aut}{Aut}
\DeclareMathOperator{\NE}{NE} 
\DeclareMathOperator{\Pic}{Pic} 

\DeclareMathOperator{\Exc}{Ex} 
\DeclareMathOperator{\Spec}{Spec}
\DeclareMathOperator{\Bl}{Bl} 

\newcommand{\ignore}[1]{}

\title{Quantum Extremal Transitions \\ and Special L-values}

\author{Shuang-Yen Lee}
\address[]{Department of Mathematics, National Taiwan University, Taipei 10617, Taiwan}
\email{d10221004@ntu.edu.tw}

\author{Chin-Lung Wang}
\address[]{Department of Mathematics and Taida Institute for Mathematical Sciences (TIMS), National Taiwan University, Taipei 10617, Taiwan}
\email{dragon@math.ntu.edu.tw}

\author{Sz-Sheng Wang}
\address[]{Department of Applied Mathematics, National Yang Ming Chiao Tung University, Hsinchu 30010, Taiwan}
\email{sswangtw@math.nctu.edu.tw}

\date{August 1, 2025}

\begin{document}

\begin{abstract} 

A threefold extremal transition $Y \searrow X$ consists of a crepant extremal contraction $\phi \colon Y \to \bar Y$ with curve class $\ell \in \NE(Y)$, followed by a smoothing $\bar Y\rightsquigarrow X$. We consider the Type II case that $\phi$ contracts a divisor $E$ to a point and prove that the quantum cohomology $QH(X)$ is obtained from $QH(Y)$ via \emph{analytic continuation}, \emph{regularization}, and \emph{specialization} in $Q^\ell$. Besides roots of unity, special $\rL$-values appear in $\lim Q^\ell$ whenever $\bar Y$ admits more than one smoothings. 

Further techniques are employed and explored beyond known tools in Gromov--Witten theory including (i) the \emph{canonical local B model} attached to $Y \searrow X$, (ii) existence of \emph{semistable reduction of double point type} for the smoothing, (iii) the \emph{modularity of the extremal function} $\rE \coloneqq E^3/\langle E, E, E\rangle^Y$, and (iv) \emph{periods integrals of Eisenstein series}. Our study provides a geometric framework linking classifications of del Pezzo surfaces, Ramanujan's theta functions, and Zagier's special ODE list via Type II transitions. 
\end{abstract}

\maketitle

\tableofcontents

\setcounter{section}{-1}

\subsection*{Notation and convention}

\begin{itemize}\itemindent = -1.5em
    \item In this paper the coefficients for cohomology $H^k(M)$ are
    in $\bC$ if not stated otherwise. 
    \item To shorten notation, we set $H(M) \coloneqq H^{\ev} (M)$, the even part of the cohomology groups.
    \item $\NE (M)$ is the integral Mori cone, and $\NE (M)_\bR = \NE (M) \otimes_\bZ \bR$. 
    \item $q^\beta$ is the Novikov variable for $\beta \in \NE (M)$.  
    \item $P^\beta = q^\beta e^{(t, \beta)}$, where $t \in H^2(M)$ is the small parameter of the $I$-function.  
    \item $Q^\beta = q^\beta e^{(\tau, \beta)}$, where $\tau \in H^2(M)$ is the small parameter of the $J$-function. 
    \item $\hat{D}$ is the quantization operator of a divisor $D$, defined by $\hat{D} P^\beta = z(D, \beta) P^\beta$.  
    \item $\delta^D$ is the power operator of $D$, defined by $\delta^D Q^\beta = (D, \beta) Q^\beta$.
    \item $\theta = P^\ell \frac{\partial}{\partial P^\ell}$ where $\ell$ is the extremal curve class for the contraction $\crpcon \colon \Xres \to \Xsing$.  
    \item $q = e^{2\pi \ii \tauq}$ where $\tauq \in \bH$.  
    \item $[z^{-k}]F$ denotes the coefficient of $z^{-k}$ in $F$. 
    \item $\bm{1}_N$ denotes the trivial Dirichlet character modulo $N$. 
    \item $\omega_N = e^{2\pi \ii/N}$ is the $N$-th root of unity. 
    \item $\fS_n$ denotes the symmetric group of degree $n$.  
    \item $R\PSR{x}$ is the formal power series ring with coefficients in $R$. 
    \item $R\Laurent{x}$ is the formal Laurent series ring with coefficients in $R$. 
\end{itemize}

\section{Introduction}

We study the relation of (big) quantum cohomologies of smooth projective threefolds $\Xres$ and $\Xsm$ when they are related by an extremal transition 
\[
    \Xres \searrow \Xsm \quad \text{via} \quad \Xres \xlongrightarrow{\crpcon} \Xsing \rightsquigarrow \Xsm,
\]
where $\crpcon \colon \Xres \to \Xsing$ is a primitive (birational) crepant contraction (so $\rho(\Xres / \Xsing) = 1$) and $\Xsing \rightsquigarrow \Xsm$ is a smoothing $\fX \to \Delta$ of $\Xsing = \fX_0$ with the fiber $\Xsm = \fX_t$ for some $t \neq 0$. 

Let $\ell \in \NE(\Xres)$ be the extremal curve class for $\crpcon$, the contraction of $ \bR_{\geq 0} \ell$. Following Wilson \cite{Wilson92}, $\crpcon$ is of Type I if it contracts finitely many curves to points $p_i \in \bar Y$, of Type II if it contracts a divisor to a point, and of Type III if it contracts a divisor to a curve $C \subseteq \bar Y$. We thus call the transition $Y \searrow X$ being of Type I, II or III accordingly. 

It is well known that for Type I transitions a symplectic deformations of it leads to a conifold transition (all $p_i$'s are ordinary double points) where the question was fully understood by Li--Ruan \cite{LR01} around 2000. In the Calabi--Yau case it states that 
\begin{equation} \label{e:TypeI}
    \sum_{\beta \mapsto \bar{\beta}} \langle -  \rangle^{\Xres}_\beta = \langle -\rangle^{\Xsm}_{\bar{\beta}}.
\end{equation} 
The sum of Gromov--Witten (GW) invariants is in fact finite and the RHS can be understood by setting $q^\ell = 1$ in the Novikov variables. Moreover the notion of linked GW invariants was developed in \cite{LLW18} in order to read out the individual terms on $\Xres$ from those on $\Xsm$.  

For Type III transitions, Wilson had shown that $C$ is a smooth curve \cite{Wilson92, Wilson97} and if $g(C) \ne 0$ then one may deform the primitive contraction $\phi$ to a Type I case with $2g(C) - 2$ singular fibers in an appropriate length counting (cf.~\cite[Proposition 4.2]{Wilson92}). Since GW invariants are invariant under deformations, it is not hard to see that \eqref{e:TypeI} is still valid in this case.  

Thus the remaining cases are Type II transitions and Type III with $g(C) = 0$. It is our goal to solve both cases. Due to the difference of techniques involved, we will solve the Type II case in the current paper and defer the Type III genus zero case to a subsequent work.  

\subsection{Statement of main results}

For a Type II transition, $\phi$ contracts a (possibly non-normal) del Pezzo surface $E \subseteq Y$ to a rational Gorenstein singularity $p \in \bar Y$. 
We call $d = E^3$ the degree of the transition. Based on earlier works of Gross and Wilson \cite{Gross97a, Wilson97}, we proved that there is also a symplectic deformation on $Y \searrow X$ so that $E$ is deformed to a smooth del Pezzo surface $S_d$ of degree $d$. 
Combining Fujita--Iskovskikh's classification theory with Reid's classification of isolated Gorenstein singularities \cite{Reid80}, we further construct a projective local model of the transition. We call it a del Pezzo transition $\Xres_d \to \Xsing_d \rightsquigarrow \Xsm_d$ where $X_d$ is a del Pezzo threefold of degree $d$ and $Y_d = \bP_{S_d}(K_{S_d} \oplus \cO)$. Since $\Xsing_d$ is smoothable, the surface $S_d$ is not isomorphic to $\bP^2$ or the Hirzebruch surface $\bF_1$. In particular, we have $1 \le d \le 8$. Up to deformations, the smoothing is unique for $d \ne 6$. For $d = 6$, there are two smoothings of $\bar Y_6$ named by $X_{6\I}$ and $X_{6\II}$.

We are able to construct semistable reductions on both the \emph{K\"ahler degeneration} $\phi\colon Y \to \bar Y$ and the \emph{complex degeneration} $\fX \to \Delta$ of $\Xsing = \fX_0$ so that we may try to prove a statement on GW invariants first in the local setting $Y_d \searrow X_d$ and then for any $Y \searrow X$ by the degeneration formula of GW theory. Furthermore, again by the classification theory of $S_d$ and $X_d$, we find directly the degree of the finite base change needed for the semistable reduction of $\fX \to \Delta$, avoiding toroidal embeddings. 

We may state our main results informally as the following: we provide analytic continuations of quantum cohomology $QH(Y)$ in $Q^\ell$ under Type II extremal transitions $Y \searrow X$ of threefolds and recover $QH(X)$ after regularization and specialization of $Q^\ell$ to certain value $Q^\ell_r$. 
Exotic behavior of $Q^\ell_r$ is discovered: under the contraction $\phi$ we have $\phi_*(\ell) = 0$, however $Q^\ell_r$ might not even be a root of unity. It could be a special $\rL$-value!

Now the formal description: a canonical local B model, i.e., a pair of $I$-functions $(I^{Y_d}, I^{X_d})$ with a distinguished variable $P^\ell$ mirror to the quantum variable $Q^\ell$ is constructed. It allows us to study analytic continuations of $QH(Y_d)$ via Picard--Fuchs equations and asymptotic analysis near $P^\ell = \infty$. Namely, $QH(X_d)$ is obtained from $QH(Y_d)$ by a regularized holomorphic limit as $P^\ell \to \infty$. In fact, this construction can be globalized through the degeneration formula: 

\begin{theorem}[= Corollary~\ref{cor:GWlimitd=12347} + Corollary~\ref{cor:GWlimitd=568}]\label{t:reg-th}
Let $n \ge 0$. For any insertions $\vec{b} = b_1 \otimes \dots \otimes b_n \in H(X)^{\otimes n}$ and $0 \neq \bar{\beta} \in \NE(X)$, we consider the generating function 
\[
    \langle \phi^*\vec{b} \rangle_{\bar{\beta}}^{\Xres \searrow \Xsm} \coloneqq \sum_{\beta\mapsto \bar{\beta}} \langle \phi^*\vec{b}\rangle_\beta^Y Q^\beta. 
\]
Then we have
\begin{equation}
\lim_{P^\ell \to \infty} \Bigl(\langle \phi^*\vec{b} \rangle_{\bar{\beta}}^{\Xres \searrow \Xsm} Q^{-\tilde{\beta}}\Bigr)^{\reg} = \langle \vec{b}\rangle_{\bar{\beta}}^X
\end{equation}
where $\tilde{\beta}$ is a lifting of $\bar{\beta}$ such that $(E, \tilde{\beta}) = 0$. Moreover, the regularization is superfluous precisely when $d \in \{1, 2, 3, 4, 7\}$. 
\end{theorem}

Here we give a rough definition of regularization $(-)^\reg$. We only need to treat the cases $d \in \{5,6\I,6\II,8\}$. For the remaining values of $d$, the limit can be taken as $P^\ell \to \infty$ without applying $(-)^\reg$. For notational convenience set $y = 1 / P^\ell$. Via the mirror transform, the quantum variable $Q^\ell$ can be considered as a function of $P^\ell$, and hence of $y$. Under analytic continuation along any given path to $y = 0$, one finds that 
\[
    \langle \phi^*\vec{b} \rangle_{\bar{\beta}}^{\Xres \searrow \Xsm} Q^{-\tilde{\beta}} \in \bC \Laurent{y}\PSR{(\log y)^{-1}}. 
\]
Define the regularization of $\langle \phi^*\vec{b} \rangle_{\bar{\beta}}^{\Xres \searrow \Xsm} Q^{-\tilde{\beta}}$ to be its image under the projection 
\[
    (-)^\reg\colon \bC\Laurent{y}\PSR{(\log y)^{-1}} \longrightarrow \bC\Laurent{y}, 
\]
i.e., by taking the constant term in $(\log y)^{-1}$. 

We remark that, to prove Theorem \ref{t:reg-th}, a more refined definition of the regularization $(-)^\reg$ is required (Definition \ref{df:regmap}). This refinement is essential for carrying out the induction. We will infer that the regularization of $\langle \phi^*\vec{b} \rangle_{\bar{\beta}}^{\Xres \searrow \Xsm} Q^{-\tilde{\beta}}$ belongs to $\bC [\![y]\!]$ and thus $\langle \vec{b}\rangle_{\bar{\beta}}^X$ is the constant term of it.

As a corollary of Theorem~\ref{t:reg-th}, the big quantum cohomology $QH(X)$ is a sub-quotient of $QH(Y)$ (see Corollary~\ref{cor:QHsubquot}). 

The limit of $Q^\ell$ is more subtle and requires tools from modular forms. We show that (see Table~\ref{tab:GammadSingfd} for the congruence subgroups $\Gamma_d$):

\begin{theorem}[= Proposition~\ref{p:MZ} + Theorem~\ref{thr:mainmodularver}] \label{t:modver}
For $d \ne 7$, there exists a natural identification of $\mathbb{P}^1_{P^\ell}$ as a modular curve $\rX(\Gamma_d)$ (a double cover of $\rX(\Gamma(1))$ if $d = 1$). 

The coordinate change is expressed via the extremal function 
\begin{equation}
\rE^{Y \searrow X} = \rE_d \coloneqq \frac{E^3}{\sum \langle E, E, E\rangle_{m\ell}^Y Q^{m\ell}} = \odv{\log Q^\ell}{\log q},
\end{equation} 
where $q = e^{2\pi \ii \tauq}$ and $Q^\ell = - q + O(q^2)$. It is an Eisenstein series of weight $3$ (with root and monodromy if $d \le 2$). 

Moreover $\langle \phi^*\vec{b} \rangle_{\bar{\beta}}^{\Xres \searrow \Xsm}Q^{-\tilde{\beta}}$ is a meromorphic quasi-modular form of weight $0$ with value $\langle \vec{b}\rangle_{\bar{\beta}}^X$ at a special cusp or elliptic point, depending on the monodromy of $\rE^{Y \searrow X}$ being logarithmic or finite. 
\end{theorem}

We call the special cusp or elliptic point $[\tauq]$ corresponding to $P^\ell = \infty$ \emph{the transition point} (Definition~\ref{def;tran_pt}). 

The corresponding limiting values of $Q^\ell$ are necessarily path dependent.\footnote{The modularity for $d = 7$ has a different nature which will be discussed elsewhere. }

\begin{theorem}[= Proposition~\ref{prop:Qlimd=1234} + Theorem~\ref{thr:limQabs} \eqref{thr:limQabs_1}] \label{t:Qlim-1}
For $d \neq 6$, where the smoothing is unique up to deformations, we have $Q^\ell \to 1$ or a root of unity. 

Moreover, for $d \in \{1, 2, 3, 4, 7\}$, $QH(Y)$ can be analytically continued along $P^\ell \in [0, \infty)$ and $Q^\ell \to 1$ as $P^\ell \to \infty$. 

\end{theorem}

This is of course expected since the curve class $\ell$ is contracted. However, when $\bar Y$ has more than one smoothings, namely $d = 6$, we show that: 

\begin{theorem}[= Theorem~\ref{thr:limQabs} \eqref{thr:limQabs_2}] \label{t:Qlim-2}
The limit of $Q^\ell$ in case $6\I$ is still a root of unity while in case $6\II$ is given by 
\begin{equation}
Q_r^\ell = \omega\cdot e^{-c \rL'(-1, \chi_{3,2})} 
\end{equation}
for some root of unity $\omega$, where the cusp $r = \frac{a}{c} \in \mathbb{Q}$ with $\gcd(a, c) = 1$ and $c \equiv 2 \pmod{6}$, and $\rL$ is a Dirichlet $\rL$-function. In particular, $Q_r^\ell$ is not a root of unity.
\end{theorem}

To our knowledge, this is the first appearance of special $\rL$-values as specialization points in the study of Gromov--Witten theory. Our proof relies on detailed analysis on period integrals of Eisenstein series. 

Some of the results extend to certain higher dimensional transitions which contract a divisor to a singular point. Also the period integral of Eisenstein series can be carried out for $q$ approaching any of the elliptic points or cusps. These extensions will be contained in the first author's 2026 PhD thesis.  

\subsection{Outline of the paper}

In this subsection we discuss the main ideas behind the proofs of our main results. The details are referred to the corresponding locations in this paper.

\subsubsection{Reduction to local models}


We will construct semistable reductions of both the K\"ahler and complex degenerations for a Type II extremal transition $\Xres \xlongrightarrow{\crpcon} \Xsing \rightsquigarrow \Xsm$ of degree $d$ in \S\ref{subsec;ss_deg}. They must be of double point type in order for the GW degeneration formula to be applicable. It is crucial to control the base change degree $n_d$ for complex degeneration $\fX \to \Delta$ so that we may first perform the base change, then perform one weighted blow-up. This can be done when $E = \Exc (\crpcon)$ is smooth (Proposition \ref{prop;locDef_sm}). Note that we have $n_d = 6, 4, 3$ for $d = 1,2,3$, respectively, and $n_d = 2$ for $d \geq 4$ by using the classification theory of del Pezzo surfaces and threefolds (Appendix \ref{sec;appdP}). In general, by deforming the complex structure on an open neighbourhood of $E$, we can deform $E$ to smooth del Pezzo surface $E'$ (Proposition \ref{prop;locDef_res}). Gluing complex structures in a $C^\infty$ way, we conclude that the (symplectic) semistable model is of double point type (Theorem \ref{thm;ssdeg}):
\[
    \cX_0 = Y' \cup X_d
\] 
where $X_d$ is a smooth del Pezzo threefold of degree $d$ and $Y' \cap X_d = E'$. Moreover, $Y$ and $Y'$ are symplectic deformation equivalent of each other. We emphasize that the result seems too good to be true at first sight! 

On the other hand, for K\"ahler degeneration the semistable reduction is simply the deformations to the normal cone which is naturally of double point type with the central fiber
\[
    \cY_0 = \Xres' \cup Y_d
\]
and $\Xres' \cap Y_d = E'$. Therefore we reduce the problem to calculating GW invariants on the local model $Y_d \searrow X_d$.

\subsubsection{Canonical local B models}

There are classifications of smooth del Pezzo surfaces (Theorem \ref{thm;dP2}) according to its degree $d$ as well as of smoothable isolated rational Gorenstein threefold singularities (by Gross, Namikawa and Reid \cite{Gross97a,Namikawa97,Reid80}). It turns out to admit a perfect match on both and we get for each $1 \le d \le 8$, $E = S_d$ (smooth del Pezzo surface of degree $d$), a local model of Type II extremal transition (Example \ref{ex:dPtrans}):
\[
    Y_d = \bP_{S_d}(K_{S_d} \oplus \cO) \searrow X_d.
\] 

The smoothing family is unique if $d \ne 6$. For $d = 6$ there are two smoothings denoted by cases $d = 6\I$ and $d = 6\II$. So we let $d \in \{1, 2, 3, 4, 5, 6\I, 6\II, 7, 8\}$. 

The classification of del Pezzo threefolds (Theorem \ref{thm;dP3claf}) allows us to find a smooth polarized variety $(\T_d, \cO_{\T_d}(1))$ with known $QH(\T_d)$ and a convex vector bundle $\cV_d$ on $\T_d$ such that $X_d = Z(s_d) \subseteq \T_d$ for $s_d \in \Gamma(T_d, \cV_d)$, and $\cO_{X_d} (1) = \cO_{T_d} (1)|_{X_d}$, see Table \ref{tab;ambXd}. 

All these geometric data fit into the following basic diagram
\[\begin{tikzcd}
    Y_d \ar[d, "\pi_d"] & &\cV_d \ar[d] \\
    S_d \ar[r, hook, "i_d"] \ar[u, "j_d", bend left = 30] & X_d \ar[r, hook] & \T_d\ar[u, "s_d", bend left = 30]
\end{tikzcd}\]
which allows us to setup a \emph{canonical local B model} of $Y_d \searrow X_d$, i.e.~a pair of closely related $I$-functions $I^{Y_d}$ and $I^{X_d}$, to compare GW invariants of $Y_d$ and $X_d$. 

\subsubsection{Picard-Fuchs equations}

The annihilator of the $I$-function is the Picard--Fuchs ideal $\operatorname{PF}$ generated by equations $\square_\beta$ over curve classes $\beta$. For $d \neq 7$, we show that $\operatorname{PF}^{X_d}$ is generated by $\square_{\bar{\gamma}} I^{X_d} = 0$ with
$$
  \square_{\bar{\gamma}} = \hat{h}^4 - P^{\bar{\gamma}}(\kappa \hat{h}^2 + \kappa z\hat{h} + \lambda z^2) - \mu P^{2\bar{\gamma}},
$$
and $\operatorname{PF}^{Y_d}$ is generated by $\square_\gamma I^{Y_d} = \square_\ell I^{Y_d} = 0$ with 
    \begin{align*}
        \square_\gamma &= \hat{E}\hat{H} - P^\gamma, \\
        \square_\ell &= \hat{F}^3 - P^\ell (\kappa \hat{F}^2 + \kappa z\hat{F} + \lambda z^2) \hat{E} - \mu P^{2\ell} (\hat{F} + z) \hat{E} (\hat{E} - z). 
    \end{align*}
(See Proposition \ref{prop:PFI}.) Here the constants $(\kappa, \lambda, \mu) = (\kappa_d, \lambda_d, \mu_d)$ are listed below:\footnote{Technically speaking, the contraction $Y_d \to \bar Y_d$ has $\rho(Y_d/\bar Y_d) > 1$ and we need to perform identifications to get the single equation $\square_\ell$. 
For $d \ne 7$ we symmetrize the Picard--Fuchs equations and GW theory to achieve this in \S\ref{subsec;GW_neq_7}, with the group $G_{S_d}$ being the Weyl group of the root system $E_{9 - d}$ in Table \ref{tab;root_sys}. However, for $d = 7$ the symmetrization is not working and we take an ad hoc approach instead (see \S\ref{subsec;GW=7}).}
    \begin{table}[H]
        \centering
        \begin{tabular}{cccccccccc}
            \toprule
            $d$ & $1$ & $2$ & $3$ & $4$ & $5$ & $6 \I$ & $6 \II$ & $8$ \\
            \midrule
            $\kappa_d$ & $432$ & $64$ & $27$ & $16$ & $11$ & $7$ & $10$ & $0$ \\ 
            \midrule
            $\lambda_d$ & $60$ & $12$ & $6$ & $4$ & $3$ & $2$ & $3$ & $0$ \\ 
            \midrule
            $\mu_d$ & $0$ & $0$ & $0$ & $0$ & $1$ & $8$ & $-9$ & $16$ \\
            \bottomrule
        \end{tabular}
        \caption{Coefficients in $\square_\ell$.}
        \label{table:klmd}
    \end{table}

To simplify notations, the subscript $d$ will often be omitted throughout this section, expect $X_d$ and $Y_d$. Since $X_d$ is Fano, we have $J^{X_d} = I^{X_d}$. For $J^{Y_d}$, we showed that the mirror transform $\tau = t - g(P) E$ depends only on the canonical variables $P = P^\ell$ and $Q = Q^\ell$ via
\begin{equation} \label{e:QP}
    Q = P e^{g(P)},
\end{equation} 
where $f \coloneqq 1 + \theta g$ is uniquely determined by the equation
\begin{equation} \label{e:Eqf}
u\, \theta^2 f + \theta u\, \theta f + (\lambda P - \mu P^2)f = 0
\end{equation} 
with $u = u(P) \coloneqq 1 + \kappa P - \mu P^2$, $\theta = P \odv{}{P}$ (Lemma~\ref{lm:mt}). 

In principle, knowing $f$ near $P = 0$ determines all genus zero GW invariants of $Y_d$ by mirror theorem and reconstructions. The key player to make the reconstruction manageable is the $3$-point function $\langle E, E, E\rangle^{Y_d}$ which is determined by \emph{the extremal function} in $P$ via
\begin{equation} \label{e:Ed}
\rE \coloneqq \frac{E^3}{\langle E, E, E\rangle^{Y_d}} = u f^3
\end{equation}
(cf.\ Definition \ref{def;extr_fun}).

\subsubsection{Regularization}

To actually compare GW invariants on $Y_d$ with GW invariants of $X_d$ we have to take advantage of the canonical local B model and the comparison turns out has to take place at $P = \infty$. A computation shows that for $1$-point and $2$-point insertions $\vec{b}$ in $H(X_d)$, the invariant $\langle \phi^*\vec{b}\rangle^{Y_d}$ lies in $\bC + O(P^{-1})$ precisely when $\mu = 0$. However, when $\mu \ne 0$ (i.e., $d \in \{5, 6\I, 6\II, 8\}$), $\langle \phi^*\vec{b}\rangle^{Y_d}$ might contain terms like $v \coloneqq \frac{\theta f}{f} = -1 + O(\frac{1}{\log P})$ near $P = \infty$ and a holomorphic regularization needs to be introduced in order to perform specializations. More precisely, the local exponents of \eqref{e:Eqf} at $y \coloneqq P^{-1}= 0$ is $(1, 1)$, hence $f \in y\,\bC\PSR{y} \oplus y\log y \,\bC\PSR{y}$. Let $f^\reg \in y\,\bC\PSR{y}$ be the non-zero \emph{monodromy invariant} solution (unique up to a scalar) of \eqref{e:Eqf} then it is clear that
\[
v^\reg \coloneqq \frac{\theta f^\reg}{f^\reg} = -1 + O(y)
\] 
and the regularization we have in mind is basically the replacement of $v$ by $v^\reg$. The point is that we need to make it compatible with reconstructions in GW theory. For this purpose, let $\sS = \bC (y)[v]$ and $w = v - v^\reg$. Consider the subring $\sR \subseteq \sS$ and the ideal $\sI$ of $\sR$ (cf.~\S \ref{subsubsec;log-derring_neq_7}):
$$
\sR = \sS\cap (\bC\PSR{y} + w\cdot \bC\Laurent{y}[w]), \qquad \sI = \sS\cap (y\cdot \bC\PSR{y} + w \cdot \bC\Laurent{y}[w]).
$$

We first prove that for certain initial 1-point and 2-point insertions $\vec{b}$ on $X_d$, the invariant $\langle \phi^*\vec{b}\rangle^{Y_d}$ belongs to $\sR[Q^{\tilde{\gamma}}]$ where $\tilde{\gamma}$ is the canonical lifting of $\gamma$ in $Y_d$ and 
\begin{equation*}
    \langle \phi^*\vec{b}\rangle^{Y_d} - \phi^*\langle \vec{b} \rangle^{X_d} \in \sI[Q^{\tilde{\gamma}}].
\end{equation*}
Once this is achieved, by working out the delicate comparison of the two reconstruction procedure on both $Y_d$ and $X_d$, we discovered a mysterious cancellation via $\rE$ to get rid off logarithmic singularities and the procedure can be carried over (Proposition \ref{proplocalmodI}). This eventually leads to the proof of Theorem \ref{t:reg-th} in the local case $Y_d \searrow X_d$. (For $d = 7$ see \S\ref{subsec;GW=7}.) Now by applying (symplectic) degeneration formula of GW invariants the global case $Y \searrow X$ is reduced to the local models. This completes the outline of proof of Theorem \ref{t:reg-th}. \footnote{Nevertheless, logically we need to firstly figure out the correct regularization statement before we can run any induction on reconstruction or degeneration analysis.}

It is also worth mentioning that, after the change of coordinates $y = P^{-1}$, the Picard--Fuchs ideal $\operatorname{PF}^Y$ has an extension $\widetilde{\operatorname{PF}}{}^Y$ over $y = 0$. However, the plausible restriction identity $\widetilde{\operatorname{PF}}{}^Y|_{y = 0} = \operatorname{PF}^X$ holds only when $\mu = 0$ (Remark \ref{l:fail}). This also explains the necessity of regularization.

\subsubsection{Riemann surfaces  and quasi-modular forms}

While the regularization/specialization formula in Theorem \ref{t:reg-th} is accurate, it will be satisfactory only if we understand the limit of the intrinsic quantum variables $Q^\beta$ on $Y$ and determine the Riemann surface (\emph{natural domain of definition}) of $\rE$ and of $QH(Y)$. 

To study the problem, we proceed to show that $f = f_d$ and $\rE = \rE_d$ are $\Gamma_d$ modular forms in $q = e^{2\pi \ii \tauq}$ of weight one and three respectively, under the change of variable 
\begin{equation} \label{e:q-Q}
\odv{\log P}{\log q} = uf^2 \qquad \text{or equivalently} \qquad \odv{\log Q}{\log q} = \rE = uf^3
\end{equation}
with $q = -P + O(P^2) = -Q + O(Q^2)$. For $f$ the result is essentially due to Maier \cite{Maier09} and Zagier \cite{Zag09} where $f(P(\tauq))$ is known as Ramanujan's theta functions (cf.~Cooper's book \cite[\S4--\S6]{Coo17}). 

For example for $d = 1$, near $q = 0$ we have $P(\tauq) = (\Eis_6(\tauq)/\Eis_4^{3/2}(\tauq) - 1)/864$ and
\[
f(P(\tauq)) = {}_2\rF_1\bigl(\tfrac16, \tfrac56; 1; -432P(\tauq)\bigr) = \Eis_4(\tauq)^{1/4}. 
\]
Since $j(\tauq)^{-1} = -P(\tauq)(1 + 432P(\tauq))$, the multi-valued function $P$ parametrizes a double cover of $\rX(\Gamma(1))$. For $d \ge 2$ and $d \ne 7$, $P$ indeed parametrize $\rX(\Gamma_d)$ and similar expressions are listed in Proposition \ref{p:MZ}. Since $P$ and hence $u$ is a weight $0$ meromorphic modular function and $f$ is weight $1$, we see that $\rE = u f^3$ is a weight $3$ modular form. 

To test modularity of GW $n$-point functions, the reconstructions makes use of differentiations which generally destroys the genuine modularity. This leads us to invoke the notion of \emph{meromorphic quasi-modular forms} instead (see Definition \ref{d:QMF}).

Let $D_q = \odv{}{\log q} = q \odv{}{q}$. It follows from~\eqref{e:q-Q} that
\[
    v = \frac{\theta f}{f} = \frac{D_q f}{\rE},
\]
which is a quasi-modular of weight $0$ and depth $1$ (with root and monodromy if $d = 1$). This implies that every element in $\sS = \bC(y)[v]$ is quasi-modular of weight $0$. 

For $d \in \{4, 5, 6\I, 6\II, 8\}$, the transition point is a cusp. Take a representative $r \in \bQ$ and an element $A = \smx {a}{*}{c}{*} \in \operatorname{SL}(2, \bZ)$ such that $A\cdot \ii\infty = r$. A brief computation leads to 
\[
v^\reg = \frac{\theta f^\reg}{f^\reg} =  \rE^{-1} \cdot \Big(D_q f -\frac{1}{2\pi \ii}\frac{c}{(a - c \tauq)}\,f \Big)
\]
as well as $v^\reg(A\tauq) = v_0^A$. By taking into account the regularization procedure, it implies that given $0 \ne \bar{\beta} \in \NE(X)$ with lifting $\tilde{\beta} \in \NE(Y)$ such that $(E, \tilde{\beta}) = 0$,
\[\langle \phi^*\vec{b} \rangle_{\bar{\beta}}^{\Xres \searrow \Xsm}Q^{-\tilde{\beta}} = \sum_m \langle \phi^*\vec{b}\rangle_{\tilde{\beta} + m\ell} Q^{m\ell} 
\]
is a meromorphic quasi-modular form of weight $0$ which coincides with $\langle \vec{b}\rangle_{\bar{\beta}}^X$ at the transition point. This completes the outline of proof of Theorem \ref{t:modver} (cf.~Theorem~\ref{thr:mainmodularver}). 

Consequently, we see that the GW invariants on $Y$ in the extremal direction, which are polynomials in $v$, are quasi-modular forms of weight $0$ on $\Gamma_d$. Since non-constant quasi-modular forms all have logarithmic singularities near the real line, it follows that the unit disk $\mathbb{D} = \{q \mid |q| < 1\}$ is exactly the Riemann surface we want (cf.~Remark~\ref{rmk:mdisD}).

\subsubsection{Periods integrals of Eisenstein series}

As a byproduct, we can compute the limit of $Q^\ell$ by integrating $\rE_d$ along some path in $\mathbb{D}$ ending at a representative of the transition point. It turns out that $\rE_d$ is an Eisenstein series, and the integration can be expressed explicitly as a combination of $\rL$-values. 

To start with, $f$ admits an analytic continuation to a multi-valued function on $\bP^1 \setminus \operatorname{Sing}_{f}$ ($\operatorname{Sing}_{f}$ is contained in the set of singular points of \eqref{e:Eqf}). For any path $p\colon [0,1] \to \bP^1$ with $p(0) = 0$, $p(1) = \infty$, and $p(s) \in \bP^1 \setminus \operatorname{Sing}_{f}$ for $s \in (0, 1)$, the limit of $Q^\ell$ along the path $p$ is 
\[
Q_p^\ell \coloneqq \lim_{s \to 1^-} Q^\ell(s) = \lim_{\varepsilon \to 0^+} p(\varepsilon)\exp\Bigl( \int_\varepsilon^1 f(p(s)) \frac{\odif{p(s)}}{p(s)}\Bigr). 
\]

For $d \in \{1, 2, 3, 4\}$, we may tale the path $p$ to be the line $[0, \infty)$ and the integral can be evaluated by techniques in complex function theory. It is based on the fact that
\[
f(P) = {}_2\rF_1\bigl(\tfrac{1}{n}, 1 - \tfrac{1}{n}; 1; -\kappa P \bigr),
\]
where $n = 6, 4, 3, 2$ for $d = 1, 2, 3, 4$ ($n = n_d$ is precisely the degree of base change needed in Proposition \ref{prop;locDef_sm}), and Euler's integral representation of hypergeometric functions. The result is $Q_p^\ell = 1$. The case $d = 7$ is treated separately, and it is in fact easier.   
Thus the first half of Theorem \ref{t:Qlim-1} is essentially classical (see Proposition~\ref{prop:Qlimd=1234}).

For the second half of Theorem \ref{t:Qlim-1}, as well as for Theorem \ref{t:Qlim-2}, we develop period integrals of Eisenstein series in \S \ref{s:wild} to evaluate the limit $Q^\ell$ to a certain value $Q^\ell_r$ (see Definition \ref{df:Q_r^l}). The first ingredient is the following extension of a similar formula in \cite{ABYZ02} for the case $\chi = \bm{1}_1$, using the method in \cite{Yang04} (see Proposition \ref{prop:Yang04ext}). Let $k\ge 2$ be an integer. For any triple $(\chi, \psi, n)$ with $\chi\psi(-1) = (-1)^k$, define
\[
\rI_k^{\chi, \psi, n} = n^{-1}\lim_{s \to 0} \rL(s+1, \chi)\rL(s+2-k, \psi). 
\]
Let $a_{\ii\infty,0}$ and $a_{0,0}$ be the cusp values of $\Eis_k^{\chi, \psi, n}(\tauq)$ at $\tauq = \ii\infty$ and $\tauq = 0$, respectively. Then 
\[
\int_{\ii\infty}^0 (\Eis_k^{\chi, \psi, n}(\tauq) - a_{\ii\infty,0} - a_{0,0}\tauq^{-k}) \odif{(2\pi \ii\tauq)} = \rI_k^{\chi, \psi, n}. 
\] 

The second ingredient is the translation formula of Eisenstein series we proved in Proposition \ref{prop:TFforEis}: Let $r = \frac ac \in \bQ$ with $\gcd (a, c) = 1$. Let $\Eis_r^{\psi}(\tauq) = \Eis_3^{\bm{1}_1,\psi,1}(\tauq + r)$. Then 
\[
\Eis_r^{\psi} = \sum_{c = c_1c_2c_3}  \sum_{\chi \in \widehat{(\bZ/c_3\bZ)^\times}} \frac{\chi(a)\fg(\bar{\chi})\psi(c_1)c_1^2}{\varphi(c_3)}\cdot  \Eis_3^{\chi, \bm{1}_{c_2c_3}\chi \psi,c_1c_2}, 
\]
where $\fg$ is the Gauss sum and $\varphi$ is Euler's totient function. It is then clear that the period integrals of Eisenstein series are reduced to sums over arithmetic functions. 

We conclude this paper by working out some explicit examples and to conclude the proofs of both Theorem \ref{t:Qlim-1} and Theorem \ref{t:Qlim-2}, We also include a complete table for the limit of $Q^\ell$ at any $r = \frac{a}{c}\notin [\ii \infty]$ in Remark \ref{r:complete}.





\subsection{Acknowledgments}

The authors are grateful to Jie Zhou for his valuable suggestions, to Yifan Yang for his suggestions on modular forms, and to Chiu-Chu Melissa Liu for her interest in this work. C.-L.~is supported by NSTC, Taiwan with grant number NSTC 113-2115-M-002-004-MY3, and by the Core Research Group of National Taiwan University. S.-Y.~is supported by the PhD fund in C.-L.'s NSTC grant above. S.-S.~is supported by NSTC, Taiwan with grant number NSTC 111-2115-M-A49-019-MY3 and 114-2115-M-A49-007-MY3.


\section{Geometry of Type II extremal transitions}\label{sec;geo_of_tran}

\subsection{Extremal rays and Type II contractions}

We recall some basic terminology from Mori theory, see \cite{KM98} for a complete treatment.

Let $\Xres$ be a smooth projective threefold. A \emph{birational contraction} on $Y$ is a projective birational morphism $\crpcon \colon \Xres \to \Xsing$ onto a normal variety $\Xsing$ such that $\crpcon_\ast \cO_Y = \cO_{\oY}$. We are interested in the case that $\crpcon$ is \emph{divisorial}, i.e, the exceptional set $\Exc (\crpcon)$ has codimension one. 

Let $\NE(Y)$ be the \emph{integral Mori cone} of effective one cycles, let $\NE(Y)_\bR = \NE (Y) \otimes_\bZ \bR$. 
A divisorial contraction $\crpcon$ is called \emph{extremal} in the sense of Mori theory if for every irreducible curve $C$ in $Y$, we have $\crpcon_\ast [C] = 0$ if and only if $[C]$ lies in some fixed ray $\bR_{\geq 0} \ell$ in $\NE (\Xres)_\bR$. 
Such ray $\bR_{\geq 0} \ell$ is an \emph{extremal ray} of $\NE (\Xres)_\bR$. It is called \emph{crepant} if the canonical divisor $K_\Xres$ is $\crpcon$-trivial.

We recall the following terminology from \cite[Definition 2.1]{Wilson92}.

\begin{definition}
A divisorial extremal contraction $\crpcon$ is called primitive of Type II if it is crepant and $\dim \crpcon (\Exc (\crpcon)) = 0$.
\end{definition}

\begin{remark}
Notice that $E = \Exc (\crpcon)$ is irreducible and therefore contracted to a point \cite[Proposition 2.5]{KM98}.

In Wilson's terminology, a birational contraction $\crpcon$ is \emph{primitive} if it cannot be factored in the algebraic category \cite[Definition 1.2]{Wilson92}. As a standard argument shows (cf.\ \cite[Theorem 3.25]{KM98}), $\crpcon$ is primitive if and only if it is a contraction of an extremal ray corresponding to $K_\Xres + E$. Wilson classified a primitive birational contraction according to the dimensions of its exceptional
set and its image. \emph{Types I and III contractions} are small contractions and divisorial contractions to curves, respectively \cite[Definition 2.1]{Wilson92}.
\end{remark}

Recall that a \emph{generalized del Pezzo surface} is a projective Gorenstein surface with anti-ample dualizing sheaf (Remark \ref{rmk;GdP2}).

\begin{proposition}[\cite{Gross97a, Wilson92, Wilson97}] \label{prop:birExc}
Let $\crpcon \colon \Xres \to \Xsing$ be a primitive Type II contraction. Then $\Xsing$ is $\bQ$-factorial and Gorenstein, and the exceptional divisor $E$ is a generalized del Pezzo surface. Moreover, if $d \coloneqq E^3 > 3$ then $E$ is either
\begin{enumerate}
    \item a normal and rational del Pezzo surface of degree $d \leqslant 9$, or
    \item a non-normal del Pezzo surface of degree $d = 7$ whose normalization is a smooth rational ruled surface.
\end{enumerate}
\end{proposition}

\begin{proof}
Since $\crpcon$ is a crepant contraction of an extremal ray, $\Xsing$ is $\bQ$-factorial \cite[Corollary 3.18]{KM98}. According to the fact that $K_\Xres$ is trivial in a neighborhood of $E$, it descends to $\Xsing$, and thus $\Xsing$ is Gorenstein.  

Recall that $E$ is irreducible. By the adjunction formula and $\crpcon$ being crepant, we get $K_E = E|_E$ and hence $- K_E$ is ample, because $- E$ is $\crpcon$-ample. Also, $E$ is Gorenstein since it is an effective divisor on a smooth threefold. Hence $E$ is a generalized del Pezzo surface (see also \cite[Proposition 2.13]{Reid80}). The second part follows from \cite[Lemma 2.3]{Wilson97} (or from \cite[Theorem 5.2]{Gross97a} when $\Xres$ is Calabi--Yau).
\end{proof}

\begin{remark}\label{rmk;can_sing_13}
With notation as in Proposition \ref{prop:birExc}, let $\crpcon (E) = \{p\}$. If $d \leq 3$, then the singularity $(\Xsing, p)$ is a hypersurface singularity with an equation of the form (cf.~\cite[Corollary 2.10]{Reid80}):
\begin{enumerate}
    \item $d = 1$, and $x_1^2 + x_2^3 + x_2 f (x_3, x_4) + g (x_3, x_4) = 0$ where $f$ (resp.\ $g$) is a sum of monomials of degree at least $4$ (resp.\ $6$).
    \item $d = 2$, and $x_1^2 + f(x_2, x_3, x_4) = 0$ where $f$ is a sum of monomials of degree at least $4$.
    \item $d = 3$, and  $x_1^3 + x_1 f(x_2, x_3, x_4) + g (x_2, x_3, x_4) = 0$ where $f$ (resp.\ $g$) is a sum of monomials of degree at least $2$ (resp.\ $3$).
\end{enumerate}
In this case the generalized del Pezzo surface $E$ is rational or is a cone over an elliptic curve, see \cite[Theorem 2.2]{HW81} and \cite[Remark 5.3]{Gross97a}.
\end{remark}

The following proposition will be used in the proof of Proposition \ref{prop;topo_defect} and the construction of (symplectic) semistable degenerations (\S\ref{subsec;ss_deg}).

\begin{proposition}[\cite{Gross97a, Wilson97}] \label{prop;locDef_res} 
Let $\crpcon \colon \Xres \to \Xsing$ be a primitive Type II contraction with exceptional divisor $E$ and $d = E^3$. Then there exists a neighborhood $U$ of $E$ and a holomorphic deformation of the complex structure on $U$ such that $E$ deforms to a smooth del Pezzo surface of degree $d$.
\end{proposition}

\begin{proof}

Let $\crpcon (E) = \{p\}$. If $d \leq 3$, then the singularity $(\Xsing, p)$ is a hypersurface singularity with an equation of the form given in Remark \ref{rmk;can_sing_13}.

Let $\alpha$ be the weight
\begin{align}\label{eqn;wt13}
    \alpha (x_1, x_2, x_3, x_4) =
    \begin{cases}
       (3, 2, 1, 1) & \text{if } d = 1,\\
       (2, 1, 1, 1) & \text{if } d = 2,\\
       (1, 1, 1, 1) & \text{if } d = 3,\\
    \end{cases}
\end{align}
and $\bP (\alpha)$ the weighted projective space. Then the morphism $\crpcon$ is the weighted blow-up of the singularity with weight $\alpha$. If we denote by $h_a$ the weighted homogeneous part of degree $a$ in a power series $h$, then $E$ is naturally embedded in $\bP(\alpha)$ with the equation:
\begin{align}\label{eqn;leadterm_13}
\begin{cases}
   x_1^2 + x_2^3 + x_2 f_4 (x_3, x_4) + g_6(x_3, x_4) = 0 & \text{if } d = 1, \\
   x_1^2 + f_4(x_2, x_3, x_4) = 0 & \text{if } d = 2, \\
   x_1^3 + x_1 f_2 (x_2, x_3, x_4) + g_3 (x_2, x_3, x_4) = 0 & \text{if } d = 3.
\end{cases}    
\end{align}
It is easily seen that there exists a deformation of the hypersurface singularity $(\Xsing, p)$ such that the exceptional locus $E'$ of the deformed singularity is a smooth del Pezzo surface of degree $d$ in $\bP(\alpha)$. Note that the singularities can be resolved in the family by means of the weighted blow-up described above.

We may take a good representative $\ocU \to \Delta$ of the deformation of $(\Xsing, p)$ such that on the boundary the family is differentiably trivial (see \cite[Theorem 2.8]{Looijenga84}). If we take $\cU \to \Delta$ to be the corresponding simultaneous resolution\footnote{It is a family of manifolds with boundary where the boundary is differentiably trivial over the disc $\Delta$.}, Ehresmann fibration theorem (with boundary) then applies to show that $\cU \to \Delta$ is itself differentiably trivial. Therefore $\cU \to \Delta$ is regarded as a family of holomorphic deformations of the complex structure on the (fixed) neighborhood $U$ of $E$ in $\Xres$, and therefore we can deform $E$ to a smooth del Pezzo surface $E'$.

From now on we will assume that $d > 3$. Let $\iota \colon E \to (\Xres, E)$ be the inclusion map of $E$ into the germ $(\Xres, E)$ and $\Def((\Xres, E), E, \iota)$ the deformation space of the inclusion map. By Proposition \ref{prop:birExc} and \cite[Lemma 5.6 (iii)]{Gross97a}, the generalized del Pezzo surface $E$ is smoothable. Moreover, the argument from \cite[pp.217-218, Theorem~5.8]{Gross97a} shows that the maps of deformation spaces $\Def((\Xres, E), E, \iota) \to \Def (E)$ and  $\Def((\Xres, E), E, \iota) \to \Def (\Xres, E)$ are surjective. Hence we deduce that a smoothing $E'$ of $E$ can be achieved by holomorphically deforming the complex structure on some neighborhood $U$ of $E$.
\end{proof}

We will also need the existence of the holomorphic tubular neighborhood.

\begin{proposition}[{\cite[Proposition 5.4]{Gross97a}}] \label{prop;hol_tubu}
Suppose that $(\Xsing, p)$ is an isolated rational Gorenstein threefold and that $\Xres \to \Xsing$ is the blow-up of $\Xsing$ at $p$ with exceptional divisor $E$. If $\Xres$ and $E$ are smooth and $E^3 \geq 5$, then $(\Xsing, p)$ is analytically isomorphic to a cone over $E$, that is, there exists a neighborhood $U \subseteq Y$ of $E$ which is biholomorphic to a neighborhood of the zero section in the total space of the normal bundle $N_{E / Y}$.
\end{proposition}

\subsection{Type II extremal transitions}

We first review the definition of geometric transitions.

\begin{definition}
Given smooth projective threefolds $\Xsm$ and $\Xres$, we say that $\Xsm$ and $\Xres$ are related by a \emph{geometric transition} if there exists a crepant contraction $\crpcon \colon \Xres \to \Xsing$ followed by a smoothing $\fX \to \Delta$ of $\Xsing = \fX_0$ with the fiber $X = \fX_t$ (for some $t \neq 0$). We write $\Xres \searrow \Xsm$ or $\typeII$ for this process. 

A geometric transition $\Xres \searrow \Xsm$ is called:
\begin{enumerate}[(i)]
    \item a \emph{del Pezzo transition} if $\Exc (\crpcon)$ is a smooth divisor and $\crpcon (\Exc (\crpcon))$ is a point.

    \item a \emph{Type II extremal transition} if $\crpcon$ is a primitive Type II contraction. In particular, the relative Picard number of $\crpcon$ is one.
\end{enumerate}
If the self-intersection number of the exceptional divisor $\Exc (\crpcon)$ is $d$, we say that $\Xres \searrow \Xsm$ is of degree $d$.
\end{definition}

Note that the exceptional divisor $\Exc (\crpcon)$ is a smooth del Pezzo surface of degree $d$ if $\Xres \searrow \Xsm$ is a del Pezzo transition of degree $d$.

\begin{remark}\label{rmk:P2F1}
If the singularity $\Xsing$ is isomorphic to the cone over $\bP^2$ or the Hirzebruch surface $\bF_1$ then it is not smoothable (see \cite[(5.5)]{Gross97a}). Therefore these two surfaces do not appear in the exceptional divisor of Type II extremal transitions. 
\end{remark}

The following are the standard examples of del Pezzo transitions, which we will study how the GW invariants are related under such transitions in \S\ref{sec;can_loc_model}.

\begin{example}\label{ex:dPtrans}
For $1 \leq d \leq 8$, we denote $S_d$ by a smooth del Pezzo surface of degree $d$ which is isomorphic to the blow-up of $\bP^1 \times \bP^1$ at $8 - d$ points, see Remark \ref{rmk;Sd}. We shall construct a del Pezzo transition $Y_d \searrow X_d$ with exceptional divisor $S_d$ for each $d$ (cf.\ Remark \ref{rmk:P2F1}).

Let $\alpha$ be the weight as in \eqref{eqn;wt13} for $1 \leq d \leq 3$ and 
\begin{align}\label{eqn;wt48}
    \alpha (x_1, \ldots, x_{d + 1}) = (1, \ldots, 1)
\end{align}
for $4 \leq d \leq 8$. Then we have an embedding $S_d \hookrightarrow \bP (\alpha)$ with 
\begin{align}\label{eqn:eaDelT}
    -K_{S_d} = \cO_{\bP (\alpha)} (1) |_{S_d}.
\end{align}
Let $\mathrm{Proj}_p \colon \bP(\alpha, 1) \dashrightarrow \bP (\alpha)$ be the rational map of weighted projective spaces given by the projection from the point $p = [0: \cdots :0:1] \in \bP (\alpha, 1)$. It is easily seen that the projective bundle $\bP_{\bP (\alpha)} (\cO(-1) \oplus \cO)$ over $\bP (\alpha)$ is the graph of the rational map $\mathrm{Proj}_p$. Note that $\bP_{\bP (\alpha)} (\cO(-1) \oplus \cO)$ is also isomorphic to the weighted blow-up of $\bP (\alpha, 1)$ at $p$ with the weight $(\alpha, 1)$.

Now we define $\oY_d$ the projective cone over $S_d$ with vertex $p$ in $\bP (\alpha, 1)$ and 
\begin{align*}
    Y_d = \bP_{S_d}(K_{S_d} \oplus \cO).
\end{align*}
According to \eqref{eqn:eaDelT}, it follows that $Y_d$ is the weighted blow-up $ \bP_{S_d}(\cO(-1) \oplus \cO)$ of $\oY_d$ at $p$ with the weight $\alpha$ and thus $S_d$ is the exceptional divisor of $Y_d \to \oY_d$. Note that $Y_d \to \oY_d$ is the restriction of the weighted blow-up $\bP_{\bP (\alpha)} (\cO(-1) \oplus \cO) \to \bP (\alpha, 1)$.
\begin{equation*}
    \begin{tikzcd}[column sep = -12pt]
       Y_d \ar[rrr,hook] \ar[d] &&& \bP_{\bP (\alpha)} (\cO(-1) \oplus \cO)  \ar[dl,swap] \ar[dr] & \\
        \oY_d \ar[rr,hook] &\qquad\qquad & \bP (\alpha, 1) \ar[rr,"\mathrm{Proj}_p",dashed] & & \bP(\alpha)  & \qquad\qquad & S_d \ar[ll,hook']
    \end{tikzcd}
\end{equation*}
Let $X_d$ be a smooth del Pezzo threefold of degree $d$ in $\bP(\alpha, 1)$. Then $X_d$ is a smoothing of the singular del Pezzo threefold $\oY_d$ and therefore we get the desired del Pezzo transition $Y_d \searrow X_d$. 

We remark that, when $d = 6$ there are two cases \eqref{thm;dP3claf_6I} and \eqref{thm;dP3claf_6II} in Theorem \ref{thm;dP3claf}, we will denote $X_{6\I}$ and $X_{6\II}$, respectively. It will be distinguished by the topological Euler characteristics, see Table \ref{tab;diffchiXY}. 
\end{example}

We shall construct the (local) semistable projective degenerations, which will be used in \S\ref{subsec;ss_deg}.

\begin{proposition}\label{prop;locDef_sm}
Suppose that $(\Xsing, p)$ is a rational Gorenstein singularity. Let $\mathfrak{V} \to \Delta$ be a good representative of a smoothing of $(\Xsing, p)$ with the special fiber $\oU \coloneqq \fV_0$. Given the weight $\alpha$ as in \eqref{eqn;wt13} and \eqref{eqn;wt48}, we denote $\crpcon \colon U \to \oU$ by the weighted blow-up at $p$ with weight $\alpha$. Let $E = \Exc (\crpcon)$ and $d = E^3$. Set $n_d = 6, 4, 3$ for $d = 1,2,3$, respectively, and $n_d = 2$ for $d \geq 4$.

If $U$ and $E$ are smooth, then there exists a semistable reduction $\cV \to \Delta$ of $\mathfrak{V} \to \Delta$ wihch is obtained by a degree $n_d$ base change $\cV' \to \Delta$ allowed by the weighted blow-up at $p \in \cV'$ with weight $(\alpha, 1)$. Moreover, the special fiber $\cV_0 = U \cup X_d$ is a simple normal crossing divisor such that $X_d$ is a smooth del Pezzo threefold of degree $d$ and $U \cap X_d = E$.
\end{proposition}

\begin{proof}
Note that the exceptional divisor $E$ of $\crpcon$ is a smooth del Pezzo surface  of degree $d$ by assumption. To study $\hcO_{\Xsing, p}$, we are going to use Reid's classification of rational Gorenstein singularities \cite[Corollary 2.10]{Reid80}.

For $1 \leq d \leq 3$, recall that $(\Xsing, p)$ is a hypersurface singularity with an equation $\tilde{F}$ of the form as in Remark \ref{rmk;can_sing_13}, whose leading term $F$ as in \eqref{eqn;leadterm_13} defines the del Pezzo surface $E$ in $\bP(\alpha)$. Also, for $d = 4$, since $(\Xsing, p)$ is Gorenstein in codimension $2$, it is a complete intersection of two equations $\tilde{Q}_1, \tilde{Q}_2$ whose leading terms are quadratic $Q_1,Q_2$, respectively. It define the del Pezzo surface $E$ of degree $4$ in $\bP^4$. For $d \geq 5$, let $C(E)$ be the affine cone with vertex $o$ over $E$, that is, 
\begin{align}\label{eqn;conesg}
    C(E) = \Spec \left( \bigoplus_{m \geq 0} H^0 (E, \cO (- mK)) \right).
\end{align}
By Proposition \ref{prop;hol_tubu} and Theorem \ref{thm;dP2}, we have 
\begin{align*}
    \hcO_{\Xsing, p} \cong \hcO_{C(E), o} \cong \bC \llbracket \underline{x} \rrbracket / \cI
\end{align*}
and $E$ is defined by the generator of $\cI$. Here we write $\underline{x}$ a short form for indeterminates $\{x_m\}$. Moreover, $\cI$ is generated by $d (d - 3) / 2$ quadrics $\{Q_k\}$.

Let us denote by $t$ the coordinate of the disk $\Delta$ and let $t$ be weight one. We let $\cV' \to \Delta$ be a degree $n_d$ base change of $\fV \to \Delta$, and let $\cV$ be the weighted blow-up of $\cV'$ at $p$ with weight $(\alpha, 1)$. Note that $\cV \to \cV'$ is an ordinary blow-up at $p$ if $d \geq 3$. We claim that $\cV \to \Delta$ is the desired degeneration. Indeed, we observe that
\begin{align*}
    \hcO_{\fV, p} \cong 
    \begin{cases}
      \bC \llbracket x_1, x_2, x_3, x_4, t\rrbracket / ( \tilde{F} - t ) & \text{if } 1 \leq d \leq 3, \\
      \bC \llbracket x_1, \ldots, x_{5}, t \rrbracket / ( \tilde{Q}_1 - t, \tilde{Q}_{2} - t)  & \text{if } d = 4\\
      \bC \llbracket x_1, \ldots, x_{d + 1}, t \rrbracket / ( Q_1 - t,\dots,Q_{d(d-3)/2} - t)  & \text{if } 5 \leq d \leq 8.
    \end{cases}
\end{align*} 
We will denote $X_d$ by the exceptional divisor of the weighted blow-up $\cV$.

For $1 \leq d \leq 3$, we find that the defining equation of $(\cV', p)$ is $\tilde{F} - t^{n_d} = 0$. Then, by construction, the exceptional divisor $X_d$ of $\cV \to \cV'$ is the hypersurface in $\bP (\alpha, 1)$ defined by the weighted homogeneous polynomial $F - t^{n_d} = 0$ of degree $n_d$ and therefore $X_d$ is a smooth del Pezzo threefold of degree $d$. We remark that if $t$ is regarded as a section $t \in H^0 (X_d, \cO(1))$ then it defines the del Pezzo surface $E = ( t = 0 )$ in $X_d$. Also, according to that $U$ is the weighted blow-up of $\oU$ at $p$ with weight $\alpha$, it follows that the proper transform of $\cV_0' = \oU$ in $\cV$ is the open neighborhood $U$ of $E$. Hence we conclude that the special fiber $\cV_0$ of $\cV \to \Delta$ over $t = 0$ consists of two irreducible components $U$ and $X_d$, which are smooth and $U \cap X_d = (F=0)$ is the del Pezzo surface $E$.

Similarly, for $d \geq 4$, the exceptional divisor $X_d$ of $\cV \to \cV'$ is defined by quadrics $Q_k - t^2$ in $\bP^{d + 1}$. It is easily seen that $X_d$ is a smooth threefold, which contains the del Pezzo surface $E$ of degree $d$ defined by the section $t \in H^0 (X_d, \cO(1))$. By the Lefschetz theorem, the restriction $\Pic (X_d) \to \Pic (E)$ is injective. Hence, using the adjunction formula, we get that $\cO_{X_d}(- K) \cong \cO_{X_d}(2)$, and thus $X_d$ is del Pezzo of degree $d$. The remaining part follows from the same argument as above.
\end{proof}

\subsection{Global topology}

In this subsection we describe topological properties of Type II extremal transitions.

\begin{proposition}\label{prop;topo_defect}
If $\Xres \searrow \Xsm$ is a Type II extremal transition of degree $d$, then  
\begin{align*}
    \chitop (Y) - \chitop(X) = 2 \chitop (S_d) - \chitop(X_d),
\end{align*}
where $S_d$ (resp.\ $X_d$) is a smooth del Pezzo surface (resp.\ threefold) of degree $d$ and $\chitop(-)$ is the topological Euler characteristic.
\end{proposition}

\begin{proof}
Let $\crpcon \colon \Xres \to \Xsing$ be the crepant contraction of $\Xres \searrow \Xsm$, $E$ the exceptional divisor of $\crpcon$ and $\crpcon (E) = \{p\}$. Since the question is topological, we apply Proposition \ref{prop;locDef_res} and replace an open neighborhood of $E$
by another open set that is diffeomorphic to it. So we may assume that $E$ is a smooth del Pezzo surface of degree $d$ and $(\Xsing, p)$ is analytically isomorphic to the cone \eqref{eqn;conesg} over $E$.
 
Let $B$ be the Milnor fiber of the smoothing of $(\Xsing, p)$ corresponding to $\Xsm$. Note that $B$ is homotopy equivalent to a finite cell complex of $\dim \leq 3$ by \cite[(5.6)]{Looijenga84}. It is checked that $B$ is homeomorphic to $X_d \setminus (X_d \cap A)$, where $X_d$ is a smooth del Pezzo threefold of degree $d$ and $A \in |-K_{X_d}|$ a general hyperplane. Write $S_d = X_d \cap A$, a smooth del Pezzo surface of degree $d$, and let $\iota \colon X_d \setminus S_d \to X_d$ be the inclusion. By using the Gysin sequence
\begin{align*}
   \cdots \to H^k (X_d) \xrightarrow{\iota^\ast} H^k (X_d \setminus S_d) \xrightarrow{\mathrm{Res}} H^{k-1} (S_d) \to H^{k + 1} (X_d) \to \cdots,
\end{align*}
we have 
\begin{align}\label{eqn;chi_Milnor}
    \chitop (B) = \chitop (X_d) - \chitop (S_d).
\end{align}
Also, by Theorems \ref{thm;dP2} and \ref{thm;dP3claf}, we get the Betti numbers $b_0 (B) = 1$, $b_2 (B) = h^{1, 1} (X_d) - 1$, 
\[
    b_3 (B) = 2 h^{2, 1}(X_d) - h^{1, 1} (X_d) + 10 - d
\]
and $b_k (B) = 0$ for other $k$. Note that $b_1 (B) = 0$ by \cite[Theorem 2]{GS83}.

Since $\Xsing$ is $\Xres$ with $E$ collapsed to the point $p$ ($\Xsing$ is homeomorphic to $\Xres / E$), the relative cohomology $H^i (\Xres, E) $ is isomorphic to the reduced cohomology $\widetilde{H}^k (\Xres / E) \cong \widetilde{H}^k (\Xsing)$ for all $k$. By the long exact sequence of the pair $(\Xres, E)$, we get
\begin{align}\label{eqn;longex_res}
    \cdots \to \widetilde{H}^k(\Xsing) \xrightarrow{\crpcon^\ast} H^k (Y) \to H^k (E) \to \widetilde{H}^{k + 1} (\Xsing) \to \cdots.
\end{align}
On the other hand, one can construct a homeomorphism between $\Xsing$ and $X/B_d$. We denote the collapsing (continuous) map by 
\begin{align}\label{eqn;collapsing_map}
    c\colon \Xsm \to \Xsing.
\end{align}
Hence there also exists the long exact sequence
\begin{align}\label{eqn;longex_sm}
    \cdots \to \widetilde{H}^k (\Xsing) \xrightarrow{c^\ast} H^k (\Xsm) \to H^k (B_d) \to \widetilde{H}^{k + 1} (\Xsing) \to \cdots. 
\end{align}
By \eqref{eqn;longex_res} and \eqref{eqn;longex_sm}, we get 
\[
    \chitop (X) - \chitop(B) = \chitop(Y)- \chitop(E).
\]
Then the proposition follows from $\chitop (E) = \chitop (S_d)$ and \eqref{eqn;chi_Milnor}. 
\end{proof}

Write $\Delta\chitop \coloneqq \chitop (Y) - \chitop(X)$, the topological difference of a Type II extremal transition $\Xres \searrow \Xsm$ of degree $d$. We give $\Delta\chitop$ in Table \ref{tab;diffchiXY} by using Theorem \ref{thm;dP2} and Table \ref{tab;HGXd}.

\begin{table}[H]
    \centering
    \begin{tabular}{cccccccccc}
        \toprule
        $d$ & $1$ & $2$ & $3$ & $4$ & $5$ & $6 \I$ & $6 \II$ & $7$ & $8$ \\
        \midrule
        $\Delta \chitop$ & $60$ & $36$ & $24$ & $16$ & $10$ & $6$ & $4$ & $4$ & $4$ \\ 
        \bottomrule
    \end{tabular}
    \caption{The topological difference $\Delta\chitop$ of $\Xres \searrow \Xsm$.}
  \label{tab;diffchiXY}
\end{table}

\begin{remark}

The differences $\frac{1}{2}\Delta\chitop$ are exactly the dual Coxeter numbers of $E_{9 - d}$ (cf.~\cite[pp.241--242]{MS97}). The root system is summarized in Table \ref{tab;root_sys}. Here we set $E_1 = E_2 = A_1$, and root systems $E_7$, $E_6$, $E_5$, $E_4$ and $E_3$ are obtained from $E_8$ by deleting one by one the vertices from the long end of its Dynkin diagram. For $d \leq 7$, these root systems are associated to a smooth del Pezzo surfaces of degree $d$, see Proposition 25.2 and Theorem 25.4 in \cite{Manin86}.

For the case of $d=6$, the transition of degree $d=6 \I$ (resp. $d=6\II$) corresponds to the direct summand $A_2$ (resp.\ $A_1$) of the root system $A_2 \times A_1$. The dual Coxeter number of $A_2$ and $A_1$ are $3$ and $2$, respectively.
\begin{table}[H]
  \centering
  \begin{tabular}{ccccccccc}
    \toprule
    $d$ & $1$ & $2$ & $3$ & $4$ & $5$ & $6$ & $7$ & $8$ \\
    \midrule
    Type of $E_{9 - d}$ & $E_8$ & $E_7$ & $E_6$ & $D_5$ & $A_4$ & $A_2 \times A_1$ & $A_1$ & $A_1$ \\
    \bottomrule
  \end{tabular}
  \caption{Root systems and del Pezzo surfaces.}
  \label{tab;root_sys}
\end{table} 
\end{remark}

The following proposition relates cohomology groups of $\Xsm$ and $\Xres$.

\begin{proposition}\label{prop;coho_seq}
Let $\typeII$ be a Type II extremal transition with exceptional divisor $E$.
\begin{enumerate}[(i)]
    \item\label{prop;coho_seq_1}  Let $\crpcon (E) = \{p\}$, and let $B$ be the Milnor fiber of the smoothing of $(\Xsing, p)$ corresponding to $\Xsm$. Then there exists an exact sequence
    \begin{align*}
        0 \to H^2(B) \to H^3(\Xsing) \to H^3 (X) \to H^3(B) \to 0,
    \end{align*}
    and $H^k (\Xsing) \xrightarrow{\sim} H^k (\Xsm)$ for all $k \neq 3$. 

    \item\label{prop;coho_seq_2} There exist exact sequences
    \begin{align}
        &0 \to H^2 (\Xsing) \xrightarrow{\crpcon^\ast} H^2 (Y) \twoheadrightarrow \bC [K_E] \subseteq H^2 (E), \label{prop;coho_seq_2-1}\\
        &0 \to H^2 (E) / \bC [K_E] \to  H^3 (\Xsing) \to H^3(\Xres) \to H^3 (E) \to 0, \label{prop;coho_seq_2-2}\\
        &0 \to H^4 (\Xsing) \to H^4 (\Xres) \to H^4 (E) = \bC [\pt] \to 0, \label{prop;coho_seq_2-3}
    \end{align}
    and $H^k (\Xsing) \xrightarrow{\sim} H^k (\Xres)$ if $k \neq 2$, $3$, $4$. 
\end{enumerate}
\end{proposition}

\begin{proof}
To prove \eqref{prop;coho_seq_1}, we use the long exact sequence \eqref{eqn;longex_sm} induced by the topological pair $(\Xsm, B)$. Recall that $B$ is homotopy equivalent to a finite cell complex of $\dim \leq 3$, and $H^1(B) = 0$. Then we get $H^k (\Xsing) \xrightarrow{\sim} H^k (\Xsm)$ for $k = 0$, $1$, $5$, $6$, and $H^k (\Xsing) \to H^k (\Xsm)$ is injective (resp.~surjective) for $k = 2$ (resp.~$k = 4$). By Poincar\'e duality, we get 
\begin{align}\label{eqn;betti_inq}
    b_2 (\Xsing) \leq b_2 (\Xsm) = b_4 (\Xsm) \leq b_4 (\Xsing).
\end{align}
Since $\Xsing$ has rational singularities and is $\bQ$-factorial, we have
\begin{align}\label{eqn;defect=0}
    b_2 (\Xsing) = b_4 (\Xsing)
\end{align}
by \cite[Theorem A]{PP24}. Then \eqref{prop;coho_seq_1} follows from \eqref{eqn;betti_inq} and \eqref{eqn;defect=0}.

The proof of \eqref{prop;coho_seq_2} uses the other long exact sequence \eqref{eqn;longex_res} induced by the topological pair $(\Xres, E)$. Since $\phi$ is an extremal contraction and $\Xsing$ has rational singularities, we have $\phi^\ast \colon H^2 (\Xsing) \to H^2 (Y)$ is injective with $\dim \mathrm{Coker}(\phi^\ast) = 1$. Therefore the image of the restriction map $H^2( \Xres) \to H^2(E)$ is generated by the class of $E|_E = K_E$ and \eqref{prop;coho_seq_2-1} follows.

On the other hand, since $\Xsing$ has only an isolated singularity, the transition $\Xres \searrow \Xsm$ does not change the fundamental group and thus the first Betti number, i.e., $b_1 (X) = b_1 (Y)$. Also, $b_1 (\Xsing) = b_1 (X)$ by \eqref{prop;coho_seq_1}. Therefore the injective map $H^1 (\Xsing) \to H^1 (\Xres)$ is an isomorphism. The assertion for $k = 0$ is trivial. Observe that $H^4 (\Xres) \to H^4 (E)$ is surjective, since it is a nonzero map and $H^4 (E) = \bC  [\pt]$. Then $H^k (\Xsing) \xrightarrow{\sim} H^k (\Xres)$ for $k = 5$, $6$. To finish \eqref{prop;coho_seq_1}, we claim that $b_4(\Xsing) = b_4(\Xres) - 1$, this will give \eqref{prop;coho_seq_2-2} and \eqref{prop;coho_seq_2-3}. To see this, we already know $b_2(\Xsing) = b_2(Y) - 1$. The claim then follows from Poincar\'e duality, $b_2 (\Xres) = b_4 (\Xres)$, and \eqref{eqn;defect=0}.
\end{proof}

\begin{remark}
For a Type II extremal transition $\Xres \searrow \Xsm$, we have
\begin{align}
    h^{1,1}(Y) - h^{1,1}(X) &= 1, \label{eqn;h11_defect}\\
    h^{2, 1} (\Xsm) - h^{2, 1} (\Xres) &= \frac{1}{2} \Delta\chitop - 1. \label{eqn;h21_defect}
\end{align}
by Propositions \ref{prop;topo_defect} and \ref{prop;coho_seq}. Indeed, the transition $\Xres \searrow \Xsm$ preserves canonical bundles and thus the Hodge number $h^{3,0}$. Note that it also preserves the Hodge numbers $h^{p, 0} = h^0 (\Omega^p)$ for $p = 1$, $2$.

When $X$ and $Y$ are smooth Calabi--Yau threefolds and the exceptional divisor $E$ of the transition is smooth, \eqref{eqn;h11_defect} is Proposition 3.1 in \cite{Kapustka209I} by applying Koll\'ar--Mori's results in \cite[\S 12.1 and 12.2]{KM92} (see also \cite[Proposition\ 3.1]{Gross97b} and \cite[Proposition 6.1]{Namikawa94}). In addition, \eqref{eqn;h21_defect} was proved in \cite[Theorem 3.3, Remark 3.6]{Kapustka209I} (see also \cite[p.761]{Namikawa97} for $E^3 =6$). Moreover, if $E^3 \geq 5$, then the miniversal Kuranishi space of $\Xsing$ is the product of the miniversal space of the singularity $(\Xsing, p)$, analytically isomorphic to the cone \eqref{eqn;conesg} over $E$, with an appropriate germ of a linear space \cite[Remark 3.6]{Kapustka209I}.
\end{remark}

\begin{remark}
It is interesting to compare \eqref{eqn;h21_defect} with the situation of conifold transitions, see Lemma 1.6 in \cite{LLW18}.
\end{remark}

\begin{corollary}
Let $\typeII$ be a Type II extremal transition. Then there exists an embedding 
\begin{align}\label{eqn;H_linear}
     H(\Xsm) \hookrightarrow H(\Xres)
\end{align}
of cohomology rings which we abuse notation slightly and denote by $\crpcon^\ast$, i.e., it preserves the cup product.
\end{corollary}

\begin{proof}
By Proposition \ref{prop;coho_seq} \eqref{prop;coho_seq_1}, the collapsing (continuous) map $c \colon \Xsm \to \Xsing$ \eqref{eqn;collapsing_map} induces a ring isomorphism $c^\ast \colon H (\Xsing) \to H(\Xsm)$. On the other hand, the ring homomorphism $\crpcon^\ast \colon H (\Xsing) \to H (\Xres)$ is injective by Proposition \ref{prop;coho_seq} \eqref{prop;coho_seq_2}. Therefore the composition $\crpcon^\ast \circ (c^\ast)^{- 1}$ gives the desired embedding \eqref{eqn;H_linear}.
\end{proof}

\begin{remark}
Let $\fX \to \Delta$ be the smoothing of the transition $\Xres \searrow \Xsm$. By Proposition \ref{prop;coho_seq} \eqref{prop;coho_seq_1} and the fact that $\Xsing$ is a deformation retract of $\fX$, we get
\begin{align}\label{eqn;H_inv}
     H (\fX) \cong H(\Xsing) \xrightarrow{\sim} H(\Xsm).
\end{align}
Hence every (even) cohomology class of $\Xsm$ has a lifting in $H (\fX)$.
\end{remark}

\begin{remark}
Using the argument of the proof of Proposition \ref{prop;coho_seq}, one can infer that $b_4(\Xsing_d) > b_2(\Xsing_d)$ for the canonical local model $\Xres_d \to \Xsing_d \rightsquigarrow \Xsm_d$, see Example \ref{ex:dPtrans}, for all $1 \leq d \leq 8$. In particular, $\Xsing_d$ is non-$\bQ$-factorial (cf.\ \cite[Theorem 3.2]{NS95} and \cite[Theorem A]{PP24}).
\end{remark}

\section{Preliminaries on Gromov--Witten theory}\label{sec;prel_GW}

\subsection{Gromov--Witten invariants}

We consider only \emph{genus zero} GW invariants.

Let $M$ be a smooth projective variety. For an effective curve class $\beta \in \operatorname{NE}(M)$, we denote by $\overline{\mathcal{M}}_{0,n}(M, \beta)$ the moduli space of $n$-pointed stable maps $f\colon (C, p_1, \dots, p_n) \to X$ with arithmetic genus $g(C) = 0$ and with degree $f_*[C] = \beta$. Let $\psi_\nu$ be the first Chern class of the $\nu$-th universal cotangent line bundle on $\oM_{0,n}(M, \beta)$ and $\operatorname{ev}_\nu\colon \overline{\mathcal{M}}_{0, n}(M, \beta) \to M$ the evaluation map $[f] \mapsto f(p_\nu)$. For $\vec{a} = a_1 \otimes \dots \otimes a_n \in H(M)^{\otimes n}$ and integers $k_1$, $\dots$, $k_n \ge 0$, we define the (genus zero, descendant) GW invariant as 
\begin{align*}
     \langle \psi^{k_1}a_1, \dots, \psi^{k_n}a_n\rangle_\beta^M = \int_{[\oM_{0,n}(M, \beta)]^\vir} \prod_{\nu = 1}^n \ev_\nu^* a_\nu \cup \psi_\nu^{k_\nu}
\end{align*}
where $[\oM_{0,n}(M, \beta)]^\vir$ denotes the \emph{virtual fundamental class} on $\oM_{0,n}(M, \beta)$ \cite{BF97, LT98}. The virtual dimension is 
\begin{align*}
    \operatorname{vdim}_\beta \coloneqq (c_1(M),\beta) + n + (\dim M -3).
\end{align*}

Let $q^\beta$ be the (formal) Novikov variable corresponding to $\beta \in \NE (M)$. 
For a parameter $\tau \in H(M)$, we define the \emph{$n$-point function} (in $\tau$) by 
\[\llangle a_1, \dots, a_n\rrangle^M(\tau) = \sum_{n,\beta} \frac{q^\beta}{n!}\langle a_1, \dots, a_n, \tau^{\otimes n}\rangle_\beta. \]
The big quantum product $\ast_\tau$ is then defined by 
\begin{equation}\label{eq:bigquantumprod}
    a_1 \ast_\tau a_2 = \sum_\mu \llangle a_1, a_2, T_\mu \rrangle^M(\tau) T^\mu,
\end{equation}
where $\{T_\mu\}$, $\{T^\mu\}$ are dual bases of $H(M)$. When $\tau \in H^2 (M)$, we write $Q^\beta = q^\beta e^{(\tau, \beta)}$, so that for $n \ge 3$ the $n$-point function reduces to the usual generating series
\begin{align*}
    \llangle a_1, \dots, a_n\rrangle^M(\tau) = \sum_\beta \langle a_1, \dots, a_n\rangle_\beta^M Q^\beta \equalscolon \langle a_1, \dots, a_n\rangle^M, 
\end{align*}
by the divisor equation. We extend this notation to include descendants and arbitrary $n \ge 0$ by declaring 
\[\langle \psi^{k_1}a_1, \dots, \psi^{k_n}a_n\rangle^M \colonequals \sum_{\beta} \langle \psi^{k_1}a_1, \dots, \psi^{k_n}a_n\rangle_\beta^M Q^\beta. \]
Given $D\in H^2(M)$, we define the \emph{power operator} $\delta^D$ by setting 
\begin{align*}
    \delta^D Q^\beta = (D, \beta) Q^\beta \quad \text{for each } \beta \in \NE (M)
\end{align*}
and extending linearly.

Applying \cite[Corollary 1]{LP04} and the WDVV equations, we have the following relations between multipoint functions, known as the reconstruction formula. 
\begin{theorem}[\cite{LP04}] \label{thr:rc}

Let $a_1\otimes \dots \otimes a_n \in H(M)^{\otimes n}$ and $D\in H^2(M)$. Pick a basis $\{T_\mu\}$ of $H^\ast(M)$ with dual basis $\{T^\mu\}$. By abuse of notation, we denote the insertions $\bigotimes_{\nu \in I} a_\nu$ briefly by $\vec{a}_I$ for any subset $I$ of $\{1, \ldots, n\}$.  
    \begin{enumerate}[(a)]
        \item For $n\ge 2$, we have 
        \begin{align*}
            \langle Da_1, a_2, \dots, a_n \rangle^M &= \langle a_1, Da_2, \dots, a_n\rangle^M + \delta^D\langle a_1, \psi a_2, \dots, a_n\rangle^M \\
            & \quad - \sum_{I_L \sqcup I_R = \{3, \dots, n\}, \mu} \delta^D\langle a_1, \vec{a}_{I_L}, T_\mu\rangle^M \cdot \langle T^\mu, a_2, \vec{a}_{I_R} \rangle^M. 
        \end{align*}
        \item For $n\ge 3$, we have 
        \begin{align*}
            \langle Da_1, a_2, \dots, a_n \rangle^M &= \langle a_1, Da_2, \dots, a_n\rangle^M \\
            &\quad + \sum_{3 \in I_L, 2\in I_R, \mu} \langle a_1, \vec{a}_{I_L}, T_\mu\rangle^M \cdot \delta^D \langle T^\mu, \vec{a}_{I_R} \rangle^M \\
            &\quad - \sum_{2, 3 \in I_R, \mu} \delta^D\langle a_1, \vec{a}_{I_L}, T_\mu\rangle^M \cdot \langle T^\mu, \vec{a}_{I_R} \rangle^M, 
        \end{align*}
        where $I_L \sqcup I_R = \{2,3,\dots, n\}$ is a partition. 
    \end{enumerate}
\end{theorem}

\subsection{Symmetrized Gromov--Witten}\label{subsec;symGW}

Let $M$ be a smooth projective variety. Suppose that there is a smooth projective algebraic family $\mathcal{M} = \{M_t\}$ such that $M_0 \cong M \cong M_1$. Then $\mathcal{M}$ induces an automorphism on $H(M)$: 
\[\sigma\colon H(M) \cong H(M_0) \xleftarrow{\ \sim\ } H(\mathcal{M}) \xrightarrow{\ \sim\ } H(M_1) \cong H(M). \]
\begin{lemma}\label{lem:GWisGinv}
    For each $\vec{a} \in H(M)^{\otimes n}$, we have
    \[\langle \sigma(\vec{a})\rangle^M = \sigma\langle \vec{a}\rangle^M, \]
    where $\sigma(Q^\beta) = Q^{\sigma(\beta)}$. 
\end{lemma}

\begin{proof}
    By comparing the coefficient of $Q^{\sigma(\beta)}$ on both sides, we need to show that $\langle \sigma(\vec{a})\rangle_{\sigma(\beta)}^M = \langle \vec{a}\rangle_\beta^M$ holds for each $\beta$. This follows from the fact that the GW invariant is deformation invariant. 
\end{proof}

Let $G \subseteq \operatorname{Aut}(H(M))$ be a finite group that consists of some of the automorphisms $\sigma$ that came from deformations as above. Let $H(M)^G$ be the $G$-invariant part of $H(M)$ and $K$ the kernel of the symmetrization
\[\begin{tikzcd}[row sep = 0pt]
    H(M) \ar[r, "\Sigma"] & H(M)^G \\
    a \ar[r, mapsto] & \displaystyle\frac{1}{|G|} \sum_{\sigma \in G} \sigma(a). 
\end{tikzcd}\]
Then $H(M) = H(M)^G \oplus K$ is an orthogonal direct sum with respect to the Poincar\'e pairing. 

If $\tau^{(2)} \in H^2(M)^G$, then $e^{(\tau^{(2)}, \beta)} = e^{(\tau^{(2)}, \sigma(\beta))}$. So we may identify the Novikov variables $Q^\beta = Q^{\sigma(\beta)}$. We also denote the identification by $\Sigma$: 
\[\begin{tikzcd}[row sep = 0pt]
    \bC \llbracket Q^{\operatorname{NE}(M)} \rrbracket \ar[r, "\Sigma"] & \bC \llbracket Q^{\operatorname{NE}(M)/G}\rrbracket \\
    Q^\beta \ar[r, mapsto] & Q^{G\beta}.
\end{tikzcd}\]

Define the $G$-symmetrized invariant as 
\begin{equation}\label{eq:syminv}
    \langle \vec{a}\rangle_G^M = \Sigma \langle \vec{a}\rangle^M,\quad \vec{a} \in H(M)^{\otimes n}. 
\end{equation}
By Lemma~\ref{lem:GWisGinv}, $\langle \vec{a}\rangle_G^M$ is $G$-invariant, i.e., 
\begin{equation}\label{eqambisGinv}
    \langle \sigma(\vec{a}) \rangle_G^M = \langle \vec{a} \rangle_G^M,\quad \vec{a} \in H(M)^{\otimes n},\quad \sigma \in G. 
\end{equation}

\begin{prop}\label{propambrec}
    The reconstruction formula for symmetrized invariants holds, i.e., if $a_1\otimes \dots\otimes a_n \in (H(M)^G)^{\otimes n}$, $D\in H^2(M)^G$ and $\{T_\mu\}$, $\{T^\mu\}$ are dual bases of $H(M)^G$, then, 
    \begin{enumerate}[(a)]
        \item for $n\ge 2$, we have 
        \begin{align*}
            \langle Da_1, a_2, \dots, a_n \rangle_G^M &= \langle a_1, Da_2, \dots, a_n\rangle_G^M + \delta^D\langle a_1, \psi a_2, \dots, a_n\rangle_G^M \\
            & \quad - \sum_{I_L \sqcup I_R = \{3, \dots, n\}, \mu} \delta^D\langle a_1, \vec{a}_{I_L}, T_\mu\rangle_G^M \cdot \langle T^\mu, a_2, \vec{a}_{I_R} \rangle_G^M; 
        \end{align*}
        \item for $n\ge 3$, we have 
        \begin{align*}
            \langle Da_1, a_2, \dots, a_n \rangle_G^M &= \langle a_1, Da_2, \dots, a_n\rangle_G^M \\
            &\quad + \sum_{3 \in I_L, 2\in I_R, \mu} \langle a_1, \vec{a}_{I_L}, T_\mu\rangle^M \cdot \delta^D \langle T^\mu, \vec{a}_{I_R} \rangle_G^M \\
            &\quad - \sum_{2, 3 \in I_R, \mu} \delta^D\langle a_1, \vec{a}_{I_L}, T_\mu\rangle^M \cdot \langle T^\mu, \vec{a}_{I_R} \rangle_G^M, 
        \end{align*}
        where $I_L \sqcup I_R = \{2,3,\dots, n\}$ is a partition.  
    \end{enumerate}
\end{prop}

\begin{proof}
    We will only prove (a), since (b) follows from (a). Let $\{\tilde{T}_\nu\}$, $\{\tilde{T}^\nu\}$ be dual bases of $K$. Since $H(M) = H(M)^G \oplus^\perp K$, the original reconstruction formula (cf.~Theorem~\ref{thr:rc}) gives 
    \begin{align*}
        \langle Da_1, a_2, \dots, a_n \rangle_G^M &= \langle a_1, Da_2, \dots, a_n\rangle_G^M + \delta^D\langle a_1, \psi a_2, \dots, a_n\rangle_G^M \\
        &\quad - \sum_{I_L \sqcup I_R = \{3, \dots, n\}, \mu} \delta^D\langle a_1, \vec{a}_{I_L}, T_\mu\rangle_G^M \cdot \langle T^\mu, a_2, \vec{a}_{I_R} \rangle_G^M \\
        &\quad - \sum_{I_L \sqcup I_R = \{3, \dots, n\}, \mu} \delta^D\langle a_1, \vec{a}_{I_L}, \tilde{T}_\nu\rangle_G^M \cdot \langle \tilde{T}^\nu, a_2, \vec{a}_{I_R} \rangle_G^M. 
    \end{align*}
    To prove the statement, it suffices to show that for $\tilde{T} \in K$, $\langle \tilde{T}, \vec{a}_{I} \rangle_G^M = 0$. Indeed, by \eqref{eqambisGinv},  
    \[\langle \tilde{T}, \vec{a}_{I}\rangle_G^M = \frac{1}{|G|}\sum_\sigma \langle \sigma(\tilde{T}), \vec{a}_{I}\rangle_G^M = \langle \Sigma(\tilde{T}), \vec{a}_{I}\rangle_G^M = 0. \qedhere\]
\end{proof}

Recall that $S_d$ is a del Pezzo surface of degree $d$ (Example \ref{ex:dPtrans}). Let $\cC$ be the minimal set of generators of $\operatorname{NE}(S_d)$, the Mori cone of $S_d$. When $d\le 7$, $\cC$ is the set of $(-1)$-curves on $S_d$ by cone theorem. When $d = 8$, $\cC = \{\ell_1, \ell_2\}$, the fibers of the projections $p_1$, $p_2\colon S_8 = \bP^1\times \bP^1\to \bP^1$. 

The following proposition will be used in \S \ref{subsec;GW_neq_7}.

\begin{proposition}\label{prop:defonSd}
    Suppose that $d\neq 7$. For any pairs of $(-1)$-curves $\ell$, $\ell' \in \cC$, there is a deformation $\mathcal{S}_d$ such that $S_{d,0} \cong S_d \cong S_{d,1}$ and the induced automorphism $\sigma$ maps $\ell$ to $\ell'$. 

    In particular, there exists a finite group $G_{S_d} \subseteq \operatorname{Aut}(H(S_d))$ such that $G_{S_d}$ acts on $\cC$ transitively. 
\end{proposition}

\begin{proof}
    If $d = 8$, then $(\bP^1 \times \bP^1, \bP^1\times *)$ is clearly isomorphic to $(\bP^1 \times \bP^1, *\times \bP^1)$ by interchanging two components.  

    If $d \le 6$, let $(S_d, \ell) \to (S_{d+1}, p)$, $(S_d, \ell') \to (S_{d+1}', p')$ be the contractions of $\ell$, $\ell'$, respectively. There is a deformation that links $(S_{d+1}, p)$ and $(S_{d+1}', p')$. Indeed, there is a family $\mathcal{S} = \{S(t)\}_{t}$ on the moduli space of del Pezzo surfaces of degree $d+1$ such that $S(0) = S_{d+1}$ and $S(1) = S_{d+1}'$. Then take any path $p(t) \in S(t)$ on the family $\mathcal{S}$ that connects $p = p(0) \in S(0)$ and $p' = p(1) \in S(1)$. Let $\ell(t)$ be the exceptional curve of the blow up $\operatorname{Bl}_{p(t)} S(t) \to S(t)$. Then the family $(\operatorname{Bl}_{p(t)} S(t), \ell(t))$ is the desired deformation. 
\end{proof}

\begin{remark}
For $1 \leq d \leq 6$, the group $G_{S_d}$ is the Weyl group $W(E_{9-d})$ of the root system $E_{9 - d}$ in Table \ref{tab;root_sys}, see for example \cite[Theorem 23.9]{Manin86}.
\end{remark}

\subsection{I-function and mirror theorem}

For $\tau = \tau^{(0)} + \tau^{(2)}\in H^0(M) \oplus H^2(M)$, we define the (small) \emph{$J$-function} of $M$ to be the generating function
    \begin{align*}
        J^M(\tau, z^{-1}) \coloneqq{}& e^{\tau/z} \sum_{\beta \ge 0} J_\beta^M(z^{-1}) Q^\beta \\
        \coloneqq{}& e^{\tau/z} \Bigl(1 + \sum_{\beta > 0} \ev_{1*}\Bigl(\frac{1}{z(z - \psi_1)} \cap \bigl[\overline{\mathcal{M}}_{0,1}(M, \beta)\bigr]^\vir \Bigr) Q^\beta\Bigr) \\
        ={} & e^{\tau/z} \Bigl(1 + \sum_{\beta > 0, k\ge 0} \frac{1}{z^{k+2}}\langle \psi^k T_\mu\rangle_\beta T^\mu Q^\beta\Bigr),
    \end{align*}
where $Q^\beta = q^\beta e^{(\tau^{(2)}, \beta)}$ and $\{T_\mu\}$ is a homogeneous basis of $H(M)$ with dual basis $\{T^\mu\}$. 

Let $\cV = \bigoplus_m \cL_m$ be a vector bundle over $M$ and $i\colon Z \hookrightarrow M$ a smooth subvariety defined by a section of $\cV$. We further assume that $\cV$ is \emph{convex}, i.e., if $H^1 (\bP^1, \varphi^\ast \cV) = 0$ for any holomorphic map $\varphi \colon \bP^1 \to M$. For $t = t^{(0)} + t^{(2)}\in H^0(M) \oplus H^2(M)$, we define the \emph{$I$-function} of $Z \hookrightarrow M$ to be the generating function
\begin{align*}
    I^Z(t, z) &\coloneqq e^{t/z} \sum_{\beta \ge 0} I_\beta^Z(z) P^\beta \\
    &\coloneqq e^{t/z} \sum_{\beta \ge 0} \bigl(J_\beta^M(z^{-1}) \cdot \cV^{\overline{\beta}}\bigr) P^\beta \\
    &\coloneqq e^{t/z} \sum_{\beta \ge 0} \Bigl(J_\beta^M(z^{-1}) \cdot \prod_m (c_1(\cL_m))^{\overline{(c_1(\cL_m), \beta)}}\Bigr) P^\beta, 
\end{align*}
where $P^\beta = e^{(t^{(2)},\beta)}$ and 
\[(D)^{\overline{n}} \coloneqq \prod_{\nu=1}^n (D + \nu z) \coloneqq \frac{\prod_{\nu = -\infty}^n (D + \nu z)}{\prod_{\nu = -\infty}^0 (D + \nu z)}. \]
The following classical mirror theorem allows us to compute the GW invariants of $Z \subseteq M$. 
\begin{theorem}[\cite{LLY99, CG07}]\label{thr:I=J}
    Suppose that $c_1(M) \ge c_1(\cV) = \sum_m c_1(\cL_m)$. Then 
    \[i_*J^Z(\tau, z^{-1}) = e(\cV)\cdot I^Z(t, z) \]
    under some mirror transform $\tau = \tau(t)$, where $i_*Q^\beta = Q^{i_*\beta}$ for $\beta \in \operatorname{NE}(Z)$. More precisely, $\tau^{(0)} = t^{(0)} + zg^{(0)}(P)$, $\tau^{(2)} = t^{(2)} + g^{(2)}(P)$ for some power series $g^{(0)}$, $g^{(2)}$ that is fully determined by $I^Z$. 
\end{theorem}

When $M$ is a weighted projective space with singularities, we can still use 

\begin{theorem}[\cite{CCLT09,Prz07}]\label{thr:wtI=J}
    Suppose that $M$ is a weighted projective space $\bP(\alpha_1, \dots, \alpha_{k+1})$. Then Theorem~\ref{thr:I=J} holds if $Z$ is a smooth complete intersection of hypersurfaces $Z_1$, $\dots$, $Z_r$ which do not intersect the singular locus of $M$ with 
    \[J_\beta^M(z^{-1}) \coloneqq \frac{1}{\prod_{n=1}^{k+1} (\alpha_n h)^{\overline{(\alpha_nh, \beta)}}}, \quad \cV = \bigoplus_{m=1}^r \cO(Z_m), \]
    where $h = c_1(\cO(1))$ is the hyperplane class of $M$. 
\end{theorem}

We will use the following theorem due to Brown, which give us the $I$-function of projective bundle $\bP_M(\cW \oplus \cO)$ for $\cW = \bigoplus_m \cL_m$ a split vector bundle. 
\begin{theorem}[{\cite[Corollary~2]{Bro14}}] \label{thm:Brown}
    Let $\gamma$ be the fiber curve class of $\pi\colon \bP_M(\cW \oplus \cO) \to M$, and let $\xi = c_1(\cO_{\bP_M(\cW\oplus \cO)}(1)) = [\bP_M(\cW)]$ be the infinity hyperplane class. Assume that $-c_1(\cL_m) \ge 0$ for all $m$ and $c_1(M) + c_1(\cW) \ge 0$. Then 
    \[J^{\bP_M(\cW \oplus \cO)}(\tau, z^{-1}) = I^{\bP_M(\cW \oplus \cO)}(t, z) \]
    under some mirror transform $\tau = \tau(t)$, where 
    \begin{equation}\label{eqbrown}
        I^{\bP_M(\cW \oplus \cO)}(t, z) = e^{t/z} \sum_{\beta} \pi^*J_\beta^M(z^{-1}) \sum_{n \ge 0}\frac{1}{\prod_m (\xi + \pi^*c_1(\cL_m))^{\overline{n + (c_1(\cL_m), \beta)}}} P^{j_*\beta + n\gamma}, 
    \end{equation}
    and $j\colon M \to \bP_{M}(\cO) \subseteq \bP_M(\cW \oplus \cO)$ is the zero section. 
\end{theorem}

\begin{lemma}\label{lemma:brownI}
    With the same assumption and notation as Theorem~\ref{thm:Brown}, let $A(z^{-1}) = \sum_{\beta} A_\beta(z^{-1}) Q^\beta \in \bC[z^{-1}]\PSR{Q^{\operatorname{NE}(M)}}$
    be an invertible element satisfies $A_0 = 1$ and $A_\beta \neq 0$ only if $(c_1(\cL_m), \beta) = 0$ for all $m$. Define $I^M(z^{-1}) = e^{t/z} \sum I_\beta^M(z^{-1}) Q^\beta$ via the equation 
    \[I^M(z^{-1}) = A(z^{-1})J^M(z^{-1}). \]
    Then $\sum_\beta A_\beta(z^{-1}) P^{j_*\beta} I^{\bP_M(\cW \oplus \cO)}(t, z)$ is equal to the RHS of \eqref{eqbrown}, with $J_\beta^M$ replaced by $I_\beta^M$. 
\end{lemma}

\begin{proof}
    The proof is a direct computation. For each $\beta$, denote
    \[\cW_{\overline{\beta}} = \sum_{n\ge 0} \frac{1}{\prod_m (\xi + \pi^*c_1(\cL_m))^{\overline{n + (c_1(\cL_m), \beta)}}} P^{n\gamma}.\]
    The assumption implies that $V_{\overline{\beta}} = V_{\overline{\beta+\beta'}}$ if $A_{\beta'} \neq 0$. Thus 
    \begin{align*}
        \sum_\beta A_\beta(z^{-1}) P^{j_*\beta}\cdot  I^{\bP_M(\cW \oplus \cO)}(t, z) &=  \sum_{\beta_1} A_{\beta_1}(z^{-1}) P^{j_*\beta_1}\cdot e^{t/z} \sum_{\beta_2} \pi^*J_\beta^M(z^{-1}) \cW_{\overline{\beta_2}} P^{j_*\beta_2} \\
        &= e^{t/z} \sum_{\beta_1, \beta_2} \pi^*(A_{\beta_1}(z^{-1})J_{\beta_2}^M(z^{-1}))\cW_{\overline{\beta_1 + \beta_2}} P^{j_*(\beta_1 + \beta_2)} \\
        &= e^{t/z} \sum_{\beta} \pi^* I_\beta^M(z^{-1}) \cW_{\overline{\beta}} P^{j_*\beta}. 
    \end{align*}
    This completes the proof. 
\end{proof}

Since $A(z^{-1}) I^{\bP_M(\cW \oplus \cO)}(t, z) = J^{\bP_M(\cW \oplus \cO)}(\tau, z^{-1})$ under some mirror transform, we can compute the Gromov--Witten invariants of $\bP_M(\cW \oplus \cO)$ from $I^M$ using this modified ``$I$-function''. 

\section{Canonical Local models}\label{sec;can_loc_model}

In this section, we will omit the Novikov variables $q^\beta$ since all the summation will be convergent.

\subsection{Geometric realizations}\label{subsec;geom_real}

We shall recall realizations of del Pezzo threefolds $X_d$ (Appendix \ref{sec;appdP}), and the del Pezzo transition $Y_d \searrow X_d$ (Example \ref{ex:dPtrans}).


For $1 \leq d \leq 8$, we denote $X_d$ with the polarization $\cO_{X_d}(1)$ by the del Pezzo threefold of degree $d$ appearing in Theorem \ref{thm;dP3claf}. Therefore the zero locus $S_d$ defined by a geneal global section of $\cO_{X_d} (1)$ is a smooth del Pezzo surface of degree $d$.

Using Theorem~\ref{thm;dP3claf}, we are going to define an ambient space $\T_d$ containing $X_d$, and associate to $X_d$ a vector bundles $\cV_d$. 
Recall that when $d = 6$ there are two cases, \eqref{thm;dP3claf_6I} and \eqref{thm;dP3claf_6II}, which we denote $X_{6\I}$ and $X_{6\II}$, respectively.

\begin{proposition}
There exist a smooth polarized variety $(\T_d, \cO_{\T_d}(1))$ and a vector bundle $\cV_d$ on $\T_d$ such that $X_d$ is the zero locus of a general global section of $\cV_d$ and $\cO_{X_d} (1) = \cO_{\T_d} (1)|_{X_d}$. All of the $\T_d$ and $\cV_d$ are given in Table \ref{tab;ambXd}. Here we define $\cO_{\T_8}(\frac12) \coloneqq \cO_{\bP^4}(1)$ and follow the notation of Remark \ref{rmk;OX7_1}. 
\end{proposition}

\begin{table}[H]
    \centering
    \begin{tabular}{ccc}
        \toprule
        $d$ & $(\T_d, \cO_{\T_d}(1))$ & $\cV_d$ \\
        \midrule
        $1$ & $(\bP(3,2,1,1,1), \cO(1))$ & $\cO_{\T_1}(6)$ \\ 
        \midrule
        $2$ & $(\bP(2,1,1,1,1), \cO(1))$ & $\cO_{\T_2}(4)$ \\ 
        \midrule
        $3$ & $(\bP^4, \cO(1))$ & $\cO_{\T_3}(3)$ \\ 
        \midrule
        $4$ & $(\bP^5, \cO(1))$ & $\cO_{\T_4}(2)^{\oplus 2}$ \\ 
        \midrule
        $5$ & $(\operatorname{Gr}(2, 5), \cO(1))$ & $\cO_{\T_5}(1)^{\oplus 3}$  \\  
        \midrule
        $6\I$ & $(\bP^2 \times \bP^2, \cO(1,1))$ & $\cO_{\T_{6\I}}(1)$  \\ 
        \midrule
        $6\II$ & $(\bP^1 \times \bP^1 \times \bP^1, \cO(1,1,1))$ & $0$  \\
        \midrule
        $7$ & $(\Bl_{\pt} \bP^3, \cO(h_1 + h_2))$ & $0$  \\ 
        \midrule 
        $8$ & $(\bP^4, \cO(2))$ & $\cO_{\T_8}(\frac12)$ \\ 
        \bottomrule
    \end{tabular}
    \caption{Ambient spaces of del Pezzo threefolds.}
  \label{tab;ambXd}
\end{table}

\begin{remark}\label{rmk;dPWd}
Each del Pezzso surface $S_d$ is defined by a general global section of $\cV_d \oplus \cO_{\T_d} (1)$. We also have $\T_d = X_d$ for $d = 6\II$ and $7$. 
\end{remark}

Recall that $Y_d$ is the projective bundle $\bP_{S_d}(K_{S_d} \oplus \cO)$ with projection $\pi_d\colon Y_d \to S_d$.
Let $i_d\colon S_d \to X_d$ denote the inclusion and $j_d\colon S_d \to Y_d$ the infinity section. These maps fit into the following diagram: 
\[\begin{tikzcd}
    Y_d \ar[d, "\pi_d"] \\
    S_d \ar[r, hook, "i_d"] \ar[u, "j_d", bend left = 30] & X_d 
\end{tikzcd}\]
We denote by $h$ the restriction of $c_1(\cO_{\T_d}(1))$ on $X_d$. We denote by $E$ (resp.~$H$) the zero section $\bP_{S_d}(\cO)$ (resp.~the infinity section $\bP_{S_d}(K_{S_d}) = j_d(S_d)$) of $Y_d$. Then $F \coloneqq H - E$ is the pullback $\pi_d^*i^*_dh$. 
It is easily seen that the first Chern classes of $X_d$ and $Y_d$ are
\begin{align}\label{eqn:c1}
    c_1(X_d) = 2h \quad \text{and} \quad c_1(Y_d) = 2H,
\end{align}
since $X_d$ is del Pezzo and $Y_d = \bP_{S_d}(K_{S_d} \oplus \cO)$.

\subsection{Gromov--Witten invariants for \texorpdfstring{$\bm{d \neq 7}$}{d neq 7}}\label{subsec;GW_neq_7}

In this subsection, we aim to compute the symmetrized GW invariants (\S \ref{subsec;symGW}) of both $X_d$ and $Y_d$, and investigate how they are related. To do this, we shall introduce the groups acting on their cohomology.

For each $d \in \{1,2,3,4,5,8\}$, it is clear that $H(X_d) = \bC[h]/(h^4)$, so we set the group $G_{X_d}$ to be trivial. We define $G_{X_{6\I}} \coloneqq \fS_2$ and $G_{X_{6\II}} \coloneqq \fS_3$. The group action of $G_{X_{6\I}}$ (resp.~$G_{X_{6\II}}$) on $\T_{6\I} = \bP^2 \times \bP^2$ (resp.~$\T_{6\II} = \bP^1 \times \bP^1 \times \bP^1$) induces automorphisms on $H(X_d)$ such that 
\[H(X_d)^{G_{X_d}} = \bC[h]/(h^4) = \bC\cdot 1 \oplus \bC \cdot h \oplus \bC \cdot h^2 \oplus \bC \cdot \pt. \] 

Since $d \neq 7$, Proposition~\ref{prop:defonSd} tells us that there is a group action of $G_{S_d}$ on $H(S_d)$, and hence on $H(Y_d)$, such that 
\begin{align*}
    H(S_d)^{G_{S_d}} &= \bC[i_d^*h] / ((i_d^*h)^3) = \bC \cdot 1 \oplus \bC \cdot i_d^*h \oplus \bC\cdot \pt, \\
    H(Y_d)^{G_{S_d}} &= \bC[E, H]/(EH, F^3) = \bC \cdot 1 \oplus \bC \cdot E \oplus \bC \cdot H \oplus \bC \cdot H^2 \oplus \bC \cdot E^2 \oplus \bC \cdot \pt.
\end{align*} 

\begin{remark}
    The groups $G_{X_{6\I}}$ and $G_{X_{6\II}}$ are the Weyl group of $A_1$ and $A_2$, respectively---the two components of the reducible root system $E_3 = A_1 \times A_2$. 
\end{remark}


\begin{convention}\label{conv;omitG}
In the remainder of this section, we will omit the groups in the notation for simplicity. We will write 
\begin{itemize}
    \item $H(X_d)$, $H(S_d)$, and $H(Y_d)$ for $H(X_d)^{G_{X_d}}$, $H(S_d)^{G_{S_d}}$, and $H(Y_d)^{G_{S_d}}$, respectively; 
    \item $\operatorname{NE}(X_d)$ and $\operatorname{NE}(Y_d)$ for $\operatorname{NE}(X_d)/G_{X_d}$ and $\operatorname{NE}(Y_d)/G_{S_d}$, respectively; 
    \item $\langle\,\cdots\rangle^{X_d}$ and $\langle \,\cdots \rangle^{Y_d}$ for $\langle\, \cdots \rangle_{G_{X_d}}^{X_d}$ and $\langle \,\cdots \rangle_{G_{S_d}}^{Y_d}$, respectively. 
\end{itemize}
\end{convention}

Choose $\bar{\gamma} \in \operatorname{NE}(X_d)$ such that $(h, \bar{\gamma})_{X_d} = 1$. When $d \le 6$ (resp.~$d = 8$), $\bar{\gamma}$ (resp.~$2\bar{\gamma}$) generate the cone $\operatorname{NE}(X_d)$. 

Let $\gamma$ denote the fiber class of the projection $Y_d \to S_d$. Choose $\ell \in \operatorname{NE}(Y_d)$ such that $(H, \ell)_{Y_d} = 0$, $(F, \ell)_{Y_d} = 1$. When $d \le 6$ (resp.~$d = 8$), $\gamma$ and $\ell$ (resp.~$2\ell$) generate the cone $\operatorname{NE}(Y_d)$.

\subsubsection{Picard--Fuchs equations}\label{subsubsec;PF_neq_7}

From now on, we will work within the invariant subspaces, and all GW invariants are symmetrized (cf.~\eqref{eq:syminv}).

Write $\cW_d = \cV_d \oplus \cO_{\T_d}(1)$. It follows from Theorems~\ref{thr:I=J},~\ref{thr:wtI=J}~and~\ref{thm:Brown}, Lemma~\ref{lemma:brownI} and Remark~\ref{rmk;dPWd} that for each $d \le 6$,
\begin{align*}
    J^{X_d}(\tau, z^{-1}) = I^{X_d}(t, z) \coloneqq{}&  e^{t/z}\sum_{m\ge 0} \bigl(J_{m\bar{\gamma}}^{\T_d}(z^{-1}) \cdot \cV_d^{\overline{m\bar{\gamma}}}\bigr) P^{m\bar{\gamma}}, \\
    J^{Y_d}(\tau, z^{-1}) = I^{Y_d}(t, z) \coloneqq{}& e^{t/z}\sum_{m, n \ge 0} \Bigl(\pi_d^*i_d^*\bigl(J_{m\bar{\gamma}}^{\T_d}(z^{-1})\cdot \cW_d^{\overline{m\bar{\gamma}}}\bigr) \tfrac{1}{(E)^{\overline{n-m}}(H)^{\overline{n}}}\Bigr) P^{m\ell + n\gamma} \\
    ={}& e^{t/z} \sum_{m, n\ge 0} \Bigl(\pi_d^*i_d^* I_{m\bar{\gamma}}^{X_d}(z)\cdot \tfrac{(F)^{\overline{m}}}{(E)^{\overline{n-m}}(H)^{\overline{n}}}\Bigr) P^{m\ell + n\gamma}
\end{align*}
under some mirror transform; for $d = 8$, we have
\begin{align*}
    J^{X_8}(\tau, z^{-1}) = I^{X_8}(t, z) \coloneqq{}&  e^{t/z}\sum_{n\ge 0} \bigl(J_{m(2\bar{\gamma})}^{\T_8}(z^{-1}) \cdot \cV_d^{\overline{m(2\bar{\gamma})}}\bigr) P^{m(2\bar{\gamma})}, \\
    J^{Y_8}(\tau, z^{-1}) = I^{Y_8}(t, z) \coloneqq{}& e^{t/z}\sum_{m, n \ge 0} \Bigl(\pi_8^*i_8^*\bigl(J_{m(2\bar{\gamma})}^{\T_8}(z^{-1})\cdot \cW_8^{\overline{m(2\bar{\gamma})}}\bigr) \tfrac{1}{(E)^{\overline{n-2m}}(H)^{\overline{n}}}\Bigr) P^{m(2\ell) + n\gamma} \\
    ={}& e^{t/z} \sum_{m, n\ge 0} \Bigl(\pi_8^*i_8^* I_{m(2\bar{\gamma})}^{X_8}\cdot \tfrac{(\frac12 F)^{\overline{m}}}{(E)^{\overline{n-2m}}(H)^{\overline{n}}}\Bigr) P^{m(2\ell) + n\gamma} 
\end{align*}
under some mirror transform.

For a divisor $D$, we define its quantization operator $\hat{D}$ via  
\[\hat{D}(e^{t/z} P^\beta) = (D + (D, \beta)z) e^{t/z} P^\beta, \]
i.e., the derivative $z \pdv{}{t_D}$ along the $D$-direction. 
\begin{lemma}\label{lm:IXd=1234}
    For $d = 1$, $2$, $3$, $4$, let $(\kappa_d, \lambda_d) = (432, 60)$, $(64, 12)$, $(27, 6)$, $(16, 4)$, respectively. Then $I^{X_d}(t)$ satisfies the Picard--Fuchs equation $\square_{\bar{\gamma}}I^{X_d} = 0$, where  
    \[\square_{\bar{\gamma}} = \hat{h}^4 - P^{\bar{\gamma}} (\kappa_d \hat{h}^2 + \kappa_d z\hat{h} + \lambda_d z^2). \]
\end{lemma}

\begin{proof}
    For $d \in \{1,2,3\}$, $X_d$ is a hypersurface of degree $n_d = \sum \alpha_i-1$ in $\bP(\alpha,1)$. By Theorem~\ref{thr:wtI=J},  
    \[I_{m\bar{\gamma}}^{X_d} = \frac{(n_d h)^{\overline{n_dm}}}{(h)^{\overline{m}}\prod_{i=0}^3 (\alpha_i h)^{\overline{\alpha_i m}}}. \] 
    It is straightforward to verify case by case that 
    \begin{align*}
        \frac{I_{(m+1)\bar{\gamma}}^{X_d}}{I_{m\bar{\gamma}}^{X_d}} &= \frac{n_d^{n_d-2}}{\prod_{1 < e\mid n_d} (n_d/e)^e}\frac{(n_dh + (n_dm+1)z)(n_dh + (n_dm+(n_d-1))z)}{(h + (m+1)z)^4} \\
        &= \kappa_d \cdot \frac{(h + (m+\frac{1}{n_d})z)(h + (m+\frac{n_d-1}{n_d})z)}{(h + (m+1)z)^4} = \frac{\kappa_d (h+mz)^2 + \kappa_d z(h+mz) + \lambda_d z^2}{(h+(m+1)z)^4}. 
    \end{align*}
    Consequently, $\square_{\bar{\gamma}} I^{X_d} = 0$. 

    In the case $d = 4$, it follows from Theorem~\ref{thr:I=J} that 
    \[I_{m\bar{\gamma}}^{X_d} = \frac{((2h)^{\overline{2m}})^2}{((h)^{\overline{m}})^6}. \]
    So 
    \[\frac{I_{(m+1)\bar{\gamma}}^{X_d}}{I_{m\bar{\gamma}}^{X_d}} = \frac{16(h+(m+\frac12)z)^2}{(h+(m+1)z)^4} = \frac{\kappa_4 (h+mz)^2 + \kappa_4 z(h+mz) + \lambda_4 z^2}{(h+(m+1)z)^4}, \]
    and hence, $\square_{\bar{\gamma}} I^{X_d} = 0$. 
\end{proof}

\begin{lemma}\label{lm:IXd=56}
    For $d = 5$, $6\I$, $6\II$, let $(\kappa_d, \lambda_d, \mu_d) = (11, 3, 1)$, $(7, 2, 8)$, $(10, 3, -9)$, respectively.  
    Then $I^{X_d}(t)$ satisfies the Picard--Fuchs equation $\square_{\bar{\gamma}}I^{X_d} = 0$, where  
    \[\square_{\bar{\gamma}} = \hat{h}^4 - P^{\bar{\gamma}} (\kappa_d \hat{h}^2 + \kappa_d z\hat{h} + \lambda_d z^2) - \mu_d P^{2\bar{\gamma}}. \]
\end{lemma}

\begin{proof}
    We prove the statement case by case. In the cases $d = 6\I$ and $6\II$, we will apply a symmetrization of the Picard--Fuchs equations.
    
    For $d = 5$, it follows from \cite[p.659]{BCKvS98} that $J^{\operatorname{Gr}(2, 5)}(t)$ satisfies the Picard--Fuchs equation $\square_{\bar{\gamma}}^{\operatorname{Gr}(2,5)} J^{\operatorname{Gr}(2, 5)} = 0$, where 
    \[\square_{\bar{\gamma}}^{\operatorname{Gr}(2,5)} = \hat{h}^7 (\hat{h} - z)^3 - P^{\bar{\gamma}} \hat{h}^3 (11 \hat{h}^2 + 11z\hat{h} + 3z^2) - P^{2\bar{\gamma}}, \]
    i.e., for each $m \ge 0$, 
    \begin{align*}
        &(h + (m+1)z)^7(h+mz)^3 J_{(m+1)\bar{\gamma}}^{\operatorname{Gr}(2, 5)} \\
        &\quad = (h+mz)^3(11(h+mz)^2 + 11z(h+mz) + 3z^2) J_{m\bar{\gamma}}^{\operatorname{Gr}(2, 5)} + J_{(m-1)\bar{\gamma}}^{\operatorname{Gr}(2, 5)}. 
    \end{align*}
    Together with $I_{m\bar{\gamma}}^{X_5} = J_{m\bar{\gamma}}^{\operatorname{Gr}(2, 5)}((h)^{\overline{m}})^3$, we get 
    \[(h+(m+1)z)^4 I_{(m+1)\bar{\gamma}}^{X_5} = (11(h+mz)^2 + 11z(h+mz) + 3z^2) I_{m\bar{\gamma}}^{X_5} + I_{(m-1)\bar{\gamma}}^{X_5}. \]
    Hence, $I^{X_5}$ satisfies the Picard--Fuchs equation $\square_{\bar{\gamma}}I^{X_5} = 0$. 
    
    For $d = 6\I$, let $h_1 = c_1(\cO(1, 0))$, $h_2 = c_1(\cO(0, 1))$ be the hyperplane classes of $\T_{6\I} = \bP^2\times \bP^2$, and let $\bar{\gamma}_1 = h_1h^2$, $\bar{\gamma}_2 = h_1^2h_2$ be the generators of $\operatorname{NE}(\bP^2\times \bP^2)$. Then $I^{X_d}$ satisfies the Picard--Fuchs equations $\square_{\bar{\gamma}_1} I^{X_d} = \square_{\bar{\gamma}_2} I^{X_d} = 0$, where 
    \begin{align*}
        \square_{\bar{\gamma}_1} &= \hat{h}_1^3 - (\hat{h}_1 + \hat{h}_2) P^{\bar{\gamma}_1}, \\
        \square_{\bar{\gamma}_2} &= \hat{h}_2^3 - (\hat{h}_1 + \hat{h}_2) P^{\bar{\gamma}_2}. 
    \end{align*}
    Define
    \[\square_{\mathrm{top}} = \hat{h}_1^2 - \hat{h}_1\hat{h}_2 + \hat{h}_2^2 - (P^{\bar{\gamma}_1} + P^{\bar{\gamma}_2}). \]
    Then $\square_{\bar{\gamma}_1} + \square_{\bar{\gamma}_2} = (\hat{h}_1 + \hat{h}_2)\square_{\mathrm{top}}$. One can also verify that $I^{X_d}$ is annihilated by $\square_{\mathrm{top}}$ as well. 
    
    We want to symmetrize the differential equations. Let $\hat{h} = \hat{h}_1 + \hat{h}_2$ and let $\hat{k} = \hat{h}_1 - \hat{h}_2$. Then  
    \begin{align*}
        \square_{\bar{\gamma}_1} &= \tfrac18 (\hat{h} + \hat{k})^3 - \hat{h} P^{\bar{\gamma}_1} = \tfrac18 \hat{h}^3 + \tfrac38 \hat{h}^2\hat{k} + \tfrac38 \hat{h}\hat{k}^2 + \tfrac18 \hat{k}^3 - \hat{h} P^{\bar{\gamma}_1}, \\
        \square_{\bar{\gamma}_2} &= \tfrac18 (\hat{h} - \hat{k})^3 - \hat{h} P^{\bar{\gamma}_2} = \tfrac18 \hat{h}^3 - \tfrac38 \hat{h}^2\hat{k} + \tfrac38 \hat{h}\hat{k}^2 - \tfrac18 \hat{k}^3 - \hat{h} P^{\bar{\gamma}_2}, \\
        \square_{\mathrm{top}} &= \tfrac14 ((\hat{h} + \hat{k})^2 - (\hat{h} + \hat{k})(\hat{h} - \hat{k}) + (\hat{h} - \hat{k})^2) - (P^{\bar{\gamma}_1} + P^{\bar{\gamma}_2}) \\
        &= \tfrac14 (\hat{h}^2 + 3\hat{k}^2) -  (P^{\bar{\gamma}_1} + P^{\bar{\gamma}_2}). 
    \end{align*}
    In order to cancel out the $\hat{k}^3$ term, we consider 
    \[3(\square_{\bar{\gamma}_1} - \square_{\bar{\gamma}_2}) - \hat{k}\square_{\mathrm{top}} = 2\hat{h}^2\hat{k} - 3\hat{h} (P^{\bar{\gamma}_1} - P^{\bar{\gamma}_2}) + \hat{k}(P^{\bar{\gamma}_1} + P^{\bar{\gamma}_2}). \]
    So, modulo the left ideal generated by $\square_{\bar{\gamma}_1}$, $\square_{\bar{\gamma}_2}$, and $\square_{\mathrm{top}}$, we get 
    \begin{align*}
        \hat{h}^4 &= \hat{h}^2(4(P^{\bar{\gamma}_1} + P^{\bar{\gamma}_2}) - 3\hat{k}^2) \\
        &= 4(P^{\bar{\gamma}_1} + P^{\bar{\gamma}_2})(\hat{h} + z)^2 - \tfrac{3}{2}\hat{k}(3 \hat{h} (P^{\bar{\gamma}_1} - P^{\bar{\gamma}_2}) - \hat{k}(P^{\bar{\gamma}_1} + P^{\bar{\gamma}_2})) \\
        &= 4(P^{\bar{\gamma}_1} + P^{\bar{\gamma}_2})(\hat{h} + z)^2 - \tfrac92 (P^{\bar{\gamma}_1} (\hat{k} + z) - P^{\bar{\gamma}_2} (\hat{k} - z))(\hat{h} + z) \\
        &\quad + \tfrac{3}{2}(P^{\bar{\gamma}_1}(\hat{k} + z)^2 + P^{\bar{\gamma}_2}(\hat{k} - z)^2) \\
        &= 4(P^{\bar{\gamma}_1} + P^{\bar{\gamma}_2})(\hat{h} + z)^2 - \tfrac92 (P^{\bar{\gamma}_1} - P^{\bar{\gamma}_2})\hat{k} (\hat{h} + z) - \tfrac92 (P^{\bar{\gamma}_1} + P^{\bar{\gamma}_2})z(\hat{h} + z) \\
        &\quad + \tfrac{1}{2}(P^{\bar{\gamma}_1} + P^{\bar{\gamma}_2})(4(P^{\bar{\gamma}_1} + P^{\bar{\gamma}_2})-\hat{h}^2) + \tfrac{3}{2}(P^{\bar{\gamma}_1} + P^{\bar{\gamma}_2})z^2. 
    \end{align*}
    Restrict this equation on the diagonal $P^{\bar{\gamma}} = P^{\bar{\gamma}_1} = P^{\bar{\gamma}_2}$, we get 
    \begin{align*}
        \hat{h}^4 &= 8P^{\bar{\gamma}}(\hat{h} + z)^2 - 9P^{\bar{\gamma}}z(\hat{h} + z) + P^{\bar{\gamma}}(8P^{\bar{\gamma}} - \hat{h}^2) + 3 P^{\bar{\gamma}} z^2 \\
        &= P^{\bar{\gamma}}(7\hat{h}^2 + 7z\hat{h} + 2z^2) + 8P^{2\bar{\gamma}}, 
    \end{align*}
    which is what we want. 

    For $d = 6\II$, let $h_1 = c_1(\cO(1, 0, 0))$, $h_2 = c_1(\cO(0, 1, 0))$, $h_3 = c_1(\cO(0, 0, 1))$ be the hyperplane classes of $\T_{6\II} = \bP^1\times \bP^1 \times \bP^1$, and let $\bar{\gamma}_1 = h_2h_3$, $\bar{\gamma}_2 = h_3h_1$, $\bar{\gamma}_3 = h_1h_2$ be the generators of $\operatorname{NE}(\bP^1\times \bP^1 \times \bP^1)$. Then $I^{X_d}$ satisfies the Picard--Fuchs equations $\square_{\bar{\gamma}_1} I^{X_d} = \square_{\bar{\gamma}_2} I^{X_d} = \square_{\bar{\gamma}_3} I^{X_d} = 0$, where
    \[\square_{\bar{\gamma}_i} = \hat{h}_i^2 - P^{\bar{\gamma}_i}. \]
    Let $\hat{h} = \hat{h}_1 + \hat{h}_2 + \hat{h}_3$, $\hat{k}_+ = \hat{h}_1 + \omega \hat{h}_2 + \omega^2 \hat{h}_3$, $\hat{k}_- = \hat{h}_1 + \omega^2 \hat{h}_2 + \omega \hat{h}_3$, where $\omega = \omega_3 = e^{2\pi \ii/3}$ is a third root unity. Then 
    \begin{align*}
        \square_{\bar{\gamma}_1} + \square_{\bar{\gamma}_2} + \square_{\bar{\gamma}_3} &= \tfrac13 \hat{h}^2 + \tfrac23\hat{k}_+ \hat{k}_- - (P^{\bar{\gamma}_1} + P^{\bar{\gamma}_2} + P^{\bar{\gamma}_3}), \\
        \square_{\bar{\gamma}_1} + \omega\square_{\bar{\gamma}_2} + \omega^2\square_{\bar{\gamma}_3} &= \tfrac13 \hat{k}_-^2 + \tfrac23\hat{h}\hat{k}_+ - (P^{\bar{\gamma}_1} + \omega P^{\bar{\gamma}_2} + \omega^2 P^{\bar{\gamma}_3}), \\
        \square_{\bar{\gamma}_1} + \omega^2\square_{\bar{\gamma}_2} + \omega \square_{\bar{\gamma}_3} &= \tfrac13 \hat{k}_+^2 + \tfrac23\hat{h} \hat{k}_- - (P^{\bar{\gamma}_1} + \omega^2 P^{\bar{\gamma}_2} + \omega P^{\bar{\gamma}_3}). 
    \end{align*}
    So, modulo the left ideal generated by $\square_{\bar{\gamma}_1}$, $\square_{\bar{\gamma}_2}$, and $\square_{\bar{\gamma}_3}$, we get 
    \begin{align}
        \hat{h}^4 &= \hat{h}^2(3(P^{\bar{\gamma}_1} + P^{\bar{\gamma}_2} + P^{\bar{\gamma}_3}) - 2 \hat{k}_+\hat{k}_-) \notag\\
        &= 3(P^{\bar{\gamma}_1} + P^{\bar{\gamma}_2} + P^{\bar{\gamma}_3}) (\hat{h} + z)^2 - \hat{h}\hat{k}_+ (- \hat{k}_+^2 + 3(P^{\bar{\gamma}_1} + \omega^2 P^{\bar{\gamma}_2} + \omega P^{\bar{\gamma}_3})) \notag\\
        &= 3(P^{\bar{\gamma}_1} + P^{\bar{\gamma}_2} + P^{\bar{\gamma}_3}) (\hat{h} + z)^2 + \tfrac12 \hat{k}_+^2 ( - \hat{k}_-^2 + 3(P^{\bar{\gamma}_1} + \omega P^{\bar{\gamma}_2} + \omega^2 P^{\bar{\gamma}_3})) \notag\\
        &\quad - 3(P^{\bar{\gamma}_1}(\hat{k}_+ + z) + \omega^2 P^{\bar{\gamma}_2}(\hat{k}_+ + \omega z) + \omega P^{\bar{\gamma}_3}(\hat{k}_+ + \omega^2 z)) (\hat{h} + z) \notag\\
        &= 3(P^{\bar{\gamma}_1} + P^{\bar{\gamma}_2} + P^{\bar{\gamma}_3}) (\hat{h} + z)^2 - \tfrac12 \hat{k}_+^2 \hat{k}_-^2  \notag\\
        &\quad + \tfrac32 (P^{\bar{\gamma}_1}(\hat{k}_+ + z)^2 + \omega  P^{\bar{\gamma}_2}(\hat{k}_+ + \omega z)^2 + \omega^2 P^{\bar{\gamma}_3}(\hat{k}_+ + \omega^2 z)^2) \notag\\
        &\quad - 3(P^{\bar{\gamma}_1} + \omega^2 P^{\bar{\gamma}_2} + \omega P^{\bar{\gamma}_3})\hat{k}_+(\hat{h}+z) - 3z(P^{\bar{\gamma}_1} + P^{\bar{\gamma}_2} + P^{\bar{\gamma}_3})(\hat{h}+z). \label{eqDmodX6II}
    \end{align}
    Since 
    \begin{align*}
        \hat{k}_+^2\hat{k}_-^2 &= \tfrac12 \hat{k}_+\hat{k}_- ( - \hat{h}^2 + 3(P^{\bar{\gamma}_1} + P^{\bar{\gamma}_2} + P^{\bar{\gamma}_3})) \\
        &= -\tfrac14 \hat{h}^2 ( - \hat{h}^2 + 3(P^{\bar{\gamma}_1} + P^{\bar{\gamma}_2} + P^{\bar{\gamma}_3})) \\
        &\quad + \tfrac32 (P^{\bar{\gamma}_1}(\hat{k}_+ + z)(\hat{k}_- + z) + P^{\bar{\gamma}_2}(\hat{k}_+ + \omega z)(\hat{k}_- + \omega^2 z) + P^{\bar{\gamma}_3}(\hat{k}_+ + \omega^2 z)(\hat{k}_- + \omega z))\\
        &= \tfrac14 \hat{h}^4 - \tfrac34 (P^{\bar{\gamma}_1} + P^{\bar{\gamma}_2} + P^{\bar{\gamma}_3}) (\hat{h} + z)^2 \\
        &\quad + \tfrac34 (P^{\bar{\gamma}_1} + P^{\bar{\gamma}_2} + P^{\bar{\gamma}_3}) ( - \hat{h}^2 + 3(P^{\bar{\gamma}_1} + P^{\bar{\gamma}_2} + P^{\bar{\gamma}_3})) \\
        &\quad + \tfrac32 (P^{\bar{\gamma}_1} + \omega^2 P^{\bar{\gamma}_2} + \omega P^{\bar{\gamma}_3}) z\hat{k}_+ + \tfrac32 (P^{\bar{\gamma}_1} + \omega P^{\bar{\gamma}_2} + \omega^2 P^{\bar{\gamma}_3}) z\hat{k}_- + \tfrac32  (P^{\bar{\gamma}_1} + P^{\bar{\gamma}_2} + P^{\bar{\gamma}_3}) z^2,  
    \end{align*}
    it follows that the restriction of \eqref{eqDmodX6II} on the diagonal $P^{\bar{\gamma}} = P^{\bar{\gamma}_1} = P^{\bar{\gamma}_2} = P^{\bar{\gamma}_3}$ is 
    \begin{align*}
        \hat{h}^4 &= 9 P^{\bar{\gamma}}(\hat{h}+z)^2 - \tfrac12 (\tfrac14 \hat{h}^4 - \tfrac94 P^{\bar{\gamma}}(\hat{h}+z)^2 - \tfrac94 P^{\bar{\gamma}} \hat{h}^2 + \tfrac {81}4 P^{2\bar{\gamma}} + \tfrac92 P^{\bar{\gamma}} z^2) \\
        &\quad + \tfrac 92 P^{\bar{\gamma}} z^2 - 9 P^{\bar{\gamma}} z(\hat{h} + z),  
    \end{align*}
    or, equivalently, 
    \[\hat{h}^4 = P^{\bar{\gamma}}(10 \hat{h}^2 + 10 z\hat{h} + 3z^2) - 9 P^{2\bar{\gamma}}. \qedhere\]
\end{proof}

\begin{lemma}\label{lm:IXdtoIYd}
    If $I^X(t)$ satisfies the Picard--Fuchs equation $\square_{\bar{\gamma}} I^X = 0$, where
    \[\square_{\bar{\gamma}} = \hat{h}^4 - P^{\bar{\gamma}}(\kappa \hat{h}^2 + \kappa z \hat{h} + \lambda z^2) - \mu P^{2\bar{\gamma}}, \]
    and $I_{m\ell + n\gamma}^Y = \pi^*i^* I_{m\bar{\gamma}}^X\cdot \frac{(F)^{\overline{m}}}{(E)^{\overline{n-m}}(H)^{\overline{n}}} $, then 
    \[I^Y(t) = e^{t/z} \sum_{m, n\ge 0} I_{m\ell + n\gamma}^Y P^{m\ell + n\gamma}\]
    satisfies the Picard--Fuchs equations $\square_\gamma I^Y = \square_\ell I^Y = 0$, where 
    \begin{align*}
        \square_\gamma &= \hat{E}\hat{H} - P^\gamma, \\
        \square_\ell &= \hat{F}^3 - P^\ell (\kappa \hat{F}^2 + \kappa z\hat{F} + \lambda z^2) \hat{E} - \mu P^{2\ell} (\hat{F} + z) \hat{E} (\hat{E} - z). 
    \end{align*}
\end{lemma}

\begin{proof}
    It follows from 
    \[\frac{I_{m\ell + (n-1)\gamma}^Y}{I_{m\ell + n\gamma}^Y} = \frac{{(E)^{\overline{n-m}}(H)^{\overline{n}}}}{{(E)^{\overline{(n-1)-m}}(H)^{\overline{n-1}}}} = (E + (n-m)z)(H + nz) \]
    that $\square_\gamma I^Y = 0$. The equation $\square_{\bar{\gamma}} I^{X_d} = 0$ gives 
    \[(h+(n+1)z)^4 I_{(n+1)\bar{\gamma}}^X = (\kappa(h+nz)^2 + \kappa z(h+nz) + \lambda z^2) I_{n\bar{\gamma}}^X + \mu I_{(n-1)\bar{\gamma}}^X. \]
    This implies 
    \begin{align*}
        (F + (m+1)z)^3 I_{(m+1)\ell + n\gamma}^Y &= (\kappa (F+mz)^2 + \kappa z(F+mz) + \lambda z^2) (E + (n-m)z) I_{m\ell + n\gamma}^Y \\
        &\quad + \mu (F+mz)(E+(n-m)z)(E+(n-m+1)z) I_{(m-1)\ell + n\gamma}^Y,  
    \end{align*}
    or, equivalently, $\square_{\ell} I^{Y} = 0$. 
\end{proof}

\begin{lemma}\label{lm:IX8IY8}
    The $I$-function $I^{X_8}(t)$ satisfies the Picard--Fuchs equation $\square_{\bar{\gamma}}I^{X_8} = 0$, where  
    \[\square_{\bar{\gamma}} = \hat{h}^4 - 16 P^{2\bar{\gamma}}; \]
    $I^{Y_8}(t)$ satisfies the Picard--Fuchs equations $\square_\gamma I^{Y_8} = \square_\ell I^{Y_8} = 0$, where 
    \begin{align*}
        \square_\gamma &= \hat{E}\hat{H} - P^\gamma, \\
        \square_\ell &= \hat{F}^3 - 16 P^{2\ell} (\hat{F} + z) \hat{E} (\hat{E} - z). 
    \end{align*}
\end{lemma}

\begin{proof}
    Since $I_{n(2\bar{\gamma})}^{X_8} = \frac{1}{((\frac12 h)^{\overline{n}})^4}$, 
    \[\frac{I_{(n+1)(2\bar{\gamma})}^{X_8}}{I_{n(2\bar{\gamma})}^{X_8}} = \frac{1}{(\frac12 h+(n+1)z)^4} = \frac{16}{(h+2(n+1)z)^4}. \]
    This means that $\square_{\bar{\gamma}} I^{X_8} = 0$. Since $I_{m(2\ell) + n\gamma}^{Y_8} = \frac{(F)^{\overline{2m}}}{((\frac12 F)^{\overline{m}})^4 (E)^{\overline{n-2m}}(H)^{\overline{n}}}$, we have 
    \begin{align*}
        \frac{I_{m(2\ell) + (n-1)\gamma}^{Y_8}}{I_{m(2\ell) + n\gamma}^{Y_8}} &= (E + (n-2m)z)(H + nz) \\
        \frac{I_{(m+1)(2\ell) + n\gamma}^{Y_8}}{I_{m(2\ell) + n\gamma}^{Y_8}}  &= \frac{(F+(2m+2)z)(F+(2m+1)z)(E+(n-2m)z)(E+(n-2m-1)z)}{(\frac12 F+(m+1)z)^4} \\
        &= \frac{16((F+2mz)+z)(E+(n-2m)z)(E+(n-2m-1)z)}{(\frac12 F+(m+1)z)^3}, 
    \end{align*}
    or, equivalently, $\square_\gamma I^{Y_8} = \square_\ell I^{Y_8} = 0$. 
\end{proof}

Summarizing Lemmas~\ref{lm:IXd=1234}, \ref{lm:IXd=56}, \ref{lm:IXdtoIYd} and \ref{lm:IX8IY8}, we have:

\begin{proposition}\label{prop:PFI}
    For each $d \neq 7$, we have 
    \begin{align*}
        \square_{\bar{\gamma}} I^{X_d} = 0 \quad \text{and} \quad \square_\gamma I^{Y_d} = \square_\ell I^{Y_d} = 0.
    \end{align*}
    Here $\square_{\bar{\gamma}} = \hat{h}^4 - P^{\bar{\gamma}}(\kappa_d \hat{h}^2 + \kappa_d z\hat{h} + \lambda_d z^2) - \mu_d P^{2\bar{\gamma}}$ and 
    \begin{align*}
        \square_\gamma &= \hat{E}\hat{H} - P^\gamma, \\
        \square_\ell &= \hat{F}^3 - P^\ell (\kappa_d \hat{F}^2 + \kappa_d z\hat{F} + \lambda_d z^2) \hat{E} - \mu_d P^{2\ell} (\hat{F} + z) \hat{E} (\hat{E} - z), 
    \end{align*}
    where $\kappa_d$, $\lambda_d$ and $\mu_d$ are listed in Table \ref{table:klmd}.
\end{proposition}


\begin{remark}
The constants $(\kappa_d, \lambda_d, \mu_d)$ appeared in Zagier's list \cite{Zag09}, which was obtained by performing a computer search for integral sequences that satisfy a certain three-term recurrence relation. For example, when $d \in \{1,2,3,4,8\}$, the constants correspond to hypergeometric solutions (cf.~\cite[pp.352-353]{Zag09}). The case $d \in \{5, 6\I, 6\II\}$ correspond to the new label $\mathbf{D}$, $\mathbf{A}$, $\mathbf{C}$ in \cite[Table~2]{Zag09}, respectively. The solutions $\mathbf{A}$-$\mathbf{F}$ in \cite[Table~2]{Zag09} are now called Zagier's sporadic sequences (cf.~\cite[\S5.10]{Coo17}).
\end{remark}

\begin{rmk} \label{l:fail}
    In \cite{MS23}, it is shown that the quantum $D$-module of $\Yloc = Y_d$, after analytic continuation and restriction, recovers the quantum $D$-module of $\Xloc = X_d$, in the case of toric ambient spaces. We now compute the Picard--Fuchs ideal in our setting. Let $\tilde{\gamma} = \gamma + \ell$ be the lifting of $\gamma$ such that $(E, \tilde{\gamma}) = 0$. In the coordinates $(P^{\tilde{\gamma}}, P^{-\ell})$, we have
    \[\hat{F} = z\theta_{P^{\tilde{\gamma}}} - z\theta_{P^{-\ell}},\quad \hat{E} = z\theta_{P^{-\ell}},\quad \hat{H} = z\theta_{P^{\tilde{\gamma}}},  \]
    where $\theta_{P^{\tilde{\gamma}}} = P^{\tilde{\gamma}}\odv{}{P^{\tilde{\gamma}}}$, $\theta_{P^{-\ell}} = P^{-\ell}\odv{}{P^{-\ell}}$. In particular, $\hat F P^\gamma = P^\gamma \hat F$.
        
    Modulo the left ideal $\operatorname{PF}^{Y_d}$ generated by $\square_\gamma$ and $\square_\ell$, we have
    \begin{equation}\label{eq:FFFH}
    \begin{aligned}
        \hat{F}^3 \hat{H} &= P^\ell (\kappa_d \hat{F}^2 + \kappa_d z \hat{F} + \lambda_d z^2) \hat{E} \hat{H} + \mu_d P^{2\ell} (\hat{F} + z)(\hat{E} - z) \hat{E} \hat{H} \\
        &= P^{\tilde{\gamma}} (\kappa_d \hat{F}^2 + \kappa_d z \hat{F} + \lambda_d z^2) + \mu_d P^{\tilde{\gamma} + \ell} (\hat{F} + z) \hat{E}.
    \end{aligned}
    \end{equation}
    
    When $\mu_d = 0$, i.e., $d \in \{1,2,3,4\}$, the restriction of \eqref{eq:FFFH} to the locus $P^{-\ell} = 0$ is 
    \[\theta_{P^{\tilde{\gamma}}}^4 = z^2 P^{\tilde{\gamma}}(\kappa_d\theta_{P^{\tilde{\gamma}}}^2 + \kappa_d \theta_{P^{\tilde{\gamma}}} + \lambda_d)\]
    which is precisely the pullback of $\square_{\bar{\gamma}}$. However, when $\mu_d \neq 0$, i.e., $d \in \{5,6\I,6\II,8\}$, the term 
    \[\mu_d P^{\tilde{\gamma} + \ell} (\hat{F} + z) \hat{E} = z^2 \mu_d P^{\tilde{\gamma}} P^{\ell} (\theta_{P^{\tilde{\gamma}}} - \theta_{P^{-\ell}} + 1)\theta_{P^{-\ell}}\]
    fails to be specialized to $\square_{\bar{\gamma}}$. In fact it is not specializable on the locus $P^{-\ell} = 0$. 

    Nevertheless, the conclusion of \cite{MS23} still holds: there exists a two parameter extension $I^{\Xloc}(tH, y)$ of $I^{\Xloc}(th)$ which is a monodromy invariant solution of $\widetilde{\operatorname{PF}}{}^Y$ near $y = P^{-\ell} = 0$. 
\end{rmk}

\begin{lemma}\label{lm:mt}~
    \begin{enumerate}[(i)]
        \item The mirror transform $t \mapsto \tau$ for $J^{X_d}(\tau) = I^{X_d}(t)$ is trivial; that is, $\tau = t$. 
        \item The mirror transform $t \mapsto \tau$ for $J^{Y_d}(\tau) = I^{Y_d}(t)$ is given by $\tau = t - g_d(P^\ell)E$, for a unique power series $g_d(P^\ell) \in P^\ell \bQ\llbracket P^\ell\rrbracket$. Moreover, letting $\theta = P^\ell \odv{}{P^\ell} = - \frac1z \hat{E}$, the function $f_d \coloneqq 1 + \theta g_d$ satisfies the differential equation
        \begin{equation}\label{eqfODE}
            (1 + \kappa_d P^\ell - \mu_d P^{2\ell}) \theta^2 f_d + (\kappa_d P^\ell - 2 \mu_d P^{2\ell}) \theta f_d + (\lambda_d P^\ell - \mu_d P^{2\ell}) f_d = 0. 
        \end{equation}
    \end{enumerate} 
\end{lemma}

\begin{proof}
    It follows from \eqref{eqn:c1} that 
    \[(c_1(X_d), n\bar{\gamma}) = (2h, \bar{\gamma}) = 2n, \quad (c_1(Y_d), m\ell + n\gamma) = (2H, m\ell + n\gamma) = 2n. \]
    For $X_d$, this shows that $I_{n\bar{\gamma}}^{X_d} = O(z^{-2n})$, and hence $I^{X_d}(t) = e^{t/z}(1 + O(z^{-2})) = 1 + tz^{-1} + O(z^{-2})$. Therefore, $\tau = t$.  

    For $Y_d$, this shows that $I_{m\ell + n\gamma}^{Y_d} = O(z^{-2n})$, and hence 
    \[I^{Y_d}(t) = e^{t/z}\Bigl(\sum_{m\ge 0} I_{m\ell}^{Y_d} P^{m\ell} + O(z^{-2})\Bigr). \]
    By $\square_\ell I^{Y_d} = 0$, we have 
    \begin{align*}
        (F + (m+1)z)^3 I_{(m+1)\ell}^{Y_d} &= (\kappa_d (F+mz)^2 + \kappa_d z(F+mz) + \lambda_d z^2) (E - mz) I_{m\ell}^{Y_d} \\
        &\quad + \mu_d(F+mz)(E-mz)(E-(m-1)z) I_{(m-1)\ell}^{Y_d}. 
    \end{align*}
    The initial condition $I_{-\ell} = 0$, $I_0 = 1$ gives 
    \begin{equation}\label{eq:IlandI2l}
        I_\ell = \lambda_d Ez^{-1} + O(z^{-2}), \quad I_{2\ell} = -\tfrac12 (2\kappa_d + \lambda_d + \mu_d)Ez^{-1} + O(z^{-2}).
    \end{equation}
    By induction on $m$, one obtains $I_{m\ell} = -\rg_m E z^{-1} + O(z^{-2})$ and the sequence $\{\rg_m\}_{m\ge 1}$ satisfies the following recursion relation 
    \[(m+1)^3 \rg_{m+1} + (\kappa_d (m^2 + m) + \lambda_d) m \rg_m - \mu_d m^2(m-1) \rg_{m-1} = 0. \]
    It follows that
    \[I^{Y_d}(t) = 1 + \Bigl(t - \sum_{m\ge 1} \rg_m P^{m\ell} \cdot E\Bigr)z^{-1} + O(z^{-2}). \] 
    Set $g_d(P^\ell) = \sum_{m\ge 1} \rg_m P^{m\ell}$. Then $f_d(P^\ell) = \sum_{m\ge 0} \rf_m P^{m\ell} \coloneqq 1 + \theta g_d$ satisfies 
    \[(m+1)^2 \rf_{m+1} + (\kappa_d(m^2 + m) + \lambda_d) \rf_m - \mu_d m^2 \rf_{m-1} = 0. \]
    Note that this also holds for $m = 1$, as follows from \eqref{eq:IlandI2l}. Therefore $f_d$ satisfies the differential equation 
    \[\theta^2 f + P^\ell (\kappa_d (\theta + 1)\theta + \lambda_d) f - \mu_d P^{2\ell}(\theta + 1)^2 = 0, \]
    which coincides with equation~\eqref{eqfODE}. 
\end{proof}

By $Q^\beta = e^{(\tau, \beta)}$, $P^\beta = e^{(t, \beta)}$ and Lemma \ref{lm:mt}, we obtain the following:
\begin{cor}\label{cor:QandP}
    We have the following change of variables
    \[Q^{\bar{\gamma}} = P^{\bar{\gamma}}, \quad Q^\gamma = P^\gamma e^{-g_d (P^\ell)}, \quad Q^\ell = P^\ell e^{g_d (P^\ell)}, \]
    and thus $Q^{\gamma + \ell} = P^{\gamma + \ell}$. In particular, $\delta^E = -f^{-1} \theta$. 
\end{cor}

\begin{proof}
    The identities follow directly from the definitions of the variables involved. For the final statement, since 
    \[\delta^E Q^{\gamma + \ell} = 0 = -f^{-1} \theta P^{\gamma + \ell} = -f^{-1} \theta Q^{\gamma + \ell}, \] 
    it suffices to prove that $\delta^E P^\ell = -f^{-1} \theta P^\ell$. Observe that
    \[\delta^E P^\ell = - Q^\ell \odv{P^\ell}{Q^\ell}  = - Q^\ell \left(\odv{Q^\ell}{P^\ell}\right)^{-1}. \]
    From $Q^\ell = P^\ell e^{g_d (P^\ell)}$, it follows that
    \[\odv{Q^\ell}{P^\ell} = e^{g_d (P^\ell)}\cdot (1 + \theta g_d (P^\ell)). \]
    Hence, 
    \[\delta^E P^\ell = - Q^\ell (e^{g_d (P^\ell)}\cdot (1 + \theta g_d (P^\ell)))^{-1} = - f^{-1} P^\ell, \]
    as claimed. 
\end{proof}

\begin{proposition}\label{prop:fdisHG}
    Let $d = 1$, $2$, $3$, $4$, and let $n_d = 6$, $4$, $3$, $2$, respectively. Then the function $f_d(P^\ell)$ is given by the hypergeometric function 
    \[{}_2\rF_1\bigl(\tfrac{1}{n_d}, 1 - \tfrac{1}{n_d}; 1; -\kappa_d P^\ell\bigr). \]
    Additionally, when $d = 8$, we have 
    \[f_8(P^\ell) = {}_2\rF_1\bigl(\tfrac12, \tfrac12; 1; 16P^{2\ell}) = f_4(-P^{2\ell}). \]
\end{proposition}

\begin{proof}
    This follows immediately from the differential equation~\eqref{eqfODE} and the initial condition $f_d(P^\ell) = 1 + O(P^\ell)$. 
\end{proof}

\subsubsection{Reconstructions on local models}\label{subsubsec;Reconst_neq_7}

In Proposition \ref{prop:PFI}, we proved that $X_d$ and $Y_d$ satisfy the Picard--Fuchs equations. In this subsection, we compute the $1$-point and $2$-point invariants using these equations.

\begin{proposition}\label{propGWXloc}
    Suppose $I^X(t)$ satisfies the Picard--Fuchs equation $\square_{\bar{\gamma}} I^X = 0$, where 
    \begin{align*}
        \square_{\bar{\gamma}} &= \hat{h}^4 - P^{\bar{\gamma}}\bigl(\kappa\hat{h}^2 + \kappa z\hat{h} + \lambda z^2\bigr) - \mu P^{2\bar{\gamma}}. 
    \end{align*}
    Let $d = h^3$. Then  
    \begin{enumerate}[(i)]
        \item\label{propGWXloc_1} $\langle \pt\rangle^X = \lambda Q^{\bar{\gamma}}$; 
        \item\label{propGWXloc_2} $\langle h^2, h^2\rangle^X = d(\kappa - 2\lambda) Q^{\bar{\gamma}}$; 
        \item\label{propGWXloc_3} $\langle \pt, \pt\rangle^X = \tfrac{\lambda^2 + \mu}{2d}Q^{2\bar{\gamma}}$.  
    \end{enumerate}
\end{proposition}

\begin{proof}
    It follows from the Picard--Fuchs equations that for $\beta = n\bar{\gamma}$, 
    \[(h + nz)^4 I_\beta^X = (\kappa h^2 + \kappa (2n-1)zh + (\kappa n^2 - \kappa n + \lambda)z^2) I_{\beta - \bar{\gamma}}^X + \mu I_{\beta - 2\bar{\gamma}}. \]
    Together with the initial condition $I_0 = 1$, we get 
    \begin{equation}\label{eqbargamma}
        \begin{split}
        I_{\bar{\gamma}}^X &= \frac{\kappa h^2 + \kappa zh + \lambda z^2}{((h)^{\overline{1}})^4}, \\
        \quad I_{2\bar{\gamma}}^X &= \frac{(2\kappa + \lambda)\lambda + \mu}{16z^4} + O(z^{-5}). 
        \end{split}
    \end{equation}
    These coefficients can also be computed directly from the $I$-functions. 
    
    Since 
    \[e^{\tau/z} \biggl(1 + \sum_{k \ge 0} \frac{1}{z^{k+2}} \langle \psi^k T_\mu\rangle^X T^\mu \biggr) = J^X(\tau) = I^X(t) = e^{t/z} \sum_{\beta} I_\beta^X P^\beta \]
    and $\tau = t$, 
    \[\langle \psi^k T\rangle^X = [z^{-(k+2)}] \int_T \sum_\beta I_\beta^X Q^\beta, \]
    where $[z^{-k}]F$ denotes the coefficient of $z^{-k}$ in $F$. So for \eqref{propGWXloc_1}, it follows from \eqref{eqbargamma} that 
    \[\langle \pt \rangle^X = [z^{-2}]\int_\pt \sum_\beta I_\beta^X Q^\beta = [z^{-2}]\int_\pt I_{\bar{\gamma}}^X Q^{\bar{\gamma}} = \lambda Q^{\bar{\gamma}}. \]
    By Proposition~\ref{propambrec} and divisor equation, 
    \begin{align*}
        \langle h^2, h^2\rangle^X &= \delta^{h} \langle h^3\rangle^X + \delta^{h} \langle h, \psi h^2\rangle^X - \delta^{h} \langle h, T_\mu\rangle^X \langle T^\mu, h^2\rangle^X \\
        &= \langle h^3\rangle^X + \langle h^3\rangle^X + \langle \psi h^2\rangle^X \\
        &= 2d\langle \pt\rangle^X + [z^{-3}]\int_{h^2} I_{\bar{\gamma}}^X Q^{\bar{\gamma}} \\
        &= 2d \cdot \lambda Q^{\bar{\gamma}} + d(\kappa - 4\lambda) Q^{\bar{\gamma}} = d(\kappa - 2\lambda) Q^{\bar{\gamma}}, 
    \end{align*} 
    which is \eqref{propGWXloc_2}.  

    Finally, it follows from Proposition~\ref{propambrec}, \eqref{propGWXloc_1}, \eqref{propGWXloc_2}, and \eqref{eqbargamma} that
    \begin{align*}
        d\langle \pt, \pt\rangle^X &= \delta^h\langle h^2, \psi\pt\rangle^X - \delta^h\langle h^2, T_\mu\rangle^X \langle T^\mu, \pt\rangle^X \\
        &= 2(\delta^h\langle h, \psi^2\pt\rangle^X - \delta^h\langle h, T_\mu\rangle^X \langle T^\mu, \psi\pt\rangle^X) - \tfrac{1}{d}\langle h^2, h^2\rangle^X \langle h, \pt\rangle^X \\
        &= 8\langle \psi^2 \pt\rangle^X - 2\langle \pt\rangle^X \langle 1, \psi\pt\rangle^X - (\kappa - 2\lambda)Q^{\bar{\gamma}}\cdot \lambda Q^{\bar{\gamma}} \\
        &= 8\cdot \tfrac{(2\kappa + \lambda)\lambda + \mu}{16} Q^{2\bar{\gamma}} - 2(\lambda Q^{\bar{\gamma}})^2 - (\kappa - 2\lambda)\lambda Q^{2\bar{\gamma}} \\
        &= \tfrac{\lambda^2 + \mu}{2} Q^{2\bar{\gamma}}. 
    \end{align*}
    which is \eqref{propGWXloc_3}. 
\end{proof}

Next, we compute the $1$-point and $2$-point invariants of $Y_d$ explicitly. As a consequence we are able to determine the 3-point function $\langle E, E, E\rangle^Y$ (and hence the extremal function $\rE$).

\begin{prop}\label{propGWYloc}
    Suppose $I^Y(t)$ satisfies the Picard--Fuchs equations $\square_\gamma I^Y = \square_\ell I^Y = 0$, where
    \begin{align*}
        \square_\gamma &= \hat{E} \hat{H} - P^\gamma, \\
        \square_\ell &= \hat{F}^3 - P^\ell\bigl(\kappa\hat{F}^2 + \kappa z\hat{F} + \lambda z^2\bigr)\hat{E} - \mu P^{2\ell}(\hat{F} + z)\hat{E}(\hat{E}-z). 
    \end{align*}
    Let $\tilde{\gamma} = \gamma + \ell$ be the lifting of $\bar{\gamma}$ such that $(E, \tilde{\gamma}) = 0$. Let $f(P^\ell) \in 1 + \bC\llbracket P^\ell\rrbracket$ be the (unique) solution of the differential equation
    \[ (1 + \kappa P^\ell - \mu P^{2\ell})\theta^2 f + (\kappa P^\ell - 2\mu P^{2\ell})\theta f + (\lambda P^\ell - \mu P^{2\ell}) f = 0,\quad \theta = P^\ell \odv{}{P^\ell}, \]
    and let $u = u(P^\ell) = 1 + \kappa P^\ell - \mu P^{2\ell}$. Let $d = E^3 = H^3$. Then  
    \begin{enumerate}[(i)]
        \item\label{propGWYloc_1} $\langle \pt\rangle^Y = P^\gamma + \lambda Q^{\tilde{\gamma}}$; 
        \item\label{propGWYloc_2} $\langle H^2, H^2\rangle^Y = d(P^\gamma + (\kappa - 2\lambda) Q^{\tilde{\gamma}})$; 
        \item\label{propGWYloc_3} $\langle \pt, \pt\rangle^Y = \frac1d(\tfrac{\lambda^2 - \mu}{2}Q^{2\tilde{\gamma}} + \lambda P^{2\gamma + \ell} + u\frac{\theta f}{f}P^{2\gamma})$;  
        \item\label{propGWYloc_4} $\langle H^2, E^2\rangle^Y = -du(f + \theta f)P^\gamma$; 
        \item\label{propGWYloc_5} $\langle E, E, E\rangle^Y = \frac{d}{uf^3}$. 
    \end{enumerate}
\end{prop}

\begin{proof}
We will prove the statements in the order \eqref{propGWYloc_1}, \eqref{propGWYloc_2}, \eqref{propGWYloc_4}, \eqref{propGWYloc_5} and finally \eqref{propGWYloc_3}.

    It follows from the Picard--Fuchs equations that for $\beta = m\ell + n\gamma$, 
    \begin{align*}
        0 &= (E + (n-m)z)(H + nz) I_\beta^Y - I_{\beta - \gamma}^Y, \\
        0 &= (F + mz)^3 I_\beta^Y \\
        &\quad - (\kappa F^2 + \kappa (2m-1)zF + (\kappa m^2 - \kappa m + \lambda)z^2)(E + (n-m+1)z) I_{\beta - \ell}^Y \\
        &\quad - \mu (F + (m-1)z)(E + (n-m+2)z)(E + (n-m+1)z) I_{\beta - 2\ell}^Y.  
    \end{align*}
    Together with the initial condition $I_0 = 1$, we get 
    \begin{equation}\label{eqgamma}
    \begin{split}
        I_\gamma^Y &= \frac{1}{(E)^{\overline{1}}(H)^{\overline{1}}}, \\
         I_{\ell + \gamma}^Y &= \frac{\kappa F^2 + \kappa z F + \lambda z^2}{((F)^{\overline{1}})^3(H)^{\overline{1}}}
         \end{split}
    \end{equation}
    and $I_{m\ell + \gamma}^Y$ is a multiple of $E$ for $m\ge 2$; 
    \begin{equation}\label{eq2gamma}
        \begin{split} 
        I_{2\gamma}^Y &= \frac{1}{(E)^{\overline{2}}(H)^{\overline{2}}}, \\
        \quad I_{\ell + 2\gamma}^Y &= \frac{\kappa F^2 + \kappa z F + \lambda z^2}{((F)^{\overline{1}})^3(E)^{\overline{1}}(H)^{\overline{2}}}, \\
        \quad I_{2\ell + 2\gamma}^Y &= \frac{(2\kappa + \lambda)\lambda + \mu}{16z^4} + O(z^{-5})
        \end{split} 
    \end{equation}
    and $I_{m\ell + 2\gamma}^Y$ is a multiple of $E$ for $m\ge 3$. These coefficients can also be computed directly from the $I$-functions. 
    
    Since 
    \[e^{\tau/z} \biggl(1 + \sum_{k \ge 0} \frac{1}{z^{k+2}} \langle \psi^k T_\mu\rangle^Y T^\mu \biggr) = J^Y(\tau) = I^Y(t) = e^{t/z} \sum_{\beta} I_\beta^Y P^\beta \]
    and $\tau = t - g E$, we have
    \[\langle \psi^k T\rangle^Y = [z^{-(k+2)}] \int_T e^{g E/z} \sum_\beta I_\beta^Y P^\beta. \] 
    So for \eqref{propGWYloc_1}, it follows from \eqref{eqgamma} that 
    \[\langle \pt \rangle^Y = [z^{-2}]\int_\pt e^{gE/z} \sum_\beta I_\beta^Y P^\beta = [z^{-2}]\int_\pt I_\gamma^Y P^\gamma + I_{\ell + \gamma}^YP^{\ell + \gamma} = P^\gamma + \lambda Q^{\tilde{\gamma}}. \]
    
    Let $a \in H^4(Y)$ and $D \in H^2(Y)$. Then, by Proposition~\ref{propambrec}, 
    \[\langle HD, a\rangle^Y = \delta^{D} \langle Ha\rangle^Y + \delta^{H} \langle D, \psi a\rangle^Y - \delta^{H} \langle D, T_\mu\rangle^Y \langle T^\mu, a\rangle^Y = \delta^{D}\langle Ha\rangle^Y + \langle Da\rangle^Y + \delta^{D}\langle \psi a\rangle^Y. \]
    If we take $a = H^2$ and $D = H$, then
    \begin{align*}
        \langle H^2, H^2\rangle^Y &= 2d\langle \pt\rangle^Y + [z^{-3}]\int_{H^2} e^{gE/z}\sum_\beta I_\beta^Y P^\beta \\
        &= 2d(P^\gamma + \lambda Q^{\tilde{\gamma}}) + [z^{-3}]\int_{H^2} I_\gamma^Y P^\gamma + I_{\ell+\gamma}^Y P^{\ell+\gamma} \\
        &= 2d(P^\gamma + \lambda Q^{\tilde{\gamma}}) + (-d)P^\gamma + d(\kappa - 4\lambda)Q^{\tilde{\gamma}} = d(P^\gamma + (\kappa - 2\lambda)Q^{\tilde{\gamma}}),  
    \end{align*}
    which is \eqref{propGWYloc_2}. 
    
    If we take $a = E^2$ and $D = E$, then
    \[0 = \langle HE, a\rangle^Y = \langle E^3\rangle^Y + \delta^E\langle \psi E^2\rangle^Y. \]
    This implies
    \[\delta^E \langle \psi E^2\rangle^Y = f\cdot d(P^\gamma + \lambda Q^{\tilde{\gamma}}) = \delta^E (-du(f+\theta f)P^\gamma), \]
    i.e., $\langle \psi E^2\rangle^Y = -du(f + \theta f)P^\gamma + cQ^{\tilde{\gamma}}$ for some constant $c$. Comparing the $Q^{\tilde{\gamma}}$-coefficients of both sides we get $c = 0$. Hence, we get \eqref{propGWYloc_4}: 
    \[\langle H^2, E^2\rangle^Y = \langle \psi E^2\rangle^Y = -du(f + \theta f)P^\gamma. \]
    On the other hand, 
    \[\langle E^2, H^2\rangle^Y = \delta^E \langle E, \psi H^2\rangle^Y - \delta^E \langle E, T_\mu\rangle^Y \langle T^\mu, H^2\rangle^Y = (\delta^E)^2 \langle \psi H^2\rangle^Y - \delta^E \langle E, T_\mu\rangle^Y \langle T^\mu, H^2\rangle^Y\]
    implies 
    \[\tfrac{1}{d}\langle E, E, E\rangle^Y \langle E^2, H^2\rangle^Y = (\delta^E)^2 (dP^\gamma) = -d\,\tfrac{f + \theta f}{f^2}\,P^\gamma. \]
    Together with \eqref{propGWYloc_4}, this gives \eqref{propGWYloc_5}. 

    Finally, it follows from Proposition~\ref{propambrec}, \eqref{propGWYloc_1}, \eqref{propGWYloc_2}, \eqref{propGWYloc_4} and \eqref{eq2gamma} that
    \begin{align*}
        d\langle \pt, \pt\rangle^Y &= \delta^H\langle H^2, \psi\pt\rangle^Y - \delta^H\langle H^2, T_\mu\rangle^Y \langle T^\mu, \pt\rangle^Y \\
        &= 2(\delta^H\langle H, \psi^2\pt\rangle^Y - \delta^H\langle H, T_\mu\rangle^Y \langle T^\mu, \psi\pt\rangle^Y) \\
        &\quad - \tfrac{1}{d}(\langle H^2, H^2\rangle^Y \langle H, \pt\rangle^Y + \langle H^2, E^2\rangle^Y \langle E, \pt\rangle^Y) \\
        &= 8\langle \psi^2 \pt\rangle^Y - 2\langle \pt\rangle^Y \langle 1, \psi\pt\rangle^Y \\
        &\quad - \bigl((P^\gamma + (\kappa - 2\lambda)Q^{\tilde{\gamma}})(P^\gamma + \lambda Q^{\tilde{\gamma}}) - u(f + \theta f)P^\gamma(\tfrac{P^\gamma}{f})\bigr) \\
        &= 8\bigl(\tfrac14 P^{2\gamma} + \tfrac{\lambda}{2} P^{\ell + 2\gamma} + \tfrac{(2\kappa + \lambda)\lambda + \mu}{16} P^{2\ell + 2\gamma}\bigr) - 2(P^\gamma + \lambda Q^{\tilde{\gamma}})^2 \\
        &\quad - (P^\gamma + (\kappa - 2\lambda)Q^{\tilde{\gamma}})(P^\gamma + \lambda Q^{\tilde{\gamma}}) + (P^{2\gamma} + \kappa P^{\ell + 2\gamma} - \mu P^{2\ell + 2\gamma}) + u\tfrac{\theta f}{f}P^{2\gamma} \\
        &= \tfrac{\lambda^2 - \mu}{2} Q^{2\tilde{\gamma}} + \lambda P^{\ell + 2\gamma} + u\tfrac{\theta f}{f}P^{2\gamma},
    \end{align*}
    which is \eqref{propGWYloc_3}. 
\end{proof}

\subsubsection{Rings arising from logarithmic derivatives}\label{subsubsec;log-derring_neq_7}


For $X = X_d$ and $Y = Y_d$, consider the lifting 
\[\begin{tikzcd}[row sep = 0pt]
    \bC\llbracket Q^{\operatorname{NE}(X)}\rrbracket \ar[r, "\phi^*"] & \bC\llbracket Q^{\operatorname{NE}(Y)}\rrbracket \\
    Q^\gamma \ar[r, mapsto] & Q^{\tilde{\gamma}} = P^\gamma P^\ell.  
\end{tikzcd}\]
Comparing Proposition~\ref{propGWXloc} and Proposition~\ref{propGWYloc}, we observe that
\begin{align*}
    \langle \pt \rangle^Y - \phi^*\langle \pt \rangle^X &= P^\gamma = P^{-\ell} Q^{\tilde{\gamma}}, \\
    \langle H^2, H^2 \rangle^Y - \phi^*\langle h^2, h^2 \rangle^X &=  dP^\gamma = dP^{-\ell} Q^{\tilde{\gamma}}, \\
    \langle \pt, \pt \rangle^Y - \phi^*\langle \pt, \pt \rangle^X &= \tfrac1d\bigl(- \mu  + \lambda P^{-\ell} + u\tfrac{\theta f}{f} P^{-2\ell}\bigr)Q^{2\tilde{\gamma}} \\
    &= \tfrac1d \bigl(-\mu (1 + \tfrac{\theta f}{f}) + (\lambda + \kappa \tfrac{\theta f}{f})P^{-\ell} + \tfrac{\theta f}{f} P^{-2\ell}\bigr)Q^{2\tilde{\gamma}}. 
\end{align*}

For each $d$, denote by $\operatorname{Sing}_{f_d}$ the set of regular singular points of the differential equation \eqref{eqfODE}. At each singularity, the corresponding pair of local exponents can be computed explicitly. These are summarized in Table \ref{tab:Singfd}. 
\begin{table}[H]
    \centering
    \begin{tabular}{ccc}
        \toprule
        $d$ & $\operatorname{Sing}_{f_d}$ & local exponents \\
        \midrule
        $1$ & $\{0, \infty, -\frac{1}{432}\}$ & $(0,0)$, $(\frac16, \frac56)$, $(0, 0)$ \\
        \midrule
        $2$ & $\{0, \infty, -\frac{1}{64}\}$ & $(0,0)$, $(\frac14, \frac34)$, $(0, 0)$ \\
        \midrule
        $3$ & $\{0, \infty, -\frac{1}{27}\}$ & $(0,0)$, $(\frac13, \frac23)$, $(0, 0)$ \\
        \midrule
        $4$ & $\{0, \infty, -\frac{1}{16}\}$ & $(0,0)$, $(\frac12, \frac12)$, $(0, 0)$ \\
        \midrule
        $5$ & $\{0, \infty, \frac{11 - 5 \sqrt{5}}{2}, \frac{11 + 5 \sqrt{5}}{2}\}$ & $(0,0)$, $(1,1)$, $(0,0)$, $(0,0)$ \\
        \midrule
        $6 \I$ & $\{0, \infty, -\frac{1}{8}, 1\}$ & $(0,0)$, $(1,1)$, $(0,0)$, $(0,0)$ \\
        \midrule
        $6 \II$ & $\{0, \infty, -\frac{1}{9}, -1\}$ & $(0,0)$, $(1,1)$, $(0,0)$, $(0,0)$ \\
        \midrule
        $8$ & $\{0, \infty, -\frac{1}{4}, \frac14\}$ & $(0,0)$, $(1,1)$, $(0,0)$, $(0,0)$ \\
        \bottomrule
    \end{tabular}
    \caption{Regular singular points of $f_d$ and their corresponding local exponents.}
    \label{tab:Singfd}
\end{table}

Let $y = P^{-\ell}$. Then, near $y = 0$, we can rewrite \eqref{eqfODE} as 
\begin{equation}\label{eqfODEiny}
    (y^2 + \kappa_d y - \mu_d) (y \odv[style-frac=\tfrac]{}{y})^2 f_d - (\kappa_d y - 2 \mu_d) (y \odv[style-frac=\tfrac]{}{y}) f_d + (\lambda_d y - \mu_d) f_d = 0. 
\end{equation}

When $\mu_d = 0$, i.e., $d \in \{1,2,3,4\}$, the local exponent at $y = 0$ is $(\frac1{n_d}, \frac{n_d-1}{n_d})$, where $n_d = 6$, $4$, $3$, $2$ for $d = 1$, $2$, $3$, $4$, respectively, just as in the proof of Lemma~\ref{lm:IXd=1234}. This means that $f_d \in y^{1/n_d} \bC\PSR{y} \oplus y^{1-1/n_d} \bC\PSR{y}$ for $d \le 3$ and $f_d \in y^{1/2}\bC\PSR{y} \oplus y^{1/2}\log y\, \bC\PSR{y}$ for $d = 4$. 

When $\mu_d \neq 0$, i.e., $d \in \{5,6\I,6\II,8\}$, the local exponent at $y = 0$ is $(1, 1)$. This means that $f_d \in y\,\bC\PSR{y} \oplus y\log y \,\bC\PSR{y}$. Let $f_d^\reg \in y\,\bC\PSR{y}$ be a nonzero monodromy invariant solution (unique up to a scalar) of \eqref{eqfODEiny}. Then
\[\frac{\theta f_d^\reg}{f_d^\reg} = \frac{- y\odv{}{y}f_d^\reg}{f_d^\reg} = -1 + O(y). \]

By this observation, we found a suitable ring $\sR_d$ and a suitable ideal $\sI_d$ such that for $\vec{b} = (\pt)$, $(h^2, h^2)$, $(\pt, \pt)$, the invariant $\langle \phi^*\vec{b}\rangle^{Y_d} \in \sR_d[Q^{\tilde{\gamma}}]$ and 
\begin{equation}\label{eq:loccomp12pt}
    \langle \phi^*\vec{b}\rangle^{Y_d} - \phi^*\langle \vec{b} \rangle^{X_d} \in \sI_d[Q^{\tilde{\gamma}}].
\end{equation}

\begin{definition}\label{defRI}
    Let $y = P^{-\ell}$ and let $v_d = \frac{\theta f_d}{f_d}$. Consider the ring $\sS_d = \bC(y)[v_d] \subseteq \bC\Laurent{y}[v_d]$. We define the subring $\sR_d \subseteq \sS_d$ and its ideal $\sI_d$ as follows: 
    \begin{itemize}\itemindent=-1.5em
        \item for $d \in \{1,2,3,4\}$, 
        \[\sR_d = \sS_d\cap (\bC + y\cdot \bC\PSR{y}[v_d]),\quad \sI_d = \sS_d \cap (y\cdot \bC\PSR{y}[v_d]); \]
        \item for $d \in \{5,6\I,6\II,8\}$, let $v_d^\reg = \tfrac{\theta f_d^\reg}{f_d^\reg} \in \bC\PSR{y}$, $w_d = v_d - v_d^\reg$, 
        \[\sR_d = \sS_d\cap (\bC\PSR{y} + w_d\cdot \bC\Laurent{y}[w_d]), \quad \sI_d = \sS_d\cap (y\cdot \bC\PSR{y} + w_d \cdot \bC\Laurent{y}[w_d]). \]
    \end{itemize}
\end{definition}

\begin{lemma}\label{lm:thetaRdinId}
    The operator $\theta = P^{\ell} \odv{}{P^\ell}$ sends elements in $\sR_d$ to $\sI_d$. 
\end{lemma}

\begin{proof}
    For $p(v_d) = \sum_n p_n(y) v_d^n \in \bC(y)[v_d] = \sS_d$, 
    \[\theta p(v_d) = \sum (\theta p_n(y))v_d^n + \sum p_n(y) n v_d^{n-1} (\theta v_d). \]
    To prove that $\theta p(v_d) \in \sS_d$, it suffices to prove that $\theta v_d \in \sS_d$. By \eqref{eqfODE} (or \eqref{eqfODEiny}), 
    \begin{align}
        \theta v_d &= \frac{\theta^2 f_d}{f_d} - \frac{(\theta f_d)^2}{f_d^2} \notag \\
        &= -\frac{\kappa_d y - 2\mu_d}{y^2 + \kappa_d y - \mu_d} v_d - \frac{\lambda_d y - \mu_d}{y^2 + \kappa_d y - \mu_d} - v_d^2 \in \sS_d \cap \bC\PSR{y}[v_d]. \label{eqthetavd} 
    \end{align}
    
    When $d = 1$, $2$, $3$, $4$, if $p = p_0 + y\sum_n p_n(y) v_d^n \in \bC + y\cdot \bC\PSR{y}[v_d]$, then 
    \begin{align*}
        \theta p &= y\sum p_n(y) v_d^n + y\sum (\theta p_n(y)) v_d^n + y\sum p_n(y) nv_d^{n-1} (\theta v_d) \in y\cdot \bC\PSR{y}[v_d]
    \end{align*}
    by \eqref{eqthetavd}. Hence, $\theta \sR_d \subseteq \sI_d$. 

    When $d = 5$, $6\I$, $6\II$, $8$, if $p = p_0(y) + \sum_{n>0} p_n(y) w_d^n \in \bC\PSR{y} + w_d\cdot \bC\Laurent{y}[w_d]$, then 
    \begin{align*}
        \theta p &= \theta p_0(y) + \sum_{n>0} (\theta p_n(y)) w_d^n + \sum_{n>0} (\theta p_n(y)) nw_d^{n-1} (\theta w_d). 
    \end{align*}
    To show that $\theta p \in \sI_d$, it suffices to show that $\theta w_d \in y\cdot\bC\PSR{y} + w_d \cdot \bC\Laurent{y}[w_d]$. Indeed, by \eqref{eqthetavd}, which also holds for $v_d^\reg$, 
    \begin{align*}
        \theta w_d &= \theta v_d^\reg - \theta v_d \\
        &= -\frac{\kappa_d y - 2\mu_d}{y^2 + \kappa_d y - \mu_d} (v_d^\reg - v_d) - ((v_d^\reg)^2 - v_d^2) \\
        &= \frac{\kappa_d y - 2\mu_d}{y^2 + \kappa_d y - \mu_d} w_d + w_d(w_d + 2v_d^\reg) \in w_d \cdot \bC\Laurent{y}[w_d]. 
    \end{align*}
    Hence, $\theta \sR_d \subseteq \sI_d$. 
\end{proof}

We now extend \eqref{eq:loccomp12pt} to any $n$-point GW invariants: 
\begin{proposition}\label{proplocalmodI}
    Fix $d \neq 7$. Consider the ring $\Rloc = \sR_d[Q^{\tilde{\gamma}}]$ and its ideal $\Iloc = \sI_d[Q^{\tilde{\gamma}}]$. Then for $n\ge 1$, $\vec{a} \in H(S_d)^{\otimes n}$, the invariant $\langle j_{d*}\vec{a}\rangle^{Y_d}$ lies in $\Rloc$ and 
    \[\langle j_{d*}\vec{a}\rangle^{Y_d}  \equiv \phi^*\langle i_{d*}\vec{a}\rangle^{X_d} \pmod{\Iloc}. \]
\end{proposition}

\begin{proof}
    Since $H(S_d)$ is generated by $i_d^*h^e$, $e = 0$, $1$, $2$, $i_{d*}H(S_d)$ is generated by $h^{e+1} = i_{d*}i_d^*h^e$ and $j_{d*}H(S_d)$ is generated by $H^{e+1} = j_{d*}i_d^*h^e$, $e = 0$, $1$, $2$. 
    
    Let $\vec{a} = a_1 \otimes \dots \otimes a_n$, $b_\nu = i_{d*}a_\nu$, $c_\nu = j_{d*}a_\nu$. By divisor equation, we may assume that each $a_\nu = i_d^*h$ or $i_d^*h^2$. We prove the result by induction on $n$ and the curve degree $\frac12 \sum \deg a_\nu$. For simplicity, we will omit ``$d$'' in the subscripts. 

    For $n \le 2$, the result follows from the fundamental class axiom, the divisor equation, Proposition~\ref{propGWXloc} and  Proposition~\ref{propGWYloc}. 

    If $i^*h$ appears in $\{a_1, \dots, a_n\}$, say $a_1 = i^*h$, then by the reconstruction formula Theorem~\ref{thr:rc} (or  Proposition~\ref{propambrec}), 
    \begin{align*}
        \langle b_1, \dots, b_n\rangle^X &= \langle h^2, b_2, b_3, \dots, b_n\rangle^X \\
        &= \langle h, hb_2, b_3, \dots, b_n \rangle^X \\
        &\quad + \sum_{3\in I_L, 2\in I_R, \mu} \langle h, \vec{b}_{I_L}, \bar{T}_\mu\rangle^X \cdot \delta^h \langle \bar{T}^\mu, \vec{b}_{I_R}\rangle^X - \sum_{2, 3\in I_R, \mu} \delta^h\langle h, \vec{b}_{I_L}, \bar{T}_\mu\rangle^X \cdot \langle \bar{T}^\mu, \vec{b}_{I_R}\rangle^X, \\
        \langle c_1, \dots, c_n\rangle^Y &= \langle H^2, c_2, c_3, \dots, c_n\rangle^Y \\
        &= \langle H, Hc_2, c_3, \dots, c_n \rangle^Y \\
        &\quad + \sum_{3\in I_L, 2\in I_R, \mu} \langle H, \vec{c}_{I_L}, T_\mu\rangle^Y \cdot \delta^H \langle T^\mu, \vec{c}_{I_R}\rangle^Y - \sum_{2, 3\in I_R, \mu} \delta^H\langle H, \vec{c}_{I_L}, T_\mu\rangle^Y \cdot \langle T^\mu, \vec{c}_{I_R}\rangle^Y. 
    \end{align*}
    Hence, to show that 
    \[\langle c_1, \dots, c_n\rangle^Y \equiv \phi^*\langle b_1, \dots, b_n\rangle^X \pmod{\Iloc}, \]
    it suffices to show that: 
    \begin{enumerate}[(i)]
        \item\label{proplocalmodI_1} $\langle H, Hc_2, c_3, \dots, c_n \rangle^Y \equiv \phi^*\langle h, hb_2, b_3, \dots, b_n \rangle^X \pmod{\Iloc}$; 
        \item\label{proplocalmodI_2} $\langle \vec{c}_I, T\rangle^Y \equiv \phi^*\langle \vec{b}_I, \bar{T}\rangle^X \pmod{\Iloc}$ for $(\bar{T}, T) = (h^e, H^e)$ with lower curve degree; 
        \item\label{proplocalmodI_3} $\langle \vec{c}_{I_L}, E\rangle^Y \langle E^2, \vec{c}_{I_R}\rangle^Y \in \Iloc$. 
    \end{enumerate}
    The \eqref{proplocalmodI_1} follows from divisor equation and the induction hypothesis: 
    \[\langle Hc_2, c_3, \dots, c_n \rangle^Y \equiv \phi^*\langle hb_2, b_3, \dots, b_n \rangle  \pmod{\Iloc}. \]
    The \eqref{proplocalmodI_2} follows from the induction hypothesis, the fundamental class axiom and the divisor equation (when $e = 0$, $1$). For \eqref{proplocalmodI_3}, we show that $\langle \vec{c}, E\rangle^Y \in f^{-1}\Iloc$ and $\langle E^2, \vec{c}\rangle^Y \in f\Rloc[v]$ for $\vec{c}\in j_*H(S)^{\otimes m}$, $1\le m < n$, so that 
    \[\langle \vec{c}_{I_L}, E\rangle^Y \langle E^2, \vec{c}_{I_R}\rangle^Y \in f^{-1}\Iloc\cdot f\Rloc[v] \subseteq \Iloc. \]
    Indeed, by divisor equation, $E\cup j_*H(S_d) = 0$, Lemma~\ref{lm:thetaRdinId} and the induction hypothesis $\langle \vec{c}\rangle^Y \in \Rloc$, we have
    \[\langle \vec{c}, E\rangle^Y = \delta^E \langle \vec{c}\rangle^Y \in f^{-1} \theta \Rloc \subseteq f^{-1}\Iloc. \]
    If $m = 1$, then $\langle E^2, \vec{c}\rangle \in f\Rloc[v]$ by Proposition~\ref{propGWYloc} \eqref{propGWYloc_4}. If $m \ge 2$, apply the reconstruction formula, we get the equation 
    \begin{align*}
        \langle E^2, \vec{c}\rangle^Y &= \sum_{1\in I_L, 2\in I_R, \mu} \langle E, \vec{c}_{I_L}, T_\mu\rangle^Y \cdot \delta^E \langle T^\mu, \vec{c}_{I_R}\rangle^Y - \sum_{1, 2\in I_R, \mu} \delta^E\langle E, \vec{c}_{I_L}, T_\mu\rangle^Y \cdot \langle T^\mu, \vec{c}_{I_R}\rangle^Y. 
    \end{align*}
    When $I_L = \varnothing$, it follows from Proposition~\ref{propGWYloc} \eqref{propGWYloc_5} that
    \begin{align*}
        \delta^E \langle E, \vec{c}_{I_L}, T_\mu\rangle^Y \cdot \langle T^\mu, \vec{c}_{I_R}\rangle^Y &= (\delta^E)^2 \langle T_\mu\rangle^Y \cdot \langle T^\mu, \vec{c}\rangle^Y \\
        &= \bigl(\tfrac{1}{d}\langle E, E, E\rangle - 1\bigr) \langle E^2, \vec{c}\rangle^Y = \bigl(\tfrac{1}{uf^3} - 1\bigr) \langle E^2, \vec{c}\rangle^Y.
    \end{align*}
    So we get
    \begin{align}
        \frac{1}{uf^3}\langle E^2, \vec{c}\rangle^Y &= \sum_{1\in I_L, 2\in I_R, \mu} \delta^E \langle \vec{c}_{I_L}, T_\mu\rangle^Y \cdot \delta^E \langle T^\mu, \vec{c}_{I_R}\rangle^Y \notag\\
        &\quad - \sum_{I_L\neq \varnothing, 1, 2\in I_R, \mu} (\delta^E)^2\langle \vec{c}_{I_L}, T_\mu\rangle^Y\cdot \langle T^\mu, \vec{c}_{I_R}\rangle^Y. \label{eqreconEE}
    \end{align}

    For $d = 1$, $2$, $3$, $4$, each term in the right hand side lies in one of the followings (by induction on $m$): 
    \begin{align*}
        \delta^E \Rloc \cdot \delta^E \Rloc &\subseteq f^{-1}\Iloc\cdot f^{-1} \Iloc = f^{-2} \Iloc^2, \\
        (\delta^E)^2 \Rloc\cdot \Rloc &\subseteq f^{-2}\Iloc\cdot \Rloc = f^{-2}\Iloc, \\
        \delta^E(f^{-1}\Iloc)\cdot \delta^E(f\Rloc[v]) &\subseteq f^{-2}\Iloc \cdot \Rloc[v] = f^{-2} \Iloc, \\
        (\delta^E)^2(f^{-1}\Iloc)\cdot f\Rloc[v] &\subseteq f^{-3}\Iloc \cdot f\Rloc[v] = f^{-2} \Iloc, \\
        f^{-1}\Iloc\cdot (\delta^E)^2(f\Rloc[v]) &\subseteq f^{-1}\Iloc\cdot f^{-1}\Rloc[v] = f^{-2}\Iloc. 
    \end{align*}
    Therefore, the right hand side of \eqref{eqreconEE} lies in $f^{-2}\Iloc$, and hence
    \[\langle E^2, c\rangle^Y \in uf\Iloc \subseteq f\Rloc[v]. \]
    
    For $d = 5$, $6\I$, $6\II$, $8$, each term in the right hand side lies in one of the followings (by induction on $m$): 
    \begin{align*}
        \delta^E \Rloc \cdot \delta^E \Rloc &\subseteq f^{-1}\Iloc\cdot f^{-1} \Iloc = f^{-2} \Iloc^2, \\
        (\delta^E)^2 \Rloc\cdot \Rloc &\subseteq f^{-2}\Iloc^2\cdot \Rloc = f^{-2}\Iloc^2, \\
        \delta^E(f^{-1}\Iloc)\cdot \delta^E(f\Rloc[v]) &\subseteq f^{-2}\Iloc^2 \cdot \Rloc[v] = f^{-2}\Iloc^2, \\
        (\delta^E)^2(f^{-1}\Iloc)\cdot f\Rloc[v] &\subseteq f^{-3}\Iloc^3 \cdot f\Rloc[v] = f^{-2} \Iloc^3, \\
        f^{-1}\Iloc\cdot (\delta^E)^2(f\Rloc[v]) &\subseteq f^{-1}\Iloc\cdot f^{-1}\Iloc = f^{-2}\Iloc^2. 
    \end{align*}
    Therefore, the right hand side of \eqref{eqreconEE} lies in $f^{-2}\Iloc^2$, and hence
    \[\langle E^2, \vec{c}\rangle^Y \in uf\Iloc^2 \subseteq f\Rloc. \]

    If $a_\nu = i_d^*h^2$ for all $\nu$, then, by the reconstruction formula and \eqref{proplocalmodI_1}, \eqref{proplocalmodI_2}, \eqref{proplocalmodI_3} above, 
    \begin{align*}
        \langle c_1, \dots, c_n\rangle^Y &= \sum \langle H^2, \vec{c}_{I_L}, T_\mu\rangle^Y \cdot \delta^H \langle T^\mu, \vec{c}_{I_R}\rangle^Y - \sum \delta^H\langle H^2, \vec{c}_{I_L}, T_\mu\rangle^Y \cdot \langle T^\mu, \vec{c}_{I_R}\rangle^Y \\
        &\equiv \phi^*\Bigl(\sum \langle h^2, \vec{b}_{I_L}, \bar{T}_\mu\rangle^X \cdot \delta^h \langle \bar{T}^\mu, \vec{b}_{I_R}\rangle^X - \sum \delta^h\langle h^2, \vec{b}_{I_L}, \bar{T}_\mu\rangle^X \cdot \langle \bar{T}^\mu, \vec{b}_{I_R}\rangle^X \Bigr) \\
        &= \phi^*\langle b_1, \dots, b_n\rangle^X \pmod{\Iloc}, 
    \end{align*}
    as desired. 
\end{proof}

\subsection{Gromov--Witten invariants for \texorpdfstring{$\bm{d = 7}$}{d = 7}}\label{subsec;GW=7}

The case of $d=7$ is slightly different from other cases since we are unable to identify the curve classes of $X_7 = \operatorname{Bl}_\pt \bP^3$ only by its degree. We will follow the same approach as in the case $d \neq 7$ and Convention \ref{conv;omitG}. 

\subsubsection{Invariants of $X_7$}\label{subsubsec;Reconst_X7}

Since $X_7$ is a toric Fano variety, we can write down its $J$-function completely. Recall that $h_1$ is the pullback, via the blow-up $X_7 \to \bP^3$, of
the hyperplane class on $\bP^3$, $e$ is the exceptional divisor in $X_7$, $h_2 = h_1 - e$, and $h = h_1 + h_2 = c_1(\cO_{X_7}(1))$ (Remark \ref{rmk;OX7_1}). Then $H(X_7) = \bC[h_1, h_2] / (h_1(h_1 - h_2), h_2^3)$. Let $\bar{\gamma}_1 = h_2^2$ and $\bar{\gamma}_2 = h_2(h_1 - h_2)$. Then $(h_i, \bar{\gamma}_j)_{X_7} = \delta_{ij}$ and the Mori cone $\operatorname{NE}(X_7)$ of $X_7$ is generated by $\bar{\gamma}_1$ and $\bar{\gamma}_2$. The $J$-function of $X_7$ is then
\[J^{X_7}(t, z^{-1}) = e^{t/z}\sum_{m_1, m_2 \ge 0} \frac{1}{(h_1)^{\overline{m_1}}(h_1 - h_2)^{\overline{m_1 - m_2}}((h_2)^{\overline{m_2}})^3} Q^{m_1\bar{\gamma}_1 + m_2\bar{\gamma}_2}. \]

We first compute the $1$-point and $2$-point invariants of $X_7$. 
\begin{proposition}\label{propGWX7} 
    We have 
    \begin{enumerate}[(i)]
        \item\label{propGWX7_1} $\langle \pt \rangle^{X_7} = Q^{\bar{\gamma}_1}$; 
        \item\label{propGWX7_2} for divisors $D_1$ and $D_2 \in H^2(X_7)$, 
        \[\langle hD_1, hD_2\rangle^{X_7} = (hD_1D_2,1) Q^{\bar{\gamma}_1} + 2(D_1, \bar{\gamma}_1)(D_2, \bar{\gamma}_1) Q^{\bar{\gamma}_1} + (D_1, \bar{\gamma}_2)(D_2, \bar{\gamma}_2) Q^{\bar{\gamma}_2}; \]
        \item\label{propGWX7_3} $\langle \pt, \pt\rangle^{X_7} = Q^{\bar{\gamma}_1 + \bar{\gamma}_2}$. 
    \end{enumerate}
\end{proposition}

\begin{proof}
    As before, 
    \[\langle \psi^k T\rangle^{X_7} = [z^{-(k+2)}] \int_T \sum_\beta J_\beta^{X_7} Q^\beta. \]
    So, for \eqref{propGWX7_1}, 
    \[\langle \pt \rangle^{X_7} = [z^{-2}] \int_\pt \sum_\beta J_\beta^{X_7} Q^\beta = [z^{-2}] \int_\pt  J_{\bar{\gamma}_1}^{X_7} Q^{\bar{\gamma}_1} + J_{\bar{\gamma}_2}^{X_7} Q^{\bar{\gamma}_2} = Q^{\bar{\gamma}_1}. \]

    For further use, we also compute that 
    \begin{equation}\label{eq:psi2pt}
        \langle \psi^2 \pt \rangle = [z^{-4}] \int_{\pt} J_{2\bar{\gamma}_1}^{X_7} Q^{2\bar{\gamma}_1} + J_{\bar{\gamma}_1 + \bar{\gamma}_2}^{X_7} Q^{\bar{\gamma}_1 + \bar{\gamma}_2} + J_{2\bar{\gamma}_2}^{X_7} Q^{2\bar{\gamma}_2} = \tfrac14 Q^{2\bar{\gamma}_1} + Q^{\bar{\gamma}_1 + \bar{\gamma}_2},
    \end{equation}
    and for $D \in H^2(X_7)$,  
    \begin{equation}\label{eq:psihD}
        \langle \psi hD \rangle = [z^{-3}] \int_{HD} J_{\bar{\gamma}_1}^{X_7} Q^{\bar{\gamma}_1} + J_{\bar{\gamma}_2}^{X_7} Q^{\bar{\gamma}_2} = (hD, -2h_1 + h_2) Q^{\bar{\gamma}_1} + (hD, h_1 - h_2) Q^{\bar{\gamma}_2}. 
    \end{equation}

    For \eqref{propGWX7_2}, by reconstruction formula, 
    \[\langle hD_1, hD_2\rangle = \langle h, hD_1D_2\rangle + \delta^{D_1}\langle h, \psi hD_2\rangle - \sum \delta^{D_1} \langle h, T_\mu\rangle \langle T^\mu, hD_2\rangle. \]
    It is easy to check that the sum is $0$ by degree argument. So, by divisor equation, \eqref{propGWX7_1} and \eqref{eq:psihD}, 
    \begin{align*}
        \langle hD_1, hD_2\rangle &= \langle hD_1D_2\rangle + \delta^{D_1}\langle h^2D_2\rangle + \delta^{D_1}\delta^h\langle \psi hD_2\rangle \\
        &= (hD_1D_2, 1) \langle \pt\rangle + (D_2, h^2) \delta^{D_1} \langle \pt\rangle \\
        &\quad + \delta^{D_1}((hD_2, -2h_1 + h_2) Q^{\bar{\gamma}_1} + (hD_2, h_1 - h_2) Q^{\bar{\gamma}_2}) \\
        &= (hD_1D_2,1) Q^{\bar{\gamma}_1} + (D_2, h^2 + h(-2h_1 + h_2)) (D_1, \bar{\gamma}_1) Q^{\bar{\gamma}_1} \\
        &\quad + (D_2, h(h_1 - h_2)) (D_1, \bar{\gamma}_2) Q^{\bar{\gamma}_2} \\
        &= (hD_1D_2,1) Q^{\bar{\gamma}_1} + 2(D_1, \bar{\gamma}_1)(D_2, \bar{\gamma}_1) Q^{\bar{\gamma}_1} + (D_1, \bar{\gamma}_2)(D_2, \bar{\gamma}_2) Q^{\bar{\gamma}_2}. 
    \end{align*}

    Finally, by reconstruction formula, \eqref{propGWX7_1}, \eqref{propGWX7_2} and \eqref{eq:psi2pt}, 
    \begin{align*}
        7\langle \pt, \pt\rangle &= \langle h^3, \pt\rangle = \delta^h \langle h^2, \psi \pt\rangle - \delta^h\langle h^2, T_\mu\rangle\langle T^\mu, \pt\rangle \\
        &= 2(\delta^h \langle h, \psi^2 \pt\rangle - \delta^h\langle h, T_\mu\rangle \langle T^\mu, \psi \pt\rangle) - \langle h^2, \bar{\gamma}_1\rangle \langle h_1, \pt\rangle - \langle h^2, \bar{\gamma}_2\rangle \langle h_2, \pt\rangle \\
        &= 8 \langle \psi^2\pt\rangle - 2(\delta^h)^2 \langle \pt\rangle \langle 1,\psi \pt\rangle - \langle h^2, h(-\tfrac12 h_1 + h_2)\rangle \delta^{h_1}\langle \pt\rangle \\
        &= 2 Q^{2\bar{\gamma}} + 8Q^{\bar{\gamma}_1 + \bar{\gamma}_2} - 2Q^{2\bar{\gamma}_1} - (h^2(-\tfrac12 h_1 + h_2),1) Q^{2\bar{\gamma}_1} \\
        &\quad - 2(h, \bar{\gamma}_1)((-\tfrac12 h_1 + h_2), \bar{\gamma}_1) Q^{2\bar{\gamma}_1} - (h, \bar{\gamma}_2)((-\tfrac12 h_1 + h_2), \bar{\gamma}_2) Q^{\bar{\gamma}_1 + \bar{\gamma}_2} \\
        &= 2 Q^{2\bar{\gamma}} + 8Q^{\bar{\gamma}_1 + \bar{\gamma}_2} - 2Q^{2\bar{\gamma}} - Q^{2\bar{\gamma}_1} + Q^{2\bar{\gamma}_1} - Q^{\bar{\gamma}_1 + \bar{\gamma}_2} = 7Q^{\bar{\gamma}_1 + \bar{\gamma}_2}, 
    \end{align*}
    which is \eqref{propGWX7_3}. 
\end{proof}

\subsubsection{Asymptotic behavior of mirror transform}\label{subsubsec;log7}

As in \S \ref{subsubsec;PF_neq_7}, since $S_7$ is a zero locus of a general global section of the bundle $\cO_{\T_7}(1) = \cO(h_1 + h_2)$ over $X_7$, we can write down the $I$-function of $Y_7$: 
\begin{align}
    I^{Y_7}(t) &= e^{t/z}\sum_{n, m_1, m_2 \ge 0} \pi_7^*i_7^* J_{m_1\bar{\gamma}_1 + m_2\bar{\gamma}_2}^{X_7}\cdot \frac{(F)^{\overline{m_1 + m_2}}}{(E)^{\overline{n-m_1-m_2}} (H)^{\overline{n}}} P^{m_1 \ell_1 + m_2 \ell_2 + n\gamma} 
    \notag \\
    &= e^{t/z}\sum_{n, m_1, m_2 \ge 0} \frac{(F_1 + F_2)^{\overline{m_1 + m_2}}}{(F_1)^{\overline{m_1}}(F_1 - F_2)^{\overline{m_1 - m_2}}((F_2)^{\overline{m_2}})^3(E)^{\overline{n-m_1-m_2}} (H)^{\overline{n}}} P^{m_1 \ell_1 + m_2 \ell_2 + n\gamma} \label{eq:IY7}
\end{align}
where, as before, $F_1 = \pi_7^*i_7^*h_1$, $F_2 = \pi_7^*i_7^*h_2$, $F = F_1 + F_2 = \pi_7^*i_7^*h$, $\gamma$ is the fiber class of the projective bundle $Y_7 \to S_7$, and $\ell_1$, $\ell_2 \in H_2(Y_7)^{G_{S_7}}$ are chosen such that $(H, \ell_i) = 0$, $(F_i, \ell_j) = \delta_{ij}$. Then $\gamma$, $\ell_1$, and $\ell_2$ generate the cone $\operatorname{NE}(Y_7)/G_{S_7}$. 
We will, again, omit the group $G_{S_7} \cong \fS_2$ in the notation for simplicity. 

Before computing the GW invariants of $Y_7$, we first analyze the mirror transform and the function $f_7$ which will be defined below.  

It follows from \eqref{eq:IY7} that the mirror transform is $\tau = t - g_7E$, where 
\begin{equation}\label{eq:g7}
    g_7 = g_7(P^{\ell_1}, P^{\ell_2}) \coloneqq \sum_{\substack{m_1 \ge m_2 \\ m_1 + m_2 > 0}} \frac{(-1)^{m_1 + m_2}}{m_1+m_2}\frac{(m_1 + m_2)!^2}{m_1!(m_1-m_2)!m_2!^3} P^{m_1\ell_1 + m_2\ell_2}. 
\end{equation}
This gives the following change of variables: 
\[Q^\gamma = P^\gamma e^{-g_7}, \quad Q^{\ell_1} = P^{\ell_1} e^{g_7},\quad Q^{\ell_2} = P^{\ell_2} e^{g_7}. \]

As before, define $\theta = -\frac1z \hat{E} = P^{\ell_1} \odv{}{P^{\ell_1}} + P^{\ell_2} \odv{}{P^{\ell_2}}$ and 
\[f_7(P^{\ell_1}, P^{\ell_2}) = 1 + \theta g_7(P^{\ell_1}, P^{\ell_2}) = \sum_{m_1 \ge m_2} \frac{(-1)^{m_1 + m_2}(m_1 + m_2)!^2}{m_1!(m_1-m_2)!m_2!^3} P^{m_1\ell_1 + m_2\ell_2}. \]
Let $\theta_1 = P^{\ell_1} \odv{}{P^{\ell_1}}$, $\theta_2 = P^{\ell_2} \odv{}{P^{\ell_2}}$. Then $f_7$ satisfies the following differential equations: 
\begin{equation}\label{eqf7ODE} 
    \begin{aligned}
    (\theta_1 (\theta_1 - \theta_2) + P^{\ell_1} (\theta_1 + \theta_2 + 1)^2) f &= 0, \\
    (\theta_2^3 + P^{\ell_2} (\theta_1 + \theta_2 + 1)^2 (\theta_1 - \theta_2)) f &= 0. 
    \end{aligned}
\end{equation}

Let $y = P^{-\ell_1}$ and let $r = P^{\ell_2 - \ell_1} = Q^{\ell_2 - \ell_1}$. Then $\theta_1 = -\theta_y - \theta_r$, $\theta_2 = \theta_r$, where we denote $\theta_y \coloneqq y \odv{}{y}$, $\theta_r \coloneqq r \odv{}{r}$. With this change of variables, the differential equation \eqref{eqf7ODE} becomes
\begin{align}
    ((\theta_y-1)^2  + y(\theta_y + \theta_r)(\theta_y + 2\theta_r))f_7 &= 0, \label{eqf7ODE1}\\
    ((3+3y-4r)\theta_y + (7y-2-8r)\theta_r - 3)f_7 &= 0. \label{eqf7ODE2}
\end{align}
Plugging \eqref{eqf7ODE2} into \eqref{eqf7ODE1}, we obtain a second-order differential equation of the form
\begin{equation}\label{eqf7ODEABC}
    (A_2(y,r)\theta_y^2 + A_1(y,r)\theta_y + A_0(y,r))f_7 = 0, 
\end{equation}
where the coefficients are given by
\begin{align*}
    A_2(y,r) &= -2(1+4r)^3 + y + 72yr + 272yr^2 + 15y^2 - 162y^2r + 8y^2r^2 + 19y^3 - 15y^3r + 7y^4, \\
    A_1(y,r) &= 4(1+4r)^3 - 15y - 184yr - 496yr^2 - 42y^2 + 280y^2r - 23y^3 - y^3r, \\
    A_0(y,r) &= -2(1+4r)^3 + 12y + 108yr + 240yr^2 + 21y^2 - 126y^2r + 7y^3. 
\end{align*}
This shows that, for $r \neq -\frac14$, $\infty$, the function $f_7|_r$ has local exponents $(1,1)$ at $y = 0$. Therefore, we may assume that 
\[f_7 = \sum_{k\ge 1} (a_k(r)\log y + b_k(r)) y^k,\quad r\neq -\tfrac14, \infty.  \]
Substituting this into~\eqref{eqf7ODE2}, we see that 
\[\bigl((-4r)a_1(r) + (-2-8r) \theta_r a_1(r)\bigr) y\log y + O(y) = 0, \]
or equivalently, 
\[a_1(r) = \operatorname{exp} \Bigl(-\int \frac{2}{1+4r}\,\odif{r}\Bigr) = \frac{c_0}{(1+4r)^{1/2}}\] 
for some constant $c_0$. To determine $c_0$, observe that when $r = 0$, 
\[f_7|_{r = 0} = \sum_{m \ge 0} (-1)^{m} P^{m\ell_1} = \frac{1}{1 + P^{\ell_1}} = \frac{y}{1 + y}. \]
This yields $c_0 = a_1(0) = 0$, and hence $a_1(r) = 0$ for all $r$. Since the solutions to \eqref{eqf7ODEABC} have local exponents $(1,1)$, it follows that $a_k(r) \equiv 0$ for all $k$. Moreover, \eqref{eqf7ODE2} together with the boundary condition again tell us that
\begin{equation}\label{eqf7expan}
    f_7 = \sum_{k\ge 1}b_k(r)y^k = \frac{1}{(1+4r)^{1/2}}y - \frac{1-2r}{(1+4r)^{5/2}}y^2 + O(y^3),
\end{equation}
where $\{b_k(r)\}$ satisfies the recurrence relation 
\begin{equation}\label{eq:bk(r)}
    k^2b_{k+1}(r) + (k^2 + 3k \theta_r b_k(r) + 2\theta_r^2 b_k) = 0,\quad k\ge 1. 
\end{equation}

Since $f_7 = 1 - \theta_y g_7$, we have
\[g_7 = \log y -  \sum_{k\ge 1} \frac{b_k(r)}{k} y^k - c(r)\]
for some function $c(r)$. Substituting $f_7 = 1 - \theta_y g_7$ into \eqref{eqf7ODE1}, we get 
\begin{align*}
    0 &= (\theta_y^2y^{-1} + (\theta_y+\theta_r)(\theta_y+2\theta_r))(1 - \theta_y g_7) \\
    &= -\theta_y (y^{-1} + (\theta_y y^{-1}\theta_y + (\theta_y+\theta_r)(\theta_y+2\theta_r))g_7). 
\end{align*}
A direct computation using \eqref{eq:g7} then gives 
\[y^{-1} + (\delta_y y^{-1}\delta_y + (\delta_y+\delta_r)(\delta_y+2\delta_r))g_7 = 0. \]
This implies that $\delta_r^2 c(r) = 0$, i.e., $c(r) = c_1 \log r + c_2$ for some $c_1$, $c_2$. To determine these constants, note that 
\[\left.g_7\right|_{r = 0} = \sum_{m \ge 1} \frac{(-1)^m}{m} P^{m\ell_1} = - \log(1 + P^{\ell_1}) = \log \Bigl(\frac{y}{1+y}\Bigr), \]
which shows that $c(r) = 0$ for all $r$. Hence, 
\begin{align}
    g_7 &= \log y - \sum_{k\ge 1} \frac{b_k(r)}{k} y^k, \label{eq:g7ex} \\
    g_{7,1} \coloneqq \theta_{P^{\ell_1}} g_7 &= -1 + \sum_{k\ge 1} \left(b_k(r) + \frac{\theta_rb_k(r)}{k}\right)y^k, \label{eq:g71}\\
    g_{7,2} \coloneqq \theta_{P^{\ell_2}} g_7 &= -\sum_{k\ge 1} \frac{\theta_rb_k(r)}{k}y^k. \label{eq:g72}
\end{align}
In particular, both $\frac{g_{7,1} + 1}{f}$ and $\frac{g_{7,2}}{f}$ are bounded as $y \to 0$. This suggests us to consider the divisors $\tilde{F}_1 = F_1 + E$, $\tilde{F}_2 = F_2$, for which 
\begin{equation}\label{eq:dtF_1F_2}
    \delta^{\tilde{F}_1} = \theta_{\tilde{\gamma}_1} - \frac{g_{7,2}}{f}\theta_y - \theta_r, \quad \delta^{\tilde{F}_2} = \frac{g_{7,2}}{f}\theta_y + \theta_r 
\end{equation}
has coefficients that remain bounded near $y = 0$, where $\theta_{\tilde{\gamma}_1} = P^{\tilde{\gamma}_1} \odv{}{P^{\tilde{\gamma}_1}} = Q^{\tilde{\gamma}_1} \odv{}{Q^{\tilde{\gamma}_1}}$. 

The asymptotic behaviors made in this subsubsection are crucial and will play a central role in \S \ref{subsubsec;Reconst_Y7}.

\subsubsection{Reconstructions on $Y_7$ and formal ring $\sR_7$}\label{subsubsec;Reconst_Y7}

The main result in this subsubsection is Proposition~\ref{proplocal7modI}, which is an analogy of Proposition~\ref{proplocalmodI} in the case $d \neq 7$. Recall that $y = P^{-\ell_1}$ and $r = P^{\ell_2 - \ell_1} = Q^{\ell_2 - \ell_1}$.

\begin{proposition}\label{propGWY7}
    Let $\tilde{\gamma}_1 = \ell_1 + \gamma$ be the lifting of $\bar{\gamma}_1$ such that $(E, \tilde{\gamma}_1) = 0$. Take $F_0 = \frac12 F_1 - F_2$ so that $(EF_0,\tilde{F}_1) = (EF_0, \tilde{F}_2) = 0$ and $(EF_0, E) = 1$. 
    
    When $r \neq -\frac14$, $\frac12$, $\infty$, we have: 
    \begin{enumerate}[(i)]
        \item\label{propGWY7_1} $\langle \pt \rangle^{Y_7} = (1 + y)Q^{\tilde{\gamma}_1}$; 
        \item\label{propGWY7_2} for divisors $D_1$, $D_2 \in H^2(Y_7)$, 
        \[\langle HD_1, HD_2\rangle^{Y_7} = ((HD_1D_2,1) (1+y) + 2(D_1, \tilde{\gamma}_1)(D_2, \tilde{\gamma}_1) + (D_1, \tilde{\gamma}_2)(D_2, \tilde{\gamma}_2) r)Q^{\tilde{\gamma}_1}; \]
        \item\label{propGWY7_3} $\langle \pt, \pt\rangle^{Y_7} = r Q^{2\tilde{\gamma}_1} + O(y)$; 
        \item\label{propGWY7_4} for divisor $D\in H^2(Y_7)$, $\langle HD, EF_0\rangle^{Y_7} \in O(y)$; 
        \item\label{propGWY7_5} near $y = 0$,  
        \begin{align*}
            \langle E, E, F_0\rangle^{Y_7} &= -(1+4r)^{1/2} y^{-1} + O(1) \\
            \langle E, F_0, F_2\rangle^{Y_7} &= -\frac32\cdot\frac{1-(1+4r)^{1/2}}{1+4r} + O(y)
        \end{align*}
    \end{enumerate}
    and all these invariants lies in the ring $\bC[Q^{\tilde{\gamma}_1}, r, (1+4r)^{-1/2},(1-2r)^{-1}]\Laurent{y}$. 
\end{proposition}

\begin{proof}
    As before, 
    \[\langle \psi^k T\rangle^{Y_7} = [z^{-(k+2)}] \int_T e^{gE/z}\sum_\beta I_\beta^{Y_7} P^\beta. \]
    So, for \eqref{propGWY7_1}, 
    \[\langle \pt \rangle^{Y_7} = [z^{-2}] \int_\pt e^{gE/z}\sum_\beta I_\beta^{Y_7} Q^\beta = [z^{-2}] \int_\pt  I_{\gamma}^{Y_7} P^{\gamma} + I_{\ell_1 + \gamma}^{Y_7} P^{\ell_1 + \gamma} = P^{\gamma} + P^{\ell_1 + \gamma} = (1+y)Q^{\tilde{\gamma}_1}. \]

    For further use, we also compute that 
    \begin{align}
        \langle \psi^2 \pt \rangle^{Y_7} = [z^{-4}] \int_{\pt} e^{gE/z}\sum_\beta I_\beta^{Y_7} Q^\beta &= \tfrac14 P^{2\gamma} + \tfrac12 P^{2\gamma + \ell_1} + Q^{\ell_1 + \ell_2 + 2\gamma} + \tfrac14 Q^{2\ell_1 + 2\gamma} \notag\\
        &= (\tfrac14 + \tfrac12 y + \tfrac14 y^2 + r) Q^{2\tilde{\gamma}_1},\label{eq:psi2ptY7}
    \end{align}
    and for $D \in H^2(X_7)$,  
    \begin{align}
        \langle \psi HD \rangle^{Y_7} &= [z^{-3}] \int_{HD} e^{gE/z}\sum_\beta I_\beta^{Y_7} Q^\beta \notag \\
        &= (HD, -2F_1 + F_2)P^{\ell_1 + \gamma} + (HD, F_1 - F_2)P^{\ell_2 + \gamma} - (HD,H)P^\gamma \notag \\
        &= ((HD, -2F_1 + F_2) - (HD,H) y + (HD, F_1 - F_2) r) Q^{\tilde{\gamma}_1}. \label{eq:psihDY7}
    \end{align}

    For \eqref{propGWY7_2}, by reconstruction formula, 
    \[\langle HD_1, HD_2\rangle^{Y_7} = \langle H, HD_1D_2\rangle^{Y_7} + \delta^{D_1}\langle H, \psi hD_2\rangle^{Y_7} - \sum \delta^{D_1} \langle H, T_\mu\rangle^{Y_7} \langle T^\mu, HD_2\rangle^{Y_7}. \]
    It is easy to check that the sum is $0$ by degree argument. So, by divisor equation, \eqref{propGWY7_1} and \eqref{eq:psihDY7}, 
    \begin{align*}
        \langle HD_1, HD_2\rangle^{Y_7} &= \langle HD_1D_2\rangle^{Y_7} + \delta^{D_1}\langle H^2D_2\rangle^{Y_7} + \delta^{D_1}\delta^H\langle \psi HD_2\rangle^{Y_7} \\
        &= (HD_1D_2, 1) \langle \pt\rangle^{Y_7} + (D_2, H^2) \delta^{D_1} \langle \pt\rangle^{Y_7} \\
        &\quad + \delta^{D_1}(((HD_2, -2F_1 + F_2) - (HD_2,H) y + (HD_2, F_1 - F_2) r) Q^{\tilde{\gamma}_1}) \\
        &= (HD_1D_2,1)(1+y)Q^{\tilde{\gamma}_1} + (D_1, H^2 + H(-2F_1 + F_2)) Q^{\tilde{\gamma}_1} \\
        &\quad + (D_2, H(F_1 - F_2)) (D_1, \tilde{\gamma}_2) Q^{\tilde{\gamma}_2} \\
        &= ((HD_1D_2,1) (1+y) + 2(D_1, \tilde{\gamma}_1)(D_2, \tilde{\gamma}_1) + (D_1, \tilde{\gamma}_2)(D_2, \tilde{\gamma}_2) r)Q^{\tilde{\gamma}_1}. 
    \end{align*}

    Write $\langle \psi E F_0\rangle^{Y_7} = G(y, r) Q^{\tilde{\gamma}_1}$. Then for any divisor $D$, 
    \begin{align*}
        \langle HD, EF_0\rangle^{Y_7} &= \delta^D \langle \psi EF_0\rangle^{Y_7} + \langle H, EF_0D\rangle^{Y_7} - \sum \delta^D \langle H, T_\mu\rangle^{Y_7} \langle T^\mu EF_0\rangle^{Y_7} \\
        &= \delta^D(G Q^{\tilde{\gamma}_1}) + (D, EF_0) (1+y)Q^{\tilde{\gamma}_1}. 
    \end{align*}
    In particular, take $D = E$, we have 
    \[0 = \delta^E(GQ^{\tilde{\gamma}_1}) + (E,EF_0)(1+y)Q^{\tilde{\gamma}_1},  \]
    or, equivalently, $\theta_y G = f_7(1+y)$. By the expansion \eqref{eqf7expan} of $f_7$, we get 
    \[G = \sum_{k\ge 0} G_k(r)y^k, \quad G_k(r) = -\frac{b_k(r) + b_{k-1}(r)}{k}\quad \text{for}\quad k\ge 1, \]
    where $b_0(r)$ is defined to be $0$. 
    
    \begin{claim*}
        We have $G|_{r = 0} = -y$, and for each $n \ge 1$, 
        \[\left.\pdv[order={n}]{G}{r}\right|_{r = 0} = \frac{yV_n(y)}{(1+y)^{3n-1}}\]
        for some polynomial $V_n(y)$ of at most degree $n$. 
    \end{claim*}

    \begin{proof}[Proof of Claim]\renewcommand{\qedsymbol}{$\blacksquare$}
        It follows from 
        \[GQ^{\tilde{\gamma}_1} = \langle \psi E F_0\rangle^{Y_7} = [z^{-3}]\int_{EF_0} \sum_\beta I_\beta^{Y_7} P^\beta = -P^\gamma - Q^{\ell_2 + \gamma} + \cdots\]
        and the identity
        \[\theta_y G = (1 + y)f = \sum_{m_1\ge m_2}\frac{(-1)^{m_1+m_2}(m_1+m_2)!^2}{m_1!(m_1-m_2)!m_2!^3}(P^{m_1\ell_1 + m_2\ell_2} + P^{(m_1-1)\ell_1 + m_2\ell_2})\]
        that 
        \begin{align*}
            G &= -P^{-\ell_1} - P^{-\ell_1 + \ell_2} + \sum_{\substack{m_1\ge m_2 \\ m_1 + m_2\ge 2}} \frac{(-1)^{m_1+m_2}(m_1+m_2)!^2}{m_1!(m_1-m_2)!m_2!^3}\left(\frac{P^{m_1\ell_1 + m_2\ell_2}}{m_1+m_2} + \frac{P^{(m_1-1)\ell_1 + m_2\ell_2}}{m_1+m_2-1}\right) \\
            &= - y - r + \sum_{\substack{m_1\ge m_2 \\ m_1 + m_2\ge 2}} \frac{(-1)^{m_1+m_2}(m_1+m_2)!^2}{m_1!(m_1-m_2)!m_2!^3}\left(\frac{P^{(m_1+m_2)\ell_1}}{m_1+m_2} + \frac{P^{(m_1+m_2-1)\ell_1}}{m_1+m_2-1}\right)r^{m_2} \\
            &= - y - \Bigl(\sum_{m\ge 0} (-1)^m (3m+1)P^{m\ell_1}\Bigr)r \\
            &\quad + \sum_{m_2\ge 2} \frac{(-r)^{m_2}}{m_2!^3}\Biggl(\sum_{m_1\ge m_2}\frac{(-1)^{m_1}(m_1+m_2)!^2}{m_1!(m_1-m_2)!(m_1+m_2)} P^{(m_1+m_2)\ell_1} \\
            &\hphantom{\quad + \sum_{m_2\ge 2} \frac{(-r)^{m_2}}{m_2!^3}} \quad - \sum_{m_1\ge m_2-1}\frac{(-1)^{m_1}(m_1+m_2+1)!^2}{(m_1+1)!(m_1+1-m_2)!(m_1+m_2)} P^{(m_1+m_2)\ell_1}\Biggr). 
        \end{align*}
        In particular, 
        \begin{align*}
            G|_{r = 0} &= -y, \\
            \left.\pdv{G}{r}\right|_{r = 0}& = -\sum_{m\ge 0} (-1)^m (3m+1)P^{m\ell_1} = -\frac{1-3P^{\ell_1}}{(1+P^{\ell_1})^2} = \frac{y(3-y)}{(1+y)^2}. 
        \end{align*}
    
        For $n\ge 2$, we compute
        \begin{align*}
            \left.\pdv[order={n}]{G}{r}\right|_{r = 0} &= \frac{(-1)^{n}}{n!^2}\sum_{m\ge n-1}\left(\frac{(-1)^m(m+n)!(m+n-1)!}{(m+1)!(m+1-n)!}(-n)(n+3m+3)\right) P^{(m+n)\ell_1} \\
            &= \frac{nP^{(2n-1)\ell_1}}{n!^2}\sum_{m\ge 0}\frac{(m+2n-1)!(m+2n-2)!}{(m+n)!m!}(4n+3m) (-P^{\ell_1})^m \\
            &= \frac{nP^{(2n-1)\ell_1}}{n!^2} (4n + 3\theta_{P^{\ell_1}}){}_2\rF_1(2n-1, 2n; n+1; -P^{\ell_1}). 
        \end{align*}
        By the linear transformation formula for hypergeometric function, 
        \begin{align*}
            {}_2\rF_1(2n-1, 2n; n+1; -P^{\ell_1}) &= (1 + P^{\ell_1})^{(n+1)-(2n-1)-(2n)} {}_2\rF_1(2-n, 1-n; n+1; -P^{\ell_1}) \\
            &= \frac{U_n(P^{\ell_1})}{(1 + P^{\ell_1})^{3n-2}}, 
        \end{align*}
        where $U_n(P^{\ell_1}) = {}_2\rF_1(2-n, 1-n; n+1; -P^{\ell_1})$ is a polynomial of degree $n-2$. Substituting $y = P^{-\ell_1}$, we obtain
        \begin{align*}
            \left.\pdv[order={n}]{G}{r}\right|_{r = 0} &= \frac{ny^{-(2n-1)}}{n!^2} (4n - 3\theta_{y})\left(\frac{y^{3n-2}U_n(y^{-1})}{(1 + y)^{3n-2}}\right) = \frac{yV_n(y)}{(1 + y)^{3n-1}}
        \end{align*}
        for some polynomial $V_n$ of degree at most $n$, completing the proof. 
    \end{proof}

    It follows from the claim that $G_0(r) = 0$ for all $r$, and hence $G = - b_1(r)y + O(y^2)$. Thus, near $y = 0$, by \eqref{eq:dtF_1F_2} and the boundedness of $\frac{g_{7,2}}{f}$, 
    \begin{align*}
        \langle H\tilde{F}_1, EF_0\rangle^{Y_7} &= \delta^{\tilde{F}_1}(G Q^{\tilde{\gamma}_1}) + (\tilde{F}_1, EF_0) (1+y)Q^{\tilde{\gamma}_1} = \delta^{\tilde{F}_1}(G Q^{\tilde{\gamma}_1}) = O(y), \\
        \langle H\tilde{F}_2, EF_0\rangle^{Y_7} &= \delta^{\tilde{F}_2}(G Q^{\tilde{\gamma}_1}) + (\tilde{F}_2, EF_0) (1+y)Q^{\tilde{\gamma}_1} = \delta^{\tilde{F}_2}(G Q^{\tilde{\gamma}_1}) = O(y), \\
        \langle HE, EF_0\rangle^{Y_7} &= 0. 
    \end{align*}
    Since $D$ is a linear combination of $\tilde{F}_1$, $\tilde{F}_2$, $E$, we have $\langle HD, EF_0\rangle^{Y_7} = O(y)$. 

    On the other hand, by reconstruction formula, 
    \[\langle HD, EF_0\rangle^{Y_7} = \delta^E (\delta^{F_0} \langle \psi HD\rangle^{Y_7} + \langle F_0HD\rangle^{Y_7}) - \delta^E\langle F_0, T_\mu\rangle^{Y_7} \langle T^\mu, HD\rangle^{Y_7}. \]
    This shows that 
    \begin{align*}
        &\langle E, E, F_0\rangle^{Y_7}\langle HD, EF_0\rangle^{Y_7} + \langle E, F_0, F_2\rangle^{Y_7} \langle H(\tfrac32 F_1 - 2F_2)), HD\rangle^{Y_7} \\
        &\quad = -(H^2D)\delta^{F_0}\tfrac{y Q^{\tilde{\gamma}_1}}{f} + (F_0HD)\tfrac{y Q^{\tilde{\gamma}_1}}{f}. 
    \end{align*}
    So we can write $\langle E, E, F_0\rangle^{Y_7}$ and $\langle E, F_0, F_2\rangle^{Y_7}$ in terms of $G$ by taking $D = \tilde{F}_1$, $\tilde{F}_2$: 
    \begin{align*}
        \langle E, E, F_0\rangle^{Y_7} &= \frac{(1+y-2r)Q^{\tilde{\gamma}_1}(-4\delta^{F_0}\tfrac{y Q^{\tilde{\gamma}_1}}{f_7} + \tfrac{y Q^{\tilde{\gamma}_1}}{f_7}) - (2-y)Q^{\tilde{\gamma}_1}(-3\delta^{F_0}\tfrac{y Q^{\tilde{\gamma}_1}}{f_7})}{(1+y-2r)\delta^{\tilde{F}_1}(GQ^{\tilde{\gamma}_1}) - (2-y)Q^{\tilde{\gamma}_1}\delta^{F_2}(GQ^{\tilde{\gamma}_1})}, \\
        \langle E, F_0, F_2\rangle^{Y_7} &= \frac{\delta^{\tilde{F}_1}(GQ^{\tilde{\gamma}_1})(-3\delta^{F_0}\tfrac{y Q^{\tilde{\gamma}_1}}{f_7}) - \delta^{F_2}(GQ^{\tilde{\gamma}_1})(-4\delta^{F_0}\tfrac{y Q^{\tilde{\gamma}_1}}{f_7} + \tfrac{y Q^{\tilde{\gamma}_1}}{f_7})}{(1+y-2r)\delta^{\tilde{F}_1}(GQ^{\tilde{\gamma}_1}) - (2-y)Q^{\tilde{\gamma}_1}\delta^{F_2}(GQ^{\tilde{\gamma}_1})}. 
    \end{align*}    
    Since 
    \begin{align*}
        \delta^{\tilde{F}_1}(GQ^{\tilde{\gamma}_1}) &= (\theta_{\tilde{\gamma}_1} - \tfrac{g_{7,2}}{f_7}\theta_y - \theta_r)\bigl(\bigl(-\tfrac{y}{(1+4r)^{1/2}} + O(y^2)\bigr)Q^{\tilde{\gamma}_1}\bigr) = \bigl(-\tfrac{y}{(1+4r)^{1/2}} + O(y^2)\bigr)Q^{\tilde{\gamma}_1}, \\
        \delta^{F_2}(GQ^{\tilde{\gamma}_1}) &= (\tfrac{g_{7,2}}{f_7}\theta_y + \theta_r)\bigl(\bigl(-\tfrac{y}{(1+4r)^{1/2}} + O(y^2)\bigr)Q^{\tilde{\gamma}_1}\bigr) = O(y^2)Q^{\tilde{\gamma}_1}, \\
        \tfrac{y Q^{\tilde{\gamma}_1}}{f_7} &= \bigl((1+4r)^{1/2} + O(y)\bigr)Q^{\tilde{\gamma}_1}, \\
        \delta^{F_0} \tfrac{y Q^{\tilde{\gamma}_1}}{f_7} &= (\tfrac{1-2r}{2}(\tfrac{1-(1+4r)^{1/2}}{1+4r}) + O(y))Q^{\tilde{\gamma}_1},
    \end{align*}
    we have the following asymptotic expansion when $r \neq \frac12$: 
    \begin{align*}
        \langle E, E, F_0\rangle^{Y_7} &= -(1+4r)^{1/2} y^{-1} + O(1), \\
        \langle E, F_0, F_2\rangle^{Y_7} &= -\frac32\cdot\frac{1-(1+4r)^{1/2}}{1+4r} + O(y). 
    \end{align*}
    Also, it follows from (\ref{eq:bk(r)}) and the initial condition $b_1(r) = (1+4r)^{-1/2}$, that the coefficients of $f_7$, $f_7^{-1}$, and hence $g_{7,2}$, all lies in $\bC[r, (1+4r)^{-1/2}]$. This means that the invariants 
    \[\langle H\tilde{F}_1, EF_0\rangle^{Y_7},\ \langle H\tilde{F}_2, EF_0\rangle^{Y_7},\ \langle E, E, F_0\rangle^{Y_7},\ \langle E, F_0, F_2\rangle^{Y_7}\]
    all lies in the ring $\bC[Q^{\tilde{\gamma}_1}, r, (1+4r)^{-1/2},(1-2r)^{-1}]\Laurent{y}$. 

    Denote $\tilde{F}^1 = - \frac{1}{2} F_1 + F_2$, $\tilde{F}^2 = F_1 - F_2$ so that $(H\tilde{F}^i, F_j) = \delta^i_j$. 
    
    Finally, by reconstruction formula, \eqref{propGWY7_1}, \eqref{propGWY7_2}, \eqref{propGWY7_4} and \eqref{eq:psi2ptY7}, 
    \begin{align*}
        7\langle \pt, \pt\rangle^{Y_7} &= \langle H^3, \pt\rangle^{Y_7} = \delta^H \langle H^2, \psi \pt\rangle^{Y_7} - \delta^H\langle H^2, T_\mu\rangle^{Y_7}\langle T^\mu, \pt\rangle^{Y_7} \\
        &= 2(\delta^H \langle H, \psi^2 \pt\rangle^{Y_7} - \delta^H\langle H, T_\mu\rangle^{Y_7} \langle T^\mu, \psi \pt\rangle^{Y_7}) \\
        &\quad - \langle H^2, H\tilde{F}^1)\rangle^{Y_7} \langle \tilde{F}_1, \pt\rangle^{Y_7} - \langle H^2, H\tilde{F}^2\rangle^{Y_7} \langle \tilde{F}_2, \pt\rangle^{Y_7} - \langle H^2, EF_0\rangle^{Y_7} \langle E,\pt\rangle^{Y_7}  \\
        &= 8 \langle \psi^2\pt\rangle^{Y_7} - 2\langle \pt\rangle \langle 1,\psi \pt\rangle^{Y_7} - \langle H^2, H\tilde{F}^1\rangle^{Y_7} \delta^{\tilde{F}_1}\langle \pt\rangle^{Y_7} - \langle H^2, EF_0\rangle^{Y_7} \tfrac{y}{f_7} Q^{\tilde{\gamma}_1}\\
        &= 8(\tfrac14 + r) Q^{2\tilde{\gamma}_1} - 2Q^{2\tilde{\gamma}_1} - (H^2\tilde{F}^1,1)Q^{2\tilde{\gamma}_1} \\
        &\quad - 2(H, \tilde{\gamma}_1)(\tilde{F}^1, \tilde{\gamma}_1) Q^{2\tilde{\gamma}_1} - (H, \tilde{\gamma}_2)(\tilde{F}^2, \tilde{\gamma}_2) rQ^{2\tilde{\gamma}_1} + O(y)\\
        &= 7r Q^{2\tilde{\gamma}_1} + O(y), 
    \end{align*}
    which is (iii), and lies in the ring $\bC[Q^{\tilde{\gamma}_1}, r, (1+4r)^{-1/2},(1-2r)^{-1}]\Laurent{y}$. 
\end{proof}

\begin{remark}
    The generating function $\langle E, E, E\rangle$ is still of the form $\frac{7}{u_7f_7^3}$. In fact, one can show that 
    \[u_7(P^{\ell_1}, P^{\ell_2}) = \frac{(1 + P^{\ell_1})^3 - P^{\ell_2} - 20P^{\ell_1}P^{\ell_2} + 8P^{2\ell_1}P^{\ell_2} + 16P^{\ell_1}P^{2\ell_2}}{1 - \frac{2}{7}P^{\ell_1} - \frac{8}{7}P^{\ell_2}}, \]
    which is equal to 
    \[\frac{1 + 2P^\ell - 17P^{2\ell} + 25P^{3\ell}}{1 - \frac{10}{7}P^\ell}\]
    when restricting on the diagonal $P^\ell = P^{\ell_1} = P^{\ell_2}$. 
\end{remark}

\begin{definition}\label{defRI7}
    We define the ring 
    \[\sR_7 = \bC[Q^{\tilde{\gamma}_1}, r, (1+4r)^{-1/2},(1-2r)^{-1}]\PSR{y}\]
    and its ideal $\sI_7 = y\cdot \sR_7$. 
\end{definition}

It is easy to see that the operators $\theta_{\tilde{\gamma}_1}$ and $\theta_r$ sends elements in $\sR_7$ to $\sR_7$, and $\theta_y$ send elements in $\sR_7$ to $\sI_7$. 

\begin{proposition}\label{proplocal7modI}
    Consider the ring $\Rloc = \sR_d[Q^{\tilde{\gamma}_1}]$ and its ideal $\Iloc = \sI_d[Q^{\tilde{\gamma}_1}]$. Then for $n\ge 1$, $\vec{a} \in H(S_7)^{\otimes n}$, 
    the invariant $\langle j_{7*}a\rangle^{Y_7}$ lies in $\Rloc$ and 
    \[\langle j_{7*}\vec{a}\rangle^{Y_7} \equiv \phi_7^*\langle i_{7*}\vec{a}\rangle^{X_7} \pmod{\Iloc}. \] 
\end{proposition}

\begin{proof}
    Since $H(S_7)$ is generated by $1$, $i_7^*h_1$, $i_7^*h_2$, $i_7^*h^2$, $i_{7*}H(S_7)$ is generated by $h$, $hh_1$, $hh_2$, $h^3$ and $j_{7*}H(S_7)$ is generated by $H$, $H\tilde{F}_1$, $H\tilde{F}_2$, $H^3$. 

    Let $\vec{a} = a_1 \otimes \dots \otimes a_n$, $b_\nu = i_{7*}a_\nu$ and $c_\nu = j_{7*}a_\nu$. By divisor equation, we may assume that each $a_\nu = i_7^*h_1$, $i_7^*h_2$ or $h^2$. We prove the result by induction on $n$ and the curve degree $\frac12 \sum \deg a_\nu$. 
    
    For $n \le 2$, the result follows from the fundamental class axiom, the divisor equation, Proposition~\ref{propGWX7} and  Proposition~\ref{propGWY7}. 

    As in the proof of Proposition~\ref{proplocalmodI}, by divisor equation and reconstruction formula, it suffices to show that: 
    \begin{enumerate}[(i)]
        \item\label{proplocal7modI_1} $\langle \tilde{F}_i, Hc_2, c_3, \dots, c_n \rangle^Y \equiv \phi^*\langle h_i, hb_2, b_3, \dots, b_n \rangle^X \pmod{\Iloc}$; 
        \item\label{proplocal7modI_2} $\langle \vec{c}_I, T\rangle^Y \equiv \phi^*\langle \vec{b}_I, \bar{T}\rangle^X \pmod{\Iloc}$ for $(\bar{T}, T) = (1,1)$, $(h_1,\tilde{F}_1)$, $(h_2, \tilde{F}_2)$, $(hh_1,H\tilde{F}_1)$, $(hh_2,H\tilde{F}_2)$, $(\pt,\pt)$ with lower curve degree; 
        \item\label{proplocal7modI_3} $\langle \vec{c}_{I_L}, E\rangle^Y \langle EF_0, \vec{c}_{I_R}\rangle^Y \in \Iloc$. 
    \end{enumerate}

    We note that if $B \in \Rloc$ and $A \in \bC[Q^{\bar{\gamma}_1}, Q^{\bar{\gamma}_2}]$ satisfies $B \equiv \phi^* A \pmod{\Iloc}$, then 
    \begin{align*}
        \delta^{\tilde{F}_1} B &= \theta_{\tilde{\gamma}_1} B - \tfrac{g_{7,2}}{f_7} \theta_y B - \theta_r B \equiv \phi^*((\theta_{\bar{\gamma}_1} + \theta_{\bar{\gamma}_2})A - \theta_{\bar{\gamma}_2}A) = \phi^*(\delta^{h_1} A), \\
        \delta^{\tilde{F}_1} B &=  \tfrac{g_{7,2}}{f_7} \theta_y B + \theta_r B \equiv \phi^*(\theta_{\bar{\gamma}_2} A) = \phi^*(\delta^{h_2} A). 
    \end{align*}
    This, together with induction hypothesis, the fundamental class axiom and the divisor equation, gives \eqref{proplocal7modI_1} and \eqref{proplocal7modI_2}. 

    For \eqref{proplocal7modI_3}, we see that 
    \[\langle \vec{c}, E\rangle^Y = \delta^E \langle \vec{c}\rangle^Y = \delta^E \langle \vec{c}\rangle^Y \in f^{-1} \theta_y \Rloc \subseteq \Rloc. \]
    If $m = 1$, then $\langle EF_0, \vec{c}\rangle^Y \in \Iloc$ by Proposition~\ref{propGWY7} \eqref{propGWY7_4}. If $m \ge 2$, apply the reconstruction formula, we get the equation 
    \begin{align*}
        \langle EF_0, \vec{c}\rangle^Y &= \sum_{1\in I_L, 2\in I_R, \mu} \langle F_0, \vec{c}_{I_L}, T_\mu\rangle^Y \cdot \delta^E \langle T^\mu, \vec{c}_{I_R}\rangle^Y - \sum_{1, 2\in I_R, \mu} \delta^E\langle F_0, \vec{c}_{I_L}, T_\mu\rangle^Y \cdot \langle T^\mu, \vec{c}_{I_R}\rangle^Y. 
    \end{align*}
    When $I_L = \varnothing$, it follows from 
    \[\langle E, F_0, \tilde{F}_1\rangle^Y + \langle E, F_0, \tilde{F}_2\rangle^Y = \langle E, F_0, H\rangle^Y = 0\]
    that 
    \begin{align*}
        &\delta^E \langle F_0, \vec{c}_{I_L}, T_\mu\rangle^Y \cdot \langle T^\mu, \vec{c}_{I_R}\rangle^Y \\
        &\quad = \delta^E \delta^{F_0} \langle E\rangle^Y \cdot \langle EF_0, \vec{c}\rangle^Y + \delta^E \delta^{F_0} \langle \tilde{F}_1\rangle^Y \cdot \langle H\tilde{F}^1, \vec{c}\rangle^Y + \delta^E \delta^{F_0} \langle \tilde{F}_2\rangle^Y \cdot \langle H\tilde{F}^2, \vec{c}\rangle^Y\\
        &\quad = \bigl(\langle E, E, F_0\rangle^Y - 1\bigr) \langle EF_0, \vec{c}\rangle^Y + \langle E, F_0, F_2\rangle (\langle H\tilde{F}^2, \vec{c}\rangle^Y - \langle H\tilde{F}^1, \vec{c}\rangle^Y). 
    \end{align*}
    So we get
    \begin{align}
       \langle E, E, F_0\rangle^Y\cdot \langle E^2, \vec{c}\rangle^Y &= \sum_{1\in I_L, 2\in I_R, \mu} \delta^E \langle \vec{c}_{I_L}, T_\mu\rangle^Y \cdot \delta^E \langle T^\mu, \vec{c}_{I_R}\rangle^Y \notag\\
        &\quad - \sum_{I_L\neq \varnothing, 1, 2\in I_R, \mu} (\delta^E)^2\langle \vec{c}_{I_L}, T_\mu\rangle^Y\cdot \langle T^\mu, \vec{c}_{I_R}\rangle^Y \notag \\
        &\quad + \langle E, F_0, F_2\rangle (\langle H\tilde{F}^2, \vec{c}\rangle^Y - \langle H\tilde{F}^1, \vec{c}\rangle^Y). \label{eqreconEF0}
    \end{align}
    Since $\delta^E = f^{-1}\theta_y$ sends elements in $\Rloc$ to $\Rloc$, this, together with Proposition~\ref{propGWY7} \eqref{propGWY7_5} implies that 
    \[\langle E^2, \vec{c}\rangle^Y \in \frac{1}{\langle E, E, F_0\rangle^Y} \Rloc + \frac{\langle E, F_0, F_2\rangle^Y}{\langle E, E, F_0\rangle^Y} \Rloc \subseteq \Iloc. \]
    We conclude that $\langle \vec{c}_{I_L}, E\rangle^Y\cdot \langle EF_0, \vec{c}_{I_R}\rangle^Y \in \Rloc \cdot \Iloc \subseteq \Iloc$. 
\end{proof}

\section{Proof of main results in the global cases}

\subsection{Two (symplectic) semistable degenerations}\label{subsec;ss_deg}
To apply the degeneration analysis in the next subsection, we shall factor a Type II transition $\Xres \searrow \Xsm$ as a composition of two semistable (symplectic) degenerations $\cX \to \Delta$ and $\cY \to \Delta$. As before, let $\crpcon \colon \Xres \to \Xsing$ be the crepant contraction of $\Xres \searrow \Xsm$, $E = \Exc (\crpcon)$ and $\crpcon (E) = \{ p \}$.

As a simple demonstration, we first assume that $\Xres \searrow \Xsm$ is a del Pezzo transition of degree $d$, that is, $E$ is a smooth del Pezzo surface of degree $d$. In this case $\cX$ and $\cY$ are semistable \emph{projective} degenerations over the disk $\Delta$. 

The \emph{K\"ahler degeneration}
\begin{align*}
    \cY \to \Delta
\end{align*}
is the deformation to the normal cone $\cY = \Bl_{E \times \{0\}}( \Xres \times \Delta)$. Since $E$ has codimension one in $\Xres$ and $E|_E = K_E$, the special fiber $\cY_0 = \Xres \cup Y_d$ is a simple normal crossing divisor with $Y_d = \bP_E (K_E \oplus \cO)$. The intersection $E = \Xres \cap Y_d$ is understood as the infinity divisor (or relative hyperplane section) of $Y_d \to E$.

The \emph{complex degeneration} 
\begin{align*}
    \cX \to \Delta
\end{align*}
is the semistable reduction of the smoothing $\fX \to \Delta$ constructed in Proposition \ref{prop;locDef_sm}. Recall that, in the notation of the proof of Proposition \ref{prop;locDef_sm}, $\cX \to \Delta$ is obtained by a degree $n_d$ base change $\cX' \to \Delta$ allowed by the weighted blow-up at $p \in \cX'$ with weight $(\alpha, 1)$. The special fiber $\cX_0 = \Xres \cup X_d$ is a simple normal crossing divisor with $X_d$ being a smooth del Pezzo threefold of degree $d$. The intersection $E = \Xres \cap X_d$ is now understood as a general member of the linear system $|{-K_{X_d}}|$. Notice that the local model $Y_d \searrow X_d$ (Example \ref{ex:dPtrans}) appears in the special fibers $\cX_0$ and $\cY_0$.

In general, let $\Xres \searrow \Xsm$ be a Type II extremal transition of degree $d$. If we work in the category of symplectic manifolds, we can still construct similar semistable degenerations to relate the local model $Y_d \searrow X_d$ as above.

To begin with, the smoothing $\fX \to \Delta$ with $\fX_t = X$ ($t \neq 0$) induces a holomorphic map
\begin{align}\label{eqn;local_smO}
    \upsigma \colon \Delta \to \Def (\Xsing) \to \Def(\Xsing, p).    
\end{align}
We may take two good representatives $\ocU \to \Delta$ and $\tcU \to \Delta$ of the deformation corresponding to $\upsigma$ such that $\ocU_t \supseteq \tcU_t$ are open subsets of $X$ and $\ocU_0 \supseteq \tcU_0$ are open neighborhoods of $p$. Set $V \coloneqq \ocU_t$ and $\oU \coloneqq \ocU_0$.

Next, by Proposition \ref{prop;locDef_res}, there exists an open neighborhood $U$ of $E$ in $ \Xres$ and a holomorphic deformation of the complex structure on $U$ such that $E$ deforms to a smooth del Pezzo surface $E'$. So $p \in \oU$ can be deformed to $p' \in \oU'$ such that $\crpcon' \colon U' \to \oU'$ is also the weighted blow-up at $p'$ with weight $\alpha$:  
\begin{equation}\label{eqn;localdef_U}
    \begin{tikzcd}
       \mathllap{E \subseteq{}} U \ar[d, "\crpcon"] \ar[d, shift right=5ex] \ar[r, rightsquigarrow]& U' \mathrlap{{}\supseteq E'} \ar[d, "\crpcon'", swap] \ar[d, shift left=5ex]\\
      \mathllap{p \,\in{}\,} \oU \ar[r, rightsquigarrow] & \oU' \mathrlap{{}\ni\, p' \,.}
    \end{tikzcd}
\end{equation}
Let $\uptau \colon \Delta \to \Def (\Xsing, p)$ be the corresponding map of the deformation of $\oU$. Then there are a holomorphic map 
\begin{align}\label{eqn;local_smN}
    \upsigma' \colon \Delta \to \Delta \times \Delta \xrightarrow{\upsigma \times \uptau} \Def (\Xsing, p)
\end{align}
and a good representative $\fV \to \Delta$ of the deformation corresponding to $\upsigma'$ such that $\fV_0 = \oU'$ and $\fV_t = V$. Therefore we get a local del Pezzo transition $U' \searrow V$ such that the exceptional divisor $E'$ of $U' \to \oU'$ is smooth. As before, applying the deformation to the normal cone and Proposition \ref{prop;locDef_sm} to $U' \searrow V$, we get two semistable (projective) degenerations $\cU' \to \Delta$ and $\cV \to \Delta$ with the special fibers $\cU'_0 = U' \cup Y_d$ and $\cV_0 = U' \cup X_d$. Moreover, $\cU' \setminus \cU'_0$ is the trivial family $U' \times (\Delta \setminus\{0\})$ and $\cV_t = V$.

Now we will use the trick in \cite{LR01, Wilson97} to construct semistable (symplectic) degenerations. Let $J_\Xres$ be the (integrable) complex structure of $\Xres$. We fix an ample class of $(\Xres, J_\Xres)$ and denote $\omega_\Xres$ by a corresponding K\"ahler form. Note that $U'$ can be viewed as a deformed complex structure on $U$ \eqref{eqn;localdef_U}. We can patch this deformed complex structure on $U$ together (in a $C^\infty$ way by a partition of unity argument) with the original complex structure $J_\Xres$ to produce an almost complex structure $J'_\Xres$ on $\Xres$. The almost complex structure $J'_\Xres$ is integrable in a neighborhood of $E$. Furthermore, it is $\omega_\Xres$-tamed because the taming condition is open (cf.\ \cite[p.153]{MS17}). For abbreviation, we let $Y'$ stand for $(\Xres, J'_\Xres)$. Notice that $Y$ and $Y'$ are symplectic deformation equivalent of each other.

Choose a suitable open neighborhood $U''$ of $E$ in $\Xres$ with $U'' \subseteq U$. Then we get a semistable (symplectic) degeneration, which we also denote by $\cY \to \Delta$, obtained by gluing the trivial family $(Y \setminus U'') \times \Delta$ and $\cU'$. The family $\cY|_{\Delta \setminus \{0\}}$ is the trivial family with the fiber $\Xres'$. The special fiber $\cY_0$ consists of $(\Xres, J'_\Xres)$ and $Y_d$ which meet transversely and
\[
    Y' \cap Y_d = U' \cap Y_d = E'.
\]
Here $E'$ is the smooth del Pezzo surface of degree $d$. We also denote this degeneration by $Y' \rightsquigarrow Y' \cup_{E'} Y_d$.

Similarly, we can glue $\fX \setminus \tcU$ and $\cV$ to get a semistable (symplectic) degeneration, also denoted by $\cX \to \Delta$, with the fiber $\cX_t = X$. The special fiber $\cX_0$ consists of $Y'$ and $X_d$ which meet transversely, and 
\[
    Y' \cap X_d = U' \cap Y_d = E'.
\]
The $\cX$ is also denoted by $X \rightsquigarrow Y' \cup_{E'} X_d$. We remark that, in contrast to the degeneration $\cY$ with the general fiber $Y' = (\Xres, J_\Xres')$, the general fiber of $\cX$ is the original smooth projective threefold $X$ because $\cV_t = \fV_t = V = \ocU_t$ is the open subset of $X$ by \eqref{eqn;local_smO} and \eqref{eqn;local_smN}.

Now we summarize the above construction in the following proposition.

\begin{theorem}\label{thm;ssdeg}
Let $\typeII$ be a Type II extremal transition of degree $d$. Then there exist two semistable (symplectic) degeneration $\cY$ and $\cX$, i.e.,
\begin{align}\label{eqn;ssdeg}
    Y' \rightsquigarrow Y' \cup_{E'} Y_d \quad \text{and} \quad X \rightsquigarrow Y' \cup_{E'} X_d,
\end{align}
such that
\begin{enumerate}
    \item $Y$ and $Y'$ are symplectic deformation equivalent of each other;

    \item $E'$ is a smooth del Pezzo surface of degree $d$ and $Y_d  = \bP_{E'} (K_{E'} \oplus \cO)$;

    \item $X_d$ is a smooth del Pezzo threefold of degree $d$.
\end{enumerate}
Moreover, up to a (local) deformation of the singularity $\Sing (\Xsing) = \{p\}$, we have:
\begin{enumerate}[(i)]
    \item If $d = 1$, the $\cX$ is obtained by a degree six base change $\cX' \to \Delta$ allowed by the weighted blow-up at $p \in \cX'$ with weight $(3,2,1,1,1)$.

    \item If $d = 2$, the $\cX$ is obtained by a degree four base change $\cX' \to \Delta$ allowed by the weighted blow-up at $p \in \cX'$ with weight $(2,1,1,1,1)$.
    
    \item If $d = 3$ (resp.\ $d \geq 4$), the $\cX$ is obtained by a degree three (resp.\ two) base change $\cX' \to \Delta$ allowed by the blow-up at $p \in \cX'$.
\end{enumerate}
\end{theorem}

\begin{remark}
For compact symplectic manifolds $(M_i, \omega_{M_i})$ for $i = 0$, $\infty$, we have the symplectic sum construction. 
Indeed, assume that $M_0$ and $M_\infty$ have a common compact symplectic divisor $D$ and there exists an isomorphism
\begin{align*} 
    \Phi \colon N_{D/M_0} \otimes N_{D/M_\infty} \cong \cO_D
\end{align*}
of \emph{complex} line bundles. The symplectic sum is a symplectic manifold $(X, \omega_Z)$ obtained from $M_0$ and $M_\infty$ by gluing the complements of tubular neighborhoods of $D$ in $M_0$ and $M_\infty$ along their common boundary as directed by $\Phi$. It produces a symplectic fibration $\cZ \to \Delta$ with special fiber $\cZ_0 = M_0 \cup_D M_\infty$, and it satisfies some conditions (see for example \cite[\S 1.2]{FTZ20} and references therein).

One can apply the symplectic sum construction to $Y'$ and $X_d$ with a common divisor $E'$, and similarly for $Y' \cup_{E'} Y_d$. However, since we assume that the global smoothing $\fX \to \Delta$ exists, we can construct the degeneration $X \rightsquigarrow Y' \cup_{E'} X_d$ directly as above. 
\end{remark}

\subsection{Degeneration analysis revisited}

To compare the (genus zero) GW invariants of $\Xres \searrow \Xsm$ for curve classes, we will use the degeneration formula to reduce the problem to local models $Y_d \searrow X_d$ as in \cite{LLW10}.

Let $(M, \omega_{M})$ be a compact symplectic manifold with a compact symplectic divisor $D$. For example, $M$ is a smooth polarized variety and $D$ is a smooth divisor. For a relative pair $(M, D)$, we let $\Gamma = (n, \beta, \rho, \mu)$ with $\mu = (\mu_1, \dots, \mu_\rho) \in \bN^\rho$ a partition of the intersection number $(D, \beta) = |\mu| \coloneqq \sum_{i = 1}^\rho \mu_i$. For $A \in H(M)^{\otimes n}$ and $\varepsilon \in H(D)^{\otimes \rho}$, the relative invariant of stable maps with topological type $\Gamma$ (i.e., with contact order $\mu_i$ in $D$ at the $i$-th contact point) is
\begin{align*}
    \langle A \mid \varepsilon, \mu\rangle_\Gamma^{(M, D)} \coloneqq \int_{[\overline{\mathcal{M}}_\Gamma(M, D)]^{\vir}} e_M^*A \cup e_D^*\varepsilon
\end{align*}
where $e_M \colon \oM_{\Gamma} (M, D) \to M^n$ and $e_D \colon \oM_{\Gamma} (M, D) \to D^\rho$ are evaluation maps on marked points and contact points, respectively. The virtual dimension of $[\overline{\mathcal{M}}_\Gamma(M, D)]^{\vir}$ is 
\begin{align*}
    \operatorname{vdim}_\Gamma \coloneqq (c_1 (M), \beta) + \rho - |\mu| +n + (\dim M - 3).
\end{align*}
Here the virtual class $[\overline{\mathcal{M}}_\Gamma(M, D)]^{\vir}$ is constructed in \cite{IP03, LR01} in the symplectic case and in \cite{Li01} in the algebraic case. In the symplectic setting, we refer the reader to \cite{IP04, FTZ21} for more details.

If $\Gamma = \bigsqcup_{\pi} \Gamma^\pi$, then the relative invariant with disconnected domain curve
is defined by the product rule:
\begin{align*}
    \langle A \mid \varepsilon, \mu\rangle_\Gamma^{\bullet(M, D)} \coloneqq \prod \langle A \mid \varepsilon, \mu\rangle_{\Gamma^\pi}^{(M, D)}.
\end{align*}

We apply the degeneration formula to the semistable (symplectic) degenerations \eqref{eqn;ssdeg} given in Theorem \ref{thm;ssdeg}. For simplicity of notation, we denote these two degenerations by 
\[
    \cW \to \Delta.
\] 
It has a smooth fiber $W \coloneqq \cW_t$ ($t \neq 0$), a special fiber $\cW_0 = M_0 \cup_D M_\infty$ and $D = M_0 \cap M_\infty$ a smooth divisor.

All cohomology classes $\vec{a} \in H(W)^{\otimes n}$ have global liftings and the restriction $\vec{a}(s)$ on $\cW_s$ is defined for all $s \in \Delta$, because $Y' \rightsquigarrow Y' \cup_{E'} Y_d$ is a degeneration of a trivial family and we apply Proposition \ref{prop;coho_seq} \eqref{prop;coho_seq_1} to $X \rightsquigarrow Y' \cup_{E'} X_d$ (see \eqref{eqn;H_inv}). Let $j_0 \colon M_0 \hookrightarrow \cW_0$ and $j_\infty \colon M_\infty \hookrightarrow \cW_0$ be the inclusion maps. Let $\{e_i\}$ be a basis of $H(D)$, with $\{e^i\}$ its dual basis. Then $\{e_I\}$ forms a basis of $H (D^\rho)$ with dual basis $\{e^I\}$, where $|I| = \rho$ and $e_I = e_{i_1}\otimes \cdots \otimes e_{i_\rho}$. The \emph{degeneration formula} expresses the absolute invariants of $W$ in terms of the relative invariants of the two smooth pairs $(M_0, D)$ and $(M_\infty, D)$ \cite{LR01, Li02, IP04, LY04, FTZ20}:
\begin{align}\label{eqn;degen_fomula}
    \langle \vec{a}\rangle_{\beta}^{W} = \sum_I \sum_{\eta \in \Omega_\beta} C_\eta \langle j_0^*\vec{a}(0) \mid e_I, \mu\rangle_{\Gamma_0}^{\bullet (M_0, D)}\langle j_\infty^*\vec{a}(0) \mid e^I, \mu\rangle_{\Gamma_\infty}^{\bullet (M_\infty, D)}.
\end{align}
Here $\eta = (\Gamma_0, \Gamma_\infty, I_\rho)$ is an \emph{admissible triple} which consists of (possibly disconnected) topological types 
\begin{align*}
    \Gamma_i = \bigsqcup\nolimits_{\pi = 1}^{|\Gamma_i|} \Gamma_i^{\pi}
\end{align*}
with the same partition $\mu$ of the contact order under the identification $I_\rho$ of the contact points. The gluing $\Gamma_0 +_{I_\rho} \Gamma_\infty$ have type $(n, \beta)$ and is connected. In particular, $\rho = 0$  if and only if one of the $\Gamma_i$ is empty. The total number of marked points $n_i$ and the total degree $\beta_i \in \NE (M_i)$ satisfy the splitting relations $n = n_0 + n_\infty$ and
\begin{align*}
    \beta = \beta_0 + \beta_\infty,
\end{align*}
where we omit the obvious pushforwards in the degree splitting in $\NE(\cW)$, and $|\Gamma_0| + |\Gamma_\infty| = \rho + 1$. The constants $C_\eta = m (\mu) / |{\Aut \eta}|$, where $m(\mu) = \prod \mu_i$ and $\Aut \eta = \{\sigma \in \fS_\rho \mid \eta^\sigma = \eta\}$. We denote by $\Omega_\beta$ the set of equivalence classes of admissible triples with fixed degree $\beta$.

\subsection{From local to global}

Now, we can prove the main results of this section. Given a Type II extremal transition $\typeII$ of degree $d$, we let $\sR_Y = \sR_d\PSR{Q^{\phi^* \operatorname{NE}(X)}}$ and $\sI_Y = \sI_d\PSR{Q^{\phi^* \operatorname{NE}(X)}}$ (see Definition~\ref{defRI} and Definition~\ref{defRI7}).


\begin{thr}\label{thrGWmodI}
    Let $n\ge 0$, $\vec{b} \in H(X)^{\otimes n}$ and $0 \neq \bar{\beta} \in \operatorname{NE}(X)$. Consider the generating function 
    \[\langle \phi^*\vec{b} \rangle_{\bar{\beta}}^{\Xres \searrow \Xsm} = \sum_{\beta\mapsto \bar{\beta}} \langle \phi^*\vec{b}\rangle_\beta^Y Q^\beta. \]
    Then $\langle \phi^*\vec{b} \rangle_{\bar{\beta}}^{\Xres \searrow \Xsm}$ lies in $\sR_Y$ and 
    \begin{align}\label{thrGWmodI_eqn}
            \langle \phi^*\vec{b} \rangle_{\bar{\beta}}^{\Xres \searrow \Xsm} \equiv \langle \vec{b}\rangle_{\bar{\beta}}^X Q^{\tilde{\beta}} \pmod{\sI_Y},
    \end{align}
    where $\tilde{\beta}$ is a lifting of $\bar{\beta}$ such that $(E, \tilde{\beta}) = 0$. 
\end{thr}

\begin{proof}
    The proof is mainly based on the argument in \cite[\S4.4]{LLW10}, using the method developed in \cite{MP06}. 
        
    We are going to apply the degeneration formula \eqref{eqn;degen_fomula} to two semistable (symplectic) degenerations 
    \begin{align*}
        Y' \rightsquigarrow Y' \cup_{E'} Y_d \quad \text{and} \quad X \rightsquigarrow Y' \cup_{E'} X_d
    \end{align*} 
    given in Theorem \ref{thm;ssdeg}. Let $\Xsm_\loc = X_d$ and $\Xres_\loc = Y_d$ (see \S\ref{subsec;geom_real}). For $\vec{b} = b_1 \otimes \dots \otimes b_n$, without loss of generality we can assume $b_k \notin H^0 (\Xsm)$ by the fundamental class axiom and $n\ge 1$ by the divisor equation. We choose the support of $b_k$ outside of $\crpcon' (E') = \{p'\}$, where $\crpcon'$ is the weighted blow-up in \eqref{eqn;localdef_U}. Since we construct semistable degenerations via weighted blow-ups at $p'$ (see \S\ref{subsec;ss_deg}), the restriction of global lifting of $b_k$ to $\Xsm_\loc$ is $0$, and similarly for $\Xres_\loc$. Denote $c_k = \phi^*b_k$ and $\vec{c} = c_1 \otimes \dots \otimes c_n$. Since  $Y$ and $Y'$ are symplectic deformation equivalent of each other, we know that $\langle\vec{c}\rangle^{Y}_\beta =\langle\vec{c}\rangle^{Y'}_\beta$ \cite{LT99,Siebert99}. By the degeneration formula \eqref{eqn;degen_fomula}, we obtain
    \begin{equation}\label{eqn;GW_local2}
    \begin{aligned}
        \langle \vec{b}\rangle_{\bar{\beta}}^X &= \sum C_\eta \langle \vec{c} \mid e_I, \mu \rangle_{\bar{\Gamma}_0}^{\bullet (Y', E')} \cdot \langle \varnothing \mid e^I, \mu\rangle_{\bar{\Gamma}_\infty}^{\bullet (X_{\text{loc}}, h)}, \\
        \langle \vec{c}\rangle_\beta^Y &= \sum C_\eta \langle \vec{c} \mid e_I, \mu \rangle_{\Gamma_0}^{\bullet(Y', E')} \cdot \langle \varnothing \mid e^I, \mu\rangle_{\Gamma_\infty}^{\bullet(Y_{\text{loc}}, H)}, 
    \end{aligned}
    \end{equation}
    where $\bar{\Gamma}_0 = \Gamma_0 = (\beta_0, \rho, \mu)$, $\bar{\Gamma}_\infty = (\bar{\beta}_\infty, \rho, \mu)$, $\Gamma_\infty = (\beta_\infty, \rho, \mu)$ with
    \[
        \bar{\beta} = \phi_* \beta_0,\quad \beta = \beta_0 + \pi_*\beta_\infty,\quad (E', \beta_0) = (h, \bar{\beta}_\infty) = (H, \beta_\infty) = |\mu|. 
    \]
    Here, $\pi_*\colon H_2(\Yloc) \to H_2(E') \to H_2(Y')$ is the natural pushforward. 
    When $|\mu| = 0$, since $\bar{\beta} \neq 0$, we must have $\beta_\infty = 0$ and $(E',\beta) = (E',\beta_0) = 0$, and this implies $\beta = \tilde{\beta}$.
    
    The virtual dimensions satisfy
    \begin{align*}
        \operatorname{vdim}_{\bar{\Gamma}_\infty} &= \operatorname{vdim}_{\bar{\beta}_\infty} + \rho - |\mu| = \rho + |\mu|\ge 2\rho, \\
        \operatorname{vdim}_{\Gamma_\infty} &= \operatorname{vdim}_{\beta_\infty} + \rho - |\mu| = \rho + |\mu|\ge 2\rho, 
    \end{align*}
    which implies that the only non-vanishing contribution arises when $(e_I, \mu) = (1, 1)^{\rho}$, $(e^I, \mu) = (\pt, 1)^{\rho}$. This also gives 
    \[\beta_0 = \tilde{\beta} - \rho \ell, \quad \bar{\beta}_\infty = \rho \bar{\gamma}, \quad \beta_\infty = \rho \tilde{\gamma} + m\ell, \]
    where $m = -(E', \beta)$ satisfies $\beta = \tilde{\beta} + m\ell$. 

    We write $\Gamma_\infty \mapsto \bar{\Gamma}_\infty$ if $|\Gamma_\infty| \to |\bar{\Gamma}_\infty|$ is an isomorphism of graphs and $\beta_\infty\mapsto \bar{\beta}_\infty$ and define 
    \[
        \langle \varnothing \mid (\pt, 1)^\rho\rangle_{\bar{\Gamma}_\infty}^{\bullet (\Yloc , H) \searrow (\Xloc, h)} \coloneqq \sum_{\Gamma_\infty \mapsto \bar{\Gamma}_\infty} \langle \varnothing \mid (\pt, 1)^\rho\rangle_{\Gamma_\infty}^{\bullet (\Yloc , H)} Q^{\beta_\infty - \rho \tilde{\gamma}}. 
    \]    
    Let $\beta_\mathrm{min} \in \operatorname{NE}(Y)$ be the minimal lifting of $\bar{\beta}$, and let $e = (E',\beta_\mathrm{min})$. Then \eqref{eqn;GW_local2} induces
    \begin{align*}
        \langle \vec{b}\rangle_{\bar{\beta}}^X Q^{\tilde{\beta}} &= \sum_{\rho = 1}^e C_\eta \langle \vec{c} \mid (1, 1)^\rho \rangle_{\bar{\Gamma}_0}^{\bullet(Y', E')}Q^{\tilde{\beta}}\langle \varnothing \mid (\pt, 1)^\rho\rangle_{\bar{\Gamma}_\infty}^{\bullet(\Xloc, h)} + \langle \vec{c} \mid \varnothing\rangle_{\tilde{\beta}}^{(Y', E')}Q^{\tilde{\beta}} \\
        \sum_{\beta\mapsto \bar{\beta}}\langle \vec{c}\rangle_{\beta}^Y Q^\beta &= \sum_{m \ge -e}\langle \vec{c}\rangle_{\tilde{\beta} + m\ell}^Y Q^{\tilde{\beta} + m\ell} \\
        &= \sum_{\rho = 1}^e C_\eta \langle \vec{c} \mid (1, 1)^\rho \rangle_{\Gamma_0}^{\bullet(Y', E')}Q^{\tilde{\beta}} \langle \varnothing \mid (\pt, 1)^\rho\rangle_{\bar{\Gamma}_\infty}^{\bullet(Y_\loc, H)\searrow (\Xloc, h)} + \langle \vec{c} \mid \varnothing \rangle_{\tilde{\beta}}^{(Y',E')} Q^{\tilde{\beta}}
    \end{align*}

    Define
    \begin{align*}
        [\rho]^{\Yloc} &\coloneqq \sum_{m \ge -\rho}\langle \varnothing \mid (\pt, 1)^\rho\rangle_{\rho \tilde{\gamma} + m\ell}^{(\Yloc, H)}Q^{m\ell} \\
        [\rho]^{\Xloc} &\coloneqq \langle \varnothing \mid (\pt, 1)^\rho\rangle_{\rho \bar{\gamma}}^{(\Xloc,h)}
    \end{align*}
    By decomposing each $\bar{\Gamma}_\infty$ into its connected component, it suffices to prove that for each $\rho > 0$, 
    \begin{align}\label{eqn;GW_loc_cong}
        [\rho]^{\Yloc} \in \sR_d \quad \text{and} \quad [\rho]^{\Yloc} \equiv [\rho]^{\Xloc} \pmod{\sI_d}.
    \end{align}
        
    Denote $[\rho]_\Gamma^{\bullet \Yloc} = \prod [\rho_i]^{\Yloc}$ and  $[\rho]_\Gamma^{\bullet \Xloc} = \prod [\rho_i]^{\Xloc}$. Applying the degeneration formula again \eqref{eqn;degen_fomula}, we get 
    \begin{align*}
        \langle \pt^{\otimes \rho}\rangle_{\rho\bar{\gamma}}^{\Xloc} Q^{\rho\tilde{\gamma}} &= \sum_{\rho' = 1}^\rho C_\eta \langle \pt^{\otimes \rho} \mid (1, 1)^{\rho'} \rangle_{\bar{\Gamma}_0}^{\bullet(\Yloc, E')}Q^{\rho\tilde{\gamma}} \cdot [\rho']_{\bar{\Gamma}_\infty}^{\bullet \Xloc} + \langle \pt^{\otimes \rho} \mid \varnothing\rangle_{\rho\tilde{\gamma}}^{(\Yloc, E')}Q^{\rho\tilde{\gamma}}, \\
        \sum_{m \ge -\rho}\langle \pt^{\otimes \rho}\rangle_{\rho\tilde{\gamma} + m\ell}^{\Yloc} Q^{\rho\tilde{\gamma} + m\ell} &= \sum_{\rho' = 1}^\rho C_\eta \langle \pt^{\otimes \rho} \mid (1, 1)^{\rho'} \rangle_{\Gamma_0}^{\bullet(\Yloc, E')}Q^{\rho\tilde{\gamma}} \cdot [\rho']_{\bar{\Gamma}_\infty}^{\bullet \Yloc} + \langle \pt^{\otimes \rho} \mid \varnothing \rangle_{\rho\tilde{\gamma}}^{(\Yloc,E')} Q^{\rho\tilde{\gamma}}, 
    \end{align*} 
    where $\bar{\Gamma}_0 = \Gamma_0 = (\rho\tilde{\gamma} - \rho'\ell, \rho', 1^{\rho'})$ and  $\bar{\Gamma}_\infty = (\rho'\bar{\gamma}, \rho', 1^{\rho'})$.  
    
    Now, by Proposition~\ref{proplocalmodI} (if $d\neq 7$) or Proposition~\ref{proplocal7modI} (if $d = 7$), we know that 
    \[\langle \pt^{\otimes \rho}\rangle_{\rho\bar{\gamma}}^{\Xloc} Q^{\rho\tilde{\gamma}} \equiv \sum_{m \ge -\rho}\langle \pt^{\otimes \rho}\rangle_{\rho\tilde{\gamma} + m\ell}^{\Yloc} Q^{\rho\tilde{\gamma} + m\ell} \pmod{\Iloc}. \]
    Moreover, among the relative invariants appearing in the degeneration formula, the only term contributing to the coefficient of $[\rho]^{\Xloc}$ (and $[\rho]^{\Yloc}$) is 
    \[\langle \pt^{\otimes \rho} \mid (1, 1)^\rho \rangle_{\rho\tilde{\gamma} - \rho \ell}^{\bullet(\Yloc, E')} = \langle \pt^{\otimes \rho} \mid (1, 1)^\rho \rangle_{\rho \gamma}^{\bullet(\Yloc, E')} = \left(\langle \pt \mid (1, 1)\rangle_{\gamma}^{(\Yloc, E')}\right)^\rho, \]
    which is nonzero (cf.~\cite[\S 1.2]{MP06}, or using the identity $\langle \pt \mid (1, 1)\rangle_{\gamma}^{(\Yloc, E')}[1]_{\gamma}^{(\Yloc, H)} = \langle \pt \rangle_\gamma = 1$). All other terms $[\rho']_{\bar{\Gamma}_\infty}^{\bullet\Xloc}$ and $[\rho']_{\Gamma_\infty}^{\bullet\Yloc}$ satisfy either $\rho' < \rho$, or has at least two connected components. 
    Therefore \eqref{eqn;GW_loc_cong} follows by induction on $\rho$. 
\end{proof}

\begin{cor}\label{cor:QHsubquot}
    The big quantum cohomology $QH(X)$ is a sub-quotient of $QH(Y)$. More precisely, for $b_1$, $b_2$, $\tau \in H(X)$, 
    \[\phi^*b_1 \star_{\phi^*\tau} \phi^*b_2 \equiv \phi^*(b_1 \star_\tau b_2) \pmod{\sI_Y^q H(Y) + \bC\PSR{q^{\operatorname{NE}(Y)}} E + \bC\PSR{q^{\operatorname{NE}(Y)}} \ell}, \]
    where $\sI_Y^q \subseteq \bC\PSR{q^{\operatorname{NE}(Y)}}$ is the $q$-analogue of $\sI_Y \subseteq \bC\PSR{Q^{\operatorname{NE}(Y)}}$, obtained by replacing the variables $Q^\beta$ with $q^\beta$. 
\end{cor}

\begin{proof}
    Let $\{\bar{T}_1, \dots, \bar{T}_N\}$ be a basis of $H(X)$, and let $\{\bar{T}^1, \dots, \bar{T}^N\}$ be the dual basis. Then $\{T_\mu = \phi^*\bar{T}_\mu, T_{N+1} = E, T_{N+2} = \ell\}$ forms a basis of $H(Y)$, with dual basis $\{T^\mu = \phi^*\bar{T}^\mu, T^{N+1} = -\ell, T^{N+2} = -E\}$. For $b_1$, $b_2$ and $\tau \in H(X)$, we recall from \eqref{eq:bigquantumprod} that the quantum products are given by 
    \begin{align*}
        b_1 \ast_\tau b_2 &= \sum_{n, \bar{\beta}, \mu} \frac{q^\beta}{n!}\langle b_1, b_2, \bar{T}_\mu, \tau^{\otimes n}\rangle_\beta^X \bar{T}^\mu, \\
        \phi^*b_1 \ast_{\phi^*\tau} \phi^*b_2 &= \sum_{n, \beta, \mu} \frac{q^\beta}{n!}\langle \phi^*b_1, \phi^*b_2, T_\mu, (\phi^*\tau)^{\otimes n}\rangle_\beta^Y T^\mu \\
        &= \sum_{n, \beta, \mu} \frac{q^\beta}{n!}\langle \phi^*b_1, \phi^*b_2, \phi^*\bar{T}_\mu, (\phi^*\tau)^{\otimes n}\rangle^Y \phi^*\bar{T}^\mu \\
        &\quad - \sum_{n} \frac{q^\beta}{n!}\langle \phi^*b_1, \phi^*b_2, E, (\phi^*\tau)^{\otimes n}\rangle_\beta^Y \ell - \sum_{n} \frac{1}{n!}\langle \phi^*b_1, \phi^*b_2, \ell, (\phi^*\tau)^{\otimes n}\rangle_\beta^Y E. 
    \end{align*}
    By Theorem~\ref{thrGWmodI} (replacing $Q^\beta$ with $q^\beta$), we get
    \[\sum_\beta \langle \phi^*b_1, \phi^*b_2, \phi^*\bar{T}_\mu, (\phi^*\tau)^{\otimes n}\rangle_\beta^Y q^\beta \equiv \phi^*\sum_{\bar{\beta}}\langle b_1, b_2, \bar{T}_\mu, \tau^{\otimes n}\rangle_{\bar{\beta}}^X q^{\bar{\beta}}\pmod{\sI_Y^q}. \]
    This yields the desired result. 
\end{proof}

For $d \in \{1,2,3,4,7\}$, we can immediately take the limit as $P^\ell \to \infty$ in Theorem \ref{thrGWmodI}.

\begin{cor}\label{cor:GWlimitd=12347}
    Suppose $d \in \{1,2,3,4,7\}$. Let $n\ge 0$, $\vec{b} \in H(X)^{\otimes n}$ and $0 \neq \bar{\beta} \in \operatorname{NE}(X)$. Then $\langle \phi^*\vec{b} \rangle_{\bar{\beta}}^{\Xres \searrow \Xsm} Q^{-\tilde{\beta}}$ is a multi-valued analytic function of $P^\ell \in \bP^1$, and 
    \begin{equation}\label{eq:lim12347}
        \lim_{P^\ell \to \infty}\langle \phi^*\vec{b} \rangle_{\bar{\beta}}^{\Xres \searrow \Xsm}Q^{-\tilde{\beta}} = \langle \vec{b}\rangle_{\bar{\beta}}^X, 
    \end{equation}
    where $\tilde{\beta}$ is a lifting of $\bar{\beta}$ such that $(E, \tilde{\beta}) = 0$. 
\end{cor}

\begin{proof}
    For $d \in \{1, 2, 3, 4\}$, since $v_d = \frac{\theta f_d}{f_d}$ is bounded near $y = P^{-\ell} = 0$, every element of $\sI_d = \sS_d \cap y \cdot \bC \llbracket y \rrbracket [v_d]$ tends to $0$ as $P^\ell \to \infty$. For $d = 7$, the elements of $\sI_7 = y\cdot \sR_7$ also tend to $0$ as $P^\ell \to \infty$. Therefore \eqref{eq:lim12347} follows from \eqref{thrGWmodI_eqn}. 
\end{proof}

For the remaining $d$, we need to take the regularization before passing to the limit. Recall that $f_d^\reg$ is a nonzero holomorphic solution of \eqref{eqfODEiny} near $P^\ell = \infty$ and $v_d^\reg = \tfrac{\theta f_d^\reg}{f_d^\reg}$ (see Definition~\ref{defRI}). 

\begin{definition}\label{df:regmap}
    For $d \in \{5,6\I,6\II,8\}$, we define the \emph{regularization map} 
    \[\begin{tikzcd}[row sep = 0pt]
        \sS_d = \bC(y)[v_d] \ar[r, "(-)^\reg"] & \bC(y)[v_d^\reg] \equalscolon \sS_d^{\reg} \subseteq \bC\Laurent{y} \\
        \hphantom{\sS_d =}v_d \ar[r, mapsto] & v_d^\reg,\hphantom{\equalscolon \sS_d^{\reg}\subseteq \bC\Laurent{y}}
    \end{tikzcd}\]
    and extend it to $\sS_d \llbracket Q^{\phi^* \operatorname{NE}(X)}\rrbracket \to \sS_d^\reg \llbracket Q^{\phi^* \operatorname{NE}(X)}\rrbracket$, which we also denote by $(-)^\reg$. 
\end{definition}

\begin{remark}\label{rmk:regmap}
    Write $f_d = C(\log y)f_d^\reg + f_{d}^{(1)}$ for some constant $C \neq 0$ and holomorphic function $f_d^{(1)}$ near $y = 0$. Then 
    \begin{align*}
        v_d &= \frac{\theta f_d}{f_d} = \frac{C(\log y)\theta f_d^\reg + (\theta f_{d}^{(1)} -Cf_d^\reg)}{C(\log y)f_d^\reg + f_{d}^{(1)}} \\
        &\in \frac{\theta f_d^\reg + (\log y)^{-1}\bC\PSR{y}}{f_d^\reg + (\log y)^{-1} \bC\PSR{y}} \subseteq v_d^\reg + (\log y)^{-1} \bC\PSR{y, (\log y)^{-1}}. 
    \end{align*}
    This means that the regularization map is the restriction of the $\bC\Laurent{y}$-linear map 
    \[\begin{tikzcd}[row sep = 0pt]
        \bC \Laurent{y}\PSR{(\log y)^{-1}} \ar[r] & \bC \Laurent{y} \\
        (\log y)^{-1} \ar[r, mapsto] & 0
    \end{tikzcd}\]
    to the subset $\sS_d = \bC(y)[v_d] \subseteq \bC \Laurent{y}\PSR{(\log y)^{-1}}$. 
\end{remark}

\begin{remark}
    When $d = 4$, the local exponents of the differential equation \eqref{eqfODE} at $P^\ell = \infty$ are $(\frac12, \frac12)$. Consequently, one may define the regularization $f^{\reg}$ of $f$ to be the solution in $\bC\PSR{y^{1/2}}$ (unique up to scalar). In particular, the regularization map is also defined when $d = 4$. 
\end{remark}

The following corollary follows by a similar argument to Corollary~\ref{cor:GWlimitd=12347}. 

\begin{cor}\label{cor:GWlimitd=568}
    Suppose $d \in \{5,6\I,6\II,8\}$. Let $n\ge 0$, $\vec{b} \in H(X)^{\otimes n}$ and $0 \neq \bar{\beta} \in \operatorname{NE}(X)$. Then the regularization $\left(\langle \phi^*\vec{b} \rangle_{\bar{\beta}}^{\Xres \searrow \Xsm} Q^{-\tilde{\beta}}\right)^\reg$ is holomorphic near $P^\ell = \infty$, and 
    \[\lim_{P^\ell \to \infty} \Bigl(\langle\phi^*\vec{b} \rangle_{\bar{\beta}}^{\Xres \searrow \Xsm}Q^{-\tilde{\beta}}\Bigr)^{\reg} = \langle \vec{b}\rangle_{\bar{\beta}}^X. \]
\end{cor}

\section{\texorpdfstring{The limit of $Q^\ell$}{Limit of Ql}}

Recall that we have the mirror transform $\tau = t - g_d(P^\ell) E$ so that $Q^\ell = P^\ell e^{g_d(P^\ell)}$ (see Corollary~\ref{cor:QandP}) and a function $f_d = 1 + \theta g_d$ that satisfies the differential equation \eqref{eqfODE} with the set $\operatorname{Sing}_{f_d}$ of regular singular points (see Table~\ref{tab:Singfd}).

As an analytic function, $f_d$ admits an analytic continuation to a multi-valued function on $\bP^1 \setminus \operatorname{Sing}_{f_d}$. We can take a path $p\colon [0,1] \to \bP^1$ such that $p(0) = 0$, $p(1) = \infty$, and $p(s) \in \bP^1 \setminus \operatorname{Sing}_{f_d}$ for $s \in (0, 1)$. Then, using $f_d = 1 + \theta g_d$,
\begin{align*}
    Q^\ell(s) \coloneqq p(s) e^{g_d(p(s))} &= p(s) \exp \Bigl(\int_0^s (f_d(p(s')) - 1) \frac{\odif{p(s')}}{p(s')} \Bigr) \\
    &= \lim_{\varepsilon \to 0^+}p(\varepsilon)\exp \Bigl(\int_\varepsilon^s f_d(p(s')) \frac{\odif{p(s')}}{p(s')}\Bigr)
\end{align*}
is well-defined for $s \in [0, 1)$ and we can define the limit of $Q^\ell$ along the path $p$ to be 
\[Q_p^\ell \coloneqq \lim_{s \to 1^-} Q^\ell(s) = \lim_{\varepsilon \to 0^+} p(\varepsilon)\exp\Bigl( \int_\varepsilon^1 f_d(p(s)) \frac{\odif{p(s)}}{p(s)}\Bigr). \]
The goal of this section is to calculate $Q_p^\ell$ along some path $p$. 

\subsection{The ordinary case}

We first consider the simplest cases, namely $d \in \{1, 2, 3, 4\}$, for which $f_d$ is a hypergeometric function. 

\begin{proposition}\label{prop:Qlimd=1234}
    For $d \in \{1,2,3,4\}$, the limit of $Q^\ell$ along $[0, \infty)$ is $1$.
\end{proposition}

\begin{proof}
    We need to show that 
    \[\lim_{\varepsilon \to 0^+}\varepsilon \exp \Bigl(\int_\varepsilon^\infty f_d(p) \frac{\odif{p}}{p}\Bigr) = 1. \]
    This is equivalent to 
    \[\lim_{\varepsilon \to 0^+} \Bigl(\log \varepsilon + \int_\varepsilon^\infty f_d(p) \frac{\odif{p}}{p} \Bigr) = 0. \]
    
    Recall from Proposition~\ref{prop:fdisHG} that 
    \[f_d(p) = {}_2\rF_1\bigl(\tfrac{1}{n}, 1 - \tfrac{1}{n}; 1; -\kappa_d p\bigr), \]
    where $n = n_d = 6$, $4$, $3$, $2$ for $d = 1$, $2$, $3$, $4$, respectively. By the Euler integral representation of the hypergeometric function, we have
    \begin{equation}\label{eq:intrepoffd}
        f_d(p) = \frac{1}{\operatorname{B}(1 - \tfrac1{n}, \tfrac1{n})} \int_0^1 \frac{\odif{s}}{s^{1/n}(1-s)^{1-1/n}(1 + \kappa_d p s)^{1/n}}, 
    \end{equation}
    where $\operatorname{B}(-,-)$ is the Beta function. Substituting this into the integral, we obtain
    \begin{align}
        \int_\varepsilon^\infty f_d(p)\,\frac{\odif{p}}{p} &= \frac{1}{\operatorname{B}(1 - \tfrac1{n}, \tfrac1{n})}\int_\varepsilon^\infty \int_0^1 \frac{\odif{s,p}}{s^{1/{n}}(1-s)^{1-1/{n}}(1 + \kappa_d ps)^{1/n} p} \notag \\
        &= \frac{1}{\operatorname{B}(1 - \tfrac1{n}, \tfrac1{n})}\int_0^1\int_\varepsilon^\infty \frac{\odif{p}}{(1 + \kappa_d ps)^{1/{n}} p}\frac{\odif{s}}{s^{1/n}(1-s)^{1-1/n}}. \label{eq:intfd=1234} 
    \end{align}
    Make the substitution $x = (1 + \kappa_d ps)^{1/n}$, so that $p = \frac{x^n-1}{\kappa_d s}$ and $\odif{p} = \frac{nx^{n-1}}{\kappa_d s}\odif{x}$. Then 
    \begin{align*}
        \int_\varepsilon^\infty  \frac{\odif{p}}{(1 + \kappa_d ps)^{1/n} p} &= \int_{(1+\kappa_d\varepsilon s)^{1/n}}^{\infty} \frac{n x^{n - 2}\,\odif{x}}{x^n-1}. 
    \end{align*}
    Denote by $\omega_n$ the $n$-th root of unity $e^{2\pi \ii/n}$. Then, for $R$ large, 
    \begin{align*}
        \int_{(1+\kappa_d\varepsilon s)^{1/n}}^R \frac{n x^{n - 2}\odif{x}}{x^n-1} &= \sum_{j = 0}^{n-1} \int_{(1+\kappa_d\varepsilon s)^{1/n}}^{R} \frac{\omega_n^{-j} \odif{x}}{x - \omega_n^j} \\
        &= \int_{(1 + \kappa_d\varepsilon s)^{1/n}}^R \frac{\odif{x}}{x-1} + \sum_{j = 1}^{n-1} \int_{1}^{R} \frac{\omega_n^{-j} \odif{x}}{x - \omega_n^j} + O(\varepsilon) \\
        &= \Bigl(\log R + \sum_{j=1}^{n-1} \omega_n^{-j} \log (R - \omega_n^j)\Bigr) \\
        &\quad - \log \Bigl(\frac{\kappa_d\varepsilon s}{n}\Bigr) - \sum_{j = 1}^{n-1} \omega_n^{-j}\log(1 - \omega_n^j) + O(\varepsilon). 
    \end{align*}
    As $R \to \infty$, we have 
    \[\log R + \sum_{j=1}^{n-1} \omega_n^{-j} \log (R - \omega_n^j) = \Bigl(1 + \sum_{j=1}^{n-1} \omega_n^{-j}\Bigr) \log R + o(1) \to 0. \]
    Thus, 
    \[\int_{(1+\kappa_d\varepsilon s)^{1/n}}^\infty \frac{n x^{n - 2} \odif{x}}{x^n-1} = - \log \Bigl(\frac{\kappa_d\varepsilon s}{n}\Bigr) - \sum_{j = 1}^{n-1} \omega_n^{-j}\log(1 - \omega_n^j) + O(\varepsilon). \]
    Putting this back into the double integral \eqref{eq:intfd=1234}, we obtain
    \begin{align}
        \lim_{\varepsilon \to 0^+}\left(\log \varepsilon + \int_\varepsilon^\infty f_d(p)\,\frac{\odif{p}}{p}\right) &= -\log \Bigl(\frac{\kappa_d}{n}\Bigr) - \sum_{j = 1}^{n-1} \omega_n^{-j}\log(1 - \omega_n^j) \notag \\
        &\quad - \frac{1}{\operatorname{B}(1 - \tfrac1n, \tfrac1n)} \int_0^1 \frac{\log s \odif{s}}{s^{1/n}(1 - s)^{1-1/n}}. \label{eq:intfd=1234-2}
    \end{align}
    
    It remains to evaluate the final integral. We have
    \begin{align*}
        \int_0^1 \frac{\log s\odif{s}}{s^{1/n}(1 - s)^{1-1/n}} &= \left.\odv{}{\alpha}\right|_{\alpha = 0}\int_0^1 \frac{s^\alpha \odif{s}}{s^{1/n}(1 - s)^{1-1/n}} \\
        &= \left.\odv{}{\alpha}\right|_{\alpha = 0} \operatorname{B}(1 - \tfrac1n + \alpha, \tfrac1n). 
    \end{align*}
    Using the derivative of the Beta function 
    \[\left.\odv{}{\alpha}\right|_{\alpha = 0} \operatorname{B}(a + \alpha,b) = \left.\odv{}{\alpha}\right|_{\alpha = 0} \frac{\Gamma(a + \alpha)\Gamma(b)}{\Gamma(a+b+\alpha)} = \operatorname{B}(a, b)(\psi^{(0)}(a) - \psi^{(0)}(a+b)), \]
    where $\psi^{(0)} = \Gamma'/\Gamma$ is the digamma function, we get
    \[\frac{1}{\operatorname{B}(1 - \tfrac1n,\tfrac1n)} \int_0^1 \frac{\log s \odif{s}}{s^{1/n}(1 - s)^{1-1/n}} = \psi^{(0)}(1 - \tfrac1n) -\psi^{(0)}(1). \]
    By Gauss' digamma formula, we get
    \[\psi^{(0)}(1 - \tfrac1n) -\psi^{(0)}(1) = \sum_{m=1}^\infty \left(\frac{1}{m} - \frac{1}{m - \frac1n}\right) = - \sum_{j=1}^{n-1} (1 - \omega_n^j) \log (1 - \omega_n^j). \]
    Substituting this back into \eqref{eq:intfd=1234-2}, we find 
    \[\lim_{\varepsilon \to 0^+} \left( \log \varepsilon + \int_\varepsilon^\infty f_d(p) \, \frac{\odif{p}}{p} \right) = -\log \Bigl(\frac{\kappa_d}{n}\Bigr) - \sum_{j = 1}^{n - 1} (\omega_n^{j} + \omega_n^{-j} - 1) \log (1 - \omega_n^j). \]
    Finally, one verifies case by case that the right hand side is exactly $0$ for $d \in \{1,2,3,4\}$. 
\end{proof}

We can now rephrase Theorem~\ref{thrGWmodI} as follows: 
\begin{cor}
    Let $\Xres \searrow \Xsm$ be a Type II transition of degree $d \in \{1,2,3,4,7\}$ and let $\bar{\beta} \neq 0$. Then the analytic function 
    \[\langle \phi^*\vec{b} \rangle_{\bar{\beta}}^{\Xres \searrow \Xsm} = \sum_{\beta \to \bar{\beta}}\langle \phi^*\vec{b}\rangle_\beta^Y Q^\beta \]
    has an analytic continuation along $Q^\ell \in [0, 1)$ and 
    \[\lim_{Q^\ell \to 1} \langle \phi^*\vec{b} \rangle_{\bar{\beta}}^{\Xres \searrow \Xsm} = \langle \vec{b}\rangle_{\bar{\beta}}^X Q^{\tilde{\beta}}, \]
    where $\tilde{\beta}$ is any lifting of $\bar{\beta}$ under $\operatorname{NE}(Y) \to \operatorname{NE}(X)$. 
\end{cor}

\begin{proof}
    For $d \in \{1,2,3,4\}$,  
    \[Q^\ell = \lim_{\varepsilon \to 0^+} \varepsilon \exp \Bigl(\int_\varepsilon^{P^\ell} f_d(p) \frac{\odif{p}}{p}\Bigr)\]
    is real-valued on the interval $(0, \infty)$, and its limit as $P^\ell \to \infty$ is equal to $1$, as shown in Proposition~\ref{prop:Qlimd=1234}. Since $v_d = \frac{\theta f_d}{f_d}$ is bounded near $P^\ell = \infty$, elements of $\sI_Y$ tend to $0$ as $P^\ell \to \infty$. Hence, 
    \begin{equation}\label{eqlimQellto1}
        \lim_{Q^\ell \to 1} \langle \phi^*\vec{b} \rangle_{\bar{\beta}}^{\Xres \searrow \Xsm} = \langle \vec{b}\rangle_{\bar{\beta}}^X Q^{\tilde{\beta}}.
    \end{equation}

    For $d = 7$, it follows from \eqref{eq:g7ex} that, as $y \to 0$,  
    \[Q^\ell = P^\ell e^{g(P^\ell)} = \exp\left(- \sum \frac{b_k(1)}{k}y^k \right) \to 1. \]
    Since elements of $\sI_Y$ tend to $0$ as $y \to 0$, we also obtain \eqref{eqlimQellto1}. 
    
    Finally, note that we may choose any lifting $\tilde{\beta}$ of $\bar{\beta}$, since any two liftings differ by a multiple of $\ell$, and $Q^\ell$ tends to $1$.  
\end{proof}

The following is the first nontrivial case where the limit differs from $1$, though, still a root of unity. 

\begin{corollary}
    For $d = 8$, the limit of $Q^\ell$ along $[0, \ii\infty)$ is $\ii$.
\end{corollary}

\begin{proof}
    It follows from Proposition~\ref{prop:fdisHG} that $f_8(P^\ell) = f_4(-P^{2\ell})$. Therefore, $g_8(P^\ell) = \frac12 g_4(-P^{2\ell})$, and hence $Q^\ell = P^\ell e^{g_8(P^\ell)} = P^\ell e^{\frac12 g_4(-P^{2\ell})}$. The result then follows easily from  Proposition~\ref{prop:Qlimd=1234}. 
\end{proof}

For $d \in \{5, 6\I, 6\II\}$, we will compute the corresponding integral using modular forms in the next subsection. 

\subsection{Relation to modular curves}

Motivated by \cite{ASYZ14, FRZZ19}, we establish a relation between meromorphic quasi-modular forms and the GW invariants of $\Yloc$, and consequently those of $Y$. 

We briefly recall the definitions of modular and quasi-modular forms. More details could be found in \cite{DiaShu, Zag08}. Let $\bH = \{\tauq \in \bC \mid \operatorname{Im} \tauq > 0\}$ denote the upper half plane, $\obH = \bH \cup \bQ \cup \{\ii\infty\}$ denote the extended upper half plane, and set $q = e^{2\pi \ii \tauq}$. The group $\operatorname{SL}(2, \bZ)$ acts on $\bH$ via 
\[A\cdot \tauq = \mx {a_A}{b_A}{c_A}{d_A}\cdot \tauq = \frac{a_A \tauq + b_A}{c_A \tauq + d_A}. \]
The automorphy factor $\rj\colon \operatorname{SL}(2, \bZ) \times \bH \to \bC$ is defined by 
\[\rj(A, \tauq) = c_A\tauq + d_A. \]

A subgroup $\Gamma \subseteq \operatorname{SL}(2, \bZ)$ is called a \emph{congruence subgroup} if $\Gamma \supseteq \Gamma(N)$ for some $N \ge 1$, where 
\[\Gamma(N) = \left\{A \in \operatorname{SL}(2, \bZ) \,\middle|\, A \equiv \mx 1001\ (\operatorname{mod}N)\right\}. \]
A cusp of $\Gamma$ is a $\Gamma$-equivalence class in $\bQ \cup \{\ii\infty\}$. A point $\tau \in \bH$ is an elliptic point of $\Gamma$ if the isotropy group $\Gamma_\tau$ is nontrivial in the sense that $\{\pm I\}\Gamma_\tau \neq \{\pm I\}$.
We will use the following congruence subgroups: 
\begin{align*}
    \Gamma_0(N) &= \left\{A \in \operatorname{SL}(2, \bZ) \,\middle|\, A \equiv \mx **0*\ (\operatorname{mod}N)\right\}, \\
    \Gamma_1(N) &= \left\{A \in \operatorname{SL}(2, \bZ) \,\middle|\, A \equiv \mx 1*01\ (\operatorname{mod}N)\right\}.
\end{align*} 

A meromorphic function $F\colon \bH \to \bP^1$ is \emph{weakly modular of weight $k$ on $\Gamma$} if for each $A \in \Gamma$,  
\[F|_k^A(\tauq) \coloneqq \frac{F(A\cdot \tauq)}{\rj(A, \tauq)^k} = F(\tauq). \]

A holomorphic function $F\colon \bH \to \bC$ is a \emph{modular form of weight $k$ on $\Gamma$} if $F$ is weakly modular and $F|_k^A(\tauq)$ is holomorphic at $\ii\infty$ for all $A \in \operatorname{SL}(2, \bZ)$. The value of $F$ at $A\cdot \ii\infty$ is defined to be $F|_k^A(\ii\infty)$. 

The notion of quasi-modular forms was introduced by Kaneko and Zagier in \cite{KZ95} (see also \cite{Dijkgraaf95}), where they are defined as the ``constant term'' of almost-holomorphic modular forms. In this paper, we adopt the following direct definition (see e.g., \cite[p.58]{Zag08}), allowing meromorphicity: 
\begin{definition} \label{d:QMF}
    For $A = \smx {a_A}{b_A}{c_A}{d_A} \in \operatorname{SL}(2, \bZ)$, define 
    \[\rk(A, \tauq) = \frac{c_A}{\rj(A, \tauq)} = \frac{c_A}{c_A \tauq + d_A}. \]
    A meromorphic function $F\colon \bH \to \bP^1$ is a \emph{meromorphic quasi-modular form} of weight $k$ and depth $\le s$ on $\Gamma$ if 
    \begin{enumerate}[(1)]
        \item there exists meromorphic functions $F_0$, $F_1$, $\dots$, $F_s$ on $\bH$, meromorphic at $\ii\infty$, such that for each $A \in \Gamma$, 
        \begin{equation}\label{eq:q-mf}
            F|_k^A(\tauq) = F_0(\tauq) + \rk(A, \tauq) F_1(\tauq) + \dots + \rk(A, \tauq)^s F_s(\tauq); 
        \end{equation}
        \item for each $A \in \operatorname{SL}(2, \bZ)$, there exists meromorphic functions $F_0^A$, $F_1^A$, $\dots$, $F_s^A$ on $\bH$, meromorphic at $\ii\infty$, such that 
        \begin{equation}\label{eq:q-mfA}
            F|_k^A(\tauq) = F_0^A(\tauq) + \rk(A, \tauq)F_1^A(\tauq) + \dots + \rk(A, \tauq)^s F_s^A(\tauq). 
        \end{equation}
    \end{enumerate}
    The value of $F$ at $A\cdot \ii\infty$ is defined to be $F_0^A(\ii\infty)$. 
\end{definition}

Applying $D_q = q\odv{}{q} = \frac{1}{2\pi \ii} \odv{}{\tauq}$ to the identity \eqref{eq:q-mf}, we get 
\[(D_qF)|_{k+2}^A(\tauq) = \sum_{n = 0}^s \rk(A, \tauq)^n D_qF_n + \rk(A, \tauq)^{n+1} \frac{k - n}{2\pi \ii} F_n. \]
It follows that $D_q$ maps a meromorphic quasi-modular form of weight $k$ and depth $\le s$ to a meromorphic quasi-modular form of weight $k+2$ and depth $\le s+1$. 

Recall that for $d \neq 7$, the function $f_d(P^\ell)$ satisfies the differential equation \eqref{eqfODE}: 
\[\theta(u_d\cdot \theta f_d) + (\lambda_d P^\ell - \mu_d P^{2\ell}) \cdot f_d = u_d \cdot \theta^2 f_d + \theta u_d\cdot \theta f_d + (\lambda_d P^\ell - \mu_d P^{2\ell}) \cdot f_d = 0\]
where $u_d = 1 + \kappa_d P^\ell - \mu_d P^{2\ell}$, $\theta = P^\ell \odv{}{P^\ell}$, and the triple $(\kappa_d, \lambda_d, \mu_d)$ is listed in Table~\ref{table:klmd}. 

For $d = 1$, consider the modular function (with root and monodromy) 
\[P_1(\tauq) = \frac{1}{864}\left(\frac{\Eis_6(\tauq)}{\Eis_4(\tauq)^{3/2}} - 1\right),\] defined near $q = 0$, where 
\[\Eis_4(\tauq) = 1 + 240 \sum_{m\ge 1} \frac{m^3q^m}{1 - q^m},\quad \Eis_6(\tauq) = 1 - 504 \sum_{m\ge 1} \frac{m^5q^m}{1 - q^m}\]
are the Eisenstein series. Then, near $q = 0$, we have (see \cite[p.52]{Maier09}, \cite[\S3]{ASYZ14}, \cite[Theorem~4.2]{Coo17})
\begin{equation}\label{eq:f14=e4}
    f_1(P_1(\tauq)) = {}_2\rF_1\bigl(\tfrac16, \tfrac56; 1; -432P_1(\tauq)\bigr) = \Eis_4(\tauq)^{1/4}. 
\end{equation}
Since the $j$-invariant $j(\tauq)$ parametrizes the modular curve $\rX(\Gamma(1)) = \Gamma(1)\backslash \obH$ and
\[
j(\tauq)^{-1} = -P_1(1 + 432P_1), 
\]
the multi-valued function $P_1$ parametrizes a double cover $\widetilde{\rX}(\Gamma(1))$ of $\rX(\Gamma(1))$, ramified at $j(\tauq)^{-1} = 1728^{-1}$ and $\infty$, which correspond to the elliptic points $[\omega_4]$ and $[\omega_6]$, respectively. In Figure~\ref{fig:fund_dom}, a fundamental domain $\mathcal{D}_1$ of $\widetilde{\rX}(\Gamma(1))$ (view as a double cover of a fundamental domain of $\rX(\Gamma(1))$) contains two cusps $[\ii\infty]$, $[0]$, which correspond to $P_1 = 0$, $-\frac{1}{432}$, respectively, and an elliptic point $[\omega_6]$, which corresponds to $P_1 = \infty$ (cf.~\cite[Figure~1]{ASYZ14}). 

We now summarize the $d = 1$ case together with the cases $d \ge 2$. 
\begin{proposition}[\cite{Maier09, Zag09}] \label{p:MZ}
    For $d = 1$, there is a meromorphic modular function $P_1(\tauq) = -q + O(q^2)$ of weight $0$ on $\Gamma(1)$ (with monodromy) such that $f_1(P_1(\tauq))^4$ is the weight $4$ modular form $\Eis_4(\tauq)$ on the group $\Gamma(1)$. 
    
    For each $d \in \{2,3,4,5,6\I,6\II,8\}$, there is a congruence subgroup $\Gamma_d$ and a meromorphic modular function $P_d(\tauq) = -q + O(q^2)$ of weight $0$ on $\Gamma_d$, such that 
    \begin{enumerate}[(i)]
        \item if $d = 2$, then $f_d(P_d(\tauq))^2$ is the weight $2$ modular form $\Theta_{D_4}(\tauq)$ on the group $\Gamma(2)$; 
        \item if $d \in \{3,4,5,6\I,6\II\}$, then $f_d(P_d(\tauq))$ is a weight $1$ modular form on $\Gamma_d$; 
        \item if $d = 8$, then $f_d(P_d(\tauq))$ is a weight $1$ modular form on $\Gamma_1(8)$. 
    \end{enumerate}
    Moreover, the induced map $P_d\colon \rX(\Gamma_d) = \Gamma_d \backslash \obH \to \bP^1$ is an isomorphism between Riemann surfaces. The points which are mapped to $\operatorname{Sing}_{f_d}$ by $P_d$ correspond to cusps or elliptic points in $\rX(\Gamma_d)$, depending on whether the monodromy is infinite or finite, as listed in Table~\ref{tab:GammadSingfd}.  
    \begin{table}[H]
    \centering
    \begin{tabular}{cccc}
        \toprule
        $d$ & $\Gamma_d$ & $\operatorname{Sing}_{f_d}$ & \textup{corresponding cusps or elliptic points} \\
        \midrule
        $1$ & - & $\{0, \infty, -\frac{1}{432}\}$ & $[\ii\infty]$, $[\omega_6]$, $[0]$ \\
        \midrule
        $2$ & $\Gamma(2)$ & $\{0, \infty, -\frac{1}{64}\}$ & $[\ii\infty]$, $[\frac{1+\ii}{2}]$, $[0]$  \\
        \midrule
        $3$ & $\Gamma_1(3)$ & $\{0, \infty, -\frac{1}{27}\}$ & $[\ii\infty]$, $[\frac{1 + \omega_6}{3}]$, $[0]$ \\
        \midrule
        $4$ & $\Gamma_1(4)$ & $\{0, \infty, -\frac{1}{16}\}$ & $[\ii\infty]$, $[\frac12]$, $[0]$ \\
        \midrule
        $5$ & $\Gamma_1(5)$ & $\{0, \infty, \frac{11 - 5 \sqrt{5}}{2}, \frac{11 + 5 \sqrt{5}}{2}\}$ & $[\ii\infty]$, $[\frac25]$, $[0]$, $[\frac12]$ \\
        \midrule
        $6 \I$ & $\Gamma_1(6)$ & $\{0, \infty, -\frac{1}{8}, 1\}$ & $[\ii\infty]$, $[\frac13]$, $[0]$, $[\frac12]$  \\
        \midrule
        $6 \II$ & $\Gamma_1(6)$ & $\{0, \infty, -\frac{1}{9}, -1\}$ & $[\ii\infty]$, $[\frac12]$, $[0]$, $[\frac13]$ \\
        \midrule
        $8$ & $\Gamma_0(8)$ & $\{0, \infty, -\frac{1}{4}, \frac14\}$ & $[\ii\infty]$, $[\frac14]$, $[0]$, $[\frac12]$ \\
        \bottomrule
    \end{tabular}
    \caption{Congruence subgroups $\Gamma_d$, singular points $\operatorname{Sing}_{f_d}$, and corresponding points in $\rX(\Gamma_d)$.}
    \label{tab:GammadSingfd}
    \end{table}
\end{proposition}

\begin{proof}
    This is a classical result, the functions $f_d(P_d(\tauq))$ are known as Ramanujan's theta functions, see \cite[\S4-\S6]{Coo17}, \cite{Maier09}, \cite{Zag09}, and \cite[Table~1.2]{Zhou14} for $d = 2$, $3$, $4$. We summarize the expressions below for the reader's convenience. 
    
    As listed in \cite{Zag09} (with $P_d = -t$) and \eqref{eq:f14=e4} for $d = 1$, the modular functions $f_d$ and their corresponding $P_d$ can be expressed in Table~\ref{table:fdPd}.
    \begin{table}[H]
    \centering
    \begin{tabular}{cccc}
        \toprule
        $d$ & index & $f_d(P_d(\tauq))$ & $P_d(\tauq)$ \\
        \midrule
        $1$ & - & $\Theta_{E_8}^{1/4} = \Eis_4(\tauq)^{1/4}$ & $-\frac{\Eis_4(\tauq)^{3/2} - \Eis_6(\tauq)}{864 \Eis_4(\tauq)^{3/2}}$ \\ 
        \midrule
        $2$ & \#20 & $\Theta_{D_4}^{1/2} = (2\Eis_2(2\tauq) - \Eis_2(\tauq))^{1/2}$ & $-\frac{\eta(2\tauq)^{24}}{\eta(\tauq)^{24} + 64 \eta(2\tauq)^{24}}$ \\ 
        \midrule
        $3$ & \#14 & $\Theta_{A_2} = \vartheta_3(\tauq)$ & $-\frac{\eta(3\tauq)^{12}}{\eta(\tauq)^{12} + 27 \eta(3\tauq)^{12}}$  \\ 
        \midrule
        $4$ & \#11 & $\Theta_{\bZ_2} = \vartheta_4(\tauq)$ & $-\frac{\eta(\tauq)^{8}\eta(4\tauq)^{16}}{\eta(2\tauq)^{24}}$ \\ 
        \midrule
        $5$ & \#9 $= \mathbf{D}$  & $f_5(\tauq)$ & $P_5(\tauq)$ \\  
        \midrule
        $6\I$ & \#5 $= \mathbf{A}$ & $\frac13 \vartheta_3(\tauq) + \frac23 \vartheta_3(2\tauq)$ & $-\frac{\eta(\tauq)^{3}\eta(6\tauq)^{9}}{\eta(2\tauq)^{3}\eta(3\tauq)^{9}}$ \\ 
        \midrule
        $6\II$ & \#8 $= \mathbf{C}$ & $\frac12 \vartheta_3(\tauq) + \frac12 \vartheta_3(2\tauq)$ & $-\frac{\eta(\tauq)^{4}\eta(6\tauq)^{8}}{\eta(2\tauq)^{8}\eta(3\tauq)^{4}}$ \\
        \midrule 
        $8$ & \#2 $= \mathbf{G}$ & $\vartheta_4(2\tauq)$ & $-\frac{\eta(2\tauq)^{4}\eta(8\tauq)^{8}}{\eta(4\tauq)^{12}}$ \\ 
        \bottomrule
    \end{tabular}
    \caption{Ramanujan's theta functions $f_d$ and their associated Hauptmoduln $P_d$.}
    \label{table:fdPd}
    \end{table}
    \noindent Here,
    \begin{align*}
        \Eis_2(\tauq) = 1 - 24 \sum_{m\ge 1} \frac{mq^m}{1 - q^m}
    \end{align*}
    is the Eisenstein series, 
    \[\eta(\tauq) = q^{1/24} \prod_{m=1}^\infty (1 - q^n)\]
    is the Dedekind eta function, and 
    \[\vartheta_3(\tauq) = 1 + 6 \sum_{m\ge 1} \chi_{3,2}(m) \frac{q^m}{1 - q^m}, \quad  \vartheta_4(\tauq) = 1 + 4 \sum_{m\ge 1} \chi_{4,2}(m) \frac{q^m}{1 - q^m}\]
    are the theta series, with $\chi_{3,2} = \left(\frac{-3}{\cdot}\right)$ and $\chi_{4,2} = \left(\frac{-4}{\cdot}\right)$ denoting the unique nontrivial Dirichlet characters modulo $3$ and $4$, respectively. 
    
    For $d = 5$, 
    \[f_5(\tauq) = 1 + \sum_{m\ge 1} \left(\tfrac{3-\ii}{2}\chi_{5,2}(m) + \tfrac{3+\ii}{2}\bar{\chi}_{5,2}(m)\right) \frac{q^m}{1 - q^m}, \quad P_5(\tauq) = -q \prod_{m=1}^\infty (1 - q^m)^{5 \chi_{5,2}^2(m)}\]
    where $\chi_{5,2}$ is the Dirichlet character modulo $5$ with $\chi_{5,2}(2) = \ii$. 
    
    Using the explicit construction of $P_d$, we obtain in Table~\ref{tab:tdPd} the relation between $P_d(\tauq)$ and the canonical Hauptmodul $t_d(\tauq)$ that parametrizes the modular curve $\rX(\Gamma_0(d))$ (cf.~\cite[Table 2]{Maier09}).
    \begin{table}[H]
    \centering
    \begin{tabular}{cccccccccc}
        \toprule
        $d$ & $2$ & $3$ & $4$ & $5$ & $6$  & $8$ \\
        \midrule
        $t_d(\tauq)$ & $- \frac{64^2 P_2}{1 + 64P_2}$ & $- \frac{27^2 P_3}{1 + 27P_3}$ & $- \frac{16^2 P_4}{1 + 16P_4}$ & $-\frac{125P_5}{1 + 11P_5 - P_5^2}$ & $- \frac{72 P_{6\I}}{1 + 8P_{6\I}} = - \frac{72 P_{6\II}}{1 + 9P_{6\II}}$ & $- \frac{32 P_8}{1 + 4P_8}$ \\
        \bottomrule
    \end{tabular}
    \caption{Relations between $P_d(\tauq)$ and the canonical Hauptmodul $t_d(\tauq)$.}
    \label{tab:tdPd}
    \end{table}
    \noindent These identities shows that $P_d$ parametrize the modular curve $\rX(\Gamma_d)$, and moreover, by the values of $t_d$ at cusps and elliptic points listed in \cite[Table 2]{Maier09}, each regular singular point corresponds to the desired cusp or elliptic point as listed in Table~\ref{tab:GammadSingfd}.  
\end{proof}

\begin{definition}\label{def;tran_pt}
    The point corresponding to $P^\ell = \infty$ in Table~\ref{tab:GammadSingfd} (and shown in Figure~\ref{fig:fund_dom}) is called the \emph{transition point} of degree $d$, denoted by $\trpt{d}$ ($d \neq 7$). 
\end{definition}

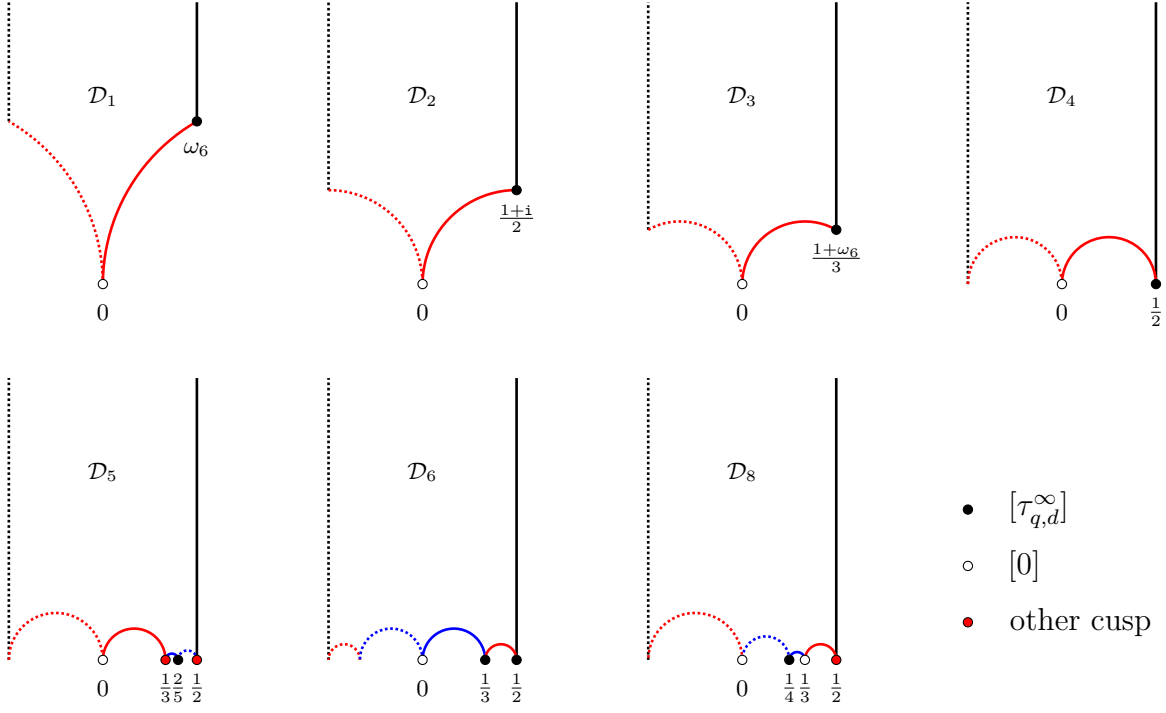
\begin{figure}[htbp!]
    \centering
    \begin{tikzpicture}[scale = 2.5]
    \begin{scope}[shift = {(0,0)}] 
        \draw [line width = 1, densely dotted] (-0.5,0.8660254037844388) -- (-0.5, 1.5); 
        \draw [line width = 1] (0.5,0.8660254037844388) -- (0.5, 1.5); 
        \draw [shift={(-1.,0.)},line width=1pt, color = red, densely dotted]  plot[domain=0.:1.0471975511965976,variable=\t]({1.*1.*cos(\t r)+0.*1.*sin(\t r)},{0.*1.*cos(\t r)+1.*1.*sin(\t r)});
        \draw [shift={(1.,0.)},line width=1pt, color = red]  plot[domain=2.0943951023931953:3.141592653589793,variable=\t]({1.*1.*cos(\t r)+0.*1.*sin(\t r)},{0.*1.*cos(\t r)+1.*1.*sin(\t r)});
        \begin{scriptsize}
        \draw [fill=white] (0,0) circle (0.7pt);   
        \draw (0,-0.15) node {$0$};
        \draw [fill=black] (0.5,0.8660254037844388) circle (0.7pt);
        \draw (0.5,0.7160254037844388) node {$\omega_6$};
        \draw (0,1) node {$\mathcal{D}_1$}; 
        \end{scriptsize}
    \end{scope}
    
    \begin{scope}[shift = {(1.7,0)}] 
        \draw [line width = 1, densely dotted] (-0.5,0.5) -- (-0.5, 1.5); 
        \draw [line width = 1] (0.5,0.5) -- (0.5, 1.5); 
        \draw [shift={(-0.5,0.)},line width=1pt, densely dotted, color = red]  plot[domain=1.5714378619914093:3.141592653589793,variable=\t]({-1.*0.5*cos(\t r)+0.*0.5*sin(\t r)},{0.*0.5*cos(\t r)+1.*0.5*sin(\t r)});
        \draw [shift={(0.5005555858538224,0.)},line width=1pt, color = red]  plot[domain=1.5714378619914093:3.141592653589793,variable=\t]({1.*0.5005555858538224*cos(\t r)+0.*0.5005555858538224*sin(\t r)},{0.*0.5005555858538224*cos(\t r)+1.*0.5005555858538224*sin(\t r)});
        \begin{scriptsize}
        \draw [fill=white] (0,0) circle (0.7pt);
        \draw (0,-0.15) node {$0$}; 
        \draw [fill=black] (0.5,0.5) circle (0.7pt);
        \draw (0.5,0.35) node {$\tfrac{1 + \ii}{2}$};
        \draw (0,1) node {$\mathcal{D}_2$}; 
        \end{scriptsize}
    \end{scope}
    
    \begin{scope}[shift = {(3.4,0)}] 
        \draw [line width = 1, densely dotted] (-0.5,0.28910290390264715) -- (-0.5, 1.5); 
        \draw [line width = 1] (0.5,0.28910290390264715) -- (0.5, 1.5); 
        \draw [shift={(-0.33333,0.)},line width=1pt, densely dotted, color = red]  plot[domain=1.0478392052023793:3.141592653589793,variable=\t]({-1.*0.333333*cos(\t r)+0.*0.333333*sin(\t r)},{0.*0.333333*cos(\t r)+1.*0.333333*sin(\t r)});
        \draw [shift={(0.33333,0.)},line width=1pt, color = red]  plot[domain=1.0478392052023793:3.141592653589793,variable=\t]({1.*0.333333*cos(\t r)+0.*0.333333*sin(\t r)},{0.*0.333333*cos(\t r)+1.*0.333333*sin(\t r)});
        \begin{scriptsize}
        \draw [fill=white] (0,0) circle (0.7pt);
        \draw (0,-0.15) node {$0$}; 
        \draw [fill=black] (0.5,0.28910290390264715) circle (0.7pt);
        \draw (0.5,0.13910290390264715) node {$\tfrac{1 + \omega_6}{3}$};
        \draw (0,1) node {$\mathcal{D}_3$}; 
        \end{scriptsize}
    \end{scope}
    
    \begin{scope}[shift = {(5.1,0)}] 
        \draw [line width = 1, densely dotted] (-0.5,0) -- (-0.5, 1.5); 
        \draw [line width = 1] (0.5,0) -- (0.5, 1.5); 
        \draw [shift={(-0.25,0.)},line width=1pt, densely dotted, color = red]  plot[domain=0.:3.141592653589793,variable=\t]({-1.*0.25*cos(\t r)+0.*0.25*sin(\t r)},{0.*0.25*cos(\t r)+1.*0.25*sin(\t r)});
        \draw [shift={(0.25,0.)},line width=1pt, color = red]  plot[domain=0.:3.141592653589793,variable=\t]({1.*0.25*cos(\t r)+0.*0.25*sin(\t r)},{0.*0.25*cos(\t r)+1.*0.25*sin(\t r)});
        
        \begin{scriptsize}
        \draw [fill=white] (0,0) circle (0.7pt);
        \draw (0,-0.15) node {$0$};
        \draw [fill=black] (0.5,0) circle (0.7pt);
        \draw (0.5,-0.15) node {$\tfrac12$};
        \draw (0,1) node {$\mathcal{D}_4$}; 
        \end{scriptsize}
    \end{scope}
    \begin{scope}[shift = {(0,-2)}] 
        \draw [line width = 1, densely dotted] (-0.5, 0) -- (-0.5, 1.5); 
        \draw [line width = 1] (0.5, 0) -- (0.5, 1.5); 
        \draw [shift={(-0.25,0.)},line width=1pt, color=red, densely dotted]  plot[domain=0.:3.141592653589793,variable=\t]({-1.*0.25*cos(\t r)+0.*0.25*sin(\t r)},{0.*0.25*cos(\t r)+1.*0.25*sin(\t r)});
        \draw [shift={(00.166666,0.)},line width=1pt, color=red]  plot[domain=0.:3.141592653589793,variable=\t]({1.*00.166666*cos(\t r)+0.*00.166666*sin(\t r)},{0.*00.166666*cos(\t r)+1.*00.166666*sin(\t r)});
        \draw [shift={(0.36666,0.)},line width=1pt, color=blue]  plot[domain=0.:3.141592653589793,variable=\t]({1.*0.033333*cos(\t r)+0.*0.033333*sin(\t r)},{0.*0.033333*cos(\t r)+1.*0.033333*sin(\t r)});
        \draw [shift={(0.45,0.)},line width=1pt, color=blue, densely dotted]  plot[domain=0.:3.141592653589793,variable=\t]({1.*0.05*cos(\t r)+0.*0.05*sin(\t r)},{0.*0.05*cos(\t r)+1.*0.05*sin(\t r)});
        \begin{scriptsize}
        \draw [fill=white] (0,0) circle (0.7pt); 
        \draw (0,-0.15) node {$0$};
        \draw [fill=black] (0.4,0) circle (0.7pt);
        \draw (0.4,-0.15) node {$\tfrac25$};
        \draw [fill=red] (0.5,0) circle (0.7pt); 
        \draw (0.5,-0.15) node {$\tfrac12$};
        \draw [fill=red] (0.3333,0) circle (0.7pt);
        \draw (0.3333,-0.15) node {$\tfrac13$};
        \draw (0,1) node {$\mathcal{D}_5$}; 
        \end{scriptsize}
    \end{scope}
    
    \begin{scope}[shift = {(1.7,-2)}] 
        \draw [line width = 1, densely dotted] (-0.5,0) -- (-0.5, 1.5); 
        \draw [line width = 1] (0.5,0) -- (0.5, 1.5); 
        \draw [shift={(-0.416666,0.)},line width=1pt, color=red, densely dotted]  plot[domain=0.:3.141592653589793,variable=\t]({1.*0.083333*cos(\t r)+0.*0.083333*sin(\t r)},{0.*0.083333*cos(\t r)+1.*0.083333*sin(\t r)});
        \draw [shift={(-0.16666,0.)},line width=1pt, color=blue, densely dotted]  plot[domain=0.:3.141592653589793,variable=\t]({1.*0.16666*cos(\t r)+0.*0.16666*sin(\t r)},{0.*0.16666*cos(\t r)+1.*0.16666*sin(\t r)});
        \draw [shift={(0.16666,0.)},line width=1pt, color=blue]  plot[domain=0.:3.141592653589793,variable=\t]({1.*0.16666*cos(\t r)+0.*0.16666*sin(\t r)},{0.*0.16666*cos(\t r)+1.*0.16666*sin(\t r)});
        \draw [shift={(0.416666,0.)},line width=1pt, color=red]  plot[domain=0.:3.141592653589793,variable=\t]({1.*0.083333*cos(\t r)+0.*0.083333*sin(\t r)},{0.*0.083333*cos(\t r)+1.*0.083333*sin(\t r)});
        \begin{scriptsize}
        \draw [fill=white] (0,0) circle (0.7pt);
        \draw (0,-0.15) node {$0$};
        \draw [fill=black] (0.5,0) circle (0.7pt);
        \draw (0.5,-0.15) node {$\tfrac12$};
        \draw [fill=black] (0.3333,0) circle (0.7pt);
        \draw (0.3333,-0.15) node {$\tfrac13$};
        \draw (0,1) node {$\mathcal{D}_6$}; 
        \end{scriptsize}
    \end{scope}
    
    \begin{scope}[shift = {(3.4,-2)}] 
        \draw [line width = 1, densely dotted] (-0.5,0) -- (-0.5, 1.5); 
        \draw [line width = 1] (0.5,0) -- (0.5, 1.5); 
        \draw [shift={(-0.25,0.)},line width=1pt, color=red, densely dotted]  plot[domain=0.:3.141592653589793,variable=\t]({1.*0.25*cos(\t r)+0.*0.25*sin(\t r)},{0.*0.25*cos(\t r)+1.*0.25*sin(\t r)});
        \draw [shift={(0.125,0.)},line width=1pt, color=blue, densely dotted]  plot[domain=0.:3.141592653589793,variable=\t]({1.*0.125*cos(\t r)+0.*0.125*sin(\t r)},{0.*0.125*cos(\t r)+1.*0.125*sin(\t r)});
        \draw [shift={(0.29166,0.)},line width=1pt, color=blue]  plot[domain=0.:3.141592653589793,variable=\t]({1.*0.04166*cos(\t r)+0.*0.04166*sin(\t r)},{0.*0.04166*cos(\t r)+1.*0.04166*sin(\t r)});
        \draw [shift={(0.416666,0.)},line width=1pt, color=red]  plot[domain=0.:3.141592653589793,variable=\t]({1.*0.083333*cos(\t r)+0.*0.083333*sin(\t r)},{0.*0.083333*cos(\t r)+1.*0.083333*sin(\t r)});
        \begin{scriptsize}
        \draw [fill=white] (0,0) circle (0.7pt);   
        \draw (0,-0.15) node {$0$};
        \draw [fill=red] (0.5,0) circle (0.7pt);
        \draw (0.5,-0.15) node {$\tfrac12$};
        \draw [fill=white] (0.3333,0) circle (0.7pt);
        \draw (0.3333,-0.15) node {$\tfrac13$};
        \draw [fill=black] (0.25,0) circle (0.7pt);
        \draw (0.25,-0.15) node {$\tfrac14$};
        \draw (0,1) node {$\mathcal{D}_8$}; 
        \end{scriptsize}
    \end{scope}

    \begin{scope}[shift = {(4.6,-0.7)}]
        \draw [fill=black] (0,-0.5) circle (0.7pt); 
        \draw (0.375,-0.5) node {$[\trpt{d}]$}; 
        \draw [fill=white] (0,-0.8) circle (0.7pt); 
        \draw (0.3,-0.8) node {$[0]$}; 
        \draw [fill=red] (0,-1.1) circle (0.7pt); 
        \draw (0.6,-1.1) node {other cusp}; 
    \end{scope}
    
    \end{tikzpicture}
    \caption{Fundamental domain $\mathcal{D}_d$ of $\rX(\Gamma_d)$ with cusps and elliptic points.}
    \label{fig:fund_dom}
\end{figure}

Since the coefficients of the differential equation \eqref{eqfODE} and the operator $\theta$ is invariant under the action of $\Gamma_d$, we see that $f_d(A \tauq) = \rj(A, \tauq) f_d(\tauq)$ is also a solution for $A \in \Gamma_d$. In particular, $\tauq f_d$ is a solution as well. This means that for any solution $f$, the following Wronskian determinant vanishes: 
\begin{align*}
    0 &= \begin{vmatrix}
        f & f_d & \tauq f_d \\ 
    \theta f & \theta f_d & \theta (\tauq f_d) \\ 
    \theta^2 f & \theta^2 f_d & \theta^2(\tauq f_d)
    \end{vmatrix} = W_2\cdot \theta^2 f + W_1 \cdot \theta f + W_0\cdot f.  
\end{align*}
Together with the differential equation \eqref{eqfODE}, this shows that 
\[\theta \log u_d = \frac{\theta u_d}{u_d} = \frac{W_1}{W_2} = \frac{-f_d\cdot (\theta^2\tauq\cdot f_d + 2\theta\tauq\cdot\theta f_d)}{\theta \tauq \cdot f_d^2} = - \frac{\theta^2 \tauq}{\theta \tauq} - 2 \frac{\theta f_d}{f_d} = -\theta\log (\theta \tauq \cdot f_d^2). \]
Hence, the quantity
\[u_df_d^2 \cdot \odv{\log q}{\,\log P^\ell} = 2\pi \ii\cdot u_df_d^2 \cdot \theta \tauq \]
is constant. Since $P^\ell = - q + O(q^2)$ near $q = 0$ and $u_df_d^2 = 1 + O(P^\ell)$, it follows that 
\begin{equation}\label{eq:dP/dq}
    \odv{\log P^\ell}{\,\log q} = u_df_d^2. 
\end{equation}
This implies
\[\odv{\log Q^\ell}{\log q} = \odv{\log P^\ell}{\log q}\cdot \odv{\log Q^\ell}{\,\log P^\ell} = u_df_d^3. \]
Also, we have $Q^\ell = -q + O(q^2)$. 

Note that, by Proposition~\ref{propGWYloc} \eqref{propGWYloc_5}, $u_df_d^3$ is equal to the ratio
\[\frac{d}{\langle E, E, E\rangle^{Y_d}} = \frac{E^3}{\sum_{m\ge 0}\langle E, E, E\rangle_{m\ell}^{Y_d} Q^{m\ell}} = \rE_d.\]
We are therefore led to the following definition: 
\begin{definition}\label{def;extr_fun}
    For a Type II extremal transition $\Xres \searrow \Xsm$ with exceptional divisor $E$ and exceptional curve $\ell$, we define its \emph{extremal function} to be 
    \[\rE^{\Xres \searrow \Xsm} = \frac{E^3}{\sum_{m\ge 0}\langle E, E, E\rangle_{m\ell}^Y Q^{m\ell}}. \]
    Then the \emph{canonical modular coordinate} $q$ of the transition is defined by
    \begin{equation}\label{eq:dfofq}
        \odv{\log q}{\,\log Q^\ell} = \frac{1}{\rE^{\Xres \searrow \Xsm}},\quad q = -Q^\ell + O(Q^{2\ell}). 
    \end{equation}
\end{definition}

The following proposition shows that extremal functions depend only on the degree of $\Xres \searrow \Xsm$:

\begin{proposition}
    If $\Xres \searrow \Xsm$ is a Type II extremal transition of degree $d$, then
    \[\rE^{\Xres \searrow \Xsm} = \rE_d. \]
\end{proposition}

\begin{proof}
    Let $E$ be the exceptional divisor appearing in $\Xres \searrow \Xsm$. By the degeneration formula~\eqref{eqn;degen_fomula}, we have
    \[\sum_{m \ge 0} \langle E, E, E\rangle_{m\ell}^Y Q^{m\ell} = \sum_{m\ge m'\ge 0}\sum_{\eta} C_\eta \langle \varnothing \mid e_I, \mu\rangle_{\Gamma_0}^{\bullet(Y, E)}Q^{m'\ell} \langle E, E, E\mid e^I, \mu \rangle_{\Gamma_\infty}^{\bullet(\Yloc, H)} Q^{(m-m')\ell}. \]
    Since $|\mu| = (E, m'\ell) = -m' \le 0$, we must have $m' = 0$, which reduces the sum to
    \[\sum_{m\ge 0} \langle E, E, E\mid \varnothing \rangle_{m\ell}^{(\Yloc, H)} Q^{m\ell}. \]
    This quantity is independent of the choice of $Y$, and therefore equals $\sum_{m\ge 0} \langle E, E, E\rangle_{m\ell}^{\Yloc} Q^{m\ell}$. We thus conclude that $\rE^{\Xres \searrow \Xsm} = \rE^{\Yloc\searrow \Xloc} = \rE_d$. 
\end{proof}

Since $u_d$ is a weight $0$ meromorphic modular function and $f_d$ is a weight $1$ modular form (with root and monodromy if $d \le 2$), we see that $\rE_d = u_d f_d^3$ is a weight $3$ modular form. 

It follows from~\eqref{eq:dP/dq} that
\[v_d = \frac{\theta f_d}{f_d} = \frac{1}{f_d}\cdot \odv{\log q}{\,\log P^{\ell}}\cdot \odv{f_d}{\log q} = \frac{D_q f_d}{u_d f_d^3}\]
where $D_q = \odv{}{\log q} = q \odv{}{q}$. Since $f_d$ is a modular form of weight $1$, $D_q f_d$ is a quasi-modular form of weight $3$ and depth $1$. Therefore, $v_d$ is a quasi-modular form of weight $0$ and depth $1$ (with root and monodromy if $d = 1$). This implies that every element in $\sS_d = \bC(y)[v_d]$ is a quasi-modular form of weight $0$. 

For the cases $d \in \{4, 5, 6\I, 6\II, 8\}$, the transition point $[\trpt{d}]$ is a cusp. Take a representative $r \in [\trpt{d}]$ and an element $A = \smx {a_A}{b_A}{c_A}{d_A} \in \operatorname{SL}(2, \bZ)$ such that $A\cdot \ii\infty = r$. Then, since 
\[f_d|_1^A(A^{-1}\tauq) = (a_A - c_A \tauq)f_d(\tauq)\quad\text{and}\quad f_d|_1^A(\tauq) \in \bC\llbracket q \rrbracket, \]
the holomorphic solution $f_d^\reg(\tauq)$ of \eqref{eqfODE} near $y = 0$ is proportional to $(a_A - c_A \tauq)f_d(\tauq)$. Hence, 
\[v_d^\reg = \frac{\theta f_d^\reg}{f_d^\reg} = \frac{1}{u_d f_d^2}\cdot \frac{D_q ((a_A - c_A \tauq)f_d(\tauq))}{(a_A - c_A \tauq)f_d(\tauq)} = \frac{D_q f_d}{u_d f_d^3} -\frac{1}{2\pi \ii}\frac{c_A}{(a_A - c_A \tauq) u_df_d^2} \]
and therefore, 
\[v_d^\reg(A\tauq) = \frac{D_q f_d|_3^A(\tauq) - \frac{1}{2\pi\ii} \rk(A, \tauq) f_d}{u_d|_0^A(\tauq) f_d|_1^A(\tauq)^3} = \left(\frac{D_q f_d}{u_d f_d^3}\right)_0^A (\tauq), \]
where $F_0^A$ is defined in \eqref{eq:q-mfA} for a meromorphic quasi-modular form $F$. 
From this, we obtain the modular version of the main result: 
\begin{theorem}\label{thr:mainmodularver}
    Let $\Xres \searrow \Xsm$ be a Type II transition of degree $d \neq 7$ and let $\bar{\beta} \neq 0$. Let $\tilde{\beta}$ be a lifting of $\bar{\beta}$ such that $(E, \tilde{\beta}) = 0$. Then, in the canonical modular coordinate $q$ of $\Xres \searrow \Xsm$, 
    \[\langle \phi^*\vec{b} \rangle_{\bar{\beta}}^{\Xres \searrow \Xsm}Q^{-\tilde{\beta}} = \sum_{\beta \to \bar{\beta}}\langle \phi^*\vec{b}\rangle_\beta^Y Q^{\beta - \tilde{\beta}} = \sum_m \langle \phi^*\vec{b}\rangle^Y_{\tilde{\beta} + m\ell} Q^{m\ell} \]
    is a meromorphic quasi-modular form of weight $0$. Moreover, its value at the transition point $[\trpt{d}]$ coincides with the invariant $\langle \vec{b}\rangle_{\bar{\beta}}^X$. 
\end{theorem}

\begin{remark}\label{rmk:mdisD}
    Recall that for all $\vec{b}\in H(X)^{\otimes n}$, the coefficient of $\langle \phi^*\vec{b} \rangle_{\bar{\beta}}^{\Xres \searrow \Xsm}$ lies in $\sS_d = \bC(P^\ell)[v_d]$. Let $\mathcal{D}$ be the maximal domain of the invariants. Then, for $d \in \{2,3,4,5,6,8\}$, since $v_d$ has to be defined on $\mathcal{D}$ and $v_d$ has logarithmic singularities at each rational point $\tauq \in \bQ$, we see that $\bH \cup \{\ii\infty\}$ is a cover of $\mathcal{D}$. Since $P^\ell$ is well-defined on $\mathcal{D}$, we have the following diagram: 
    \[\begin{tikzcd}
        \bH \cup \{\ii\infty\} \ar[r, "\varphi"] & \mathcal{D} \ar[r, "\psi"] \ar[rd, "P^\ell"'] & \rX(\Gamma_d) \ar[d, "\wr"]\\
        && \bP^1.
    \end{tikzcd}\]
    If $\varphi(\tauq') = \varphi(\tauq'')$ for some $\tauq'$, $\tauq'' \in \bH$, then $\psi(\varphi(\tauq')) = \psi(\varphi(\tauq''))$ implies that $\tauq'' = A\tauq'$ for some $A \in \Gamma_d$. Since $v_d$ is defined on $\mathcal{D}$, we get $v_d(A\tauq) = v_d(\tauq)$ near $\tauq = \tauq'$, and hence $v_d(A\tauq) = v_d(\tauq)$ for all $\tauq \in \bH$, as they are both holomorphic on $\bH$. Using  
    \[v_d(A\tauq) = v_d(\tauq) +  \frac{\rk(A, \tauq)}{2\pi \ii}\frac{1}{u_d f_d^2}, \]
    we must have $\rk(A, \tauq) = 0$, i.e., $A = \smx {1}{b_A}{0}{1}$ for some $b_A$. This means that the maximal domain
    \[\mathcal{D} = \{ \smx 1{b_A}01\} \backslash (\bH \cup \{\ii\infty\})\]
    is exactly the unit disk $\mathbb{D} = \{q \mid |q| < 1\}$. 
\end{remark}

\subsection{The wild case: periods of Eisenstein series} \label{s:wild}

Let $k \geq 3$ be an integer, and let $\chi$ and $\psi$ be Dirichlet characters modulo $N_\chi$ and $N_\psi$, respectively, so that $\chi\psi(-1) = (-1)^k$. Given a positive integer $n$, we denote by $\Eis_k^{\chi, \psi, n}(\tauq)$ the normalized Eisenstein series of weight $k$ with parameter $(\chi, \psi, n)$, which has the following Fourier expansion:
\begin{equation}\label{eq:defEis}
    \Eis_k^{\chi, \psi, n}(\tauq) = \frac12 \delta_{\bm{1}_1, \chi}\rL(1-k, \psi) + \sum_{m\ge 1} \psi(m)m^{k-1} \sum_{\nu=1}^{N_\chi} \frac{\chi(\nu) q^{\nu mn}}{1 - q^{N_\chi mn}}, 
\end{equation}
where $\delta_{\bm{1}_1, \chi}$ is $1$ if $\chi = \bm{1}_1$ and is $0$ otherwise. 
\begin{proposition}\label{prop:EdinEis}
    For $d \in \{3,4,5,6\}$, the weight $3$ modular form $\rE_d$, as a function of $\tauq$, is the Eisenstein series 
    \[\Eis_{\ii\infty, 3}^{\Gamma_1(d)}(\tau_q) \coloneqq \sum_{(a,c) \equiv (1,0)\,\mathrm{mod}\,d} \frac{1}{(-c\tauq + a)^3} \]
    of weight $3$ associate with the cusp $\ii\infty$ of the group $\Gamma_1(d)$. In particular, for $d \in \{3,4,5,6,8\}$, $\rE_d$ can be written in terms of the normalized Eisenstein series: 
    \begin{enumerate}[(i)]
        \item $\rE_3 = -9\Eis_3^{\bm{1}_1,\chi_{3,2},1}$; 
        \item $\rE_4 = -4\Eis_3^{\bm{1}_1,\chi_{4,2},1}$;  
        \item $\rE_5 = (-1 + \frac{\ii}{2})\Eis_3^{\bm{1}_1,\chi_{5,2},1} + (-1 - \frac{\ii}{2})\Eis_3^{\bm{1}_1,\chi_{5,3},1}$; 
        \item $\rE_6 = -\Eis_3^{\bm{1}_1,\chi_{3,2},1} - 8\Eis_3^{\bm{1}_1,\chi_{3,2},2}$; 
        \item $\rE_8 = -4\Eis_3^{\bm{1}_1,\chi_{4,2},2}$, 
    \end{enumerate}
    where $\chi_{3,2} = \left(\frac{-3}{\cdot}\right)$ and $\chi_{4,2} = \left(\frac{-4}{\cdot}\right)$ are the unique nontrivial Dirichlet characters modulo $3$ and $4$, respectively, $\chi_{5, 2}$ is the Dirichlet character modulo $5$ defined by $\chi_{5,2}(2) = \ii$, and $\chi_{5,3} = \bar{\chi}_{5,2} = \chi_{5,2}^3$. 
\end{proposition}

\begin{proof}
    It follows from Table~\ref{tab:Singfd} and Table~\ref{tab:GammadSingfd} that $\rE_d = u_d f_d^3$ vanishes at all cusps except at $[\ii\infty]$. Since $\rE_d = 1 + O(q)$, this means that the difference $\rE_d - \Eis_{\ii\infty, 3}^{\Gamma_1(d)}$ is a weight $3$ cusp form on $\Gamma_1(d)$. 
    
    By the dimension formula for spaces of cusp forms (see, e.g., \cite[Theorem 3.6.1]{DiaShu}), we have $\dim \mathcal{S}(\Gamma_1(d)) = 0$ for $d \in \{3,4,5,6\}$. Hence, the difference must be zero, i.e., $\rE_d = \Eis_{\ii\infty, 3}^{\Gamma_1(d)}$. 
    
    Using the explicit formula for Eisenstein series given in \cite[Equation~(4.6), Theorem~4.2.3]{DiaShu}, we obtain the expressions of $\rE_d = \Eis_{\ii\infty, 3}^{\Gamma_1(d)}$ in (i) - (iv). For (v), we compute
    \[\rE_8(\tauq) = u_8(\tauq)f_8(\tauq)^3 = u_4(2\tauq)f_4(2\tauq)^3 = \rE_4(2\tauq) = - 4 \Eis_3^{\bm{1}_1, \chi_{4,2}, 2}. \qedhere\]
\end{proof}

Any path $p\colon [0,1] \to \bP^1$ can be lifted to a path $\tilde{p} \colon [0,1] \to \obH$ via the projection $\obH \to \rX(\Gamma_d) \cong \bP^1$. Suppose $p(0) = 0$, $p(1) = \infty$; then $\tilde{p}(0) = \ii\infty$ and $\tilde{p}(1)$ will be a rational number $r \in \bQ$ that represents the transition point $[\trpt{d}]$. It follows from the definition of $q$ in~\eqref{eq:dfofq} that, along the path $\tilde{p}$, we have
\[Q^\ell(s) = - e^{2\pi \ii \tilde{p} 
(s)}\exp\left(\int_{\ii \infty}^{\tilde{p}(s)} (\rE_d(\tauq) - 1) \odif{(2\pi \ii \tauq)}\right). \]
In particular, the limit $\lim_{s\to 1}Q^\ell(s)$ exists and depends only on the endpoint $r$. We denote this limiting value by $Q_r^\ell$. 

\begin{definition}\label{df:Q_r^l}
    For a rational number $r$, we define
    \[Q_r^\ell = - e^{2\pi \ii r} \exp \left(\int_{\ii \infty}^{r} (\rE_d(\tauq) - 1) \odif{(2\pi \ii \tauq)}\right). \] 
\end{definition}

Consider the translation $\tauq \mapsto \tauq + r$, we have 
\[\int_{\ii \infty}^{r} (\rE_d(\tauq) - 1) \odif{(2\pi \ii \tauq)} = \int_{\ii \infty}^{0} (\rE_d(\tauq + r) - 1) \odif{(2\pi \ii \tauq)}. \]
By Proposition~\ref{prop:EdinEis}, when $d = 5$, $6\I$, $6\II$, $8$, $\rE_d$ can be written as a linear combination of Eisenstein series $\sum_{\psi, n} \alpha_{\psi, n}\Eis_3^{\bm{1}_1,\psi,n}$. So 
\[\int_{\ii \infty}^{0} (\rE_d(\tauq + r) - 1) \odif{(2\pi \ii \tauq)} = \int_{\ii\infty}^{0} \left(\sum \alpha_{\psi, n}\Eis_3^{1,\psi,n}(\tauq + r) - 1\right) \odif{(2\pi \ii \tauq)}\]

The following proposition expresses the translated Eisenstein series $\Eis_3^{\bm{1}_1, \psi, n}(\tauq + r)$ explicitly as a linear combination of Eisenstein series. 
\begin{prop}[Translation formula] \label{prop:TFforEis}
    Let $r = \frac ac$ be a rational number with $\gcd (a, c) = 1$. Define 
    \[\Eis_r^{\psi}(\tauq) = \Eis_3^{\bm{1}_1,\psi,1}(\tauq + r).\] 
    Then 
    \[\Eis_r^{\psi} = \sum_{c = c_1c_2c_3}  \sum_{\chi \in \widehat{(\bZ/c_3\bZ)^\times}} \frac{\chi(a)\fg(\bar{\chi})\psi(c_1)c_1^2}{\varphi(c_3)}\cdot  \Eis_3^{\chi, \bm{1}_{c_2c_3}\chi \psi,c_1c_2}, \]
    where $\fg(-)$ denotes the Gauss sum of a Dirichlet character and $\varphi(-)$ denotes the Euler's totient function. 
\end{prop}

\begin{remark}
    The case $n > 1$ reduces to the case $n = 1$ by scaling: for any integer $n > 1$, we have 
    \[\Eis_3^{\bm{1}_1, \psi, n}(\tauq + r) = \Eis_3^{\bm{1}_1, \psi}(n\tauq + nr) = \Eis_{nr}^\psi(n\tauq). \]
\end{remark}

\begin{proof}
    Let $\omega = e^{2\pi \ii r}$. Define the non-constant part of the translated Eisenstein series by 
    \[\Eis_{r,+}^\psi \coloneqq \Eis_r^{\psi} - a_0(\Eis_r^{\psi}) = \sum \psi(m) m^2 \frac{\omega^m q^m}{1 - \omega^m q^m}. \]
    Expanding the denominator, we write
    \[\frac{\omega^m q^m}{1 - \omega^m q^m} = \sum_{n = 1}^c \frac{\omega^{nm} q^{nm}}{1 - q^{cm}}. \]
    Let $c_1 = \gcd (c, m)$ and write $c = c_1c'$, $m = c_1m'$. Then 
    \[\sum_{n = 1}^c \omega^{nm} q^{nm} = \frac{1 - q^{cm}}{1 - q^{c'm}}\sum_{n = 1}^{c'} e^{2\pi \ii anm'/c'} q^{nm}. \]
    Next, for each $n$, let $c_2 = \gcd (n, c')$ and write $c' = c_2c_3$, $n = c_2n'$. Applying orthogonality of characters on $(\mathbb{Z}/c_3\mathbb{Z})^\times$, we get
    \begin{align*}
        \sum_{n = 1}^{c'} e^{2\pi \ii anm'/c'} q^{nm} &= \sum_{c' = c_2c_3} \sum_{n'\perp c_3} e^{2\pi \ii an'm'/c_3} q^{c_1c_2n'm'} \\
        &= \sum_{c' = de} \sum_{n'\perp c_3} \Biggl(\frac{1}{\varphi(c_3)}\sum_{\chi \in \widehat{(\mathbb{Z}/c_3\mathbb{Z})^\times}} \chi(a)\chi(m') \fg(\bar{\chi})\cdot \chi(n')\Biggr) q^{c_1c_2n'm'}. 
    \end{align*}
    Here we use the notation $n_1 \perp n_2$ to mean that $n_1$ and $n_2$ are coprime. 
    
    Putting everything together: 
    \begin{align*}
        \Eis_{r,+}^\psi &= \sum \psi(m) m^2 \sum_{n = 1}^c  \frac{\omega^{nm} q^{nm}}{1 - q^{cm}} \\
        &= \sum_{c = c_1c_2c_3} \sum_{m' \perp c_2c_3} \sum_{n'\perp c_3} \sum_{\chi \in \widehat{(\mathbb{Z}/c_3\mathbb{Z})^\times}} \frac{\chi(a)\fg(\bar{\chi})\psi(c_1)c_1^2}{\varphi(c_3)}\cdot\psi(m')\chi(m')(m')^2\cdot \frac{\chi(n') q^{c_1c_2n'm'}}{1 - q^{c_1c_2c_3m'}} \\
        &= \sum_{c = c_1c_2c_3}  \sum_{\chi \in \widehat{(\mathbb{Z}/c_3\mathbb{Z})^\times}} \frac{\chi(a)\fg(\bar{\chi})\psi(c_1)c_1^2}{\varphi(c_3)}\cdot \sum_{m' = 1}^\infty (\bm{1}_{c_2c_3}\chi \psi)(m')(m')^2\sum_{n' = 1}^{c_3} \frac{\chi(n') (q^{c_1c_2})^{n'm'}}{1 - (q^{c_1c_2})^{c_3m'}} \\
        &= \sum_{c = c_1c_2c_3}  \sum_{\chi \in \widehat{(\mathbb{Z}/c_3\mathbb{Z})^\times}} \frac{\chi(a)\fg(\bar{\chi})\psi(c_1)c_1^2}{\varphi(c_3)}\cdot \Eis_{3,+}^{\chi, \bm{1}_{c_2c_3}\chi \psi,c_1c_2}, 
    \end{align*}
    where $\Eis_{3,+}^{\chi, \bm{1}_{c_2c_3}\chi \psi,c_1c_2}$ denotes the non-constant part of $\Eis_{3,+}^{\chi, \bm{1}_{c_2c_3}\chi \psi,c_1c_2}$. Since both sides are non-constant parts of weight $3$ modular forms, they must have the same constant term. This completes the proof.
\end{proof}

Using the translation formula, we can write 
\[\rE_d(\tauq + r) = \sum \alpha_{\psi, n} \Eis_3^{\bm{1}_1, \psi, n}(\tauq + r) = \sum \alpha_{\chi',\psi', n'}' \Eis_3^{\chi', \psi', n'}(\tauq) \]
for some coefficients $\alpha_{\chi', \psi', n'}' \in \bC$. It remains to compute the contribution of each $\Eis_3^{\chi', \psi', n'}$ to the integral. 

\begin{prop}[\cite{ABYZ02}, extended by \cite{Yang04}] \label{prop:Yang04ext}
    Let $k\ge 2$ be an integer. For any triple $(\chi, \psi, n)$ with $\chi\psi(-1) = (-1)^k$, define
    \[\rI_k^{\chi, \psi, n} = n^{-1}\lim_{s \to 0} \rL(s+1, \chi)\rL(s+2-k, \psi). \]
    Let $a_{\ii\infty,0}$ and $a_{0,0}$ denote the cusp values of $\Eis_k^{\chi, \psi, n}(\tauq)$ at $\tauq = \ii\infty$ and $\tauq = 0$, respectively. 
    Then the integral
    \[\int_{\ii\infty}^0 (\Eis_k^{\chi, \psi, n}(\tauq) - a_{\ii\infty,0} - a_{0,0}\tauq^{-k}) \odif{(2\pi \ii\tauq)} = \rI_k^{\chi, \psi, n}. \] 
    In particular, when $\chi$ is even and nontrivial, i.e., $\chi(-1) = 1$ but $\chi \neq \bm{1}_N$ for all $N$, the integral is equal to $0$. 
\end{prop}

\begin{proof}
    This identity was originally established by \cite{ABYZ02} in the special case $\chi = \bm{1}_1$, i.e., $a_{0,0} = 0$. To handle the general case, we follow the method in \cite{Yang04}. 

    Let $x = - \log q = - 2\pi \ii \tauq$. Define $\Eis_{k,+}^{\chi, \psi, 1}(\tauq) = \Eis_k^{\chi, \psi, 1}(\tauq) - a_{\ii\infty, 0}$. Then, by the definition of Eisenstein series \eqref{eq:defEis} and the Mellin transform on exponential function, we have
    \begin{align*}
        \Eis_{k,+}^{\chi, \psi, 1}(\tauq) &= \sum_{m\ge 1}\psi(m)m^{k-1} \sum_{\nu \ge 1} \chi(\nu) e^{-m\nu x} \\
        &= \sum_{m\ge 1}\psi(m)m^{k-1} \sum_{\nu \ge 1} \chi(\nu) \cdot \frac{1}{2\pi \ii}\int_{\Lambda - \ii \infty}^{\Lambda + \ii \infty} \Gamma(s) (m \nu x)^{-s} \odif{s} \\
        &= \frac{1}{2\pi \ii}\int_{\Lambda - \ii \infty}^{\Lambda + \ii \infty} \sum_{m\ge 1} \frac{\psi(m) m^{k-1}}{m^s}\cdot \sum_{\nu \ge 1}\frac{\chi(\nu)}{\nu^s} \cdot \Gamma(s) x^{-s}\odif{s} \\
        &= \frac{1}{2\pi \ii} \int_{\Lambda - \ii \infty}^{\Lambda + \ii \infty} \rL(s+1-k, \psi) \rL(s,\chi) \Gamma(s) x^{-s}\odif{s}, 
    \end{align*}
    where $\Lambda > 0$ is a sufficiently large number. Thus, for $0 < \varepsilon < R < \infty$, 
    \begin{align}
        &n\int_{nR\ii/2\pi }^{n\varepsilon \ii/2\pi} \Eis_{k,+}^{\chi, \psi, n}(\tauq)\odif{(2\pi\ii \tauq)} \notag\\
        &\quad = \int_{R\ii/2\pi}^{\varepsilon \ii/2\pi} \Eis_{k,+}^{\chi, \psi, 1}(\tauq)\odif{(2\pi\ii \tauq)} \notag\\
        &\quad = \int_\varepsilon^R \Eis_{k,+}^{\chi, \psi, 1}\odif{x} \notag\\
        &\quad = \frac{1}{2\pi \ii} \int_{\Lambda - \ii \infty}^{\Lambda + \ii \infty} \rL(s+1-k, \psi) \rL(s,\chi) \Gamma(s) \frac{\varepsilon^{1-s} - R^{1-s}}{s - 1}\odif{s} \notag\\
        &\quad = \frac{1}{2\pi \ii} \int_{\Lambda + 1 - \ii \infty}^{\Lambda + 1 + \ii \infty} \rL(s+2-k, \psi) \rL(s+1,\chi) \Gamma(s) (\varepsilon^{-s} - R^{-s})\odif{s}. \label{eq:Yang04}
    \end{align}
    As $R \to \infty$, the $R^{-s}$ term vanishes, and we are left with
    \[\frac{1}{2\pi \ii} \int_{\Lambda + 1 - \ii \infty}^{\Lambda + 1 + \ii \infty} \rL(s+2-k, \psi) \rL(s+1,\chi) \Gamma(s) \varepsilon^{-s}\odif{s}. \]
    To evaluate this integral, we shift the contour of integration left ward to $\operatorname{Re} s = -\delta$ for a small $\delta > 0$. In the area $-\delta < \operatorname{Re} s < \Lambda + 1$, there are two possible poles: if $\psi = \bm{1}_N$, then at $s = k - 1$, 
    \begin{align*}
        \mathop{\operatorname{Res}}_{s = k-1} \rL(s+2-k, \psi) \rL(s+1,\chi) \Gamma(s) \varepsilon^{-s} &= a \varepsilon^{-(k-1)}
    \end{align*}
    for some constant $a$ only depends on $\chi$, $\psi$ and $k$; at $s = 0$, 
    \[\mathop{\operatorname{Res}}_{s = 0} \rL(s+2-k, \psi) \rL(s+1,\chi) \Gamma(s) \varepsilon^{-s} = \lim_{s \to 0}\rL(s+2-k,\psi)\rL(s+1,\chi) = \rI_k^{\chi, \psi, 1}. \]
    By the residue theorem, 
    \begin{align*}
        &\frac{1}{2\pi \ii} \int_{\Lambda + 1 - \ii \infty}^{\Lambda + 1 + \ii \infty} \rL(s+2-k, \psi) \rL(s+1,\chi) \Gamma(s) \varepsilon^{-s}\odif{s} \\
        &\quad = \frac{1}{2\pi \ii} \int_{-\delta - \ii \infty}^{-\delta + \ii \infty} \rL(s+2-k, \psi) \rL(s+1,\chi) \Gamma(s) \varepsilon^{-s}\odif{s} + a\varepsilon^{-(k-1)} + \rI_k^{\chi, \psi, 1}. 
    \end{align*}
    As $\varepsilon \to 0$, the integral over the shifted path tends to $0$. Hence, by~\eqref{eq:Yang04},    
    \[\lim_{\varepsilon \to 0}\int_{\ii\infty}^{n\varepsilon \ii/2\pi} \Eis_{k,+}^{\chi, \psi, n}(\tauq)\odif{(2\pi\ii \tauq)} - \frac{a}{n} \varepsilon^{-(k-1)} - \rI_k^{\chi, \psi, n} = 0. \]
    Now observer that 
    \[\int_{\ii\infty}^{n\varepsilon\ii/2\pi} \tauq^{-k}\odif{(2\pi \ii \tauq)} = c_0\varepsilon^{-(k-1)} \]
    for some constant $c_0$, and since the limit 
    \[\lim_{\varepsilon \to 0}\int_{\ii \infty}^{n\varepsilon \ii /2\pi} (\Eis_{k, +}^{\chi, \psi, n}(\tauq) - a_{0,0} \tau^{-k})\odif{(2\pi \ii\tauq)}\]
    exists and is finite, we must have $\frac{a}{n} = a_{0,0}c_0$. Therefore,  
    \[\int_{\ii \infty}^{n\varepsilon \ii /2\pi} (\Eis_{k}^{\chi, \psi, n}(\tauq) - a_{\ii\infty, 0} - a_{0,0} \tau^{-k})\odif{(2\pi \ii\tauq)} = \rI_k^{\chi, \psi, n}, \]
    as claimed. 

    For the last statement, simply note that, in this case, $\rL(s, \chi)$ is bounded near $s = 1$ and $\psi(-1) = (-1)^k$, so 
    \[\lim_{s \to 0} \rL(s + 1, \chi)\rL(s + 2 - k, \psi) = \rL(1, \chi)\rL(2 - k, \psi) = 0. \qedhere\]
\end{proof}

Let $a_{\ii\infty,0}(\Eis_3^{\chi', \psi', n'})$ and $a_{0,0}(\Eis_3^{\chi', \psi', n'})$ denote the cusp values of $\Eis_3^{\chi', \psi', n'}(\tauq)$ at $\ii\infty$ and $0$, respectively. Since $\rE_d(\tauq) = 1 + O(q)$ and $\rE_d(\tauq)$ vanishes at $r \in [\trpt{d}]$, we have
\[\sum \alpha_{\chi', \psi', n'}'a_{\ii\infty,0}(\Eis_3^{\chi', \psi', n'}) = 1,\quad \sum \alpha_{\chi', \psi', n'}'a_{0,0}(\Eis_3^{\chi', \psi', n'}) = 0. \]
Therefore, 
\begin{align*}
    &\int_{\ii\infty}^0 \left(\sum \alpha_{\chi', \psi', n'}'\Eis_3^{\chi', \psi', n'} - 1\right) \odif{(2\pi \ii \tauq)} \\
    &\quad = \int_{\ii\infty}^0 \sum \alpha_{\chi', \psi', n'}'\left(\Eis_3^{\chi', \psi', n'} - a_{\ii\infty,0}(\Eis_3^{\chi', \psi', n'}) - a_{0,0}(\Eis_3^{\chi', \psi', n'})\right) \odif{(2\pi \ii \tauq)} \\
    &\quad = \sum \alpha_{\chi', \psi', n'}' \rI_3^{\chi', \psi', n'}. 
\end{align*}
This means that the limit can be computed explicitly in terms of special values of Dirichlet $\rL$-functions. 

\begin{example}
    When $d = 6$, we have 
    \[\rE_6 = - \Eis_3^{\bm{1}_1,\chi_{3,2},1} - 8\Eis_3^{\bm{1}_1,\chi_{3,2},2}.\] 
    To simplify notation, we will write $\chi = \chi_{3,2}$ throughout this example. Take $r_\I = \frac13$ that represents the transition point $[\trpt{6\I}]$. Then, by Proposition~\ref{prop:TFforEis}, we have
    \begin{align*}
        \Eis_3^{\bm{1}_1,\chi,1}(\tauq + r_\I) &= \Eis_{r_\I}^{\chi}(\tauq) = \Eis_3^{\bm{1}_1, \chi,3} - \frac12 \Eis_3^{\bm{1}_3,\chi, 1} + \frac{\sqrt{3} \ii}{2} \Eis_3^{\chi,\bm{1}_3,1}. \\
        \Eis_3^{\bm{1}_1,\chi,2}(\tauq + r_\I) &= \Eis_{2r_\I}^{\chi}(2\tauq) = \Eis_3^{\bm{1}_1, \chi,6} - \frac12 \Eis_3^{\bm{1}_3,\chi, 2} - \frac{\sqrt{3} \ii}{2} \Eis_3^{\chi,\bm{1}_3,2}.
    \end{align*}
    So by Proposition~\ref{prop:Yang04ext}, 
    \begin{align*}
        \int_{\ii \infty}^{0} (\rE_6(\tauq + r_\I) - 1) \odif{(2\pi \ii \tauq)} &= - \Bigl(\rI_3^{\bm{1}_1, \chi,3} - \frac12 \rI_3^{\bm{1}_3,\chi, 1} + \frac{\sqrt{3} \ii}{2} \rI_3^{\chi,\bm{1}_3,1}\Bigr) \\
        &\quad - 8\Bigl(\rI_3^{\bm{1}_1, \chi,6} - \frac12 \rI_3^{\bm{1}_3,\chi, 2} - \frac{\sqrt{3} \ii}{2} \rI_3^{\chi,\bm{1}_3,2}\Bigr). 
    \end{align*}
    Since 
    \begin{align*}
        \rI_3^{\bm{1}_1, \chi,3} &= \frac13 \rI_3^{\bm{1}_1,\chi, 1}, \\
        \rI_3^{\bm{1}_3, \chi, 1} &= \lim_{s\to 0} \rL(s+1, \bm{1}_3) \rL(s-1, \chi) \\
        &= \lim_{s\to 0} (1 - 3^{-(s+1)})\rL(s+1, \bm{1}_1) \rL(s-1, \chi) = \frac23 \rI_3^{\bm{1}_1,\chi, 1}, 
    \end{align*}
    it follows that $\rI_3^{\bm{1}_1, \chi,3} = \frac12 \rI_3^{\bm{1}_3,\chi, 1}$. Similarly, $\rI_3^{\bm{1}_1, \chi,6} = \frac12 \rI_3^{\bm{1}_3,\chi, 2}$. Thus, it remains to compute 
    \[- \frac{\sqrt{3} \ii}{2} \rI_3^{\chi,\bm{1}_3,1} + 8 \frac{\sqrt{3} \ii}{2} \rI_3^{\chi,\bm{1}_3,2} = \frac{3\sqrt{3} \ii}{2} \rL(1, \chi) \rL(-1, \bm{1}_3), \]
    which evaluates to $\frac{3\sqrt{3} \ii}{2}\cdot \frac{\pi}{3\sqrt{3}}\cdot \frac16 = \frac{\pi \ii}{12}$. Hence, the resulting limit is 
    \[Q_{r_\I}^\ell = -e^{2\pi \ii\cdot r_\I} \cdot e^{\frac{\pi \ii}{12}} = e^{-\frac{1}{8}\cdot 2\pi \ii}, \]
    which is an $8$-th root of unity. 

    Now take $r_\II = \frac12$ which represents another transition point $[\trpt{6\II}]$. Then, by Proposition~\ref{prop:TFforEis}, 
    \begin{align*}
        \Eis_3^{\bm{1}_1,\chi,1}(\tauq + r_\II) &= \Eis_{r_\II}^{\chi}(\tauq) = -4\Eis_3^{\bm{1}_1, \chi,2} + \Eis_3^{\bm{1}_1,\bm{1}_2\chi, 2} - \Eis_3^{\bm{1}_2,\bm{1}_2\chi,1}. \\
        \Eis_3^{\bm{1}_1,\chi,2}(\tauq + r_\II) &= \Eis_{2r_\II}^{\chi}(2\tauq) = \Eis_3^{\bm{1}_1,\chi,2}.
    \end{align*}
    Thus, by Proposition~\ref{prop:Yang04ext}, 
    \begin{align*}
        \int_{\ii \infty}^{0} (\rE_d(\tauq + r_\II) - 1) \odif{(2\pi \ii \tauq)} &= -\bigl(-4 \rI_3^{\bm{1}_1, \chi, 2} + \rI_3^{\bm{1}_1, \bm{1}_2\chi, 2} - \rI_3^{\bm{1}_2, \bm{1}_2\chi, 1}\bigr) - 8\rI_3^{\bm{1}_1, \chi, 2} \\
        &= -4 \rI_3^{\bm{1}_1, \chi, 2} - \rI_3^{\bm{1}_1, \bm{1}_2\chi, 2} + \rI_3^{\bm{1}_2, \bm{1}_2\chi, 1}. 
    \end{align*}
    Similar to the case $6\I$, we have  
    \begin{align*}
        \rI_3^{\bm{1}_1, \chi,2} &= \frac12 \rI_3^{\bm{1}_1,\chi, 1}, \\
        \rI_3^{\bm{1}_1, \bm{1}_2\chi,2} &= \frac12 \lim_{s\to 0} \rL(s+1, \bm{1}_1) \rL(s-1, \bm{1}_2\chi) \\
        &= \frac12 \lim_{s\to 0} \bigl(1 - \chi(2)2^{-(s-1)}\bigr)\rL(s+1, \bm{1}_1) \rL(s-1, \chi) = \frac32 \rI_3^{\bm{1}_1,\chi, 1}, \\
        \rI_3^{\bm{1}_2, \bm{1}_2\chi,1} &= \lim_{s\to 0} \rL(s+1, \bm{1}_2) \rL(s-1, \bm{1}_2\chi) \\
        &= \lim_{s\to 0} \bigl(1 - 2^{-(s+1)}\bigr)\bigl(1 - \chi(2)2^{-(s-1)}\bigr) \rL(s+1, \bm{1}_1) \rL(s-1, \chi) = \frac32 \rI_3^{\bm{1}_1,\chi, 1}. 
    \end{align*}
    Finally, since
    \[\rI_3^{\bm{1}_1,\chi, 1} = \lim_{s\to 0} \rL(s+1, \bm{1}_1) \rL(s-1, \chi) = \rL'(-1, \chi), \]
    we conclude that the limit is 
    \[Q_\II^\ell = -e^{2\pi \ii\cdot r_\II} \cdot e^{-2 \rL'(-1, \chi)} = e^{-2 \rL'(-1, \chi)}.  \]
    Here, using the functional equation for Dirichlet $\rL$-series, we have 
    \begin{equation}\label{eq:L-1chi}
        \rL'(-1, \chi_{3,2}) = \frac{3\cdot \fg(\chi)}{2\cdot 2\pi \ii}\rL(2, \bar{\chi}) = \frac{3\sqrt{3}}{4\pi}\rL(2, \chi) \approx 0.3230659... 
    \end{equation}
\end{example}

\begin{example}
    When $d = 5$, we have 
    \[\rE_5 = (-1 + \frac {\ii}{2})\Eis_3^{\bm{1}_1,\chi_{5,2},1} + (-1 - \frac {\ii}{2})\Eis_3^{\bm{1}_1,\chi_{5,3},1}. \]
    To simplify notation, let us write $\chi = \chi_{5,2}$ throughout this example. Then, by Proposition~\ref{prop:TFforEis}, 
    \begin{align*}
        \Eis_3^{\bm{1}_1,\chi,1}(\tauq + r) &= \Eis_{r}^{\chi}(\tauq) = \Eis_3^{\bm{1}_1, \chi,5} - \frac14 \Eis_3^{\bm{1}_5, \chi, 1} + \frac {\ii }{4} \fg(\chi^3)\Eis_3^{\chi, \chi^2, 1} - \frac{1}4\fg(\chi^2) \Eis_3^{\chi^2, \chi^3, 1} - \frac {\ii }4 \fg(\chi)\Eis_3^{\chi^3, \bm{1}_5, 1}, \\
        \Eis_3^{\bm{1}_1,\chi^3,1}(\tauq + r) &= \Eis_{r}^{\chi^3}(\tauq) = \Eis_3^{\bm{1}_1, \chi^3,5} - \frac14 \Eis_3^{\bm{1}_5, \chi^3, 1} + \frac {\ii }{4} \fg(\chi^3)\Eis_3^{\chi, \bm{1}_5, 1} - \frac{1}4 \fg(\chi^2)\Eis_3^{\chi^2, \chi, 1} - \frac {\ii }4\fg(\chi) \Eis_3^{\chi^3, \chi^2, 1}.
    \end{align*}
    As in the case $d = 6$, $\rI_3^{\bm{1}_1, \chi,5} = \frac14 \rI_3^{\bm{1}_5, \chi, 1}$ and $\rI_3^{\bm{1}_1, \chi^3,5} = \frac14 \rI_3^{\bm{1}_5, \chi^3, 1}$. Since $\chi^2$ is even and nontrivial, we find that $\rI_3^{\chi^2, \chi, 1} = \rI_3^{\chi^2, \chi^3, 1} = 0$. We now compute the remaining constants
    \begin{align*}
        \fg(\chi^3)\rL(1, \chi) &= -\pi \ii B_{1,\chi^3} = \frac{\pi(1+3\ii)}{5}, \\
        \fg(\chi)\rL(1, \chi^3) &= -\pi \ii B_{1,\chi} = \frac{\pi(-1+3\ii)}{5}, \\
        \rL(-1, \bm{1}_5) &= -\frac{B_{2, \bm{1}_5}}{2} = \frac{1}{3}, \\
        \rL(-1, \chi^2) &= -\frac{B_{2, \chi^2}}{2} = -\frac{2}{5}. 
    \end{align*}
    Therefore, 
    \begin{align*}
        \int_{\ii \infty}^0 (\rE_5(\tauq + r) - 1) \odif{(2\pi \ii \tauq)} &= \frac{\ii}{4}\fg(\chi^3) \rL(1, \chi) \left(\left(-1+\tfrac{\ii}{2}\right)\rL(-1, \chi^2) + \left(-1-\tfrac{\ii}{2}\right) \rL(-1, \bm{1}_5)\right) \\
        &\quad - \frac{\ii}{4}\fg(\chi) \rL(1, \chi^3) \left(\left(-1+\tfrac{\ii}{2}\right)\rL(-1, \bm{1}_5) + \left(-1-\tfrac{\ii}{2}\right) \rL(-1, \chi^2)\right) \\
        &= \frac{7\pi \ii}{60}. 
    \end{align*}
    Thus, we obtain 
    \[Q_r^\ell = -e^{2\pi \ii r} \cdot e^{7\pi \ii /60} = e^{-\frac{1}{24}\cdot 2\pi \ii}, \]
    which is a $24$-th root of unity. 
\end{example}

\begin{theorem}\label{thr:limQabs}
    For $d = 4$, $5$, $6\I$, $6\II$, $8$ and $r \in [\trpt{d}]$, we have 
    \begin{enumerate}[(i)]
        \item\label{thr:limQabs_1} $Q_r^\ell$ is a root of unity if $d \neq 6\II$; 
        \item\label{thr:limQabs_2} if $d = 6 \II$ and $r = \frac{a}{c}$ with $\gcd(a, c) = 1$ and $c \equiv 2 \pmod{6}$, then 
        \[Q_r^\ell = \omega\cdot e^{-c \rL'(-1, \chi_{3,2})} \]
        for some root of unity $\omega$. In particular, $Q_r^\ell$ is not a root of unity (cf.~\eqref{eq:L-1chi}). 
    \end{enumerate}
\end{theorem}

\begin{proof}
    We write $\rE_d = \sum \alpha_{\psi, n} \Eis^{\bm{1}_1, \psi, n}$ as in Proposition~\ref{prop:EdinEis}, and write $r = \frac{a}{c}$ with $\gcd (a, c) = 1$. 
    
    For fixed $n$, write $c = \gcd (c, n)\cdot c_0$, $n = \gcd (c, n) \cdot n_0$. Then $nr = \frac{an_0}{c_0}$ with $\gcd (an_0, c_0) = 1$. By Proposition~\ref{prop:Yang04ext}, the contribution from $\Eis_3^{\bm{1}_1, \psi, n}(\tauq + r)$ is equal to 
    \begin{align*}
        &\sum_{c_0 = c_1 c_2 c_3} \frac{\bm{1}_{c_3}(n_0a)\fg(\bm{1}_{c_3}) \psi(c_1) c_1^2}{\varphi(c_3)}\cdot (c_1c_2n)^{-1} \lim_{s \to 0} \rL(s+1, \bm{1}_{c_3}) \rL(s-1, \bm{1}_{c_2c_3} \psi) \\
        &+ \sum_{c_0 = c_1 c_2 c_3} \sum_{\substack{\chi \in \widehat{(\bZ/c_3\bZ)^\times} \\ \chi(-1) = -1}} \frac{\chi(n_0a)\fg(\bar{\chi})\psi(c_1)c_1^2}{\varphi(c_3)}\cdot  (c_1c_2n)^{-1} \rL(1, \chi) \rL(-1, \bm{1}_{c_2c_3} \chi\psi). 
    \end{align*}
    We denote these two summations by $\rI_{r, \mathrm{triv}}^{\psi, n}$ and $\rI_{r, \mathrm{nontriv}}^{\psi, n}$, respectively. Then the limit
    \begin{align}
        Q_r^\ell &= -\exp\left(\sum \alpha_{\psi, n} (\rI_{r, \mathrm{triv}}^{\psi, n} +\rI_{r, \mathrm{nontriv}}^{\psi, n})\right) \notag\\
        &= -\exp\left(\sum \alpha_{\psi, n} \rI_{r, \mathrm{triv}}^{\psi, n}\right)\exp\left(\sum \alpha_{\psi, n}\rI_{r, \mathrm{nontriv}}^{\psi, n}\right). \label{eqlogabsQ}
    \end{align}
    For the trivial part, 
    \begin{align*}
        \rI_{r, \mathrm{triv}}^{\psi, n} &= \sum_{c_0 = c_1c_2c_3} \frac{\mu(c_3) \psi(c_1) c_1}{\varphi(c_3) c_2 n} \cdot \prod_{p \mid c_3} (1 - p^{-1}) \cdot \rL'(-1, \bm{1}_{c_2c_3}\psi) \\
        &= \sum_{c_0 = c_1c_2c_3}  \frac{\mu(c_3) \psi(c_1) c_1}{c_2c_3 n}  \cdot \rL'(-1, \bm{1}_{c_2c_3}\psi) \\
        &= \sum_{c_1 \mid c_0} \frac{\psi(c_1)c_1^2}{c_0n} \rL'(-1, \bm{1}_{c_0/c_1} \psi) \sum_{c_2 c_3 = c_0/c_1} \moe(c_3) = \frac{\psi(c_0)c_0}{n} \rL'(-1, \psi). 
    \end{align*}
    Here $\moe$ is the M\"obius function and we use the fact that $\sum_{m\mid n} \moe(m) = \delta_{1, n}$. 
    
    Thus, we find 
    \[\sum \alpha_{\psi, n} \rI_{r, \mathrm{triv}}^{\psi, n} = \sum \alpha_{\psi, n} \frac{\psi(\frac{c}{\gcd (c, n)})\frac{c}{\gcd (c, n)}}{n} \rL'(-1, \psi). \]
    In particular, one can verify case by case that this sum is $0$ when $d = 4$, $5$, $6\I$, $8$, and is equal to 
    \[ - \frac{\chi_{3,2}(2) c}{1} \rL'(-1, \chi_{3,2}) - 8 \frac{\chi_{3,2}(1) \frac{c}{2}}{2} \rL'(-1, \chi_{3,2}) = -c \rL'(-1, \chi_{3,2}) \]
    when $d = 6\II$. 

    By \eqref{eqlogabsQ}, it remains to show that $\sum \alpha_{\psi, n}\rI_{r, \mathrm{nontriv}}^{\psi, n}$ lies in $2\pi\ii \bQ$. We show that $\frac{1}{2\pi \ii} \sum \alpha_{\psi, n}\rI_{r, \mathrm{nontriv}}^{\psi, n}$ is $\sigma$-invariant for any Galois action $\sigma\colon \bar{\bQ} \to \bar{\bQ}$. 
    
    Suppose $\chi \in \widehat{(\bZ/c_3\bZ)^\times}$ with $\chi(-1) = -1$, and let its conductor be $N \mid c_3$, so that $\chi = \bm{1}_{c_3} \chi_N$ for some primitive character $\chi_N$ modulo $N$. Then the term 
    \[\rI_{c_1, c_2, c_3, n_0a}^{\chi, \psi, n} \coloneqq \frac{1}{2\pi \ii}\frac{\chi(n_0a)\fg(\bar{\chi})\psi(c_1)c_1^2}{\varphi(c_3)}\cdot  (c_1c_2n)^{-1} \rL(1, \chi) \rL(-1, \bm{1}_{c_2c_3} \chi\psi) \]
    can be rewritten using properties of Gauss sums and special $\rL$-values as 
    \begin{align*}
        &\frac{1}{2\pi \ii}\frac{\chi(n_0a) (\fg(\bar{\chi}_N)\mu(\frac{c_3}{N})\chi_N(\frac{c_3}{N}))\psi(c_1) c_1}{\varphi(c_3)c_2 n} \cdot \prod_{p \mid c_3}(1 - p^{-1}\chi_N(p))\rL(1, \chi_N) \cdot \rL(-1, \bm{1}_{c_2c_3} \chi \psi) \\
        &= \frac{1}{2\pi \ii}\frac{\chi(n_0a) \fg(\bar{\chi}_N) \mu(\frac{c_3}{N})\chi_N(\frac{c_3}{N})\psi(c_1) c_1}{\varphi(c_3)c_2 n} \prod_{p \mid c_3}(1 - p^{-1}\chi_N(p)) \left(\frac{\pi \ii \fg(\chi_N)B_{1, \bar{\chi}_N}}{N}\right)\left(-\frac{B_{2, \bm{1}_{c_2c_3}\chi\psi}}{2}\right) \\
        &= -\frac{\chi(n_0a)  \mu(\frac{c_3}{N})\chi_N(\frac{c_3}{N})\psi(c_1) c_1}{4\varphi(c_3)c_2 n} \prod_{p \mid c_3}(1 - p^{-1}\chi_N(p)) B_{1, \bar{\chi}_N}B_{2, \bm{1}_{c_2c_3}\chi\psi} 
    \end{align*}
    where, for a Dirichlet character $\chi' \in \widehat{(\bZ/N'\bZ)^\times}$, $B_{m, \chi'}$ denotes the $m$-th generalized Bernoulli number associated to $\chi'$, defined by the series expansion
    \[\sum_{\nu = 1}^{N'} \chi'(\nu) \frac{x e^{\nu x}}{e^{N'x} - 1} = \sum_{m = 0}^\infty \frac{B_{m, \chi'}}{m!} x^m. \]

    Because $\chi^\sigma = \sigma \circ \chi = \bm{1}_{c_3}\chi_N^\sigma$, $\sigma(B_{m, \chi'}) = B_{m, (\chi')^\sigma}$, and $\bar{\chi}_N^\sigma = \overline{\chi_N^\sigma}$, it follows that 
    \[\sigma(\alpha_{\psi, n}\rI_{c_1, c_2, c_3, n_0a}^{\chi, \psi, n}) = \sigma(\alpha_{\psi, n})\rI_{c_1, c_2, c_3, n_0a}^{\chi^\sigma, \psi^\sigma, n} = \alpha_{\psi^\sigma, n} \rI_{c_1, c_2, c_3, n_0a}^{\chi^\sigma, \psi^\sigma, n}. \]
    Hence, the sum $\frac{1}{2\pi \ii} \sum \alpha_{\psi, n}\rI_{r, \mathrm{nontriv}}^{\psi, n}$ is $\sigma$-invariant, as desired. 
\end{proof}

\begin{rmk} \label{r:complete}
    Using the above method, along with induction on the denominator of $r$ and change of coordinates $\tauq\mapsto A\tauq$ (as suggested by Y.~Yang), the limit of $Q^\ell$ is computed in \cite{LeePhD} at any $r = \frac{a}{c}\notin [\ii \infty]$. Below are the explicit result in the cases $d = 3$, $4$, $5$, $6$. 
    \begin{enumerate}[(i)]
        \item For $d = 3$, 
        \[Q_r^\ell = (-1)^ce^{-9c\rL'(-1, \chi_{3,2})},\quad r \in [0],\ c \equiv 1 \pmod{3}. \]
        \item For $d = 4$, 
        \[Q_r^\ell = \begin{cases}
            -e^{-4c\rL'(-1, \chi_{4,2})} & \text{ if } r \in [0],\ c \equiv 1 \pmod{4}, \\
            1 & \text{ if } r \in [\frac12]. 
        \end{cases} \]
        \item For $d = 5$, 
        \[Q_r^\ell = \begin{cases}
            (-1)^c e^{c\operatorname{Re}((-2+\ii)\rL'(-1, \chi_{5,2}))} & \text{ if } r \in [0],\ c \equiv 1 \pmod{5}, \\
            e^{c\operatorname{Re}((-1-2\ii)\rL'(-1, \chi_{5,2}))} & \text{ if } r \in [\frac12],\ c \equiv 2 \pmod{5}, \\
            \omega_{120}^{-c} & \text{ if } r \in [\trpt{5}] = [\frac25],\ c\equiv 0\pmod{5},\ a \equiv 2 \pmod{5}.
        \end{cases}\]
        In particular, for some $r \in [\trpt{5}]$, for example, $r = \frac{7}{120}$, one has $Q_r^\ell = 1$. In other words, there exists a path on $\mathbb{P}^1$ connecting $P^\ell = 0$ and $P^\ell = \infty$ along which $Q^\ell$ tends to $1$. 
        \item For $d = 6$, 
        \[Q_r^\ell = \begin{cases}
            -e^{-5c\rL'(-1, \chi_{3,2})} & \text{ if } r = \frac ac \in [0], \ c \equiv 1 \pmod{6}, \\
            \omega_{24}^{-c} & \text{ if } r = \frac ac \in [\trpt{6\I}] = [\tfrac13],\ c\equiv 3 \pmod{6},\ a\equiv 1 \pmod{3}, \\
            e^{-c\rL'(-1, \chi_{3,2})} & \text{ if } r = \frac ac \in [\trpt{6\II}] = [\tfrac12],\ c \equiv 2 \pmod{6}. 
        \end{cases}\]
        This shows that, in Theorem~\ref{thr:limQabs} \eqref{thr:limQabs_2}, the root of unity $\omega$ is actually superfluous. 
    \end{enumerate}
\end{rmk}


Here, we plot the images of several straight line segments $[0, \omega_n)$ under the maps $P_d$ and $Q^\ell$, for $d = 6\I$ and $6\II$ (corresponding to Figure~\ref{fig:limit6I} and Figure~\ref{fig:limit6II}, respectively), to illustrate how paths lift under the projection $\mathbb{D}^\times \to \bP_{P^\ell}^1 \setminus \operatorname{Sing} f_d$, and how $Q^\ell$ tends to $Q_r^\ell$, $r = 1/n$. Note that the function $Q^\ell(q)$ is the same in both figures. 

Comparing the image curves $Q^\ell([0, \omega_N])$ with $N = 8, 10$ in Figure~\ref{fig:limit6II} and $N = 9$ in Figure~\ref{fig:limit6I} reveals the complexity of the K\"ahler moduli associated to the function $Q^\ell(q)$. We plan to study the geometry of the K\"ahler moduli in details in a subsequent work.

\newpage

\begin{figure}[htbp!]
    \centering
    \includegraphics{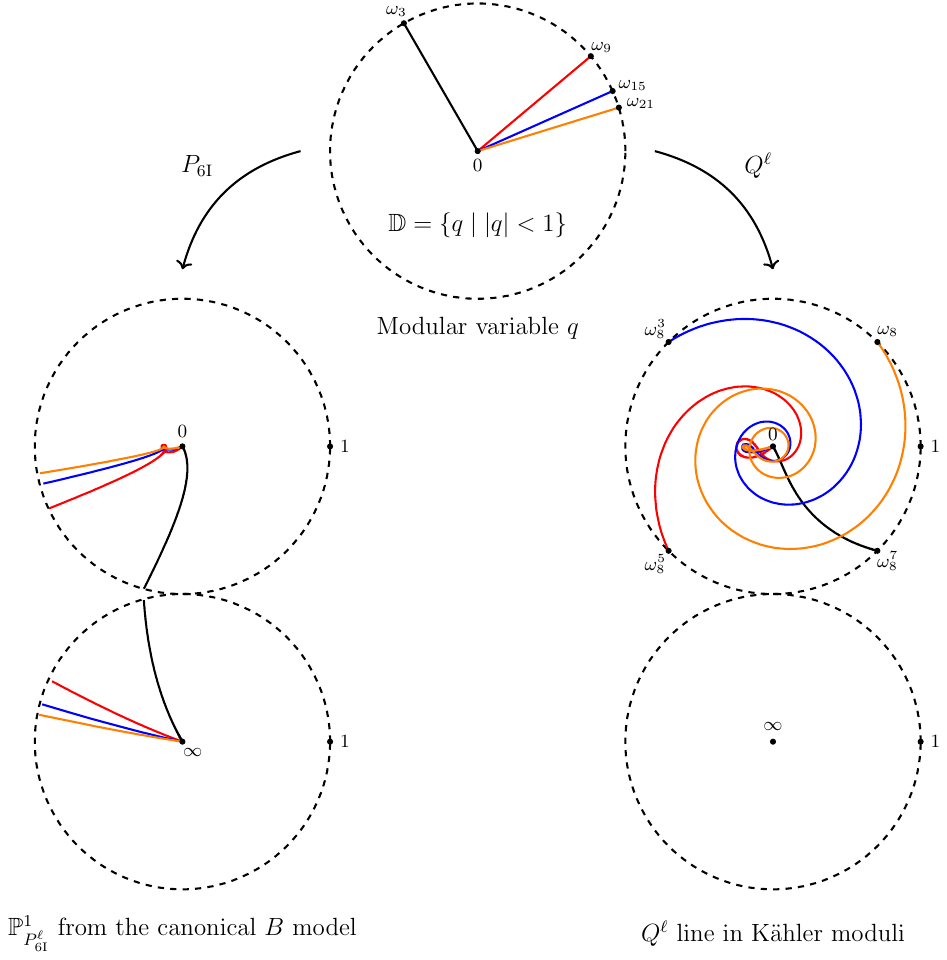}
    \caption{$d = 6\I$}
    \label{fig:limit6I}
\end{figure}

\begin{figure}[htbp!]
    \centering
    \includegraphics{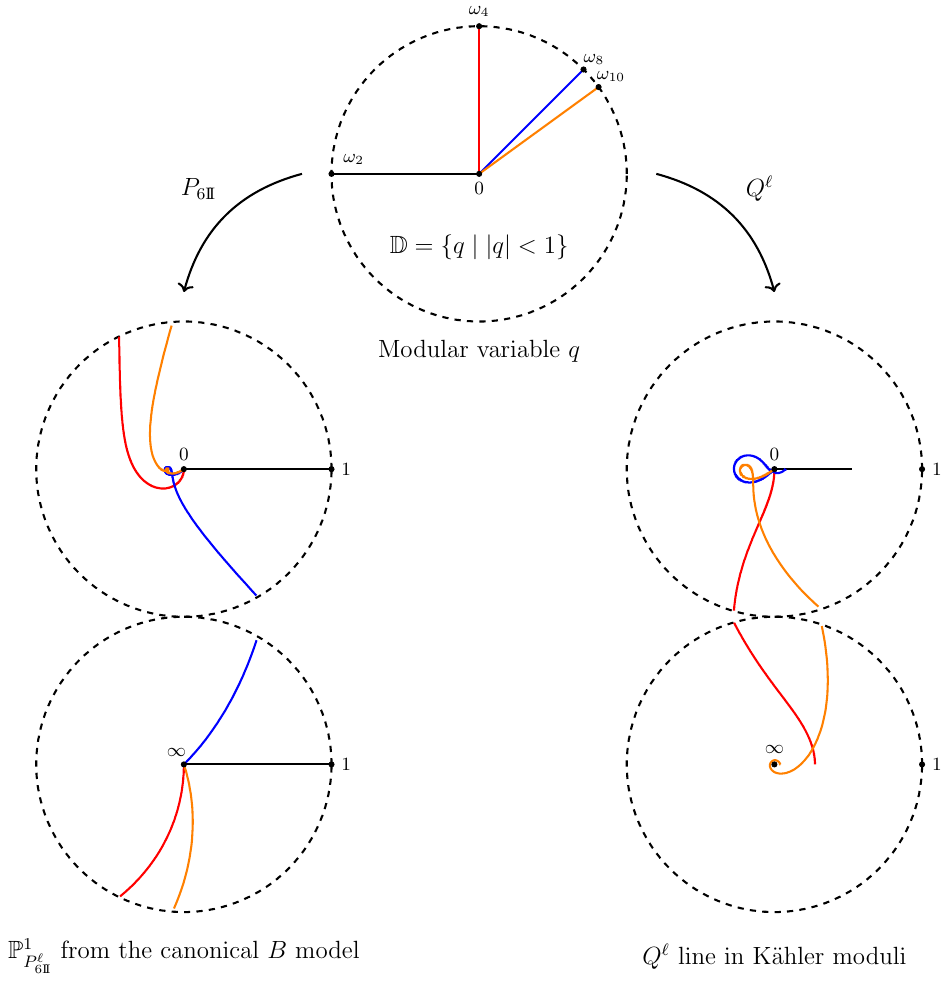}
    \caption{$d = 6\II$}
    \label{fig:limit6II}
\end{figure}

\newpage

\appendix
\section{On del Pezzo surfaces and threefolds} \label{sec;appdP}

We briefly review the classification of del Pezzo manifolds of dimensions two and three. The standard reference is \cite{bookAGV}.

In classical geometry, a (smooth) \emph{del Pezzo surface} is a smooth projective surface $S$ such that the anti-canonical divisor $- K_S$ is ample. The self-intersection number $d = K_S^2$ is called the \emph{degree} of the del Pezzo surface $S$. It satisfies $1 \leq d \leq9$. The description of $S$ is well-known as below, see \cite[\S 24]{Manin86} for example. 

\begin{theorem}\label{thm;dP2}
Let $S$ be a smooth del Pezzo surface of degree $d$. Set the anti-canonical algebra $R = \bigoplus_{m \geq 0} R_m$ with grading $R_m = H^0 (S, \cO (-m K))$. Then $S = \Proj R$ is isomorphic to either a blow-up of $\bP^2$ at $9 - d$ general points or $\bP^1 \times \bP^1$ with $d = 8$, and has Betti numbers
\begin{align*}
    b_2(S) = h^{1,1} (S) = 10 - d,
\end{align*}
$b_0(S) = b_4(S) = 1$ and $b_1(S) = b_3(S) = 0$. Furthermore, we have:
\begin{enumerate}[(i)]
    \item\label{thm;dP2-1} If $d = 1$, then $R$ is generated by $R_1$, $R_2$ and $R_3$ and $S$ is a sextic $(6) \subseteq \bP (3,2,1,1)$. To be precise, $R \cong \bC [x_1, x_2, x_3, x_4] / (F)$ with $x_3, x_4 \in R_1$, $x_2 \in R_2$ and $x_1 \in R_3$ and a weighted homogeneous polynomial $F = F (x_1, x_2, x_3, x_4)$ of weighted degree six.

    \item If $d = 2$, then $R$ is generated by $R_1$ and $R_2$ and $S$ is a quartic $(4) \subseteq \bP (2,1,1,1)$, that is, $R \cong \bC [x_1, x_2, x_3, x_4] / (F)$ with $x_2, x_3, x_4 \in R_1$ and $x_1 \in R_2$ and $F$ of weighted degree four.

    \item If $d = 3$, then $R \cong \bC [x_1, x_2, x_3, x_4] / (F)$ with $x_i \in R_1$ $(1 \leq i \leq 4)$ and a homogeneous polynomial $F$ of degree three. That is, $S$ is a cubic surface in $\bP^3$.
        
    \item\label{thm;dP2-4} If $d \geq 4$, then $R \cong \bC [x_1, \ldots, x_{d + 1}] / \cI$ with $x_i \in R_1$ $(1 \leq i \leq d + 1)$ and $\cI$ generated by $d (d - 3) / 2$ quadrics. In particular, if $d = 4$ then $S$ is a complete intersection of two quadrics in $\bP^4$.
\end{enumerate}
\end{theorem}

\begin{remark}\label{rmk;Sd}
If a del Pezzo surface $S$ of degree $d$ is a blow-up of $\bP^2$ at more than one point, then it can also be described as a blow-up of $\bP^1 \times \bP^1$ at $8 - d$ points.   
\end{remark}

\begin{remark}\label{rmk;GdP2}
The \eqref{thm;dP2-1}-\eqref{thm;dP2-4} in Theorem \ref{thm;dP2} still holds when $S$ is normal and Gorenstein as shown by Hidaka and Watanabe \cite{HW81}. This holds even for a (not necessarily normal) generalized del Pezzo surface as derived from the works of Reid \cite{Reid94}. 
\end{remark}

Now, we recall the definition of del Pezzo manifolds, which is a natural higher dimensional polarized version of the notion of del Pezzo surfaces. 

A smooth polarized variety $(X, \cO_X (1))$ of dimension $n \geq 3$ is \emph{del Pezzo} if $\cO_X (- K) = \cO_X (n - 1)$. An ample divisor $H$ is called a \emph{fundamental divisor} of $X$ if $\cO_X (H) \cong \cO_X(1)$. The corresponding linear system $|H|$ is called the \emph{fundamental system} of $X$. The self-intersection number $d = H^n$ is also called the degree of the del Pezzo manifold $X$. Note that $H^i (X, \cO_X) = $ for all $i > 0$ by Kodaira vanishing Theorem.

The classification of del Pezzo manifolds was obtained by Fujita and Iskovskikh. We are mainly interested in smooth del Pezzo threefolds in this paper.

For the following proposition, see \cite[Propositions 3.2.3, 3.2.4]{bookAGV} and the references therein.

\begin{proposition} \label{p:Xd}
Let $(X, \cO_X (1))$ be a smooth del Pezzo threefold of degree $d$. Then:
\begin{enumerate}
    \item If $S$ is the zero locus defined by a general global section of $\cO_X (1)$, then it is a smooth del Pezzo surface and $\cO_S (- K) \cong \cO_X(1)|_S$.

    \item If $d \geq 3$, then the fundamental system of $X$ determines an embedding $X \hookrightarrow \bP^{d + 1}$ which is projectively normal and $\cO_X (1) \cong \cO_{\bP^{d + 1}} (1)|_X$.

    \item\label{deld_prop3} If $d \geq 4$, then the image $X \hookrightarrow \bP^{d + 1}$ is an intersection of quadrics.
\end{enumerate}
\end{proposition}

We have the following classification of smooth del Pezzo threefolds, see \cite[Theorem 3.3.1]{bookAGV} and the references therein.

\begin{theorem}\label{thm;dP3claf}
Let $(X, \cO_X (1))$ be a smooth del Pezzo threefold of degree $d$. Then $X$ is one of the following:
\begin{enumerate}
    \item $d = 1$ and $X$ is a sextic hypersurface in $\bP (3, 2, 1, 1, 1)$.

    \item $d = 2$ and $X$ is a quartic hypersurface in $\bP (2, 1, 1, 1, 1)$.

    \item $d = 3$ and $X \hookrightarrow \bP^{4}$ is a cubic hypersurface.

    \item $d = 4$ and $X \hookrightarrow \bP^{5}$ is a complete intersection of two quadrics.

    \item $d = 5$ and $X \hookrightarrow \bP^{6}$ is a linear section of Pl\"ucker embedding of $Gr(2,5) \subseteq \bP^9$ by codimension $3$ subspace.
    \myitem{6I}\label{thm;dP3claf_6I} $d = 6$ and $X \hookrightarrow \bP^{7}$ is a divisor on $\bP^2 \times \bP^2$ of bidegree $(1,1)$, and $\bP^2 \times \bP^2 \hookrightarrow \bP^8$ by Segre embedding.
    
    \myitem{6I\!I}\label{thm;dP3claf_6II} $d = 6$ and $X = \bP^1 \times \bP^1 \times \bP^1 \hookrightarrow \bP^7$ by Segre embedding. 

    \setcounter{enumi}{6}
    
    \item $d = 7$ and $X \hookrightarrow \bP^{8}$ is a blow-up of $\bP^3$ at a point.

    \item $d = 8$ and $X = \bP^{3} \hookrightarrow \bP^{9}$ with $\cO_X (1) = \cO_{\bP^3} (2)$.
\end{enumerate}
The Hodge numbers of del Pezzo threefolds are given in Table \ref{tab;HGXd}.
\end{theorem}

\begin{table}[H]
    \centering
    \begin{tabular}{cccccccccc}
        \toprule
        $d$ & $1$ & $2$ & $3$ & $4$ & $5$ & $6 \I$ & $6 \II$ & $7$ & $8$ \\
        \midrule
        $h^{1,1} (X)$ & $1$ & $1$ & $1$ & $1$ & $1$ & $2$ & $3$ & $2$ & $1$ \\
        $h^{2,1} (X)$ & $21$ & $10$ & $5$ & $2$ & $0$ & $0$ & $0$
        & $0$ & $0$ \\
        \bottomrule
    \end{tabular}
    \caption{Hodge numbers of del Pezzo threefolds.}
  \label{tab;HGXd}
\end{table}

\begin{remark}\label{rmk;OX7_1}
Let $h_1$ denote the pullback, via the blow-up $X_7 \to \bP^3$, of the hyperplane class on $\bP^3$. Let $e$ be the exceptional divisor in $X_7$. Write $h_2 = h_1 - e$. Then $\cO_{X_7} (1) = \cO (h_1 + h_2)$.
\end{remark}

\bibliography{main}

\begin{thebibliography}{10}

\bibitem{ABYZ02}
S.~Ahlgren, B.~C. Berndt, A.~J. Yee, and A.~Zaharescu.
\newblock Integrals of {E}isenstein series and derivatives of {$L$}-functions.
\newblock {\em Int. Math. Res. Not.}, 2002(32):1723--1738, 2002.

\bibitem{ASYZ14}
M.~Alim, E.~Scheidegger, S.-T. Yau, and J.~Zhou.
\newblock Special polynomial rings, quasi modular forms and duality of topological strings.
\newblock {\em Adv. Theor. Math. Phys.}, 18(2):401--467, 2014.

\bibitem{BCKvS98}
V.~Batyrev, I.~Ciocan-Fontanine, B.~Kim, and D.~van Straten.
\newblock Conifold transitions and mirror symmetry for {C}alabi-{Y}au complete intersections in {G}rassmannians.
\newblock {\em Nuclear Phys. B}, 514(3):640--666, 1998.

\bibitem{BF97}
K.~Behrend and B.~Fantechi.
\newblock The intrinsic normal cone.
\newblock {\em Invent. Math.}, 128(1):45--88, 1997.

\bibitem{Bro14}
J.~Brown.
\newblock Gromov-{W}itten invariants of toric fibrations.
\newblock {\em Int. Math. Res. Not. IMRN}, 2014(19):5437--5482, 2013.

\bibitem{CG07}
T.~Coates and A.~Givental.
\newblock Quantum {R}iemann-{R}och, {L}efschetz and {S}erre.
\newblock {\em Ann. of Math. (2)}, 165(1):15--53, 2007.

\bibitem{CCLT09}
T.~Coates, Y.-P. Lee, A.~Corti, and H.-H. Tseng.
\newblock The quantum orbifold cohomology of weighted projective spaces.
\newblock {\em Acta Math.}, 202(2):139--193, 2009.

\bibitem{Coo17}
S.~Cooper.
\newblock {\em Ramanujan's theta functions}.
\newblock Springer, Cham, 2017.

\bibitem{DiaShu}
F.~Diamond and J.~Shurman.
\newblock {\em A first course in modular forms}, volume 228 of {\em Graduate Texts in Mathematics}.
\newblock Springer-Verlag, New York, 2005.

\bibitem{Dijkgraaf95}
R.~Dijkgraaf.
\newblock Mirror symmetry and elliptic curves.
\newblock In {\em The moduli space of curves ({T}exel {I}sland, 1994)}, volume 129 of {\em Progr. Math.}, pages 149--163. Birkh\"auser Boston, Boston, MA, 1995.

\bibitem{FRZZ19}
B.~Fang, Y.~Ruan, Y.~Zhang, and J.~Zhou.
\newblock Open {G}romov-{W}itten theory of {$K_{\mathbb{P}^2}$}, {$K_{\mathbb{P}^1\times\mathbb{P}^1}$}, {$K_{W\mathbb{P}[1,1,2]}$}, {$K_{\mathbb{F}_1}$} and {J}acobi forms.
\newblock {\em Comm. Math. Phys.}, 369(2):675--719, 2019.

\bibitem{GS83}
G.-M. Greuel and J.~Steenbrink.
\newblock On the topology of smoothable singularities.
\newblock In {\em Singularities, {P}art 1 ({A}rcata, {C}alif., 1981)}, volume~40 of {\em Proc. Sympos. Pure Math.}, pages 535--545. Amer. Math. Soc., Providence, RI, 1983.

\bibitem{Gross97a}
M.~Gross.
\newblock Deforming {C}alabi-{Y}au threefolds.
\newblock {\em Math. Ann.}, 308(2):187--220, 1997.

\bibitem{Gross97b}
M.~Gross.
\newblock Primitive {C}alabi-{Y}au threefolds.
\newblock {\em J. Differential Geom.}, 45(2):288--318, 1997.

\bibitem{HW81}
F.~Hidaka and K.~Watanabe.
\newblock Normal {G}orenstein surfaces with ample anti-canonical divisor.
\newblock {\em Tokyo J. Math.}, 4(2):319--330, 1981.

\bibitem{IP03}
E.~Ionel and T.~Parker.
\newblock Relative {G}romov-{W}itten invariants.
\newblock {\em Ann. of Math. (2)}, 157(1):45--96, 2003.

\bibitem{IP04}
E.~Ionel and T.~Parker.
\newblock The symplectic sum formula for {G}romov-{W}itten invariants.
\newblock {\em Ann. of Math. (2)}, 159(3):935--1025, 2004.

\bibitem{bookAGV}
V.~A. Iskovskikh and Yu.~G. Prokhorov.
\newblock {\em Algebraic geometry. {V}}, volume~47 of {\em Encyclopaedia of Mathematical Sciences}.
\newblock Springer-Verlag, Berlin, 1999.
\newblock Fano varieties, A translation of {{\i}t Algebraic geometry. 5} (Russian), Ross. Akad. Nauk, Vseross. Inst. Nauchn. i Tekhn. Inform., Moscow, Translation edited by A. N. Parshin and I. R. Shafarevich.

\bibitem{KZ95}
M.~Kaneko and D.~Zagier.
\newblock A generalized {J}acobi theta function and quasimodular forms.
\newblock In {\em The moduli space of curves ({T}exel {I}sland, 1994)}, volume 129 of {\em Progr. Math.}, pages 165--172. Birkh\"auser Boston, Boston, MA, 1995.

\bibitem{Kapustka209I}
G.~Kapustka and M.~Kapustka.
\newblock Primitive contractions of {C}alabi-{Y}au threefolds. {I}.
\newblock {\em Comm. Algebra}, 37(2):482--502, 2009.

\bibitem{KM92}
J.~Koll\'{a}r and S.~Mori.
\newblock Classification of three-dimensional flips.
\newblock {\em J. Amer. Math. Soc.}, 5(3):533--703, 1992.

\bibitem{KM98}
J.~Koll\'ar and S.~Mori.
\newblock {\em Birational geometry of algebraic varieties}, volume 134 of {\em Cambridge Tracts in Mathematics}.
\newblock Cambridge University Press, Cambridge, 1998.
\newblock With the collaboration of C. H. Clemens and A. Corti, Translated from the 1998 Japanese original.

\bibitem{LeePhD}
S.-Y. Lee.
\newblock Ph.{D}. thesis.
\newblock In preparation, National Taiwan University, 2026.

\bibitem{LLW10}
Y.-P. Lee, H.-W. Lin, and C.-L. Wang.
\newblock Flops, motives, and invariance of quantum rings.
\newblock {\em Ann. of Math. (2)}, 172(1):243--290, 2010.

\bibitem{LLW18}
Y.-P. Lee, H.-W. Lin, and C.-L. Wang.
\newblock Towards {$A+B$} theory in conifold transitions for {C}alabi-{Y}au threefolds.
\newblock {\em J. Differential Geom.}, 110(3):495--541, 2018.

\bibitem{LP04}
Y.-P. Lee and R.~Pandharipande.
\newblock A reconstruction theorem in quantum cohomology and quantum {$K$}-theory.
\newblock {\em Amer. J. Math.}, 126(6):1367--1379, 2004.

\bibitem{LR01}
A.-M. Li and Y.~Ruan.
\newblock Symplectic surgery and {G}romov-{W}itten invariants of {C}alabi-{Y}au 3-folds.
\newblock {\em Invent. Math.}, 145(1):151--218, 2001.

\bibitem{Li01}
J.~Li.
\newblock Stable morphisms to singular schemes and relative stable morphisms.
\newblock {\em J. Differential Geom.}, 57(3):509--578, 2001.

\bibitem{Li02}
J.~Li.
\newblock A degeneration formula of {GW}-invariants.
\newblock {\em J. Differential Geom.}, 60(2):199--293, 2002.

\bibitem{LT98}
J.~Li and G.~Tian.
\newblock Virtual moduli cycles and {G}romov-{W}itten invariants of algebraic varieties.
\newblock {\em J. Amer. Math. Soc.}, 11(1):119--174, 1998.

\bibitem{LT99}
J.~Li and G.~Tian.
\newblock Comparison of algebraic and symplectic {G}romov-{W}itten invariants.
\newblock {\em Asian J. Math.}, 3(3):689--728, 1999.

\bibitem{LLY99}
B.~H. Lian, K.~Liu, and S.-T. Yau.
\newblock Mirror principle. {III}.
\newblock {\em Asian J. Math.}, 3(4):771--800, 1999.

\bibitem{LY04}
C.-H. Liu and S.-T. Yau.
\newblock A degeneration formula of {G}romov-{W}itten invariants with respect to a curve class for degenerations from blow-ups.
\newblock {\em arXiv:0408147}, 2004.

\bibitem{Looijenga84}
E.~J.~N. Looijenga.
\newblock {\em Isolated singular points on complete intersections}, volume~77 of {\em London Mathematical Society Lecture Note Series}.
\newblock Cambridge University Press, Cambridge, 1984.

\bibitem{Maier09}
R.~S. Maier.
\newblock On rationally parametrized modular equations.
\newblock {\em J. Ramanujan Math. Soc.}, 24(1):1--73, 2009.

\bibitem{Manin86}
Y.~I. Manin.
\newblock {\em Cubic forms}, volume~4 of {\em North-Holland Mathematical Library}.
\newblock North-Holland Publishing Co., Amsterdam, second edition, 1986.
\newblock Algebra, geometry, arithmetic, Translated from the Russian by M. Hazewinkel.

\bibitem{MP06}
D.~Maulik and R.~Pandharipande.
\newblock A topological view of {G}romov-{W}itten theory.
\newblock {\em Topology}, 45(5):887--918, 2006.

\bibitem{MS17}
D.~McDuff and D.~Salamon.
\newblock {\em Introduction to symplectic topology}.
\newblock Oxford Graduate Texts in Mathematics. Oxford University Press, Oxford, third edition, 2017.

\bibitem{MS23}
R.~Mi and M.~Shoemaker.
\newblock Extremal transitions via quantum {S}erre duality.
\newblock {\em Math. Ann.}, 386(1-2):821--876, 2023.

\bibitem{MS97}
D.~R. Morrison and N.~Seiberg.
\newblock Extremal transitions and five-dimensional supersymmetric field theories.
\newblock {\em Nuclear Phys. B}, 483(1-2):229--247, 1997.

\bibitem{Namikawa94}
Y.~Namikawa.
\newblock On deformations of {C}alabi-{Y}au {$3$}-folds with terminal singularities.
\newblock {\em Topology}, 33(3):429--446, 1994.

\bibitem{Namikawa97}
Y.~Namikawa.
\newblock Deformation theory of {C}alabi-{Y}au threefolds and certain invariants of singularities.
\newblock {\em J. Algebraic Geom.}, 6(4):753--776, 1997.

\bibitem{NS95}
Y.~Namikawa and J.~H.~M. Steenbrink.
\newblock Global smoothing of {C}alabi-{Y}au threefolds.
\newblock {\em Invent. Math.}, 122(2):403--419, 1995.

\bibitem{PP24}
S.~G. Park and M.~Popa.
\newblock Lefschetz theorems, {Q}-factoriality, and {H}odge symmetry for singular varieties.
\newblock {\em arXiv:2410.15638v2}, 2024.

\bibitem{Prz07}
V.~V. Przyjalkowski.
\newblock Quantum cohomology of smooth complete intersections in weighted projective spaces and in singular toric varieties.
\newblock {\em Mat. Sb.}, 198(9):107--122, 2007.

\bibitem{Reid80}
M.~Reid.
\newblock Canonical {$3$}-folds.
\newblock In {\em Journ\'{e}es de {G}\'{e}ometrie {A}lg\'{e}brique d'{A}ngers, {J}uillet 1979/{A}lgebraic {G}eometry, {A}ngers, 1979}, pages 273--310. Sijthoff \& Noordhoff, Alphen aan den Rijn---Germantown, Md., 1980.

\bibitem{Reid94}
M.~Reid.
\newblock Nonnormal del {P}ezzo surfaces.
\newblock {\em Publ. Res. Inst. Math. Sci.}, 30(5):695--727, 1994.

\bibitem{Siebert99}
B.~Siebert.
\newblock Algebraic and symplectic {G}romov-{W}itten invariants coincide.
\newblock {\em Ann. Inst. Fourier (Grenoble)}, 49(6):1743--1795, 1999.

\bibitem{FTZ20}
M.~F. Tehrani and A.~Zinger.
\newblock The refined symplectic sum formula for {G}romov-{W}itten invariants.
\newblock {\em Internat. J. Math.}, 31(4):2050032, 60, 2020.

\bibitem{FTZ21}
M.~F. Tehrani and A.~Zinger.
\newblock On the rim tori refinement of relative {G}romov-{W}itten invariants.
\newblock {\em Commun. Contemp. Math.}, 23(5):Paper No. 2050051, 50, 2021.

\bibitem{Wilson92}
P.~M.~H. Wilson.
\newblock The {K}\"{a}hler cone on {C}alabi-{Y}au threefolds.
\newblock {\em Invent. Math.}, 107(3):561--583, 1992.

\bibitem{Wilson97}
P.~M.~H. Wilson.
\newblock Symplectic deformations of {C}alabi-{Y}au threefolds.
\newblock {\em J. Differential Geom.}, 45(3):611--637, 1997.

\bibitem{Yang04}
Y.~Yang.
\newblock On integrals of {E}isenstein series and derivatives of {$L$}-series.
\newblock {\em Int. Math. Res. Not.}, 2004(6):303--307, 2004.

\bibitem{Zag08}
D.~Zagier.
\newblock Elliptic modular forms and their applications.
\newblock In {\em The 1-2-3 of modular forms}, Universitext, pages 1--103. Springer, Berlin, 2008.

\bibitem{Zag09}
D.~Zagier.
\newblock Integral solutions of {A}p\'ery-like recurrence equations.
\newblock In {\em Groups and symmetries}, volume~47 of {\em CRM Proc. Lecture Notes}, pages 349--366. Amer. Math. Soc., Providence, RI, 2009.

\bibitem{Zhou14}
J.~Zhou.
\newblock {\em Arithmetic {P}roperties of {M}oduli {S}paces and {T}opological {S}tring {P}artition {F}unctions of {S}ome {C}alabi-{Y}au {T}hreefolds}.
\newblock ProQuest LLC, Ann Arbor, MI, 2014.
\newblock Thesis (Ph.D.)--Harvard University.

\end{thebibliography}
\bibliographystyle{plain}

\end{document}